\documentclass[11pt]{article}

\usepackage[margin=1in]{geometry}
\usepackage{amsthm,amsmath,amsfonts,amssymb}
\usepackage[numbers,sort&compress]{natbib}
\usepackage[colorlinks,citecolor=blue,urlcolor=blue]{hyperref}
\usepackage{graphicx}

\usepackage{mathtools}
\usepackage{float}
\usepackage{array}
\usepackage{mathbbol}
\usepackage{comment}
\usepackage{xcolor}
\usepackage{listings}

\usepackage{etoc}
\newcommand{\ignore}[1]{}

\theoremstyle{plain}

\newtheorem{theorem}{Theorem}[section]
\newtheorem{lemma}[theorem]{Lemma}
\newtheorem{corollary}[theorem]{Corollary}

\theoremstyle{definition}

\title{Minimaxity and Admissibility of Bayesian Neural Networks}

\author{
Daniel Andrew Coulson\\
Department of Statistics and Data Science, Cornell University\\
\texttt{dac382@cornell.edu}
\and
Martin T. Wells\\
Department of Statistics and Data Science, Cornell University\\
\texttt{mtw1@cornell.edu}
}

\date{} % leave empty, or write \date{\today}

\begin{document}

\maketitle

\begin{abstract}
Bayesian neural networks (BNNs) offer a natural probabilistic formulation for inference in deep learning models. Despite their popularity, their optimality has received limited attention through the lens of statistical decision theory. In this paper, we study decision rules induced by deep, fully connected feedforward ReLU BNNs in the normal location model under quadratic loss. We show that, for fixed prior scales, the induced Bayes decision rule is not minimax. We then propose a hyperprior on the effective output variance of the BNN prior that yields a superharmonic square-root marginal density, establishing that the resulting decision rule is simultaneously admissible and minimax. We further extend these results from the quadratic loss setting to the predictive density estimation problem with Kullback--Leibler loss. Finally, we validate our theoretical findings numerically through simulation.
\end{abstract}


\section{Introduction}
Neural networks have grown rapidly in popularity over the past several years and have demonstrated strong performance across a wide range of tasks including image classification, time-series forecasting, and language modeling. Their success is largely attributable to their modeling flexibility as well as advances in computational hardware that have improved their tractability, such as the widespread adoption of graphics processing units. Bayesian neural networks (BNNs) extend standard neural networks by placing prior distributions over the weights, thereby enabling probabilistic modeling and uncertainty quantification, for example through credible intervals \cite{arbel2023primer, papamarkouposition}. Due to their flexibility and ability to represent uncertainty, BNNs are widely used in uncertainty-critical applications such as medicine, finance, and weather forecasting. For example, \cite{lisboa2003bayesian} uses a BNN for the prognosis in patients after breast surgery. Similarly, \cite{chandra2021bayesian} use BNNs to forecast the stock price before and during the COVID-19 pandemic. In \cite{marzban2001bayesian}, two BNNs are developed: one for predicting the size of severe hail and another for classifying the size of hail. 

A substantial body of work has examined the theoretical properties of BNNs. For example, \cite{neal2012bayesian} show that, in the limit, BNNs with infinitely many hidden units converge to a Gaussian process. Subsequent studies have explored this Gaussian process behavior in greater depth, including \cite{matthews2018gaussian}. \cite{gal2016dropout} shows that the application of dropout during both training and inference approximately corresponds to Bayesian inference in a deep Gaussian process. Other lines of work establish posterior concentration results. For example, \cite{polson2018posterior} proves such results for a spike-and-slab prior. Similarly, \cite{egels2025posterior} establish posterior contraction results for BNNs with heavy-tailed prior distributions and extend these results to a variational Bayes analog. Nevertheless, many theoretical studies rely on highly technical and unrealistic assumptions, which limit their applicability. For example, unlike much of the theoretical Bayesian deep learning literature, our analysis does not require any architectural scaling regime in which depth or width grows with sample size. The results hold for arbitrary fixed, finite architectures, which makes them directly relevant to the settings used in practice. 

Despite the abundance of theoretical work on BNNs, there has been little work from a statistical decision-theoretic perspective. Statistical decision theory provides a principled framework for choosing estimators under uncertainty. This perspective could help explain the strong empirical performance of BNNs across tasks and provide guidance for architectural choices, such as prior distributions and network depth. In this work, we study the performance of BNNs from a decision-theoretic perspective. 

In particular, we study their risk in the normal location model under quadratic loss. Although the normal location problem is somewhat simplistic, it already allows us to identify which BNN modeling choices yield well performing estimators in the minimax sense. The minimax optimality of estimators in the normal location problem has a rich literature, encompassing a variety of minimax criteria and proof techniques that highlight the impact of the prior distribution, the induced posterior distribution, and the resulting decision rule. Therefore, the normal location problem provides a magnifying glass that highlights which aspects of BNNs work well and where standard BNN priors could be improved, for example by introducing shrinkage priors. 

A central challenge in BNNs is the construction of informative proper priors on network weights that are both computationally tractable and favor solutions with desirable frequentist properties. Indeed, \cite{papamarkouposition} identifies prior specification as one of the foremost unresolved problems in Bayesian deep learning, emphasizing that the prior over network parameters induces function-space behavior that ultimately governs generalization. In this context, our contribution is to provide a sharp decision-theoretic account of this issue in a canonical setting. Specifically, we show that the Bayes rules induced by standard BNNs are not minimax for the normal location problem under quadratic loss, demonstrating that widely used Bayesian specifications can fail to satisfy this basic criterion of optimality. Crucially, this deficiency is not inherent to Bayesian neural modeling itself, but rather arises from the choice of hyperprior: with an appropriate hyperprior, the induced Bayes rule is both minimax and admissible. By extending these results to predictive density estimation, we further show that the consequences of prior design persist beyond point estimation and directly affect predictive performance. More broadly, these findings suggest that the future of Bayesian deep learning depends not only on making neural Bayesian procedures expressive, but also on ensuring that the priors they employ induce decision-theoretically sound rules. In this way, the paper addresses a central concern in contemporary Bayesian deep learning by replacing heuristic prior selection with a theoretically justified criterion for determining when neural procedures are, and are not, decision-theoretically sound. 

This perspective is especially timely given the emergence of Prior-Data Fitted Networks (PFNs), introduced by \cite{mullertransformers}, which train transformers to approximate Bayesian prediction over tasks sampled from a prior. Approaches such as TabPFN \cite{hollmanntabpfn} demonstrate that this paradigm can be exceptionally powerful in practice, with a transformer-based PFN achieving state-of-the-art performance on small tabular prediction problems. Precisely because PFNs learn to approximate the predictive distribution induced by a chosen prior, our results imply that prior specification is not merely a modeling convenience but the central determinant of whether the learned predictor is decision-theoretically well founded. In this sense, the rise of PFNs makes the present analysis especially consequential: as PFN-style methods become increasingly prominent, understanding when the underlying prior yields minimax and admissible rules becomes essential.

The paper is organized as follows. In Section 1, we introduce notation, a review of statistical decision-theoretic results for the normal location model used throughout, describe the general form of the prior density induced by a fixed-scale ReLU BNN, and derive a more convenient representation of this prior as a scale mixture of normals. In Section 2, we show that the square root of the marginal density induced by a fixed-scale ReLU BNN prior is not superharmonic. We then derive the corresponding decision rule and show that it is not minimax. In Section 3, we introduce a hyperprior on the scales of the BNN prior and show that the resulting prior induces a minimax decision rule. In Section 4, we extend these results to predictive density estimation under Kullback-Leibler loss, showing that the proposed hyperprior likewise induces a minimax rule in that setting. Finally, in Section 5, we illustrate our theoretical results through simulation, comparing the fixed-scale BNN prior and the proposed hierarchical BNN prior with decision rules induced by other popular priors, including a BNN with dropout and the horseshoe prior. Proof sketches are provided in the main text, with full proofs deferred to the Supplementary Material. 

\subsection{Notation}
Boldface letters denote vectors; uppercase Latin letters denote certain functions and matrices. The symbol $\|\cdot\|$ denotes the Euclidean norm and $\boldsymbol{x}$ is a fixed covariate. Finally, $\gtrsim$ and $\lesssim$ denote inequality up to a positive constant, and $a \asymp b$ indicates that $a$ and $b$ are bounded by each other up to positive multiplicative constants. 

\subsection{Introduction to decision theory and minimax optimality}

Decision theory studies the choice of estimators (also called decision rules) for estimating a quantity of interest. Let $D$ denote the class of all estimators of $\boldsymbol{\theta}$. For a given decision rule, we incur a loss depending on how far its output is from the true value of $\boldsymbol{\theta}$. To quantify this, we use a loss function $L(\boldsymbol{\theta}, \boldsymbol{\delta})$. Because the loss depends on the (random) data, we consider the risk function 
\begin{equation*}
    R(\boldsymbol{\theta}, \boldsymbol{\delta}) = \mathbb{E}_{\boldsymbol{\theta}}[L(\boldsymbol{\theta},\boldsymbol{\delta})].
\end{equation*}
This raises the question: which decision rule should we use? There are many ways to answer this question, but a common criterion is to choose a decision rule that achieves minimax risk. That is, it minimizes the maximum risk over $\boldsymbol{\theta}$. Formally, a decision rule $\boldsymbol{\delta}_{*}$ is minimax if \begin{equation*}
    \sup\limits_{\boldsymbol{\theta} \in \Theta} R(\boldsymbol{\theta}, \boldsymbol{\delta}_{*}) = \inf\limits_{\boldsymbol{\delta} \in D} \sup\limits_{\boldsymbol{\theta} \in \Theta} R(\boldsymbol{\theta}, \boldsymbol{\delta}). 
\end{equation*} 

\subsection{Decision Problem }

We study the performance of BNNs in the normal location model with $\boldsymbol{\theta} \in \mathbb{R}^{p}$. In particular, the normal location model is 
\begin{equation*}
    \boldsymbol{Y}\mid\boldsymbol{\theta}\sim N_{p}(\boldsymbol{\theta}, I_{p}). 
\end{equation*}
This is the classical normal location model, with the prior distribution specified in Section~\ref{sec:1.4}. We study the estimation of the mean vector under quadratic loss.  For a decision rule $\boldsymbol{\delta}$, we measure loss by \begin{equation*}
    L(\boldsymbol{\theta}, \boldsymbol{\delta}) = ||\boldsymbol{\theta} - \boldsymbol{\delta}||^{2}.  
\end{equation*}
Quadratic loss is the canonical choice for studying estimator performance in the normal location model. Many decision-theoretic results and minimax criteria are explicitly stated for quadratic loss, and we exploit these results in this work. It is known that non-trivial improper Bayes minimax estimators exist in dimension $p \geq 3$ (e.g. \cite{fourdrinier1998construction}), and that proper Bayes minimax estimators exist in dimension $p\geq 5$ (see \cite{strawderman1971proper}). In Section 2, we consider $p \geq 3$ since that section focuses on proving non-minimaxity and the distinction between $p \geq 3$ and $p \geq 5$ does not arise. For the remainder of the paper, we assume $p \geq 5$. It is known that the minimax risk in the normal location model equals $p$. Moreover, a Bayes estimator (minimizing posterior risk) corresponding to a prior $\pi(\boldsymbol{\theta})$ has the form \begin{equation*}
    \boldsymbol{\delta}_{\pi} (\boldsymbol{Y}) = \boldsymbol{Y} + \frac{\nabla m(\boldsymbol{Y})}{m(\boldsymbol{Y})}, 
\end{equation*} where $m(\boldsymbol{Y})$ denotes the marginal density (see \cite{brown1971admissible} or \cite{fourdrinier2018shrinkage}). Based on this, we use two results to establish minimaxity and non-minimaxity. To show non-minimaxity, we use Stein's unbiased risk estimator (SURE, see \cite{stein1981estimation}), which provides an unbiased estimate of the risk in the normal location model. The approach assumes that the decision rule can be written in Baranchik \cite{baranchik1970family} form as $\delta(\boldsymbol{y}) = \boldsymbol{y} + g(\boldsymbol{y})$, where $g(\cdot)$ is weakly differentiable. Then the unbiased risk estimate is \begin{equation*}
    \text{SURE}(\delta) = p + ||g(\boldsymbol{Y})||^{2}+ 2\operatorname{div} g
\end{equation*} where $\operatorname{div}$ denotes the divergence operator. Taking expectations then yields the risk of any estimator of this form. As illustrated in the proof of Theorem 2.6, such arguments can be lengthy. However, a sufficient condition (see \cite{stein1981estimation} or Theorem 3.1 in \cite{fourdrinier2018shrinkage}) provides a quicker route to prove minimaxity. For a Bayes estimator of the form above, suppose that 
\begin{equation*}
    \mathbb{E}_{\boldsymbol{\theta}}[||\frac{\nabla m(\boldsymbol{Y})}{m(\boldsymbol{Y})}||^{2}]< \infty. 
\end{equation*}
Then the decision rule is minimax provided $\Delta \sqrt{m(\boldsymbol{y})} \leq 0$. That is, under this regularity condition, if $\sqrt{m(\boldsymbol{y})}$ is superharmonic, then the associated decision rule is minimax. This implication follows from the SURE representation above. The minimax risk in the normal location model is $p$. Therefore, by SURE, any estimator with risk strictly greater than $p$ cannot be minimax. Substituting the Bayes form above into the SURE representation, the excess risk over $p$ can be expressed in terms of $\Delta \sqrt{m(\boldsymbol{y})}$. In particular, the estimator is minimax if \begin{equation*}
    \Delta\sqrt{m(\boldsymbol{y})} \leq 0, \forall \boldsymbol{y}, \text{ such that } m(\boldsymbol{y})>0. 
\end{equation*}

\subsection{Prior}
\label{sec:1.4}
In this work, we use the probability distribution induced by a deep ReLU BNN as the prior distribution for $\boldsymbol{\theta}$. To define this neural network, we adopt the notation of \cite{zavatone2021exact}, where the neural network is defined recursively. Consider a neural network $f: \mathbb{R}^{n_{0}} \rightarrow \mathbb{R}^{p}$ with $d$ layers and $n_{i}$ denoting the width of layer $i \in \{1, \dots, d \}$. Let the input to the network be a fixed covariate $\boldsymbol{x}$, the layer outputs be $\boldsymbol{h}_{\ell}$, the weight matrices be $W_{\ell}$, and activation functions $\phi_{\ell}$ for $\ell \in \{0, 1,\dots, d \} $
\begin{align*}
    \boldsymbol{h}_{0} &= \boldsymbol{x},\\ 
    \boldsymbol{h}_{\ell} &= W_{\ell}\phi_{\ell-1}(\boldsymbol{h}_{\ell -1}), \ell = 1,\dots, d, \\ 
    [W_{\ell}]_{ij} & \overset{\text{i.i.d.}}{\sim} N(0, \sigma_{\ell}^{2}).  
\end{align*} Note that $\phi_{\ell}(\cdot) = \text{ReLU}(\cdot)$ for $\ell = 1, \dots, d-1$, and the activations functions $(\phi_{0}, \phi_{d})$ are identity maps, so that $\boldsymbol{h}_{d} = \boldsymbol{\theta}$ with $n_{d} = p$. Furthermore, the scales $\sigma_{\ell}$ are fixed. 
The prior density of a depth $d$ Bayesian ReLU neural network, as derived in \cite{zavatone2021exact} is given by 
\begin{align} \label{bnnprior}
   &p_{d}(\boldsymbol{h}_{d}; \sigma_{1}\dots\sigma_{d}||\boldsymbol{x}||; n_{1}, \dots, n_{d} ) = (1- \frac{(2^{n_{1}}-1)\dots(2^{n_{d-1}}-1)}{2^{n_{1} + \dots n_{d-1}}})\mathbf{1}\{\boldsymbol{h}_{d}=0\}\\ 
   &+ \frac{1}{2^{n_{1}+\dots+n_{d-1}}}\sum\limits_{k_{1}=1}^{n_{1}}\dots\sum\limits_{k_{d-1} =1}^{n_{d-1}}\begin{pmatrix}
n_{1} \nonumber \\
k_{1}
\end{pmatrix} \dots\begin{pmatrix}
n_{d-1}  \\
k_{d-1}
\end{pmatrix} \frac{\prod\limits_{\ell = 1}^{d-1}\frac{1}{\Gamma(\frac{k_{\ell}}{2})}}{(2^{d}\pi\sigma_{1}^{2}\dots\sigma_{d}^{2}||\boldsymbol{x}||^{2})^{\frac{n_{d}}{2}}} \nonumber \\ &\times G_{0,d}^{d,0}(\frac{||\boldsymbol{h}_{d}||^{2}}{2^{d}\sigma_{1}^{2}\dots\sigma_{d}^{2}||\boldsymbol{x}||^{2}}| \begin{matrix}
-\\
0, \frac{k_{1}-n_{d}}{2}, \dots, \frac{k_{d-1}-n_{d}}{2}
\end{matrix} ). \nonumber
\end{align}
Note that $\mathbf{1}\{\boldsymbol{h}_{d}=0\}$ denotes the Dirac measure at $\boldsymbol{0}_{p}$ and that $G_{m,n}^{p,q}$ denotes the Meijer-G function. The Meijer-G function arises naturally when working with products of independent mean-zero normal random variables, as in a standard BNN. For example, the product of two such variables has a density expressible in terms of a modified Bessel function of the second kind, which is itself a special case of the Meijer-G function. More generally, the density of a product of arbitrarily many independent mean-zero normal random variables can be written in terms of a Meijer-G function.  For our results, we will only use the continuous part of the prior, that is, we do not place any prior mass at $\boldsymbol{h}_{d} = \boldsymbol{0}_{p}$. This corresponds to a neural network prior that does not output the zero vector. 

However, we derive a more convenient form of the prior distribution, which reveals that it is a scale mixture of normal distributions. 
\begin{lemma}
The prior density of a depth $d$ Bayesian ReLU neural network given in (\ref{bnnprior}) can represented as
\begin{align*}
    \pi(\boldsymbol{h}_{d}) &= \int\limits_{0}^{\infty} \phi_{p}(\boldsymbol{h}_{d}, \boldsymbol{0}_{p}, vI_{p})g(v) dv,
    \end{align*}
where $g(v) = \sum\limits_{k} w_{k} g_{k}(v), w_{k} = \frac{1}{2^{n_{1}+\dots+n_{d-1}}} \prod\limits_{\ell = 1}^{d-1} \binom{n_{\ell}}{k_{\ell}}, k_{\ell} \in \{ 1, \dots, n_{\ell} \} $ and $g_{k}(v)$
is the density function of the random variable $V_{\boldsymbol{k}} = 2^{d-1} ||\boldsymbol{x}||^{2}(\prod\limits_{\ell = 1}^{d} \sigma_{\ell}^{2})(\prod\limits_{\ell = 1}^{d-1}T_{\ell})$ for  $T_{\ell} \sim \Gamma(\frac{k_{\ell}}{2},1).$
 
\end{lemma}
We provide a full proof of Lemma 1.1 in the Supplementary Material. 

\section{Deep Bayesian ReLU network with fixed scales is not minimax} 

In this section, we investigate the minimax optimality of Bayes decision rules induced by standard deep ReLU BNNs. We now use the generic notation $\boldsymbol{\theta}$ for the parameter of interest instead of the layer-wise $\boldsymbol{h}_{d}$  notation in Section 1.4. We first study the superharmonicity of the induced marginal density by establishing a stretched exponential upper bound, then apply Stein's Unbiased Risk Estimate (SURE) to analyze the risk of the induced Bayes decision rule, especially in the regime where  $||\boldsymbol{\theta}||$ is large. 

To prove that the square root of the induced marginal density is not superharmonic, we first establish a key property of a certain class of radial functions.  We then show that the marginal density induced by a fixed scale ReLU BNN admits a stretched exponential upper bound. Using a proof by contradiction, we then show that induced marginal density does not satisfy this key property that holds for analogous superharmonic functions. 

A radial function on Euclidean space is a function whose value at a point depends only on its distance from a fixed center. It is well known that radial functions satisfy the following differential equation. 

\begin{lemma}
    Let $p \geq 2$ and $u(\boldsymbol{x}) = \phi(r), r= ||\boldsymbol{x}||, \boldsymbol{x} \in \mathbb{R}^{p}$ and $\phi \in C^{2}(0,\infty)$. Then for $r>0$ \begin{equation*}
        \Delta u(\boldsymbol{x}) = \phi''(r) + \frac{p-1}{r}\phi'(r). 
    \end{equation*}
\end{lemma}

This is a well-established result, and we provide a proof in the Supplementary Material. This shows that the Laplacian depends only on how $\phi(r)$ varies with radius. The second derivative captures the radial curvature. In $\mathbb{R}^{p}$ there are $p-1$ linearly independent tangential directions along the sphere at a given point. The added first derivative term captures the contributions of these tangential directions to the divergence of the gradient. We use Lemma 2.1 in the proof of Lemma 2.2.

\begin{lemma}
    Let $q:[0,\infty) \rightarrow (0, \infty)$ be $C^{2}$ on $(0,\infty)$ and radial in $\mathbb{R}^{p}$ with $p \geq 3$. Suppose \begin{align*}
        &\text{(a) } q(r) \rightarrow 0 \text { as } r\rightarrow
\infty\\
& \text{(b) } \Delta q(r) = q''(r) + \frac{p-1}{r}q'(r) \leq 0,  \forall r \geq R_{0}, \text{ for some } R_{0}>0. 
    \end{align*}Then, $\exists R \geq R_{0}$ for every $r \geq R, q(r) \geq cr^{2-p}, c:= -\frac{R^{p-1}q'(R)}{p-2}, q'(R) <0$. 
\end{lemma}

We provide a full proof of this result in the Supplementary Material. The proof relies primarily on the decay assumption, which allows us to show the existence of $R_{0}$ such that the result holds. 

This lemma shows that if a positive radial function in $C^{2}$ is superharmonic outside the ball centered at $0$ with radius $R_{0}$ and tends to zero at infinity, then it cannot decay arbitrarily fast: its tails are bounded below, up to a constant by $r^{2-p}.$ We use this to show that the induced marginal density of the fixed scale BNN is not superharmonic, since its tails decay faster. Shrinkage priors often perform well for estimation. Two key aspects of a shrinkage prior are that it shrinks estimates towards a target value (such as zero) and assigns a sufficient amount of probability mass to large signals. Thus, exponential tails assign too little probability mass to signals far from zero, leading to poor estimation due to over shrinkage. As discussed above, improper Bayes minimax estimators require $p \geq 3$. Lemma 2.2 makes this condition natural. This is because one can show that \begin{equation*}
    q(r) \geq -q'(R)R^{p-1} \int\limits_{r}^{\infty}s^{1-p}ds,
\end{equation*}
and the integral is finite only if $p\geq 3$. We now establish the following result. 

\begin{lemma}
    The marginal (prior predictive) density $m(\boldsymbol{y})$ induced by the fixed scale BNN prior in (\ref{bnnprior}) satisfies $m(\boldsymbol{y}) \leq C\exp \{ - \kappa ||\boldsymbol{y}||^{\frac{2}{d}} \}$, for some constants $C, \kappa >0$.  
\end{lemma}

We provide a full proof of Lemma 2.3 in the Supplementary Material. The proof starts from the Gaussian mixture representation of the marginal density and then splits the integral into the regions $||\boldsymbol{\theta}|| \leq \frac{r}{2}$ and $||\boldsymbol{\theta}|| > \frac{r}{2}$, where $r = ||\boldsymbol{y}||.$ The first integral is bounded by a Gaussian tail via the reverse triangle inequality, and the second part is bounded by the prior tail probability $\mathbb{P}_{\pi} (||\boldsymbol{\theta}|| > \frac{r}{2}).$ Using a known Meijer-G asymptotic (see Appendix A of \cite{gaunt2025variance}), the prior tail is stretched exponential of order $\exp \{-cr^{\frac{2}{d}} \}$; combining the two bounds yields the desired upper bound. 

This lemma characterizes the prior predictive tail behavior of the fixed scale BNN; in particular, the marginal density is light tailed, with at most stretched exponential decay. Intuitively, this means that extreme observations are exponentially unlikely on the scale $r^{\frac{2}{d}}$ and the prior predictive places most of its mass on regions of $\mathbb{R}^{p}$ with moderate radius. The exponent $\frac{2}{d}$ highlights a result from \cite{zavatone2021exact}: increasing depth corresponds to heavier tails, but fixed d still yields stretched exponential decay which is typically too fast for minimaxity. Having established the above results, we can then prove Theorem 2.4. 

\begin{theorem}
    The square root of the marginal density induced by a BNN with fixed scale prior in (\ref{bnnprior}) is not superharmonic in $\mathbb{R}^{p} \text{ for } p \geq 3$.
\end{theorem}

We provide a full proof of Theorem 2.4 in the Supplementary Material. The proof proceeds by contradiction: assuming that $q=\sqrt{m(\boldsymbol{y})}$ is superharmonic, Lemma 2.2 then forces a polynomial lower tail bound $q(r) \gtrsim r^{2-p}$, while Lemma 2.3 implies a stretched-exponential upper tail bound $q(r) \lesssim \exp \{ - cr^{\frac{2}{d}} \}$. The stretched exponential decay dominates the polynomial bound, yielding a contradiction. Thus, superharmonicity is incompatible with the tail behavior of the fixed-scale BNN prior predictive density.

The obstruction arises from a fundamental barrier: positive radial superharmonic functions vanishing at infinity cannot decay faster than $r^{2-p}$. Intuitively, the square root of the fixed-scale BNN marginal density is too light-tailed to be superharmonic in dimension $p \geq 3$. We note that replacing ReLU activations with identity maps yields a linear network; nevertheless, the resulting network still induces a Meijer-G type density and therefore does not admit a superharmonic square root marginal density.

Minimaxity provides a uniform bound on worst-case risk, whereas superharmonicity is a pointwise criterion on the square root of the marginal density. When superharmonicity fails, it indicates only that the local quantity entering the density becomes unfavorable in some region of the sample space, which would tend to inflate risk in that region. However, since risk averages over all data values, a region where the criterion is violated may carry negligible probability mass and thus need not by itself preclude minimaxity. Nevertheless, the failure of superharmonicity is a warning sign, and a separate argument is required to conclude that the induced decision rule is not minimax.

\begin{theorem}
    The induced Bayes decision rule of the fixed scale BNN is given by 
    \begin{equation*}
    \delta_{\text{BNN, fixed}}(\boldsymbol{y}) = \mathbb{E}[\frac{V}{V+1}|\;||\boldsymbol{Y}||^{2} = ||\boldsymbol{y}||^{2}] \; \boldsymbol{y} \text{ where } V \text{ is as defined in Lemma 1.1}. 
\end{equation*}
\end{theorem}
The full proof of Theorem 2.5 is provided in the Supplementary Material and relies on the representation of the prior distribution as a normal scale mixture and recognizing 
\begin{equation*}
    \boldsymbol{\theta} | \boldsymbol{Y} = \boldsymbol{y}, V= v \sim N_{p}(\frac{v}{1+v} \boldsymbol{y}, \frac{v}{1+v}I_{p}). 
\end{equation*} 
Using the law of total expectation gives the desired form of the decision rule. By using Bayes' theorem it can be shown that $p(v|\boldsymbol{y})$ only depends on $\boldsymbol{y}$ through $||\boldsymbol{y}||^{2}$. We then prove the following theorem. 

\begin{theorem}
    $\delta_{BNN, fixed}(\boldsymbol{y})$ is not minimax. 
\end{theorem}

Note that we can write the estimator in Shrinkage form. In particular, \begin{equation*}
    \delta(\boldsymbol{y}) = a(||\boldsymbol{y}||)\boldsymbol{y} \text{ with } a(||\boldsymbol{y}||) = \mathbb{E}[\frac{V}{V+1}|\;||\boldsymbol{Y}||^{2} = ||\boldsymbol{y}||^{2}]. 
\end{equation*}
From this expression, the estimator has a natural shrinkage form, since it shrinks toward $0$. However, minimaxity depends on the shrinkage profile uniformly over $||\boldsymbol{Y}||$. As we show in the proof, when $||\boldsymbol{\theta}||$ is large, the risk exceeds the minimax level. This is because the fixed scales induce a prior predictive density that is too light tailed, leading to shrinkage that is insufficiently adaptive for large signals. This means fixed scale BNN priors can yield procedures with suboptimal worst case performance, despite behaving well on typical data sets. Therefore, in Section 3 we introduce scale mixtures to recover minimax guarantees and achieve minimax risk (or less) for all $||\boldsymbol{\theta}|| >0. $ We provide a full proof of Theorem 2.6 in the Supplementary Material. We first rewrite the fixed scale BNN  Bayes rule in the Baranchik form: 
\begin{equation*}
    \delta_{\text{BNN,fixed}}(\boldsymbol{y}) = (1-\psi(||\boldsymbol{y}
    ||^{2}))\boldsymbol{y} \text{ with } \psi(u):= \mathbb{E}[\frac{1}{1+V}|\;||\boldsymbol{Y}||^{2} = u]\in (0,1). 
\end{equation*} 
Thus, the procedure is completely characterized by the scalar shrinkage function $\psi$. We then derive a convenient representation of $\psi(u) = N(u)/D(u)$ where $N(u)$ and $D(u)$ are defined as the expectations of relevant functions of $u$ with respect to the measure $\pi(dv)$. This allows us to show that $\psi$ is differentiable, with \begin{equation*}
    \psi'(u) =  - \frac{1}{2}Var(\frac{1}{1+V}|||\boldsymbol{Y}||^{2} = u) \leq 0. 
\end{equation*} 
This implies $0 \geq \psi'(u) \geq - \frac{1}{2}\psi(u)$. These regularity and monotonicity conditions imply that
\begin{equation*}
    g(\boldsymbol{y}) := \delta_{\text{BNN,fixed}}(\boldsymbol{y}) - \boldsymbol{y} = - \psi(||\boldsymbol{y}||^{2})\boldsymbol{y}
\end{equation*}
is weakly differentiable, allowing us to apply Stein's unbiased risk estimator to analyze risk and establish non-minimaxity. By Stein's identity, the risk can be written as $R(\boldsymbol{\theta}, \delta) = p + \mathbb{E}_{\boldsymbol{\theta}}[B(U)]$ where  $U = ||\boldsymbol{Y}||^{2}$ and $B(u) := \psi(u)^{2}u - 2p\psi(u) - 4u\psi'(u)$.  Since $\psi'(u) \leq 0$, we have $B(u)>0$ whenever $u\psi(u) > 2p$. 

Therefore, to prove non-minimaxity, it suffices to show that $u\psi(u) \rightarrow \infty$ as $u \rightarrow \infty$ and then consider a region where $||\boldsymbol{\theta}||$  is large and $U = ||\boldsymbol{Y}||^{2}$ concentrates on values of $u$ for which $u\psi(u) >2p$. To analyze $\psi(u)$, we rewrite the posterior $\pi(t|U=u, \boldsymbol{K} = \boldsymbol{k})$ of the latent Gamma random vector $T = (T_{1}, \dots, T_{d-1})$, where $\boldsymbol{k} = (k_{1}, \dots, k_{d-1})$, using Bayes' theorem. For the next section of the proof we work with realizations of random variables denoted by lower case letters. We then perform a change of variables that separates the scale of the product $\prod\limits_{\ell = 1}^{d-1} t_{\ell}$ from its shape. Conditional on $\boldsymbol{K} = \boldsymbol{k}$, write $x:= (\prod\limits_{\ell=1}^{d-1}t_{\ell})^{\frac{1}{d-1}}$, so that $v= Cx^{d-1},C := 2^{d-1}||\boldsymbol{x}||^{2}(\prod\limits_{\ell = 1}^{d}\sigma_{\ell}^{2})$ and parameterize $t_{\ell} = xs_{\ell}$. The Jacobian contributes a factor of $(d-1)x^{d-2}$, and the posterior kernel factorizes into an $x-$dependent part and an $s-$dependent part. Integrating over $\boldsymbol{s}$ yields 
\begin{align*}
    \pi(x|u, \boldsymbol{k}) &\propto (1+Cx^{d-1})^{-\frac{p}{2}} \exp \{ - \frac{u}{2(1+Cx^{d-1})} \} x^{\alpha_{\cdot}-1}
H(x),\text{where }\\ 
H(x) &:= \int\exp \{-xA(\boldsymbol{s}) \} Q(\boldsymbol{s})d\boldsymbol{s},\\ \alpha_{\cdot} &=\sum\limits_{i=1}^{d-1}\frac{k_{\ell}}{2}, A(\boldsymbol{s}) = \sum\limits_{\ell =1}^{d-1}s_{\ell}, \text{ and } Q(\boldsymbol{s}) = \prod\limits_{\ell =1}^{d-2}s_{\ell}^{\alpha_{\ell}-\alpha_{d-1}-1}.\end{align*} 
We then deduce bounds on $H(x)$. By the Arithmetic Mean–Geometric Mean Inequality we show that $A(\boldsymbol{s}) \geq d-1$ with equality only at the unique minimizer $\boldsymbol{s}^{*} = \boldsymbol{1}_{d-2}$. It can be shown that $A$ is strictly convex and admits a uniform quadratic expansion around $\boldsymbol{s}^{*}$. Specifically,  for some constants $C_{-}$ and $C_{+}$,\begin{equation*}
    (d-1)+ C_{-}||\boldsymbol{s}- \boldsymbol{1}||^{2} \leq A(\boldsymbol{s}) \leq (d-1) + C_{+}||\boldsymbol{s}-\boldsymbol{1}||^{2}. 
\end{equation*} On a small neighborhood of $\boldsymbol{s}^{*}, Q(\boldsymbol{s})$ is bounded above and below. Therefore, the dominant contribution to $H(x)$ comes from this neighborhood. Rescaling via $\boldsymbol{z} = \sqrt{\boldsymbol{x}}(\boldsymbol{s} - \boldsymbol{s}^{*})$  yields a Gaussian integral, producing the polynomial prefactor $x^{-\frac{d-2}{2}}$ and the exponential prefactor $e^{-(d-1)x}$ from the bound on $A(\boldsymbol{s})$. We finally split the integral defining $H$ over a  small ball around $\boldsymbol{s}^{*}$ and its complement, which yields 
\begin{equation*}
    H(x) \asymp x^{-\frac{d-2}{2}}e^{-(d-1)x}, x\rightarrow \infty. 
\end{equation*} 
Using this asymptotic, we obtain 
\begin{align*}
    \pi(x|u, \boldsymbol{k}) & \asymp x^{\alpha_{\cdot}-1-(\frac{d-2}{2})} (1+Cx^{d-1})^{- \frac{p}{2}}\exp \{ - f_{u}(x) \}\; \text{ and } \\   
    f_{u}(x) &:= (d-1) x+\frac{u}{2(1+Cx^{d-1})}.
\end{align*} 
Therefore, the $u$-dependence is only through the function $f_{u}$. Let $m := d-1$ and $x_{u}:= (u/2C)^{\frac{1}{m+1}}$ and write $x = \lambda x_{u}$. Then $f_{u}(\lambda x_{u}) = x_{u}\phi_{u}(\lambda)$, where $\phi_{u}(\lambda) \rightarrow \phi_{\infty}(\lambda) = m\lambda + \lambda^{-m}$ uniformly on compact $\lambda$ intervals as $u \rightarrow \infty$. Since $\phi_{\infty}$ has a unique minimizer at $\lambda = 1$ and is uniformly separated from its minimum outside any fixed neighborhood of $1$, the same separation holds for $\phi_{u}$ for all large $u$. Consequently, for every $\epsilon \in (0,1)$ there exist $u_{0}$ and $c(\epsilon)>0$ such that, for all $u\geq u_{0}$,
\begin{equation*}
    f_{u}(x) \geq f_{u}(x_{u}) + c(\epsilon)x_{u} \text{ whenever } x\notin [(1-\epsilon)x_{u}, (1+\epsilon)x_{u}]. 
\end{equation*}  To control the conditional distribution $\Pi_{u, \boldsymbol{k}}(\cdot) = \mathbb{P}(X\in \cdot|U=u, \boldsymbol{K} = \boldsymbol{k})$, we first show that the small $x$ region carries negligible probability mass. Fix $x_{0}>0$ and note that, on $(0,x_{0})$ the factor $\exp\{ - u/ 2(1+Cx^{m}) \}$ is uniformly bounded by $\exp\{ -c_{0}u \}$ for $c_{0}:= (2(1+Cx_{0}^{m}))^{-1}$. This yields the exponential upper bound 
\begin{equation*}
    \int\limits_{0}^{x_{0}} \pi_{u, \boldsymbol{k}}(x) dx \leq K_{0}e^{-c_{0}u}.
\end{equation*}
This follows after integrating out $H(x)$ using the normalization identity 
\begin{equation*}
    \int Q(\boldsymbol{s}) A(\boldsymbol{s})^{-\alpha_{\cdot}} d\boldsymbol{s} = \frac{\prod\limits_{i=1}^{m}\Gamma(\alpha_{i})}{m\Gamma(\alpha_{\cdot})}. 
\end{equation*} 
On the other hand, by lower bounding the remaining $x-$ dependent factors on a fixed compact interval $I_{\delta} \subset [x_{0}, \infty),$ we obtain $\int\limits_{0}^{\infty} \pi_{u, \boldsymbol{k}}(x) dx \geq K_{1}e^{-c_{1}u}$ for some $c_{1} < c_{0}$. This implies $\Pi_{u, \boldsymbol{k}}((0, x_{0})) \leq (K_{0}/K_{1}) \exp \{ -(c_{0}-c_{1})u \}$. We rescale $x  = \lambda x_{u}$ using $x_{u} = (u/2C)^{\frac{1}{m+1}}$ and write the $\lambda-$ density as $\tilde{\pi}_{u, \boldsymbol{k}}(\lambda) := x_{u}\pi_{u, \boldsymbol{k}}(\lambda x_{u})$ which satisfies the two sided bound
\begin{equation*}
    K_{1}x_{u}^{\gamma + 1} q_{u}(\lambda) \leq \tilde{\pi}_{u, \boldsymbol{k}}(\lambda) \leq K_{2}x_{u}^{\gamma + 1} q_{u}(\lambda) \text{ with } q_{u}(\lambda) = \lambda^{\gamma} (1+A_{u}\lambda^{m})^{-\frac{p}{2}} \exp \{ - x_{u}\phi_{u}(\lambda) \}. 
\end{equation*}
Having isolated the dominant bound, we can establish the desired posterior concentration result. It follows that, for any measurable $S\subseteq [x_{0}/x_{u}, \infty)$, \begin{equation*}
    \Pi_{u, \boldsymbol{k}}(S|X\geq x_{0}) = \frac{\int\limits_{S}\tilde{\pi}_{u, \boldsymbol{k}}(\lambda) d\lambda}{\int\limits_{\frac{x_{0}}{x_{u}}}^{\infty} \tilde{\pi}_{u, \boldsymbol{k}}(\lambda)d\lambda} \leq \frac{K_{2}}{K_{1}} \frac{\int\limits_{S}q_{u}(\lambda)d\lambda}{\int\limits_{\frac{x_{0}}{x_{u}}}^{\infty} q_{u}(\lambda) d\lambda}. 
\end{equation*} 
The denominator is bounded below by restricting to a shrinking neighborhood of the unique minimizer $\lambda_{u}^{*} \rightarrow 1$, which gives 
\begin{equation*}
    \int\limits_{\frac{x_{0}}{x_{u}}}^{\infty} q_{u}(\lambda) d\lambda \geq c_{in}x_{u}^{-\frac{1}{2}} (1+A_{u})^{-\frac{p}{2}}
e^{-x_{u}\phi_{u}(\lambda_{u}^{*})}, c_{in}>0.\end{equation*} 
For the numerator, we split the integral into three fixed parts representing small, moderate, and large values of $\lambda$. On each part, we show that $\phi_{u}(\lambda) \geq \phi_{u}(\lambda_{u}^{*}) + C(\epsilon)$, so the numerator is at most $K_{\epsilon} e^{-x_{u}(\phi_{u}(\lambda_{u}^{*}) + C(\epsilon))}$ up to polynomial factors. Taking the ratio leaves a polynomial factor times $e^{-C(\epsilon)x_{u}}$, and for large u the exponential dominates, yielding
\begin{equation*}
    \Pi_{u, \boldsymbol{k}}(|\frac{X}{x_{u}}-1| \geq \epsilon |X \geq x_{0}) \leq e^{-C(\epsilon)x_{u}}. 
\end{equation*} 
An unconditional bound then follows:
\begin{equation*}
    \Pi_{u, \boldsymbol{k}}(|\frac{X}{x_{u}}-1| \geq \epsilon) \leq e^{-C(\epsilon)x_{u}} \text{ for large}\; u. 
\end{equation*}
After some algebra, this gives, for sufficiently large $u$,
\begin{equation*}
    \mathbb{P}(V \asymp u^{1-\frac{1}{d}}|U=u, \boldsymbol{K} = \boldsymbol{k}) \geq \frac{1}{2}. 
\end{equation*} 
Letting $\mathcal{A}_{u} := \{ c_{1}u^{1-\frac{1}{d}} \leq V \leq c_{2} u^{1-\frac{1}{d}} \}$, for some constants $c_{1},c_{2}>0$ it then follows that,  
\begin{equation*}
    \psi(u) = \mathbb{E}[\frac{1}{1+V}|U=u] \geq \mathbb{E}[\frac{1}{1+V}1_{\mathcal{A}_{u}}|U=u] \geq \frac{1}{2(1+c_{2}u^{1-\frac{1}{d}})},
\end{equation*} 
for sufficiently large $u$. From this it follows that $u\psi(u) \geq (4c_{2})^{-1} u^{\frac{1}{d}} \rightarrow \infty$ as $u \rightarrow \infty$.  Therefore, we can choose $u_{0}$ such that $u\psi(u) \geq 4p$ for all $u \geq u_{0}.$ Combined with the inequality $B(u) \geq \psi(u) (u\psi(u)-2p)$, this yields a uniform lower bound $B(u) \geq (p/2c_{2})u^{\frac{1}{d}-1}$ on $[u_{0}, \infty)$. We therefore restrict the expectation to the event $A_{\boldsymbol{\theta}} = \{ u_{0} \leq U \leq b_{\boldsymbol{\theta}} \}$, where $B(u)$ is controlled. Taking $b_{\theta} = 2.25||\boldsymbol{\theta}||^{2}$ and using the reverse triangle inequality shows that $A_{\boldsymbol{\theta}}$ contains $E_{\boldsymbol{\theta}} = \{ ||\boldsymbol{Z}|| \leq ||\boldsymbol{\theta}|| /2 \}$ whenever $||\boldsymbol{\theta}||$ is sufficiently large with $\boldsymbol{Z}\sim N_{p}(\boldsymbol{0}_{p}, I_{p})$. Therefore, $\mathbb{P}_{\boldsymbol{\theta}}(A_{\boldsymbol{\theta}})$ is bounded below by a positive constant. Using Markov's inequality for $S = ||\boldsymbol{Z}||^{2} \sim \chi_{p}^{2}$ gives a lower bound for $\mathbb{P}(E_{\boldsymbol{\theta}})$ for sufficiently large $||\boldsymbol{\theta}||. $ This implies that $\mathbb{E}_{\boldsymbol{\theta}}[B(U)] >0$ for large $||\boldsymbol{\theta}||$ and hence $R(\boldsymbol{\theta}, \boldsymbol{\delta}) > p$ eventually, which shows the non-minimaxity. 

Theorem 2.6 highlights that, despite the attractive radial form and monotone shrinkage profile, minimaxity is dependent on how $\psi(u)$ behaves uniformly over the full range of $||\boldsymbol{y}||^{2}$. In particular, for large $||\boldsymbol{\theta}||^{2}, ||\boldsymbol{Y}||^{2}$ is typically of the same order as $||\boldsymbol{\theta}||^{2}$, and the growth rate $u\psi(u) \gtrsim u^{\frac{1}{d}}$ indicates that increasing depth does not eliminate non-minimaxity but pushes its appearance to progressively larger signal norms. While this result is specific to the normal location model, it offers a mathematical heuristic for why deeper neural networks may appear to perform better in typical regimes. It is worth noting that, unlike much of the Bayesian deep learning literature, we do not impose any growth conditions on depth as a function of sample size. Rather, our results hold for any fixed, finite depth of the kind used in practice. 
 
It is important to note that width plays a different role from depth in our results. For fixed depth, increasing width changes the polynomial prefactors in the posterior concentration argument, and hence modifies the constants appearing in the lower bound for $\psi(u)$, but it does not change the governing power law $u\psi(u) \gtrsim u^{\frac{1}{d}}$. In this sense, greater width does not alter the asymptotic mechanism behind non-minimaxity: it can shift the threshold at which the excess risk becomes visible, but it does not remove the eventual risk inflation for sufficiently large signal norms. Thus, whereas depth affects the exponent controlling how slowly the shrinkage vanishes, width only affects lower order terms. This suggests that wider networks may improve behavior in moderate signal regimes by changing finite sample constants, but they do not change the fundamental large signal asymptotics established here. As with depth, this conclusion holds for arbitrary fixed, finite widths of the kind used in practice. 
 
The proof shows that the marginal density is too light tailed, which yields a shrinkage function that makes the excess risk integrand positive on a non-negligible set where $||\boldsymbol{Y}||^{2}$ concentrates when $||\boldsymbol{\theta}||$ is large. That is, shrinkage does not diminish sufficiently rapidly for large signals, causing excess shrinkage that drives risk above the minimax level, despite benign behavior for small $||\boldsymbol{\theta}||$ as illustrated in Figures 1-3 of Section 5. This motivates the scale hyperprior introduced in next section, which enriches tail behavior through mixing, preventing the excess-risk integrand from accumulating over regions of non-negligible mass.
 
\section{Example of a scale hyperprior that induces a minimax Bayes decision rule}

From Section 2, we know that the marginal density induced by a fixed scale Bayesian Neural Network (BNN) prior is not superharmonic and that the induced decision rule is not minimax. Intuitively, this follows from the light tails of the BNN prior density. Therefore, we place a hyperprior on the scales to induce sufficiently heavy tails, which can yield a minimax decision rule: when $||\boldsymbol{y}||$ is large, the decision rule is pushed toward $\boldsymbol{y}$. We choose a $\operatorname{BetaPrime}(1, \frac{p}{2}-2)$ prior, since it has heavy tails and allows large signals to be explained by increased variance rather than being forcibly shrunk. This yields the following theorem. 

\begin{theorem}
    The Bayes decision rule induced by a ReLU, depth d Bayesian neural network prior with output dimension $p\geq 5$ is minimax when its effective output variance is given the hyperprior $W \sim \mathrm{Beta Prime}(1,\frac{p}{2}-2)$ where $W:= 2^{d-1} ||\boldsymbol{x}||^{2} \prod\limits_{\ell = 1}^{d} \sigma_{\ell}^{2} \prod\limits_{\ell=1}^{d-1}T_{\ell} , T_{\ell} \sim \Gamma(\frac{k_{\ell}}{2},1), k_{\ell} \in \{ 1, \dots, n_{\ell} \}, \ell \in \{ 1, \dots,d-1 \}$. 
\end{theorem}

A full proof of Theorem 3.1 is provided in the Supplementary Material. The proof starts from the exact prior density of the output of a deep ReLU BNN, which represents the density of the network output as a mixture of linear subnetworks indexed by $k = \{ k_{1}, \dots, k_{d-1}\}$ and latent Gamma distributed random variables $T_{\ell}$, yielding a Gaussian scale mixture. In particular, conditional on $S = \prod\limits_{\ell = 1}^{d-1} T_{\ell}$ and the scale factor $U = 2^{d-1}||\boldsymbol{x}||^{2} \prod\limits_{\ell = 1}^{d}\sigma_{\ell}^{2},$ the effective output variance is $US$, and
\begin{equation*}
    \boldsymbol{\theta} |(U,S,k) \sim N_{p}(\boldsymbol{0}_{p}, (US)I_{p}).
\end{equation*}
The key step is to choose a hyperprior such that $W := US$ has a Beta-prime distribution. Drawing $W$ independently of $(S,k)$ and setting $U = W/S$ makes the induced marginal prior 
\begin{equation*}
    \pi(\boldsymbol{\theta}) =N \int\limits_{0}^{\infty} \phi_{p}(\boldsymbol{\theta}; \boldsymbol{0}_{p}, wI_{p})h(w) dw
\end{equation*}
which is independent of the subnetwork index $k$, so the combinatorial mixture weights sum to a finite constant $N$ and conveniently cancels.
With this prior, the marginal density \begin{equation*}
    m(\boldsymbol{y}) = \int \phi_{p}(\boldsymbol{y} - \boldsymbol{\theta}; \boldsymbol{0}_{p}, I_{p}) \pi(\boldsymbol{\theta}) d\boldsymbol{\theta}
\end{equation*}
reduces, by Gaussian convolution and a one dimensional change of variables, to a radial integral. Writing $r= ||\boldsymbol{y}||$,
\begin{equation*}
    m(r) = C\int\limits_{0}^{1} u^{p-3}e^{-\frac{r^{2}}{2}u}du,C>0.
\end{equation*}
Differentiating $m(r)$ and computing $\Delta m$ shows that the minimax condition $\Delta \sqrt{m} \leq 0$ is equivalent to an inequality among these radial integrals, which can be rewritten in terms of incomplete gamma functions. This then reduces to checking positivity of an auxiliary function $F(\lambda)$. Verifying monotonicity and showing $\lim\limits_{\lambda \downarrow0} F(\lambda) = 0$ implies that $F(\lambda) >0$ for all $\lambda >0$. Hence $\Delta \sqrt{m} \leq 0$ for $p \geq 5$. Using a standard dominated convergence theorem argument to justify differentiation under the integral sign, we show that
\begin{equation*}
    \mathbb{E}_{\boldsymbol{\theta}} [|| \frac{\nabla m(\boldsymbol{Y})}{m(\boldsymbol{Y})}||^{2}]< \infty.
\end{equation*}
Therefore, the induced Bayes estimator is minimax. In particular, the Bayes rule can be written as
\begin{equation*}
    \delta_{\text{BNN, hyper}}(\boldsymbol{Y}) = \mathbb{E}[\frac{W}{1+W}|\boldsymbol{Y}]\boldsymbol{Y}  = (1-\mathbb{E}[(1+W)^{-1}|\boldsymbol{Y}])\boldsymbol{Y}. 
\end{equation*} 

Conditional on $W$, the ReLU BNN output induces a Gaussian normal means prior with variance multiplier $W$, which can be interpreted as the effective output variance. Thus, the network architecture yields a particular scale mixture representation. Therefore, the choice of the distribution of $W$ is crucial for ensuring appropriate tail behavior. The chosen BetaPrime prior calibrates tail thickness and is heavy tailed enough that shrinkage decreases at the correct rate for large signals (unlike the fixed scale BNN prior) while still providing meaningful shrinkage near the origin. This shows that minimaxity is achieved by focusing on $W$ rather than fine tuning individual layer scales. Depth and width affect $W$ only through the factor $2^{d-1}||\boldsymbol{x}||^{2}\prod\limits_{\ell=1}^{d}\sigma_{\ell}^{2}$ and the latent product $\prod\limits_{\ell = 1}^{d-1}T_{\ell}$. The BetaPrime hyperprior effectively absorbs this complexity. It is also worth highlighting that this is an exact finite-width mixture representation; thus, the minimax guarantee is not asymptotic and is compatible with finite neural networks. 

It is worth noting that, although heavy tailed distributions are theoretically motivated in this paper, they have also been observed to arise practically in Bayesian deep learning. In particular, it is known that, during the training of feedforward neural networks using stochastic gradient descent, the weights become increasingly heavy tailed. Therefore, increasing the depth and, more importantly, introducing an appropriate variance hyperprior may help mitigate the effects of a likely misspecified prior, such as the standard Gaussian prior on the weights, as discussed in Section 1.4. For further information on this behavior and its relation to the cold posterior effect, see \cite{fortuinbayesian}. 

Given the form of the mixing density, we can also deduce admissibility for the induced Bayes decision rule.
\begin{corollary}
   The proper Bayes estimate 
   \begin{equation*}
       \delta(\boldsymbol{Y}) = \mathbb{E}[\frac{W}{1+W}|\boldsymbol{Y}]\; \boldsymbol{Y}, W \sim \operatorname{BetaPrime}(1,\frac{p}{2}-2) \text{ is admissible. }
   \end{equation*}
\end{corollary}
The complete proof is provided in the Supplementary Material, with the key component recognizing $h(w) \sim bw^{-(b+1)}$ and using Theorem 3.15 in \cite{fourdrinier2018shrinkage}.

This corollary shows that the proposed construction is not only minimax but also admissible, and therefore cannot be uniformly improved upon within the normal location problem. In particular, the result is constructive: the BetaPrime mixing distribution yields an explicit Bayes shrinkage rule with a transparent posterior shrinkage factor $\mathbb{E}[W/(1+W)|\boldsymbol{Y}]$. Thus, the corollary strengthens the main result of this section by showing that an appropriate hyperprior does not merely recover minimaxity, but in fact produces a fully decision-theoretically justified procedure. 

The Strawderman prior \cite{strawderman1971proper} places a Beta hyperprior on a shrinkage factor in the normal means problem, inducing a proper Bayes estimator that is both minimax and admissible. The key mechanism is that the induced marginal density has sufficiently heavy tails — specifically, its square root is superharmonic — which is precisely the condition that guarantees minimaxity. In the BNN setting, a Beta-Prime hyperprior on the output variance likewise yields a superharmonic square root marginal, and hence a minimax admissible Bayes rule. The Beta-Prime construction is thus the natural analog of the Strawderman prior, adapted to the parameterization arising from the network architecture.

Under the construction in Theorem 3.1, with $W$ independent of the latent Gamma variables and $R = W/(cS)$, the product of layer variances has the following conditional density. 
\begin{corollary}
For $||\boldsymbol{x}|| \neq 0$, conditional on the subnetwork index vector $\boldsymbol{k}$, the random variable $R = \prod\limits_{\ell=1}^{d} \sigma_{\ell}^{2}$ has density 
\begin{equation*}
    f_{R|\boldsymbol{k}}(r) = \frac{2^{d-1}||\boldsymbol{x}||^{2}}{\Gamma(b)\prod\limits_{\ell = 1}^{d-1}\Gamma(\alpha_{\ell})}G_{d,1}^{1,d}(2^{d-1}||\boldsymbol{x}||^{2}r|\begin{matrix}
        -b, -\alpha_{1}, \dots, -\alpha_{d-1}\\
        0
    \end{matrix}), r>0,
\end{equation*}
where $b=p/2-2$, $\alpha_{\ell} = \frac{k_{\ell}}{2}$, and $c = 2^{d-1}||\boldsymbol{x}||^{2}$. 
\end{corollary}
A proof is given in the Supplementary Material. This induced prior on the product $\prod\limits_{\ell=1}^{d}\sigma_{\ell}^{2}$ does not uniquely determine the hyperpriors on the individual variance parameters. For example, one may fix $\sigma_{1}^{2} = \dots= \sigma_{d-1}^{2}=1$ and set $\sigma_{d}^{2}  = W/(cS)$, with $W$ independent of $(S,k)$. Note that the unconditional density is given by $f_{R}(r) = N^{-1}\sum\limits_{\boldsymbol{k}}w_{\boldsymbol{k}}f_{R|\boldsymbol{k}}(r)$. 

\section{Minimaxity and admissibility in the predictive density problem }
In this section we consider the problem of estimating a predictive density discussed in \cite{brown2008admissible} and \cite{george2006improved}.  Particularly, we can extend our admissible Bayes decision rule for the Normal location model under quadratic loss, to the predictive density estimation setting. 
\subsection{Decision problem}
From \cite{george2006improved} let $\boldsymbol{X}|\boldsymbol{\theta} \sim N_{p}(\boldsymbol{\theta}, v_{x}I_{p})$ and $\boldsymbol{Y}|
\boldsymbol{\theta} \sim N_{p}(\boldsymbol{\theta}, v_{y}I_{p})$ be independent with common unknown mean $\boldsymbol{\theta}$ and known $v_{x}$ and $v_{y}$ greater than zero. Based on only observing $\boldsymbol{X} = \boldsymbol{x}$ we wish to estimate the density $p(\boldsymbol{y}|\boldsymbol{\theta})$. We measure the proximity of a density estimate $\hat{p}(\boldsymbol{y}|\boldsymbol{x})$ by the Kullback-Leibler loss \begin{equation*}
    L(\boldsymbol{\theta}, \hat{p}(\boldsymbol{y}|\boldsymbol{x})) = \int p(\boldsymbol{y}|\boldsymbol{\theta}) \log(\frac{p(\boldsymbol{y}|\boldsymbol{\theta})}{\hat{p}(\boldsymbol{y}|\boldsymbol{x})})d\boldsymbol{y}
\end{equation*} and evaluate $\hat{p}$ by its associated risk function
\begin{equation*}
    R_{KL}(\boldsymbol{\theta}, \hat{p}) = \int p(\boldsymbol{x}|\boldsymbol{\theta}) L(\boldsymbol{\theta}, \hat{p}(\boldsymbol{y}|\boldsymbol{x}))d\boldsymbol{x}. 
\end{equation*}
For a given prior distribution $\pi(\boldsymbol{\theta})$, according to Lemma 2 of \cite{george2006improved}, the Bayes predictive density is given by \begin{align*}
    \hat{p}(\boldsymbol{y}|\boldsymbol{x}) &= \frac{m_{\pi}(\boldsymbol{w};v_{w})}{m_{\pi}(\boldsymbol{x};v_{x})}\hat{p}_{U}(\boldsymbol{y}|\boldsymbol{x}), \boldsymbol{W} = \frac{v_{y}\boldsymbol{X} + v_{x}\boldsymbol{Y}}{v_{x} + v_{y}}, v_{w} = \frac{v_{x}v_{y}}{v_{x}+v_{y}} \\
    \hat{p}_{U}(\boldsymbol{y}|\boldsymbol{x}) &=  (2\pi(v_{x}+v_{y}))^{-\frac{p}{2}}\exp\{ - \frac{||\boldsymbol{y}-\boldsymbol{x}||^{2}}{2(v_{x}+v_{y})} \}, \text{ and } m_{\pi}(\boldsymbol{x}) = \int p(\boldsymbol{x}|\boldsymbol{\theta})\pi(\boldsymbol{\theta})d\boldsymbol{\theta}. 
\end{align*}
\subsection{Minimaxity and admissibility}
In this section we show that the Bayes predictive density induced by the prior distribution resulting from a deep ReLU BNN with a BetaPrime hyperprior on the effective output variance is minimax and admissible. 
\begin{theorem}
    Let $p\geq5$ and fix $v_0\geq v_x$. Let $\pi_{v_0}$ denote the deep ReLU BNN prior described in Section 3 with effective output variance $W_{\mathrm{eff}}=v_0B$, where $B\sim\operatorname{BetaPrime}(1,p/2-2)$. Then the induced Bayes predictive density $\hat p_{\pi_{v_0}}$ is minimax under Kullback--Leibler loss.
\end{theorem}

We provide a full proof in the Supplementary Material. Fix $v_{0} \geq v_{x}$. For $0 < v \leq v_{0}$, set $\rho = v/v_{0}$ and $\boldsymbol{y} = \boldsymbol{z}/\sqrt{v}$. Gaussian convolution and change of variables gives the Gaussian mixture marginal as 
\begin{equation*}
    m_{\pi_{v_{0}}}(\boldsymbol{z};v) = C_{\rho,v}G_{\rho}(\boldsymbol{y}), \qquad C_{\rho,v}>0, 
\end{equation*}
where, with $a=p-2$, 
\begin{equation*}
    G_{\rho}(\boldsymbol{y}) = \int\limits_{0}^{1} u^{a-1}q_{\rho}(u)e^{-\frac{||\boldsymbol{y}||^{2}u}{2}}du, \qquad q_{\rho}(u) = \{ \rho + (1-\rho)u \}^{-\frac{a}{2}}. 
\end{equation*}
Writing $\lambda = ||\boldsymbol{y}||^{2}/2$ and 
\begin{equation*}
    I_{j}(\lambda) = \int\limits_{0}^{1}u^{a-1+j}q_{\rho}(u)e^{-\lambda u} du, \qquad j = 0,1,2, 
\end{equation*}
differentiation reduces square-root superharmonicity to 
\begin{equation*}
    D_{\rho}(\lambda) = I_{0}(2\lambda I_{2} - pI_{1}) - \lambda I_{1}^{2} \leq 0. 
\end{equation*}
Two integration-by-parts identities and Chebyshev's association inequality, applied under the probability measure obtained by normalizing the integrand of $I_{0}$, establish this inequality. Consequently, 
\begin{equation*}
    \Delta_{\boldsymbol{z}}\sqrt{m_{\pi_{v_{0}}}(\boldsymbol{z};v)} = C_{\rho,v}^{\frac{1}{2}}v^{-1}\Delta_{\boldsymbol{y}}\sqrt{G_{\rho}(\boldsymbol{y})} \leq 0. 
\end{equation*}
Since $0<v_{w}< v_{x}  \leq v_{0}$, this holds throughout $[v_{w}, v_{x}]$. Moreover, 
\begin{equation*}
    0 < m_{\pi_{v_{0}}}(\boldsymbol{z};v) \leq N(2\pi v)^{-\frac{p}{2}} < \infty.
\end{equation*}
Theorem 1(ii) of \cite{george2006improved} therefore establishes minimaxity under KL loss. 

We have established that the prior resulting from a deep ReLU Bayesian neural network with a BetaPrime hyperprior induces a minimax decision rule in both estimating the mean of a normal location model under quadratic risk and estimating the predictive density in a normal location model setting under Kullback-Leibler risk. For each fixed $v_{0}$, the predictive minimax guarantee holds whenever $v_{x} \leq v_{0}$, for any $v_{y}>0$, because square-root superharmonicity holds throughout the required variance interval $[v_{w}, v_{x}]$. In particular, in view of Lemma 2 in \cite{george2006improved}, the prior distribution induces a similar shrinkage behavior in both problems, by shrinking the default estimator. In the quadratic risk setting, this is the maximum likelihood estimator $\boldsymbol{Y}$, and in the Kullback-Leibler risk setting it is the Bayes predictive density under the uniform prior. We also prove the induced Bayes predictive density is admissible. 

\begin{corollary}
    The induced Bayes predictive density from the deep ReLU BNN with BetaPrime hyperprior on the effective output variance is admissible with respect to the set of all proper densities on $\mathbb{R}^{p}$ given by $A_{0} = \{ g: \mathbb{R}^{p} \rightarrow \mathbb{R} \text{ such that } g(\boldsymbol{y}) \geq 0 \text{ and } \int g(\boldsymbol{y}) d\boldsymbol{y} = 1 \}$. 
\end{corollary}
We provide a full proof of Corollary 4.2 in the Supplementary Material. The Bayes predictive density 
\begin{equation*}
    \hat{p}_{\pi_{v_{0}}}(\boldsymbol{y}|\boldsymbol{x}) = \int p(\boldsymbol{y}|\boldsymbol{\theta}) \pi_{v_{0}} (d\boldsymbol{\theta}|\boldsymbol{x})
\end{equation*}
is well defined because the Gaussian likelihood is uniformly bounded and the prior density is integrable, ensuring finiteness. By \cite{aitchison1975goodness} (expression 2.6), $\hat{p}_{\pi_{v_{0}}}$ minimizes the Bayes Kullback-Leibler risk 
\begin{equation*}
    B_{KL}(\pi_{v_{0}}, \hat{p}) = \int R_{KL}(\boldsymbol{\theta}, \hat{p}) \pi_{v_{0}} (d \boldsymbol{\theta}). 
\end{equation*}
Suppose, for contradiction, that another predictive rule $\tilde{p}$ dominates $\hat{p}_{\pi_{v_{0}}}$. Integrating the resulting pointwise risk inequality against $\pi_{v_{0}}$ yields equality due to the Bayes risk minimization of $\hat{p}_{\pi_{v_{0}}}$.  Subtracting the two Bayes risks and rearranging gives 
\begin{equation*}
    \int m_{\pi_{v_{0}}}(\boldsymbol{x}) KL(\hat{p}_{\pi_{v_{0}}}(\boldsymbol{y}|\boldsymbol{x}), \tilde{p}(\boldsymbol{y}|\boldsymbol{x})) dx = 0, 
\end{equation*}
where $m_{\pi_{v_{0}}}(\boldsymbol{x})$ is the marginal density of $\boldsymbol{x}$. Since KL divergence is non-negative, it follows that 
\begin{equation*}
    KL(\hat{p}_{\pi_{v_{0}}}(\boldsymbol{y}|\boldsymbol{x}), \tilde{p}(\boldsymbol{y}|\boldsymbol{x})) = 0, m_{\pi_{v_{0}}}-a.e. \boldsymbol{x}, 
\end{equation*}
and hence $\tilde{p}(\boldsymbol{y}|\boldsymbol{x}) = \hat{p}_{\pi_{v_{0}}} (\boldsymbol{y}|\boldsymbol{x})$ almost everywhere. This implies \begin{equation*}
    R_{KL}(\boldsymbol{\theta}, \tilde{p}) = R_{KL}(\boldsymbol{\theta}, \hat{p}_{\pi_{v_{0}}}), \forall \boldsymbol{\theta}, 
\end{equation*}
contradicting strict domination. Therefore, $\hat{p}_{\pi_{v_{0}}}$ is admissible. 

We have shown that the Bayes predictive density induced by a strictly positive prior density - such as the deep ReLU BNN with BetaPrime hyperprior - is admissible under KL risk over the full action space $A_{0}$ of all proper densities on $\mathbb{R}^{p}$. Consequently, it is not dominated by any competing proper predictive density, making it a globally non-dominated rule in the canonical predictive problem. This extends our results from Section 3 by showing that hyperprior design governs not only the quality of point decisions, but also the decision-theoretic validity of the full predictive distribution. The result is also relevant for prior-fitted methodologies, where a neural predictor is trained to emulate the Bayesian predictive distribution induced by a chosen prior: in that context, our theory identifies a setting in which the predictive target itself is well founded. 
\section{Simulated example}
In this section, we empirically examine the risks of $\delta_{\text{BNN,fixed}}(\boldsymbol{y})$ and $\delta_{\text{BNN,hyper}}(\boldsymbol{y})$ using simulation. We also compare these decision rules with the decision induced by a fixed scale BNN prior with dropout, $\delta_{\text{BNN,dropout}}(\boldsymbol{y})$, and with the decision rule induced by a Horseshoe prior \cite{carvalho2010horseshoe}, $\delta_{HS}(\boldsymbol{y})$. Since the Horseshoe prior depends on the sparsity regime of the mean vector, we consider that case separately.
\subsection{Radial decision rule simulation}
Recall from Section 2 that the decision rule induced by a fixed scale Bayesian Neural Network (BNN) is
\begin{equation*}
    \delta_{\text{BNN, fixed}}(\boldsymbol{y}) = \mathbb{E}[\frac{V}{V+1}|||\boldsymbol{Y}||^{2} = ||\boldsymbol{y}||^{2}]\boldsymbol{y}, V = 2^{d-1}(\prod\limits_{\ell = 1}^{d}\sigma_{\ell}^{2})(\prod\limits_{\ell = 1}^{d-1}T_{\ell}), T_{\ell} \sim \Gamma(\frac{k_{\ell}}{2},1). 
\end{equation*} Similarly, the decision rule induced by a BNN with a BetaPrime hyperprior on the effective output variance is 
\begin{equation*}
    \delta_{\text{BNN,hyper}}(\boldsymbol{y}) = \mathbb{E}[\frac{W}{W+1}|||\boldsymbol{Y}||^{2}=||\boldsymbol{y}||^{2}]\boldsymbol{y}, W\sim \operatorname{BetaPrime}(1, \frac{p}{2}-2).
\end{equation*} 
We also consider a fixed scale BNN prior with dropout. Let $q_{\ell}$ denote the probability of keeping in the hidden layer $\ell$, and let $N_{\ell}$ denote the number of active units in layer $\ell$. The corresponding decision rule is 
\begin{equation*}
    \delta_{\text{BNN, fixed, dropout}}(\boldsymbol{y}) = \mathbb{E}[\frac{V}{V+1}|||\boldsymbol{Y}||^{2} = ||\boldsymbol{y}||^{2}]\boldsymbol{y} , 
    \end{equation*}
where, 
    \begin{equation*}
        V= 2^{d-1}(\prod\limits_{\ell = 1}^{d}\sigma_{\ell}^{2})(\prod\limits_{\ell=1}^{d-1} q_{\ell}^{-1}) (\prod\limits_{\ell = 1}^{d-1}T_{\ell})\boldsymbol{1} \{ N_{1}>0, \dots, N_{d-1}>0\},
    \end{equation*}
with 
\begin{equation*}
    N_{\ell} \sim Bin(n_{\ell}, q_{\ell}) \text{ and }T_{\ell}|k_{\ell} \sim \Gamma(\frac{k_{\ell}}{2},1). 
\end{equation*}

The simulations are conducted in the R programming language \cite{rlanguage}, and all code is available in a GitHub repository at the following link: \href{https://github.com/DanielCoulson/Risk_simulations}{Risk\_simulations}. We consider the normal location model with $p= 5, 50$, and $100$, where $\boldsymbol{\theta} \in \mathbb{R}^{p}$ is fixed but unknown, and performance is evaluated under quadratic risk. We compare four estimators: (i) the maximum likelihood estimator $\delta(\boldsymbol{Y}) = \boldsymbol{Y}$, whose risk is equal to $p$; (ii) the decision rule induced by a fixed scaled BNN, $\delta_{\text{BNN,fixed}}(\boldsymbol{y})$; (iii) the decision rule induced by a BNN with a minimax BetaPrime hyperprior on the effective output variance, $\delta_{\text{BNN, hyper}}(\boldsymbol{y})$; and (iv) the decision rule induced by a fixed scale BNN with dropout $\delta_{\text{BNN,fixed,dropout}}(\boldsymbol{y})$. The resulting Bayes rules are radial shrinkage estimators and admit the form $\delta(\boldsymbol{Y}) = a(||\boldsymbol{Y}||^{2})\boldsymbol{Y}$ where $a(\cdot)$ is a scalar shrinkage function discussed below.

For the fixed scale BNN prior, we generate Monte Carlo draws from the effective output variance $V$. The network depth is set to $d = 3$, with hidden layer widths $n_{1} = n_{2} = 20$, layer scales $\sigma_{1} = \sigma_{2} = \sigma_{3} = 1$, and $||\boldsymbol{x}|| = 1$. We draw
\begin{equation*}
    V = 2^{d-1}\prod\limits_{\ell = 1}^{d}\sigma_{\ell}^{2} \prod\limits_{\ell = 1}^{d-1} T_{\ell},
\end{equation*}
where each $T_{\ell}$ is Gamma-distributed conditional on an index $k_{\ell}$. Rather than summing over all $k_{\ell} \in \{1, \dots, n_{\ell} \}$, we sample $k_{\ell}$ with probability proportional to 
\begin{equation*}
\binom{n_{\ell}}{k_{\ell}}.     
\end{equation*}
We then sample $T_{\ell}|k_{\ell} \sim \Gamma(k_{\ell}/2, 1)$. The product $\prod\limits_{\ell = 1}^{2}T_{\ell}$ is then multiplied by the constant $2^{2}\prod\limits_{\ell =1}^{3}\sigma_{\ell}^{2}$ to obtain a single draw of $V$. We generate $M_{v} = 200,000$ i.i.d. draws $\{ V_{m} \}_{m=1}^{M_{v}}$ to approximate the mixing distribution $V$.
Given the scale mixture prior $\boldsymbol{\theta}|V \sim N_{p}(\boldsymbol{0}_{p}, vI_{p})$, the posterior mean takes the form 
\begin{equation*}
    \mathbb{E}[\boldsymbol{\theta} |\boldsymbol{Y}] = a(||\boldsymbol{Y}||^{2})\boldsymbol{Y} \text{ with } a(s) = \mathbb{E}[\frac{V}{1+V}|||\boldsymbol{Y}||^{2} = s].
\end{equation*}
For the fixed scale BNN decision rule, $a(s)$ is computed by Monte Carlo importance weighting over the sampled $\{V_{m}\}$. For each $s$ the weights $w_{m}(s)$ are proportional to 
\begin{equation*}
    (1+V_{m})^{-\frac{p}{2}} \exp \{ -\frac{s}{2(1+V_{m})} \}, 
\end{equation*}
and we approximate the shrinkage factor by
\begin{equation*}
    \sum\limits_{m}\frac{w_{m}(s)V_{m}}{1+V_{m}}. 
\end{equation*}
For computational efficiency, the shrinkage factor is precomputed on a grid $s \in [0,s_{max}]$ with 2500 points, where $s_{max} = (500 + 6\sqrt{p})^{2}$, and then evaluated at arbitrary $s$, by linear interpolation.  

For the decision rule induced by the BNN with the BetaPrime hyperprior, the posterior mean shrinkage factor admits the closed form expression
\begin{equation*}
    a(s) = 1- \frac{\gamma(p-1,\frac{s}{2})}{\frac{s}{2}\gamma(p-2, \frac{s}{2})}. 
\end{equation*}
This expression is implemented using gamma distribution functions, with the case $s = 0$ handled by its limiting value. For numerical stability, the resulting shrinkage factors are constrained to lie in $[0,1]$. 

For the fixed scale BNN with dropout, we again generate Monte Carlo draws from the effective output variance $V$, now under random dropout of the hidden units. The network depth is set to $d=3$, with hidden layer widths $n_{1} = n_{2} =20$, layer scales $\sigma_{1} = \sigma_{2} = \sigma_{3} =1$ and keep probabilities $q_{1} = q_{2} = 0.8$, and $||\boldsymbol{x}|| =1$. Using inverted dropout, each Monte Carlo draw is constructed as follows. For each hidden layer $\ell = 1,2$, we first sample the number of active units $N_{\ell} \sim Bin(n_{\ell}, q_{\ell})$.  If $N_{\ell} = 0$ for any hidden layer, the network is taken to be inactive and we set $V=0$. Otherwise, conditional on $N_{\ell}$, we sample an index $k_{\ell} \in \{ 1, \dots, N_{\ell} \}$ with probability proportional to $\binom{N_{\ell}}{k_{\ell}},$ and then sample $ T_{\ell} |k_{\ell} \sim \Gamma(k_{\ell}/2,1).$
%\begin{equation*}
 %   T_{\ell} |k_{\ell} \sim \Gamma(\frac{k_{\ell}}{2},1). 
%\end{equation*}
Then one draw of the effective output variance is
\begin{equation*}
    V = 2^{2} \prod\limits_{j=1}^{3}\sigma_{j}^{2} \prod\limits_{\ell = 1}^{2}q_{\ell}^{-1} \prod\limits_{\ell = 1}^{2} T_{\ell} = 6.25 \prod\limits_{\ell =1}^{2} T_{\ell}. 
\end{equation*}
We generate $M_{v}=200,000$ i.i.d. draws, $\{ V_{m} \}_{m=1}^{M_{v}}$, to approximate the dropout induced mixing distribution of $V$. Given the scale mixture prior $\boldsymbol{\theta}|V \sim N_{p}(\boldsymbol{0}_{p}, VI_{p})$, the posterior mean again has the form $\mathbb{E}[\boldsymbol{\theta}|\boldsymbol{Y}] = a(||\boldsymbol{Y}||^{2})\boldsymbol{Y}$, where  
%\begin{equation*}
%    \mathbb{E}[\boldsymbol{\theta}|\boldsymbol{Y}] = a(||\boldsymbol{Y}||^{2})\boldsymbol{Y}, 
%\end{equation*}
\begin{equation*}
    a(s) = \mathbb{E}[\frac{V}{1+V}\||\boldsymbol{Y}\|^{2} =s]. 
\end{equation*}
For the dropout BNN decision rule, $a(s)$ is computed using the same Monte Carlo importance-weighting scheme as for the fixed-scale BNN. As the sampled values $\{ V_{m} \}$ include the event $V_{m} = 0$, the approximation automatically accounts for the point mass at zero arising from inactive network realizations. As in the other fixed-scale cases, the resulting shrinkage factors are truncated to $[0,1]$ for numerical stability and are then evaluated by linear interpolation from a precomputed grid. 

Risk is estimated as a function of the signal strength $r = ||\boldsymbol{\theta}||$ on the grid $r = 0,1, \dots, 500$. For each $r$, the direction of $\boldsymbol{\theta}$ is randomized by drawing $\boldsymbol{u} \sim N_{p}(\boldsymbol{0}_{p}, I_{p})$ and normalizing it to unit length, $\boldsymbol{u} \leftarrow \frac{\boldsymbol{u}}{||\boldsymbol{u}||}$, which yields a uniform direction on the sphere. We then set $\boldsymbol{\theta} = r\boldsymbol{u}$.  For each such $\boldsymbol{\theta}$, we generate $N_{mc}$ i.i.d. samples $\boldsymbol{Y}^{(i)} \sim N_{p}(\boldsymbol{\theta}, I_{p})$, compute the corresponding estimates, $\boldsymbol{\delta}(\boldsymbol{Y}^{(i)})$, and average $||\boldsymbol{\delta}(\boldsymbol{Y}^{(i)}) - \boldsymbol{\theta}||^{2}$ across draws to approximate $R(\boldsymbol{\theta}, \boldsymbol{\delta})$. To reduce Monte Carlo variability due to the random direction, we repeat this procedure over $K_{dir}$ independent directions at each $r$ and average the resulting risk estimates. For $p=$ 5,50,100, respectively, we use $(N_{mc},K_{dir}) = (50000,10), (15000,6), (8000,4)$. The results are displayed in Figures 1, 2, and 3. 
\begin{figure}[!htbp] 
    \centering
    \includegraphics[width=\linewidth]{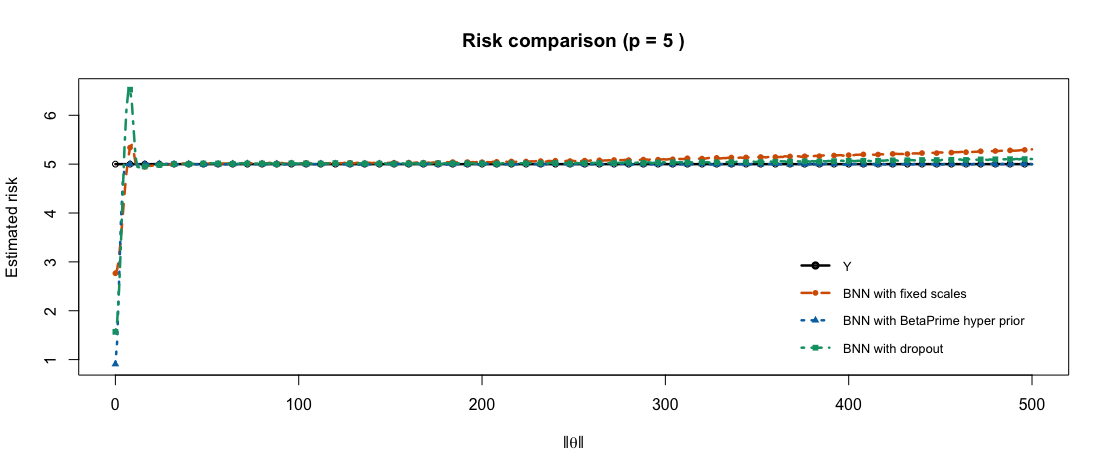} 
    \caption{Estimated risk for several decision rules in dimension $p=5$ as a function of $||\boldsymbol{\theta}||$. The plotted rules are the MLE, the fixed-scale BNN rule, the Beta-prime minimax shrinkage rule, and the dropout-BNN rule. For the BNN-based rules, the network depth is $d=3$, the hidden layer widths are $n_1=n_2=20$, and the layer scales are $\sigma_1=\sigma_2=\sigma_3=1$; for the dropout-BNN rule, the keep probabilities are $q_1=q_2=0.8$ with inverted dropout.}
    \label{fig:risk} 
\end{figure}
\begin{figure}[!htbp] 
    \centering
    \includegraphics[width=\linewidth]{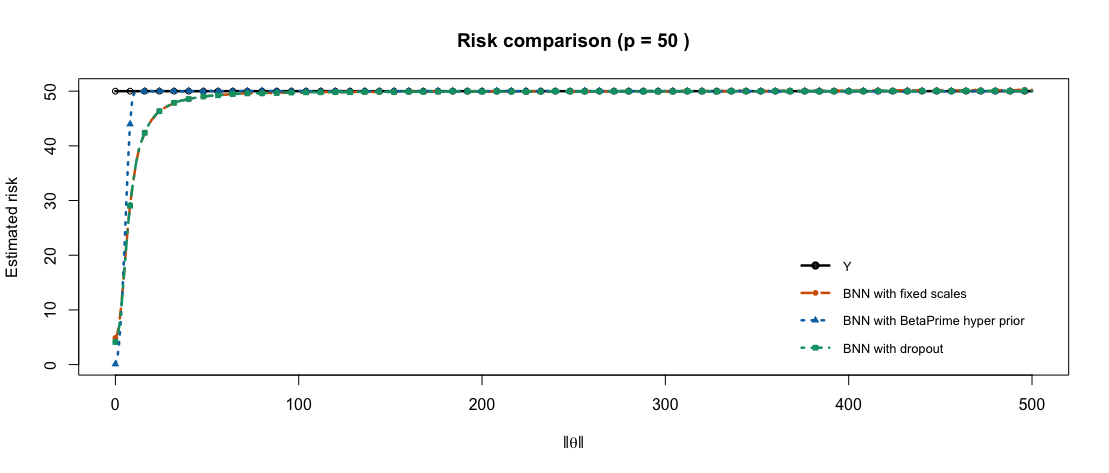} 
    \caption{Estimated risk for several decision rules in dimension $p=50$ as a function of $||\boldsymbol{\theta}||$. The plotted rules are the MLE, the fixed-scale BNN rule, the Beta-prime minimax shrinkage rule, and the dropout-BNN rule. For the BNN-based rules, the network depth is $d=3$, the hidden layer widths are $n_1=n_2=20$, and the layer scales are $\sigma_1=\sigma_2=\sigma_3=1$; for the dropout-BNN rule, the keep probabilities are $q_1=q_2=0.8$ with inverted dropout.}
    \label{fig:risk} 
\end{figure}
\begin{figure}[!htbp] 
    \centering
    \includegraphics[width=\linewidth]{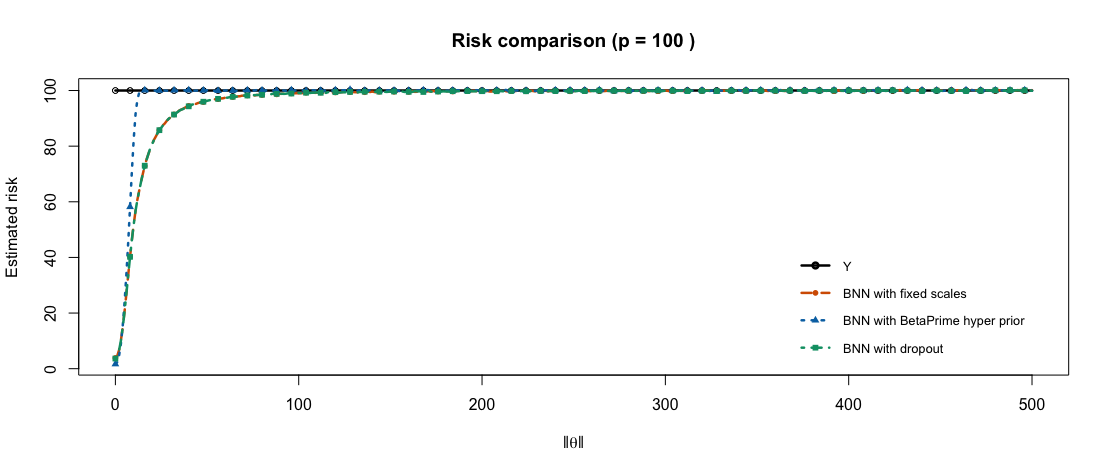} 
    \caption{Estimated risk for several decision rules in dimension $p=100$ as a function of $||\boldsymbol{\theta}||$. The plotted rules are the MLE, the fixed-scale BNN rule, the Beta-prime minimax shrinkage rule, and the dropout-BNN rule. For the BNN-based rules, the network depth is $d=3$, the hidden layer widths are $n_1=n_2=20$, and the layer scales are $\sigma_1=\sigma_2=\sigma_3=1$; for the dropout-BNN rule, the keep probabilities are $q_1=q_2=0.8$ with inverted dropout.}
    \label{fig:risk} 
\end{figure}

Across Figures 1-3, we compare the risk of the induced decision rules from different BNN priors with the minimax benchmark, which is equal to $p$ and is attained by the maximum likelihood estimator. The main observation is that the risk of $\delta_{\text{BNN, hyper}}$ tracks the minimax boundary almost exactly. For $p = 5$, its risk begins well below 5 near the origin, rises to the minimax level (up to Monte Carlo error), and then remains essentially flat at that level over the remainder of the plot. The same pattern appears for $p = 50 $ and $p = 100$, where its risk approaches $p$ from below and does not exhibit any systematic exceedance. In contrast, $\delta_{\text{BNN, fixed}}$ clearly violates the minimax bound for $ p = 5$. After improving on the maximum likelihood estimator near $||\boldsymbol{\theta}|| = 0$, its risk rises above $5$, likely due in part to numerical approximation error, then returns below $5$, and finally drifts upward again, ending noticeably above the benchmark for large $||\boldsymbol{\theta}||$. The dropout version, $\delta_{\text{BNN, fixed, dropout}}$, exhibits similar behavior, although the upward drift above the minimax level is milder. For $p = 50$ and $p =100$, these non-minimax departures are much smaller on the scale of the plots, but the qualitative distinction remains the same: the risk of $\delta_{\text{BNN, hyper}}$ approaches the constant risk level $p$ from below, whereas the risks of $\delta_{\text{BNN, fixed}}$ and $\delta_{\text{BNN, fixed, dropout}}$ can exceed it. Thus, the simulations are consistent with the theory: the BetaPrime hyperprior yields a minimax Bayes rule, whereas the fixed scale prior and its dropout augmented variant do not.
 
\subsection{Sparsity dependent simulation}
In this simulation, we compare the risk of $\delta_{\text{BNN, hyper}}$ with that of the decision rule induced by the Horseshoe prior \cite{carvalho2010horseshoe} in both sparse and dense normal location settings. Since $\delta_{\text{BNN, hyper}}$ is radial, its risk depends only on the signal magnitude $||\boldsymbol{\theta}||$, whereas the decision rule induced by the Horseshoe prior is sensitive to the sparsity structure of $\boldsymbol{\theta}$, and is therefore evaluated separately over a range of sparsity regimes. 

To study the effect of sparsity explicitly, for each $p$ and signal strength $r = ||\boldsymbol{\theta}||$, we consider $k$-sparse mean vectors of the form 
\begin{equation*}
    \boldsymbol{\theta}_{r,k} = (\frac{r}{\sqrt{k}}, \dots, \frac{r}{\sqrt{k}}, 0 , \dots,0), 
\end{equation*}
with exactly $k$ nonzero entries, so that $||\boldsymbol{\theta}_{r,k}|| = r$. The magnitude of the signal is varied over the interval $[0,2.5\sqrt{p}]$, implemented numerically as six equally spaced points. For each dimension, we consider the collection of sparsity levels
\begin{equation*}
    k \in \{ 1,2,5,10, \left\lfloor{0.1p} \right \rfloor, \left\lfloor{0.2p} \right \rfloor, \left\lfloor{0.5p} \right \rfloor, p \},
\end{equation*}
after truncation to $\{1, \dots, p\}$ and removal of duplicates. This allows the comparison to span regimes ranging from highly sparse signals to fully dense signals. 

As in the previous simulation, we estimate the risk of $\delta_{\text{BNN, hyper}}$ by Monte Carlo using $N_{mc} = 5000$ independent draws from $N_{p}(\boldsymbol{\theta}, I_{p})$ at each value of $r$. Since this decision rule is radial, its risk depends only on $||\boldsymbol{\theta}||$, and not on the support pattern of $\boldsymbol{\theta}$. Consequently, for each $r$, we compute a single risk estimate using a reference 1-sparse vector with norm $r$, and use this same estimate for all values of $k$. 

For the Horseshoe prior, the posterior mean is computed separately for each observed vector $\boldsymbol{y}$ using a Gibbs sampler based on the standard scale-mixture representation of the prior. Conditional on the current local- and global scale parameters, the coordinates of the mean vector are updated from their Gaussian full conditional distributions, after which the local- and global variance components are updated via inverse-gamma latent variable steps. To improve numerical efficiency, we use a Rao-Blackwellized estimator of the posterior mean at each iteration, averaging the conditional posterior means rather than the raw sampled coefficients themselves. For each observation, the Markov chain is run for 3,000 iterations, with the first 1,000 discarded as burn-in and every second draw retained thereafter. In repeated risk calculations, the final scale values from one Monte Carlo replication are used to initialize the next chain, providing a warm start and reducing computational cost. The risk under the Horseshoe is then estimated by averaging squared error over $N_{mc} = 500$ independent draws $\boldsymbol{Y}^{(i)} \sim N_{p}(\boldsymbol{\theta}_{r,k}, I_{p})$ for each pair $(r,k)$. 

\begin{figure}[!htbp] 
    \centering
    \includegraphics[width=\linewidth]{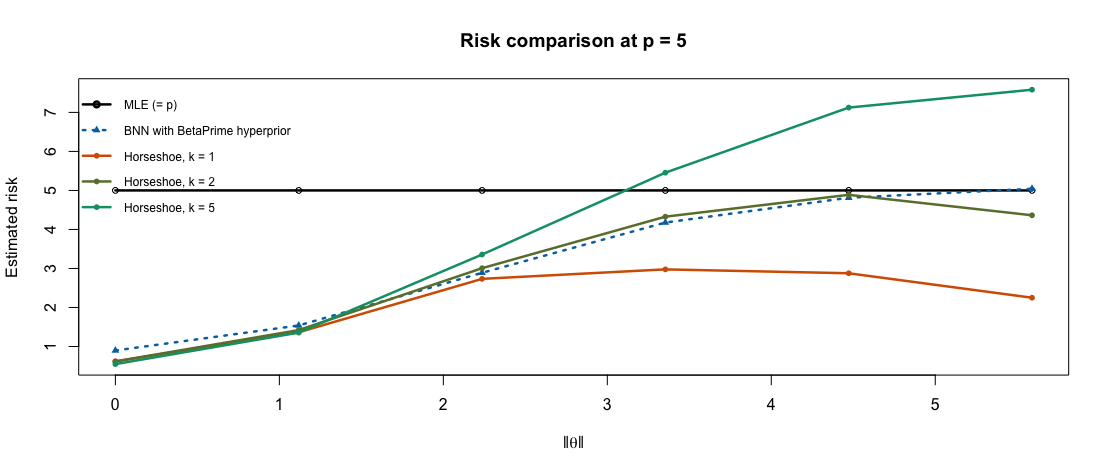} 
    \caption{Estimated risk for several decision rules in dimension $p=5$ as a function of $\|\boldsymbol{\theta}\|$ under several sparsity regimes. The plotted rules are the MLE, the Beta-prime minimax shrinkage rule, and the horseshoe posterior mean. The true sparsity levels considered are $1$, $2$, and $5$.}
    \label{fig:risk} 
\end{figure}
\begin{figure}[!htbp] 
    \centering
    \includegraphics[width=\linewidth]{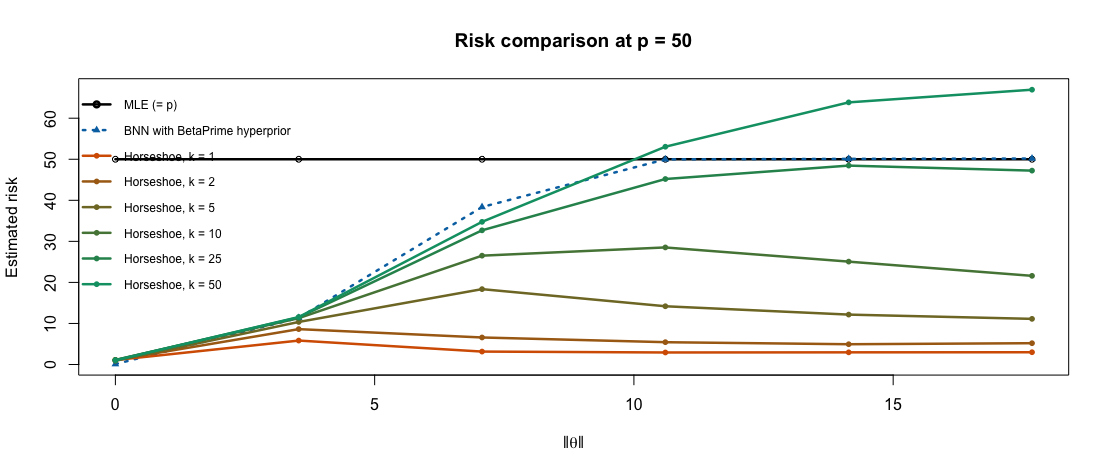} 
    \caption{Estimated risk for several decision rules in dimension $p=50$ as a function of $\|\boldsymbol{\theta}\|$ under several sparsity regimes. The plotted rules are the MLE, the Beta-prime minimax shrinkage rule, and the horseshoe posterior mean. The true sparsity levels considered are $1$, $2$, $5$, $10$, $25$, and $50$.}
    \label{fig:risk} 
\end{figure}
\begin{figure}[!htbp] 
    \centering
    \includegraphics[width=\linewidth]{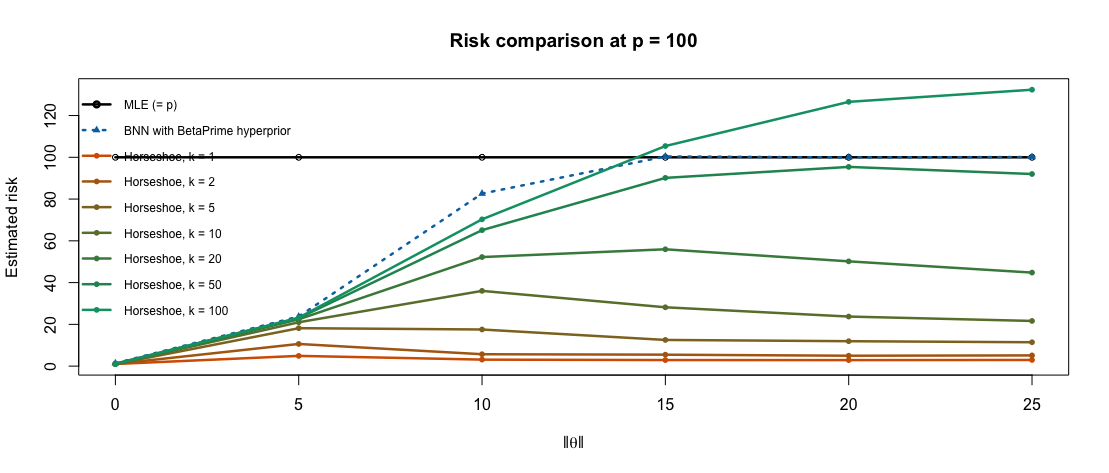} 
    \caption{Estimated risk for several decision rules in dimension $p=100$ as a function of $\|\boldsymbol{\theta}\|$ under several sparsity regimes. The plotted rules are the MLE, the Beta-prime minimax shrinkage rule, and the horseshoe posterior mean. The true sparsity levels considered are $1$, $2$, $5$, $10$, $20$, $50$, and $100$.}
    \label{fig:risk} 
\end{figure}
In Figures 4-6, $\delta_{\text{BNN, hyper}}$ depends only on $||\boldsymbol{\theta}||$ and remains close to the constant benchmark $p$, approaching it from below in a manner consistent with its theoretical minimax behavior. By contrast, the risk of the Horseshoe estimator strongly depends on the sparsity level $k$. Although the risk can lie below that of $\delta_{\text{BNN, hyper}}$ for sparse configurations, it increases substantially as $k$ grows and, for sufficiently dense signals, exceeds the minimax benchmark by a wide margin. In particular, in the dense case $k = p$, the Horseshoe risk increases to approximately 7.5 for $p=5$, $66$ for $p =50$, and $130$ for $p=100$. Thus, while the Horseshoe estimator is highly sensitive to the underlying sparsity pattern, $\delta_{\text{BNN, hyper}}$ exhibits stable uniform risk control across the full range considered, making it the strongest procedure from a minimax perspective. 
\section{Conclusion}
We have shown that the decision rule induced by a deep, fixed-scale ReLU BNN is not minimax in the normal location model under quadratic loss, because the prior predictive density has stretched exponential tails that apply overly conservative shrinkage to large signals. We then proposed a Beta-Prime hyperprior on the effective output variance of the network prior, which recovers minimaxity by inducing sufficiently heavy tails: it shrinks strongly toward the origin for weak signals and reduces shrinkage sufficiently rapidly for large signals. We further established admissibility of the induced Bayes rule and extended both the minimaxity and admissibility results to predictive density estimation under Kullback--Leibler loss. These theoretical properties were validated in a numerical simulation study under quadratic loss. An interesting direction for future work is to characterize broader families of hyperpriors that induce minimax decision rules; one natural avenue is to exploit Fox-H functions \cite{mathai2009h}, which encompass many hyperprior families, including the Beta-Prime hyperprior proposed here.
\clearpage
\appendix

% Reset counters for supplement
\setcounter{section}{0}
\setcounter{subsection}{0}
\setcounter{subsubsection}{0}
\setcounter{equation}{0}
\setcounter{figure}{0}
\setcounter{table}{0}

% Supplement numbering
\renewcommand{\thesection}{S\arabic{section}}
\renewcommand{\thesubsection}{S\arabic{section}.\arabic{subsection}}
\renewcommand{\thesubsubsection}{S\arabic{section}.\arabic{subsection}.\arabic{subsubsection}}

\numberwithin{equation}{section}
\numberwithin{figure}{section}
\numberwithin{table}{section}

% Start supplement as its own container
\part*{Supplementary Material}
\etocsetlocaltop.toc{part}

% Local TOC for supplement only
\etocsettocstyle{\section*{Contents of the Supplement}}{}
\localtableofcontents
\clearpage

\section{Proof of Section 1}
\subsection{Proof of Lemma 1.1}
\begin{proof}
From \cite{zavatone2021exact} we know that \begin{align*}
   &p_{d}(\boldsymbol{h}_{d}; \sigma_{1}\dots\sigma_{d}||\boldsymbol{x}||; n_{1}, \dots, n_{d} ) = (1- \frac{(2^{n_{1}}-1)\dots(2^{n_{d-1}}-1)}{2^{n_{1} + \dots n_{d-1}}})\delta(\boldsymbol{h}_{d})\\ &+ \frac{1}{2^{n_{1}+\dots+n_{d-1}}}\sum\limits_{k_{1}=1}^{n_{1}}\dots\sum\limits_{k_{d-1} =1}^{n_{d-1}}\begin{pmatrix}
n_{1}\\
k_{1}
\end{pmatrix} \dots\begin{pmatrix}
n_{d-1}\\
k_{d-1}
\end{pmatrix}  p_{d}^{lin}(\boldsymbol{h}_{d}; \kappa_{d}; k_{1},\dots, k_{d-1}, n_{d}) 
\end{align*} 
where we will utilize the continuous component of the prior and \begin{equation*}
    p_{d}^{lin}(\boldsymbol{h}_{d}|\boldsymbol{x}) = \frac{\gamma_{d}}{(2^{d}\pi \kappa_{d}^{2})^{\frac{n_{d}}{2}}}G_{0,d}^{d,0}(\frac{||\boldsymbol{h}_{d}||^{2}}{2^{d} \kappa_{d}^{2}}|\begin{matrix}
        - \\ 
        0, \frac{n_{1}-n_{d}}{2}, \dots, \frac{n_{d-1}-n_{d}}{2}
    \end{matrix}), 
\end{equation*}    
\begin{equation*}    
    \kappa_{d}:= \sigma_{1}\dots\sigma_{d}||\boldsymbol{x}|| \text{ and } \gamma_{d} := \prod\limits_{\ell =1}^{d-1} \frac{1}{\Gamma(\frac{n_{\ell}}{2})}. 
\end{equation*} 
Fix $\boldsymbol{k} = ( k_{1},\dots, k_{d-1} )$ where $k_{\ell} \in \{ 1,\dots, n_{\ell} \}, \ell \in \{ 1, \dots, d-1 \}$. We know from the supplemental material of \cite{zavatone2021exact} that \begin{equation*}
    p_{d}^{lin}(\boldsymbol{h}_{d}|\boldsymbol{x}, \boldsymbol{k}) = \gamma_{d}(2^{d}\pi \kappa_{d}^{2})^{-\frac{n_{d}}{2}} f_{d-1}(\frac{||\boldsymbol{h}_{d}||^{2}}{2^{d}\kappa_{d}^{2}}; \nu_{1},\dots,\nu_{d-1}), 
\end{equation*} where 
\begin{equation*}
     \nu_{\ell} = \frac{k_{\ell}-p}{2} \text{ and } f_{d-1} (z; \nu_{1},\dots, \nu_{d-1}) := [\prod\limits_{\ell = 1}^{d-1} \int\limits_{0}^{\infty}dt_{\ell} t_{\ell}^{\nu_{\ell}-1}\exp \{ - t_{\ell} \}]\exp \{-\frac{z}{t_{1}\dots t_{d-1}} \}.
\end{equation*} Since for each $\ell$, $\nu_{\ell} = \frac{k_{\ell}-p}{2}$ we have $t_{\ell}^{\nu_{\ell}-1} = t_{\ell}^{\frac{k_{\ell}-p}{2}-1} = t_{\ell}^{\frac{k_{\ell}}{2}-1}t_{\ell}^{-\frac{p}{2}}$. Let $z = \frac{||\boldsymbol{h}_{d}||^{2}}{2^{d}\kappa_{d}^{2}}$. Then, 
\begin{equation*}
    p_{d}^{lin}(\boldsymbol{h}_{d}|\boldsymbol{x}, \boldsymbol{k}) = \gamma_{d}(2^{d}\pi \kappa_{d}^{2})^{- \frac{p}{2}} [\prod\limits_{\ell = 1}^{d-1} \int\limits_{0}^{\infty} dt_{\ell} t_{\ell}^{\frac{k_{\ell}}{2}-1}e^{- t_{\ell}}t_{\ell}^{-\frac{p}{2}}] \exp \{ - \frac{||\boldsymbol{h}_{d}||^{2}}{2^{d}\kappa_{d}^{2}t_{1}\dots t_{d-1}} \}. 
\end{equation*}
Observe that $\prod\limits_{\ell = 1}^{d-1} t_{\ell}^{-\frac{p}{2}}  = (t_{1}\dots t_{d-1})^{-\frac{p}{2}}$. Then, 
\begin{equation*}
    p_{d}^{lin} (\boldsymbol{h}_{d}|\boldsymbol{x}, \boldsymbol{k}) = \gamma_{d} [\prod\limits_{\ell =1}^{d-1}\int\limits_{0}^{\infty} dt_{\ell} t_{\ell}^{\frac{k_{\ell}}{2}-1} e^{-t_{\ell}}  ](2^{d}\pi \kappa_{d}^{2})^{-\frac{p}{2}} (t_{1}\dots t_{d-1})^{- \frac{p}{2}} \exp \{ - \frac{||\boldsymbol{h}_{d}||^{2}}{2^{d}\kappa_{d}^{2}t_{1}\dots t_{d-1}} \}.
\end{equation*} Recall for $\boldsymbol{k}$, $\gamma_{d} = \prod\limits_{\ell = 1}^{d-1} \Gamma(\frac{k_{\ell}}{2})^{-1}$. Therefore, 
\begin{equation*}
    p_{d}^{lin}(\boldsymbol{h}_{d}|\boldsymbol{x}, \boldsymbol{k}) = [\prod\limits_{\ell=1}^{d-1} \int_{0}^{\infty} dt_{\ell} \frac{t_{\ell}^{\frac{k_{\ell}}{2}-1} e^{-t_{\ell}}}{\Gamma(\frac{k_{\ell}}{2})}] (2^{d}\pi \kappa_{d}^{2})^{-\frac{p}{2}}(t_{1}\dots t_{d-1})^{-\frac{p}{2}} \exp \{- \frac{||\boldsymbol{h}_{d}||^{2}}{2^{d}\kappa_{d}^{2}t_{1}\dots t_{d-1}}\}, 
\end{equation*}
observing $ T_{\ell} \sim \Gamma(\frac{k_{\ell}}{2},1)$. Let $S_{k}:= \prod\limits_{\ell = 1}^{d-1} T_{\ell}$ and $U:= 2^{d-1} \kappa_{d}^{2}$. Then $2^{d} \kappa_{d}^{2} S_{k} = 2US_{k}$. So, $(2^{d} \pi \kappa_{d}^{2})^{-\frac{p}{2}}S_{k}^{-\frac{p}{2}} = (2\pi U S_{k})^{-\frac{p}{2}}$. This implies, 
\begin{equation*}
    \exp \{ - \frac{||\boldsymbol{h_{d}}||^{2}}{2^{d} \kappa_{d}^{2}S_{k}} \} = \exp \{ - \frac{||\boldsymbol{h}_{d}||^{2}}{2US_{k}} \}. 
\end{equation*}
Therefore, 
\begin{equation*}
    (2\pi US_{k})^{-\frac{p}{2}} \exp \{ - \frac{||\boldsymbol{h}_{d}||^{2}}{2US_{k}} \} = \phi_{p}(\boldsymbol{h}_{d}, \boldsymbol{0}_{p}, US_{k}I_{p}). 
\end{equation*}
Hence, 
\begin{align*}
    p_{d}^{lin}(\boldsymbol{h}_{d}|\boldsymbol{x}, \boldsymbol{k}) &= \int\limits_{0}^{\infty} \dots \int\limits_{0}^{\infty} \phi_{p}(\boldsymbol{h}_{d}; \boldsymbol{0}_{p}, (U\prod\limits_{\ell = 1}^{d-1}t_{\ell})I_{p}) \prod\limits_{\ell = 1}^{d-1} \frac{t_{\ell}^{\frac{k_{\ell}}{2}-1}e^{-t_{\ell}}}{\Gamma(\frac{k_{\ell}}{2})} dt_{1}\dots dt_{d-1} \\ 
    &= \mathbb{E}[\phi_{p}(\boldsymbol{h}_{d}; \boldsymbol{0}_{p}, US_{k} I_{p})]. 
\end{align*}
That is, 
\begin{equation*}
    \boldsymbol{h}_{d}|S_{k} , \boldsymbol{k} \sim N_{p}(\boldsymbol{0}_{p}, US_{k}I_{p}), \text{ where } S_{k} = \prod\limits_{\ell = 1}^{d-1} T_{\ell} , T_{\ell} \sim \Gamma(\frac{k_{\ell}}{2},1). 
\end{equation*}
Let $f_{S_{k}}$ denote the density of $S_{k}$. Then, 
\begin{equation*}
    p_{d}^{lin}(\boldsymbol{h}_{d}|\boldsymbol{x}, \boldsymbol{k}) = \int\limits_{0}^{\infty} \phi_{p}(\boldsymbol{h}_{d}, \boldsymbol{0}_{p}, UsI_{p}) f_{S_{k}}(s)ds. 
\end{equation*}
Let $V_{k} := US_{k}$, then $V_{k}>0$ and its density is 
\begin{equation*}
    g_{k}(v) = \frac{1}{U} f_{S_{k}} (\frac{v}{U}), v > 0. 
\end{equation*}
Then, 
\begin{equation*}
    p_{d}^{lin}(\boldsymbol{h}_{d}|\boldsymbol{x}, \boldsymbol{k}) = \int\limits_{0}^{\infty} \phi_{p}(\boldsymbol{h}_{d}; \boldsymbol{0}_{p}, vI_{p})g_{k}(v) dv. 
\end{equation*}
So, 
\begin{equation*}
    \pi(\boldsymbol{h}_{d}) = \sum\limits_{k}w_{k}p_{d}^{lin}(\boldsymbol{h}_{d}| \boldsymbol{x}, \boldsymbol{k}), \text{ where } w_{k} = \frac{1}{2^{n_{1}+\dots+n_{d-1}}}\prod\limits_{\ell = 1}^{d-1}\binom{n_{\ell}}{k_{\ell}}, k_{\ell} \in \{ 1, \dots, n_{\ell} \}. 
\end{equation*}
Then 
\begin{equation*}
    \pi(\boldsymbol{h}_{d}) = \sum\limits_{k}w_{k} \int\limits_{0}^{\infty} \phi_{p}(\boldsymbol{h}_{d}; \boldsymbol{0}_{p}, vI_{p}) g_{k}(v) dv. 
\end{equation*}
Since the sum is finite, by the linearity of integration, 
\begin{equation*}
    \pi(\boldsymbol{h}_{d}) = \int \limits_{0}^{\infty} \phi_{p} (\boldsymbol{h}_{d}; \boldsymbol{0}_{p}, vI_{p}) (\sum\limits_{k}w_{k}g_{k}(v))dv. 
\end{equation*}
So 
\begin{equation*}
    \pi (\boldsymbol{h}_{d}) = \int\limits_{0}^{\infty} \phi_{p} (\boldsymbol{h}_{d}; \boldsymbol{0}_{p}, vI_{p}) g(v) dv,
\end{equation*}
where $g(v) = \sum\limits_{k} w_{k} g_{k}(v)$ is the finite mixture of the laws of $V_{k}$. 
\end{proof}
\section{Proofs of Section 2 }
We now use $\boldsymbol{\theta}$ instead of $\boldsymbol{h}_{d}$ in view of Section 1.4.
\subsection{Proof of Lemma 2.1 }
\begin{proof}
    Let $r = (x_{1}^{2}+\dots x_{p}^{2})^{\frac{1}{2}}$ and $S = \sum\limits_{i=1}^{p}x_{i}^{2}$. Then, \begin{equation*}
        \frac{\partial r}{\partial x_{i}} = \frac{1}{2}S^{-\frac{1}{2}}\cdot 2x_{i} = \frac{x_{i}}{S^{\frac{1}{2}}} = \frac{x_{i}}{r}. 
    \end{equation*}So, \begin{equation*}
        \frac{\partial u}{\partial x_{i}} = \frac{\partial \phi}{\partial r}\frac{\partial r}{\partial x_{i}} = \phi'(r)\frac{x_{i}}{r}. 
    \end{equation*} Similarly, 
    \begin{align*}
        \frac{\partial^{2}u}{\partial x_{i}^{2}} &= \frac{\partial }{\partial x_{i}} \phi'(r)\frac{x_{i}}{r}\\ 
        & = \frac{x_{i}}{r}\frac{\partial}{\partial x_{i}} \phi'(r) + \phi'(r) \frac{\partial}{\partial x_{i}} \frac{x_{i}}{r}\\
        & = \frac{x_{i}}{r}\frac{\partial \phi'}{\partial r}\frac{\partial r}{\partial x_{i}} + \phi'(r) \frac{\partial}{\partial x_{i}}(x_{i}r^{-1})\\
        &= \frac{x_{i}}{r}\phi''(r)\frac{x_{i}}{r} + \phi'(r) [\frac{1}{r} + x_{i}\frac{\partial}{\partial x_{i}}\frac{1}{r}]\\ 
        &= \frac{x_{i}^{2}}{r^{2}}\phi''(r) + \phi'(r) [\frac{1}{r} + x_{i}(-r^{-2})\frac{x_{i}}{r}]\\
        &= \frac{x_{i}^{2}}{r^{2}} \phi''(r) + \phi'(r) [\frac{1}{r} - \frac{x_{i}^{2}}{r^{3}}]. 
    \end{align*}Then, \begin{align*}
        \Delta u &= \sum\limits_{i=1}^{p} \frac{\partial^{2}u}{\partial x_{i}^{2}} \\ 
        & = \sum\limits_{i=1}^{p}\frac{x_{i}^{2}}{r^{2}}\phi''(r) + \phi'(r) [\frac{1}{r} - \frac{x_{i}^{2}}{r^{3}}]\\
        &= \frac{\phi''(r)}{r^{2}} \sum\limits_{i=1}^{p}x_{i}^{2} + \phi'(r) [\sum\limits_{i=1}^{p}\frac{1}{r} - \frac{1}{r^{3}}\sum\limits_{i=1}^{p}x_{i}^{2}]\\
        & = \frac{\phi''(r)}{r^{2}}r^{2}  + \phi'(r) [\frac{p}{r} - \frac{1}{r^{3}}r^{2}]\\
        &= \phi''(r) + \phi'(r) [\frac{p}{r} - \frac{1}{r}]\\ 
        & = \phi''(r) +(\frac{p-1}{r})\phi'(r). 
    \end{align*}
\end{proof}
\subsection{Proof of Lemma 2.2}
\begin{proof}
Let $w(r) := r^{p-1}q'(r)$. Then, \\
\begin{align*}
    w'(r) &= \frac{d}{dr} \{ r^{p-1} q'(r) \}\\
    &= q'(r) (p-1)r^{p-2} + r^{p-1}q''(r)\\
    &= r^{p-1}(q''(r) + \frac{p-1}{r}q'(r))\\
    &=r^{p-1}\Delta q(r) \leq 0. 
\end{align*}
Let $R_{0}$ be such that $\Delta q(r) \leq 0, \forall r \geq R_{0}$. Then, $w'(r) \leq 0$ for $r \geq R_{0}$. That is $w$ is non-increasing on $[R_{0}, \infty )$. \\
\\ 
Now assume $w(r) \geq 0, \forall r \geq R_{0}$. Then,
\begin{equation*}
    q'(r) = \frac{w(r)}{r^{p-1}} \geq 0 ,\forall r\geq R_{0}. 
\end{equation*}So, q is non-decreasing on $[R_{0}, \infty )$. But, $q(R_{0})>0$ which would imply $q(r) \geq q(R_{0}) >0, \forall r \geq R_{0}$. This would mean $q(r) \not\rightarrow 0$, a contradiction. \\
\\
Therefore, $\exists R \geq R_{0}$ with $w(R) < 0$ which implies $q'(R) <0$. This means $w(r) \leq w(R)<0$ implies \begin{equation*}
    q'(r) = \frac{w(r)}{r^{p-1}}<0. 
\end{equation*}
Now, $\forall s \geq R$ we know $w$ is non-increasing. That is, $w(s) \leq w(R)$. This implies, \begin{equation*}
    q'(s) = \frac{w(s)}{s^{p-1}} \leq \frac{w(R)}{s^{p-1}}. 
\end{equation*} Then by  the fundamental theorem of calculus, $q(T) - q(r) = \int\limits_{r}^{T} q'(s) ds$, which implies $- q(r) = \int\limits_{r}^{T} q'(s) ds - q(T)$. Now, taking the limit $T \rightarrow \infty$ gives
\begin{align*}
    - q(r) &= \lim\limits_{T \rightarrow \infty} \{ \int \limits_{r}^{T}q'(s)ds - q(T) \} \\ 
    &= \int\limits_{r}^{\infty}q'(s) ds\\
    & \leq \int\limits_{r}^{\infty} \frac{w(R)}{s^{p-1}}ds\\
    & = w(R) \int\limits_{r}^{\infty} s^{1-p}ds\\
    & = w(R)[\frac{s^{2-p}}{2-p}]_{s=r}^{\infty}\\ 
    & = w(R) \lim\limits_{S \rightarrow \infty} [\frac{S^{2-p}}{2-p} - \frac{r^{2-p}}{2-p}]\\ 
    & = - w(R) \frac{r^{2-p}}{2-p}\\
    & = w(R) \frac{r^{2-p}}{p-2}. 
\end{align*}
That is, 
\begin{align*}
    -q(r) \leq \frac{w(R)}{p-2}r^{2-p} \iff q(r) &\geq -\frac{w(R)}{p-2}r^{2-p}\\
    &= - \frac{R^{p-1}q'(R)}{p-2}r^{2-p}\\
    &=cr^{2-p}, c:= - \frac{R^{p-1}q'(R)}{p-2}.
\end{align*}
\end{proof}
\subsection{Proof of Lemma 2.3}
\begin{proof}
    Let $\boldsymbol{y} \in \mathbb{R}^{p}$ and $r = ||\boldsymbol{y}||$. Then the marginal density is given by, 
    \begin{equation*}
        m(\boldsymbol{y})  = \int\limits_{||\boldsymbol{\theta}|| \leq \frac{r}{2}} \phi_{p}(\boldsymbol{y}|\boldsymbol{\theta})\pi(\boldsymbol{\theta})d\boldsymbol{\theta} + \int\limits_{||\boldsymbol{\theta}||> \frac{r}{2}} \phi_{p}(\boldsymbol{y}|\boldsymbol{\theta})\pi(\boldsymbol{\theta}) d\boldsymbol{\theta}. 
    \end{equation*}
Note that if $||\boldsymbol{\theta}|| \leq \frac{r}{2}$. Then, 
\begin{equation*}
    ||\boldsymbol{y}- \boldsymbol{\theta}|| \geq |||\boldsymbol{y}|| - ||\boldsymbol{\theta}||| \geq \frac{r}{2}, \text{ by the reverse triangle inequality}.
\end{equation*}
Therefore, $\phi_{p}(\boldsymbol{y}|\boldsymbol{\theta}) \leq (2\pi)^{- \frac{p}{2}} \exp \{ - \frac{1}{2}(\frac{r}{2})^{2}  \} = (2\pi)^{- \frac{p}{2}} \exp \{ - \frac{r^{2}}{8} \}$. This implies,  
\begin{align*}
    \int\limits_{||\boldsymbol{\theta}|| \leq \frac{r}{2}} \phi_{p}(\boldsymbol{y}|\boldsymbol{\theta}) \pi(\boldsymbol{\theta}) d\boldsymbol{\theta} &\leq (2\pi)^{- \frac{p}{2}} \exp\{- \frac{r^{2}}{8} \}\int\limits_{||\boldsymbol{\theta}|| \leq \frac{r}{2}} \pi(\boldsymbol{\theta}) d\boldsymbol{\theta}\\ 
    & \leq (2\pi)^{- \frac{p}{2}} \exp \{ - \frac{r^{2}}{8} \}\cdot1, \text{ since} \int\limits_{||\boldsymbol{\theta}|| \leq \frac{r}{2}} \pi(\boldsymbol{\theta}) d\boldsymbol{\theta} \leq1 \\ 
    & = (2\pi)^{- \frac{p}{2}}\exp \{ - \frac{r^{2}}{8} \}.
\end{align*}Similarly, 
\begin{align*}
    \int\limits_{||\boldsymbol{\theta}||> \frac{r}{2}} \phi_{p}(\boldsymbol{y} | \boldsymbol{\theta}) \pi(\boldsymbol{\theta})d \boldsymbol{\theta} &\leq \int\limits_{||\boldsymbol{\theta}||> \frac{r}{2}}(2\pi)^{- \frac{p}{2}}\pi(\boldsymbol{\theta})d\boldsymbol{\theta} \\
    & = (2\pi)^{- \frac{p}{2}}\int\limits_{||\boldsymbol{\theta}||> \frac{r}{2}} \pi(\boldsymbol{\theta}) d\boldsymbol{\theta}\\ 
    &= (2\pi)^{- \frac{p}{2}} \mathbb{P}_{\pi}(||\boldsymbol{\theta}||> \frac{r}{2}).
\end{align*} 
Thus, \begin{equation*}
    m(\boldsymbol{y}) \leq (2\pi)^{-\frac{p}{2}}e^{- \frac{r^{2}}{8}}+ (2\pi)^{- \frac{p}{2}}\mathbb{P}_{\pi}(||\boldsymbol{\theta}||> \frac{r}{2}). 
\end{equation*}
Now it remains to bound $\mathbb{P}_{\pi}(||\boldsymbol{\theta}||> \frac{r}{2})$. 
From Appendix A of \cite{gaunt2025variance} we know that
\begin{equation}
    G_{p,q}^{q,0}(x| \begin{matrix}
a_{1}, \dots, a_{p}\\
b_{1},\dots, b_{q}\end{matrix}) \sim \frac{(2\pi)^{\frac{\sigma-1}{2}}}{\sigma^{\frac{1}{2}}}x^{{\theta}}\exp \{ - \sigma x^{\frac{1}{\sigma}} \} \text{ as } x \rightarrow \infty, 
\end{equation}
where $\sigma = q-p$ and ${\theta} = \sigma^{-1} \{ \frac{1-\sigma}{2} + \sum\limits_{i=1}^{q}b_{i} - \sum\limits_{i=1}^{p}a_{i} \}$. Applying this result to a given Meijer-G function in the prior density gives, 
\begin{equation*}
    G_{0,d}^{d,0}(\frac{||\boldsymbol{\theta}||^{2}}{B}|\begin{matrix}
-\\
0, \frac{k_{1}-n_{d}}{2}, \dots, \frac{k_{d-1}-n_{d}}{2})
\end{matrix}) \sim (2\pi)^{\frac{d-1}{2}}d^{-\frac{1}{2}}t^{\mu}\exp\{ -dt^{\frac{1}{d}} \}, \text{ as } t \rightarrow \infty 
\end{equation*}\begin{equation*}
t = \frac{||\boldsymbol{\theta}||^{2}}{B}, \mu = \frac{1}{d}(\frac{1-d}{2} + \sum\limits_{j=1}^{d-1}\frac{k_{j}-n_{d}}{2}), \text{ and } B= 2^{d}\sigma_{1}^{2}\dots\sigma_{d}^{2}||\boldsymbol{x}||^{2}. 
\end{equation*}
That is, 
\begin{equation*}
    \frac{G_{0,d}^{d,0}(t|\begin{matrix}
-\\
0, \frac{k_{1}-n_{d}}{2}, \dots, \frac{k_{d-1}-n_{d}}{2}
\end{matrix})}{(2\pi)^{\frac{d-1}{2}}d^{-\frac{1}{2}}t^{\mu}\exp\{ -dt^{\frac{1}{d}} \}} \rightarrow 1, \text{ as } t \rightarrow \infty. 
\end{equation*}
By the definition of limit, 
\begin{equation*}
    |\frac{f(t)}{g(t)} - 1| < \epsilon. 
\end{equation*}
Let $\epsilon = 1$. Then, 
\begin{equation*}
    -1 < \frac{f(t)}{g(t)} -1 <1.
\end{equation*} This implies, 
\begin{equation*}
    0 < \frac{f(t)}{g(t)}<2. 
\end{equation*}Hence, eventually $f(t) \leq 2g(t)$. 
Therefore, we have \begin{align*}
    G_{0,d}^{d,0}(t|\begin{matrix}
-\\
0, \frac{k_{1}-n_{d}}{2}, \dots, \frac{k_{d-1}-n_{d}}{2}
\end{matrix}) &\leq 2Ct^{\mu}\exp\{ -dt^{\frac{1}{d}} \}, , \text{ where } C= (2\pi)^{\frac{d-1}{2}}d^{- \frac{1}{2}}\\ 
& = Kt^{\mu}\exp\{-d{t}^{\frac{1}{d}} \}, K =2C \\
&= K(\frac{||\boldsymbol{\theta}||^{2}}{B})^{\mu}\exp \{ - d(\frac{||\boldsymbol{\theta}||^{2}}{B})^{\frac{1}{d}} \}\\ 
& = KB^{- \mu}||\boldsymbol{\theta}||^{2\mu} \exp \{ - dB^{-\frac{1}{d}}||\boldsymbol{\theta}||^{\frac{2}{d}} \}\\ 
&= K_{\mu} ||\boldsymbol{\theta}||^{2\mu}\exp \{ -c||\boldsymbol{\theta}||^{\frac{2}{d}} \}, c := dB^{- \frac{1}{d}}, K_{\mu}:= KB^{-\mu}. 
\end{align*}
Substituting this upper bound into the prior density gives, 
\begin{equation*}
    \pi(\boldsymbol{\theta}) \leq \frac{1}{2^{n_{1} + \dots n_{d-1}}}\sum\limits_{k_{1}=1}^{n_{1}}\dots\sum\limits_{k_{d-1}=1}^{n_{d-1}} \prod\limits_{\ell = 1}^{d-1}\begin{pmatrix}
        n_{\ell}\\
        k_{\ell}
    \end{pmatrix}\frac{\prod\limits_{\ell = 1}^{d-1} \frac{1}{\Gamma(\frac{k_{\ell}}{2})}}{(\pi B)^{\frac{n_{d}}{2}}}K_{\mu(k_{1},..., k_{d-1})}||\boldsymbol{\theta}||^{2\mu(k_{1},..k_{d-1})}\exp \{ -c||\boldsymbol{\theta}||^{\frac{2}{d}} \}. 
\end{equation*}
Let $\mu_{\text{max}} = \max\limits_{1\leq k_{1}\leq n_{1}, \dots, 1\leq k_{d-1} \leq n_{d-1}} \mu(k_{1},\dots, k_{d-1}) = \mu(n_{1}, \dots ,n_{d-1}) $. Since $t \rightarrow \infty$, we know $\boldsymbol{\theta} \rightarrow \infty$, so $||\boldsymbol{\theta}||$ is large. Therefore, $||\boldsymbol{\theta}||^{2\mu_{\text{max}}}$ is increasing in $\mu$. Since $\mu(k_{1},\dots, k_{d-1}) \leq \mu_{\text{max}}$ for every possible $k_{1}, \dots ,k_{d}, ||\boldsymbol{\theta}||^{2\mu(k_{1},..., k_{d-1})} \leq ||\boldsymbol{\theta}||^{2\mu_{\text{max}}}$. Therefore, 
\begin{align*}
    \pi(\boldsymbol{\theta}) &\leq ||\boldsymbol{\theta}||^{2 \mu_{\text{max}}} \exp \{ - c||\boldsymbol{\theta}||^{\frac{2}{d}} \}\\ &\times \frac{1}{2^{n_{1} + \dots n_{d-1}}} 
     \sum\limits_{k_{1} = 1}^{n_{1}}\dots\sum\limits_{k_{d-1}}^{n_{d-1}}\begin{pmatrix}
n_{1}\\
k_{1}
\end{pmatrix} \dots\begin{pmatrix}
n_{d-1}\\
k_{d-1}
\end{pmatrix}K_{\mu(k_{1}, \dots, k_{d-1})} \frac{\prod\limits_{\ell = 1}^{d-1}\frac{1}{\Gamma(\frac{k_{\ell}}{2})}}{(2^{d}\pi\sigma_{1}^{2}\dots\sigma_{d}^{2}||\boldsymbol{x}||^{2})^{\frac{n_{d}}{2}}}\\
&= C_{1}||\boldsymbol{\theta}||^{2\mu_{\text{max}}}\exp\{-c||\boldsymbol{\theta}||^{\frac{2}{d}} \}, \text{ where} 
\end{align*}
\begin{equation*}
    C_{1} := \frac{1}{2^{n_{1} + \dots n_{d-1}}} \sum\limits_{k_{1} = 1}^{n_{1}}\dots\sum\limits_{k_{d-1}}^{n_{d-1}}\begin{pmatrix}
n_{1}\\
k_{1}
\end{pmatrix} \dots\begin{pmatrix}
n_{d-1}\\
k_{d-1}
\end{pmatrix} K_{\mu(k_{1}, \dots, k_{d-1})}\frac{\prod\limits_{\ell = 1}^{d-1}\frac{1}{\Gamma(\frac{k_{\ell}}{2})}}{(\pi B)^{\frac{n_{d}}{2}}}
\end{equation*}

Now, let $\boldsymbol{\theta} = \rho \boldsymbol{\omega}$, where $\rho = ||\boldsymbol{\theta}|| \in (0, \infty), \boldsymbol{\omega} = \frac{\boldsymbol{\theta}}{||\boldsymbol{\theta}||} \in \mathcal{S}^{p-1}$,the unit sphere with $d\boldsymbol{\theta} = \rho^{p-1} d \boldsymbol{\omega} d \rho$ where $d\boldsymbol{\omega}$ is the surface measure on $\mathcal{S}^{p-1}$. Then the integral becomes \begin{align*}
    \mathbb{P}_{\pi}(||\boldsymbol{\theta}||>u) &= \int\limits_{u}^{\infty}\int_{\mathcal{S}^{p-1}} \pi(\rho\boldsymbol{\omega})\rho^{p-1} d\boldsymbol{\omega} d \rho.\\
    &= \int\limits_{u}^{\infty} \rho^{p-1} \int\limits_{\mathcal{S}^{p-1}}\pi(\rho \boldsymbol{\omega}) d\boldsymbol{\omega} d\rho\\
    & = \int\limits_{u}^{\infty} \rho ^{p-1} \tilde{\pi}(\rho) \int\limits_{\mathcal{S}^{p-1}}d\boldsymbol{\omega} d\rho, \text{ as } \tilde{\pi}(\rho) = \pi(\rho \boldsymbol{\omega}) \text{ is radial}  \\
     &= |\mathcal{S}^{p-1}| \int\limits_{u}^{\infty} \rho^{p-1} \tilde{\pi}(\rho ) d \rho\\
    &\leq |\mathcal{S}^{p-1}|C_{1} \int\limits_{u}^{\infty} \rho^{p-1}\rho^{2\mu_{\text{max}}}\exp \{ -c\rho^{\frac{2}{d}} \} d\rho\\
    &= K_{2} \int\limits_{u}^{\infty}\rho^{p-1+2\mu_{\text{max}}}\exp \{ -c\rho^{\frac{2}{d}} \} d\rho, \text{ where } K_{2} = |\mathcal{S}^{p-1}|C_{1}. 
\end{align*}
Let $\nu = \rho^{\frac{2}{d}} \implies \rho = \nu^{\frac{d}{2}} \implies \frac{d\rho}{d\nu} =\frac{d}{2}\nu^{\frac{d}{2}-1} \implies d\rho = \frac{d}{2}\nu^{\frac{d}{2}-1} d\nu$. Then, 
\begin{align*}
    \mathbb{P}_{\pi}(||\boldsymbol{\theta}||> u) &\leq  K_{2} \int\limits_{u^{\frac{2}{d}}}^{\infty}(\nu^{\frac{d}{2}})^{p-1+2\mu_{\text{max}}}\exp\{-c(\nu^{\frac{d}{2}})^{\frac{2}{d}} \}\frac{d}{2}\nu^{\frac{d}{2}-1}d\nu\\
    & = K_{3} \int\limits_{u^{\frac{2}{d}}}^{\infty} \nu^{\frac{d}{2}(p-1+2\mu_{\text{max}})}\exp\{-c \nu\}\nu^{\frac{d}{2}-1}d\nu, K_{3}= \frac{K_{2}d}{2}\\
    &= K_{3}\int\limits_{u^{\frac{2}{d}}}^{\infty}\nu^{\frac{d}{2}(p+2\mu_{\text{max}})-1}\exp\{ - c\nu \}d\nu\\ 
    &= K_{3}\int\limits_{u^{\frac{2}{d}}}^{\infty}\nu^{A-1}\exp\{-c\nu \} d\nu, A := \frac{d}{2}(p+2\mu_{\max})\\
    &\leq K_{3}C_{A}(u^{\frac{2}{d}})^{A-1}\exp\{ - cu^{\frac{2}{d}} \}, \text{ for } C_{A}>0 \text{ and large } u^{\frac{2}{d}}, 
\end{align*}
since under such conditions, 
\begin{equation*}
    \int\limits_{x}^{\infty} \nu^{A-1} e^{-c\nu}d\nu \leq C_{A}x^{A-1}e^{-cx}. 
\end{equation*} Particularly, let $I(x) := \int\limits_{x}^{\infty} \nu^{A-1}e^{-c\nu} d\nu, A>1, c>0$. Let \begin{equation*}
    u = \nu^{A-1} \text{ and } \frac{dv}{d\nu} = e^{-c\nu}. 
\end{equation*}
Then, \begin{equation*}
    du = (A-1)\nu^{A-2} d\nu \text{ and } v = \int e^{-c\nu} d\nu = - \frac{1}{c} e^{-c\nu}. 
\end{equation*} Therefore, 
\begin{align*}
    I(x) &= [-\frac{1}{c} \nu ^{A-1} e^{-c \nu}]_{x}^{\infty} + \frac{A-1}{c} \int\limits_{x}^{\infty}\nu^{A-2}e^{-c\nu}d\nu\\ 
    & = \frac{1}{c} x^{A-1}e^{-cx} + \frac{A-1}{c} \int\limits_{x}^{\infty} \nu^{A-2} e^{-c\nu} d\nu. 
\end{align*}
For $\nu > x, \nu^{A-2} = \nu^{-1}\nu^{A-1} \leq x^{-1}\nu^{A-1}$. This implies \begin{equation*}
    \int\limits_{x}^{\infty} \nu ^{A-2} e^{-c\nu} d\nu \leq \frac{1}{x} \int\limits_{x}^{\infty} \nu^{A-1}e^{-c \nu} d\nu = \frac{1}{x}I(x). 
\end{equation*} Therefore, \begin{equation*}
    I(x) \leq \frac{1}{c} x^{A-1} e^{-cx} + \frac{A-1}{c} \cdot \frac{1}{x}I(x). 
\end{equation*} Hence, \begin{equation*}
    (1- \frac{A-1}{cx}) I(x) \leq \frac{1}{c} x^{A-1}e^{-cx}. 
\end{equation*} So for large $x$, $I(x) \leq C_{A}x^{A-1} e^{-cx}$ for some $C_{A}>0$. Therefore, 
\begin{equation*}
    \mathbb{P}_{\pi}(||\boldsymbol{\theta}||> u) \leq K_{4} u^{\frac{2}{d}(A-1)}\exp\{ -c u^{\frac{2}{d}}\}, K_{4}:= K_{3}C_{A}.
\end{equation*}
Now, 
\begin{align*}
    u^{\frac{2}{d}(A-1)}\exp\{ -cu^{\frac{2}{d}} \} &= \exp\{ \log(u^{\frac{2}{d}(A-1)}) \}\exp\{-cu^{\frac{2}{d}}\}\\
    &= \exp \{ \frac{2}{d}(A-1) \log(u)-cu^{\frac{2}{d}}\}. 
\end{align*}
Note that \begin{equation*}
    \frac{\log(u)}{u^{\frac{2}{d}}} \rightarrow 0 \text{ as } u \rightarrow \infty. 
\end{equation*}
Fix $\epsilon \in (0,c)$ and let \begin{equation*}
    \delta := \frac{\epsilon}{\frac{2}{d}(A-1)}. 
\end{equation*} By the definition of limit \begin{align*}
    \frac{\log(u)}{u^{\frac{2}{d}}} \leq \delta &\iff  \frac{2}{d}(A-1)\frac{\log(u)}{u^{\frac{2}{d}}} \leq \epsilon\\ 
    &\iff \frac{2}{d}(A-1) \log(u) \leq \epsilon u^{\frac{2}{d}}\\ 
    & \iff - \frac{2}{d}(A-1)\log(u) \geq -\epsilon u^{\frac{2}{d}}\\ 
    & \iff cu^{\frac{2}{d}} - \frac{2}{d}(A-1)\log(u) \geq c u^{\frac{2}{d}} - \epsilon u^{\frac{2}{d}} \\ 
    & \iff cu^{\frac{2}{d}} - \frac{2}{d} (A-1) \log(u) \geq (c-\epsilon)u ^{\frac{2}{d}}. 
\end{align*}
So, $\forall u \geq u_{0}$ for some $u_{0}$, \begin{equation*}
    cu^{\frac{2}{d}} - \frac{2}{d}(A-1) \log(u) \geq (c- \epsilon)u^{\frac{2}{d}}. 
\end{equation*} Therefore, 
\begin{equation*}
    \frac{2}{d}(A-1)\log(u) - cu^{\frac{2}{d}}  
 \leq -(c-\epsilon)u^{\frac{2}{d}}.     
\end{equation*}
Thus, \begin{equation*}
    u^{\frac{2}{d}(A-1)}\exp\{-cu^{\frac{2}{d}} \} \leq \exp \{ - (c-\epsilon)u^{\frac{2}{d}} \}, u \geq u_{0}. 
\end{equation*}
This implies, \begin{equation*}
    \mathbb{P}_{\pi} (||\boldsymbol{\theta}||>u) \leq K_{4} \exp\{ -c'u^{\frac{2}{d}} \}, \text{ where } c'= c-\epsilon.
\end{equation*} This gives, 
\begin{align*}
    \mathbb{P}_{\pi}(||\boldsymbol{\theta}||> \frac{r}{2}) &\leq K_{4}\exp \{ -c'(\frac{r}{2})^{\frac{2}{d}} \}\\
    &= K_{4} \exp \{ -c''r^{\frac{2}{d}} \}, c'' = c'2^{-\frac{2}{d}}. 
\end{align*} Then, for large $r$, 
\begin{equation*}
    m(\boldsymbol{y}) \leq (2\pi)^{- \frac{p}{2}}e^{-\frac{r^{2}}{8}} + (2\pi)^{-\frac{p}{2}}K_{4}\exp \{-c''r^{\frac{2}{d}}\}. 
\end{equation*}
For $r,d\geq1, e^{-\frac{r^{2}}{8}} \leq e^{-\frac{r^{\frac{2}{d}}}{8}}$. This implies, 
\begin{equation*}
    m(\boldsymbol{y}) \leq (2\pi)^{-\frac{p}{2}}\exp\{- \frac{r^{\frac{2}{d}}}{8} \} + (2\pi)^{-\frac{p}{2}}K_{4}\exp \{ -c''r^{\frac{2}{d}} \}
\end{equation*}
Let $\kappa:= \min \{ \frac{1}{8} , c'' \}$. Then, 
\begin{align*}
    m(\boldsymbol{y}) &\leq (2\pi)^{-\frac{p}{2}} \exp \{ -\kappa r^{\frac{2}{d}}  \} + (2\pi)^{-\frac{p}{2}}K_{4}\exp \{ - \kappa r^{\frac{2}{d}} \}\\
    & = [(2\pi)^{-\frac{p}{2}} + (2\pi)^{-\frac{p}{2}}K_{4}] \exp \{ - \kappa r^{\frac{2}{d}} \}\\
    & = K_{5}\exp\{ - \kappa r^{\frac{2}{d}} \}, K_{5}:= (2\pi)^{-\frac{p}{2}} + (2\pi)^{-\frac{p}{2}}K_{4}
\end{align*}
\end{proof}
\subsection{Proof of Theorem 2.4}
\begin{proof}
    Let $q := \sqrt{m(\boldsymbol{y})}$. Assume $q$ is superharmonic. \\
    Since $q$ is radial, this implies $\exists R >0$ such that $\forall r \geq R$ \begin{equation*}
        Q''(r) + \frac{p-1}{r}Q'(r) \leq 0, Q(r) = q(\boldsymbol{y}), \text{ where } r = ||\boldsymbol{y}||, \text{ in view of Lemma 2.2.}
    \end{equation*}
Note $Q(r) \rightarrow 0$ as $r \rightarrow \infty$. Therefore, we can apply Lemma 2.2. That is, $\exists c>0$ such that $\forall r \geq R, Q(r) \geq cr^{2-p}$. By lemma 2.3, $\exists C >0, \kappa >0, R_{0} >0$ such that $\forall r \geq R_{0}, Q(r) \leq \sqrt{C}\exp \{ - \frac{\kappa}{2} r^{\frac{2}{d}} \}$. Thus, 
\begin{equation*}
\forall r \geq \max \{ R, R_{0} \}, cr^{2-p} \leq \sqrt{C} \exp \{ - \frac{\kappa}{2} r^{\frac{2}{d}} \} \text{ iff }    \log(c) + (2-p)\log(r) \leq \log(\sqrt{C}) - \frac{\kappa}{2} r^{\frac{2}{d}}. 
\end{equation*}That is, 
\begin{equation*}
    \frac{\kappa}{2} r^{\frac{2}{d}} \leq \log(\frac{\sqrt{C}}{c}) + (p-2)\log(r). 
\end{equation*} This gives, 
\begin{equation*}
    \frac{\kappa}{2} \leq \frac{\log(\frac{\sqrt{C}}{c})}{r^{\frac{2}{d}}} + (p-2)\frac{\log(r)}{r^{\frac{2}{d}}}. 
\end{equation*}
As $r \rightarrow \infty$, $\frac{\log(\frac{\sqrt{C}}{c})}{r^{\frac{2}{d}}} \rightarrow 0$ and $\lim\limits_{r \rightarrow \infty} \frac{\log(r)}{r^{\frac{2}{d}}} = \lim\limits_{r\rightarrow \infty} \frac{d}{2r^{\frac{2}{d}}} \rightarrow 0$, by L'Hôpital's rule. Therefore, for sufficiently large r, 
\begin{equation*}
    \frac{\log(\frac{\sqrt{C}}{c})}{r^{\frac{2}{d}}} + (p-2)\frac{\log(r)}{r^{\frac{2}{d}}} < \frac{ \kappa}{2}. 
\end{equation*} This is a contradiction. Therefore $q$ is not superharmonic. 
\end{proof}
\subsection{Proof of Theorem 2.5}
\begin{proof}
Under quadratic loss, the Bayes estimator is the posterior mean. So, 
\begin{equation*}
    \delta_{\text{BNN,fixed}} (\boldsymbol{y}) = \mathbb{E}[\boldsymbol{\theta}|\boldsymbol{Y} = \boldsymbol{y}]. 
\end{equation*}
Thus it suffices to compute the posterior mean. Fix $V=v$. Then the model becomes 
\begin{equation*}
    \boldsymbol{Y}|\boldsymbol{\theta} \sim N_{p}(\boldsymbol{\theta}, I_{p}) \text{ and } \boldsymbol{\theta}|V=v \sim N_{p}(\boldsymbol{0}_{p}, vI_{p}). 
\end{equation*}
Hence, 
\begin{align*}
    p(\boldsymbol{\theta} | \boldsymbol{y}, v) &= \frac{p(\boldsymbol{y}, \boldsymbol{\theta}|v)}{p(\boldsymbol{y}|v)}\\ 
    &= \frac{p(\boldsymbol{y}|\boldsymbol{\theta},v)p(\boldsymbol{\theta}|v)}{p(\boldsymbol{y}|v)}\\
    &\propto p(\boldsymbol{y}|\boldsymbol{\theta})p(\boldsymbol{\theta}|v)\\ 
    &\propto \exp \{ - \frac{1}{2}||\boldsymbol{y} - \boldsymbol{\theta}||^{2} \}\exp\{ - \frac{1}{2v}||\boldsymbol{\theta}||^{2} \}\\ 
    &= \exp \{ -\frac{1}{2}||\boldsymbol{y} - \boldsymbol{\theta}||^{2} - \frac{1}{2v}||\boldsymbol{\theta}||^{2} \}\\
    &= \exp \{ - \frac{1}{2}[||\boldsymbol{y}||^{2} - 2\boldsymbol{y}^{T}\boldsymbol{\theta} + (1+\frac{1}{v})||\boldsymbol{\theta}||^{2}] \}. 
\end{align*}
Now complete the square. Let $a:= 1+v^{-1}$. Then the expression becomes 
\begin{equation*}
    a||\boldsymbol{\theta}||^{2}- 2\boldsymbol{y}^{T}\boldsymbol{\theta} + ||\boldsymbol{y}||^{2}. 
\end{equation*}
We want to write this as $a||\boldsymbol{\theta}- \boldsymbol{m}||^{2} +$ a constant, where $\boldsymbol{m}$ and the constant need to be determined. Expanding gives 
\begin{equation*}
    a||\boldsymbol{\theta} - \boldsymbol{m}||^{2} = a||\boldsymbol{\theta}||^{2} - 2a\boldsymbol{m}^{T}\boldsymbol{\theta} + a||\boldsymbol{m}||^{2}. 
\end{equation*}
Matching terms gives $a\boldsymbol{m} = \boldsymbol{y}$ which implies $\boldsymbol{m} = a^{-1}\boldsymbol{y} = (v/(1+v))\boldsymbol{y}$. The constant term does not depend on $\boldsymbol{\theta}$ and can be absorbed into the constant of proportionality. Then, 
\begin{equation*}
    p(\boldsymbol{\theta}|\boldsymbol{y},v) \propto \exp \{- \frac{1}{2}(1 + \frac{1}{v})||\boldsymbol{\theta} - \frac{v}{1+v}\boldsymbol{y}||^{2} \}. 
\end{equation*}
This implies 
\begin{equation*}
    \boldsymbol{\theta} | \boldsymbol{Y} = \boldsymbol{y},V=v \sim N_{p}(\frac{v}{1+v}\boldsymbol{y}, \frac{v}{1+v}I_{p}).
\end{equation*}
Then, 
\begin{equation*}
    \mathbb{E}[\boldsymbol{\theta}|\boldsymbol{Y} = \boldsymbol{y}, V=v] = \frac{v}{1+v} \boldsymbol{y}.
\end{equation*}
Then by the law of total expectation, 
\begin{align*}
    \mathbb{E}[\boldsymbol{\theta}|\boldsymbol{Y} = \boldsymbol{y}] &= \mathbb{E}[\mathbb{E}[\boldsymbol{\theta}|\boldsymbol{Y} = \boldsymbol{y}, V=v]|\boldsymbol{Y} = \boldsymbol{y}]\\
    &= \mathbb{E}[\frac{V}{1+V}\boldsymbol{y}|\boldsymbol{Y} = \boldsymbol{y}]\\ 
    &= \mathbb{E}[\frac{V}{1+V}|\boldsymbol{Y} = \boldsymbol{y}]\boldsymbol{y}. 
\end{align*}
Observe that 
\begin{align*}
    p(\boldsymbol{y}|v) &= \int\limits_{\mathbb{R}^{p}} p(\boldsymbol{y}|\boldsymbol{\theta}, v) p(\boldsymbol{\theta}|v)d\boldsymbol{\theta}\\ 
    & = \int\limits_{\mathbb{R}^{p}} \phi_{p}(\boldsymbol{y} - \boldsymbol{\theta}; \boldsymbol{0}_{p}, I_{p}) \phi_{p}(\boldsymbol{\theta}; \boldsymbol{0}, vI_{p}) d\boldsymbol{\theta}\\ 
    &= \phi_{p}(\boldsymbol{y}; \boldsymbol{0}_{p} , (1+v)I_{p}), \text{ by convolution of Gaussians}. 
\end{align*}
That is, $\boldsymbol{Y}|V=v \sim N_{p}(\boldsymbol{0}_{p}, (1+v)I_{p})$. Then, 
\begin{equation*}
    p(\boldsymbol{y}|v) = (2\pi)^{-\frac{p}{2}} (1+v)^{-\frac{p}{2}} \exp \{ - \frac{||\boldsymbol{y}||^{2}}{2(1+v)} \}. 
\end{equation*}
By Bayes' theorem, 
\begin{equation*}
    p(v|\boldsymbol{y}) \propto (1+v)^{-\frac{p}{2}}\exp \{ - \frac{||\boldsymbol{y}||^{2}}{2(1+v)} \} \pi_{V}(dv), 
\end{equation*}
which depends on $\boldsymbol{y}$ only through $||\boldsymbol{y}||^{{2}}$. Hence, 
\begin{equation*}
    \mathbb{E}[\frac{V}{1+V} |\boldsymbol{Y} = \boldsymbol{y}] = \mathbb{E}[\frac{V}{1+V}|||\boldsymbol{Y}||^{2} = ||\boldsymbol{y}||^{2}]. 
\end{equation*}
Therefore, 
\begin{equation*}
    \delta_{\text{BNN, fixed}} (\boldsymbol{y}) = \mathbb{E}[\frac{V}{1+V}|||\boldsymbol{Y}||^{2} = ||\boldsymbol{y}||^{2}]\boldsymbol{y}.
\end{equation*}
\end{proof}
\subsection{Proof of Theorem 2.6}
\begin{proof}
    We know that \begin{equation*}
        \delta(\boldsymbol{Y}) = \delta_{\text{BNN,fixed}}(\boldsymbol{Y}) = \mathbb{E}[\frac{V}{V+1}|||\boldsymbol{Y}||^{2}]\boldsymbol{Y}. 
    \end{equation*} Let $U:=||\boldsymbol{Y}||^{2}  = \sum\limits_{i=1}^{p}Y_{i}^{2}$. Then, \begin{equation*}
        a(u):= \mathbb{E}[\frac{V}{V+1}|U=u] \text{ and } \psi(u) := 1-a(u) = \mathbb{E}[\frac{1}{V+1}|U=u]. 
    \end{equation*} Therefore, \begin{equation*}
        \delta(\boldsymbol{y}) = (1-\psi(||\boldsymbol{y}||^{2}))\boldsymbol{\boldsymbol{y}}, \text{ with } 0 < \psi(u) < 1.  
    \end{equation*} 
Let $\pi(v)$ denote the unconditional density of $V$. Let 
\begin{equation*}
    g(\boldsymbol{y}) = \delta(\boldsymbol{y}) - \boldsymbol{y} =  -\psi(||\boldsymbol{y}||^{2})\boldsymbol{y}
\end{equation*}
Then, 
\begin{align*}
        ||g(\boldsymbol{y})|| &= ||-\psi(||\boldsymbol{y}||^{2})\boldsymbol{y}|| \\ 
        & = |-\psi(||\boldsymbol{y}||^{2})|||\boldsymbol{y}||\\
        & = \psi(||\boldsymbol{y}||^{2})||\boldsymbol{y}|| \\ 
        &\leq ||\boldsymbol{y}||, \text{ since } 0 < \psi(||\boldsymbol{y}||^{2}) < 1. 
    \end{align*} Then, 
    \begin{align*}
        \mathbb{E}_{\boldsymbol{\theta}}[||g(\boldsymbol{Y})||^{2}] & \leq  \mathbb{E}_{\boldsymbol{\theta}}[||\boldsymbol{Y}||^{2}]\\ 
        & = \mathbb{E}_{\boldsymbol{\theta}}[\sum\limits_{i=1}^{p}Y_{i}^{2}] \\
        &= \sum\limits_{i=1}^{p}\mathbb{E}_{\boldsymbol{\theta}}[Y_{i}^{2}]\\ 
        & = \sum\limits_{i=1}^{p} 1+ \boldsymbol{\theta}_{i}^{2}\\ 
        & = p +||\boldsymbol{\theta}||^{2} < \infty.
    \end{align*} Let $\Sigma_{V} := (1+V)I_{p}$. Since $\boldsymbol{Y}|V=v \sim N_{p} (\boldsymbol{0}_{p}, \Sigma_{v})$ we know that $\boldsymbol{Y}|V=v  \overset{d}{=}  \sqrt{1+v}\boldsymbol{Z}$, $\boldsymbol{Z} \sim N_{p}(\boldsymbol{0}_{p}, I_{p}).$ Then, $U = ||\boldsymbol{Y}||^{2} = ||\sqrt{1+v}\boldsymbol{Z}||^{2} = |\sqrt{1+v}|^{2}||\boldsymbol{Z}||^{2} = (1+v)W,$ where $W:= ||\boldsymbol{Z}||^{2}$. That is, $U|V = v \overset{d}{=} (1+v)W$. Since $u = (1+v)w$ we know \begin{equation*}
        w = \frac{u}{1+v} \implies \frac{dw}{du} = \frac{1}{1+v}. 
    \end{equation*} Then, 
\begin{align*}
    f_{U|V=v}(u|v) & = f_{W}(\frac{u}{1+v})|\frac{1}{1+v}|, \text{ where } W \overset{d}{=} \chi_{p}^{2}\\ 
    &= \frac{1}{1+v} \frac{1}{2^{\frac{p}{2}}\Gamma(\frac{p}{2})} (\frac{u}{1+v})^{\frac{p}{2}-1} \exp \{ - \frac{u}{2(1+v)} \}\\ 
    &= \frac{1}{2^{\frac{p}{2}}\Gamma(\frac{p}{2})} u^{\frac{p}{2}-1}(1+v)^{-\frac{p}{2}} \exp\{ - \frac{u}{2(1+v)} \}\\ 
    &\propto u^{\frac{p}{2}-1} (1+v)^{-\frac{p}{2}} \exp \{ - \frac{u}{2(1+v)} \}. 
\end{align*} Therefore, 
\begin{align*}
    f_{V|U=u}(v|u) &\propto u^{\frac{p}{2}-1} (1+v)^{-\frac{p}{2}} \exp \{ - \frac{u}{2(1+v)} \}\pi(v)\\
    & \propto (1+v)^{-\frac{p}{2}}\exp \{ - \frac{u}{2(1+v)} \}\pi(v). 
\end{align*} Let \begin{align*}
    D(u) := \mathbb{E}[(1+V)^{-\frac{p}{2}} \exp \{ - \frac{u}{2(1+V)} \}]  &= \int\limits_{0}^{\infty}\pi(v) (1+v)^{-\frac{p}{2}} \exp \{ - \frac{u}{2(1+v)} \}dv\\ 
    & \text{and}\\ 
    N(u):= \mathbb{E}[(1+V)^{-\frac{p}{2}-1}\exp\{ - \frac{u}{2(1+V)} \}] &= \int\limits_{0}^{\infty} \pi(v) (1+v)^{-\frac{p}{2}-1} \exp \{ - \frac{u}{2(1+v)} \}dv. 
\end{align*} Then \begin{equation*}
    f_{V|U=u}(v|u) = \frac{1}{D(u)} \pi(v) (1+v)^{-\frac{p}{2}} \exp \{ - \frac{u}{2(1+v)} \}. 
\end{equation*} Recall \begin{align*}
    \psi(u) &=  \mathbb{E}[\frac{1}{V+1}|U=u]\\ 
    & = \int\limits_{0}^{\infty} \frac{1}{1+v}\pi(v|u)dv\\
    &= \frac{1}{D(u)} \int\limits_{0}^{\infty} \frac{1}{1+v} \pi(v) (1+v)^{-\frac{p}{2}} \exp \{ - \frac{u}{2(1+v)} \}dv\\ 
    & = \frac{1}{D(u)} \int\limits_{0}^{\infty} \pi(v) (1+v)^{-\frac{p}{2} - 1} \exp \{ - \frac{u}{2(1+v)} \}dv \\ 
    & = \frac{N(u)}{D(u)}. 
\end{align*}
Let \begin{equation*}
    f(u,v):= (1+v)^{-\frac{p}{2}}\exp \{ - \frac{u}{2(1+v)} \}. 
\end{equation*} Note that \begin{align*}
    \frac{\partial }{\partial u} \{ (1+v)^{-\frac{p}{2}} \exp \{ - \frac{u}{2(1+v)} \} \} &= (1+v)^{-\frac{p}{2}} \frac{\partial}{\partial u} \{ \exp\{- \frac{u}{2(1+v)} \} \}\\ 
    & = (1+v)^{-\frac{p}{2}} (- \frac{1}{2(1+v)}) \exp \{ - \frac{u}{2(1+v)} \} \\ 
    &= - \frac{1}{2} (1+v)^{-\frac{p}{2}-1} \exp \{ - \frac{u}{2(1+v)} \}. 
\end{align*} For $h\neq 0$, \begin{equation*}
\frac{D(u+h) - D(u)}{h} = \int\limits_{0}^{\infty}\frac{f(u+h, v) - f(u,v)}{h} \mu(dv). 
\end{equation*} For fixed $v$, apply the single variable mean value theorem to $t \rightarrow f(t,v)$ on the interval between $u$ and $u+h$, $\exists \xi = \xi(h,v) \in (0,1)$ such that \begin{equation*}
    \frac{f(u+h,v) - f(u,v)}{h} = \frac{\partial}{\partial u} f(u+\xi h, v). 
\end{equation*} Therefore, \begin{align*}
    |\frac{f(u+h, v) - f(u,v)}{h}| & = |\frac{\partial}{\partial u} f(u+\xi h, v)| \\ 
    & = \frac{1}{2}(1+v)^{-\frac{p}{2}-1} \exp \{ - \frac{u+\xi h}{2(1+v)} \} \\ 
    &\leq \frac{1}{2} (1+v)^{-\frac{p}{2}-1 }.
\end{align*}
Note that \begin{equation*}
    g(v) := \frac{1}{2}(1+v)^{-\frac{p}{2}-1}
\end{equation*}  is $\mu-$ integrable. Particularly $0 \leq (1+v)^{-\frac{p}{2}-1} \leq 1, \forall v\geq 0$. Therefore, \begin{equation*}
    \int g(v) \mu(dv) \leq \frac{1}{2} \int 1\mu(dv) \leq \frac{1}{2}< \infty. \text{ Since } \lim\limits_{h \rightarrow 0 } \frac{f(u+h, v) - f(u,v)}{h} \rightarrow \frac{\partial}{\partial u} f(u,v), 
\end{equation*} by the dominated convergence theorem, \begin{align*}
    D'(u) & = \lim\limits_{h \rightarrow 0} \frac{D(u+h) - D(u)}{h} \\ 
    & = \int \lim\limits_{h \rightarrow 0} \frac{f(u+h, v) - f(u,v)}{h} \mu(dv)\\ 
    &= \int \frac{\partial}{\partial u} f(u,v) \mu(dv)\\ 
    &= -\frac{1}{2} \int (1+v)^{-\frac{p}{2}-1} \exp\{ - \frac{u}{2(1+v)} \} \mu(dv). 
\end{align*} That is, 
\begin{align*}
    D'(u) &= - \frac{1}{2}\mathbb{E}[(1+V)^{-\frac{p}{2}-1} \exp \{ - \frac{u}{2(1+V)} \}] \\ 
    & = - \frac{1}{2}N(u). 
\end{align*} Similarly, \begin{equation*}
    M(u) := \mathbb{E}[(1+V)^{-\frac{p}{2}-2} \exp \{ - \frac{u}{2(1+V)} \}], \text{ with } N'(u) = - \frac{1}{2} M(u). 
\end{equation*} Then, \begin{align*}
    \psi'(u) &= \frac{d}{du} \{ N(u) D(u)^{-1} \}\\ 
    &= D(u)^{-1} N'(u) + N(u)[-D(u)^{-2}D'(u)] \\ 
    & = \frac{N'(u)}{D(u)} - \frac{N(u)D'(u)}{D(u)^{2}}\\ 
    & = -\frac{1}{2} \frac{M(u)}{D(u)} + \frac{1}{2}(\frac{N(u)}{D(u)})^{2}. 
\end{align*} Observe that $\mathbb{P}(U \in du, V \in dv) =f_{U|V=v}(u|v) du \mu(dv)$. Then,  \begin{equation*}
    \mu(dv|u) = \frac{f_{U|V=v}(u|v)\mu(dv)}{ \int f_{U|V=v}(u|t) \mu(dt)}. 
\end{equation*} Now, \begin{align*}
    \mathbb{E}[(\frac{1}{1+V})^{2}|U=u] &= \int \frac{1}{(1+v)^{2}}\mu(dv|u) \\ 
    & = \frac{\int (1+v)^{-2}f_{U|V=v}(u|v)\mu(dv)}{\int f_{U|V=v} (u|t) \mu(dt)}\\ 
    & = \frac{\frac{u^{\frac{p}{2}-1}}{2^{\frac{p}{2}}\Gamma(\frac{p}{2})}\int (1+v)^{-\frac{p}{2}-2} \exp \{ - \frac{u}{2(1+v) } \} \mu(dv)}{\frac{u^{\frac{p}{2}-1}}{2^{\frac{p}{2}}\Gamma(\frac{p}{2})}\int(1+t)^{-\frac{p}{2}} \exp \{ - \frac{u}{2(1+t)} \}\mu (dt)} \\ 
    & = \frac{M(u)}{D(u)}. 
\end{align*} Similarly \begin{equation*}
    \mathbb{E}[\frac{1}{1+V}|U=u] = \frac{N(u)}{D(u)} = \psi(u). 
\end{equation*} Therefore \begin{align*}
    \psi'(u) & = -\frac{1}{2}(\mathbb{E}[(\frac{1}{1+V})^{2}|U=u] - \mathbb{E}[\frac{1}{1+V}|U=u]^{2})\\ 
    &= - \frac{1}{2} Var(\frac{1}{1+V}|U=u) \leq  0. 
\end{align*} This means \begin{align*}
    0 &\leq Var(\frac{1}{1+V}|U=u) \\ 
    &= \mathbb{E}[(\frac{1}{1+V})^{2}|U=u] - \mathbb{E}[\frac{1}{1+V}|U=u]^{2} \\ 
    &\leq  \mathbb{E}[(\frac{1}{1+V})^{2}|U=u] \\
    & \leq \mathbb{E}[\frac{1}{1+V}|U=u], \text{ since } \frac{1}{1+V} \in (0,1]\\
    &= \psi(u). 
\end{align*}
This implies \begin{align*}
    & - Var(\frac{1}{1+V}|U=u) \geq -\mathbb{E}[\frac{1}{1+V}|U= u]\\ 
    &\iff - \frac{1}{2}Var(\frac{1}{1+V}|U=u) \geq - \frac{1}{2}\mathbb{E}[\frac{1}{1+V}|U=u]\\ 
    &\iff \psi'(u) \geq - \frac{1}{2}\psi(u). 
\end{align*} Therefore, \begin{equation*}
    0 \geq \psi'(u) \geq - \frac{1}{2}\psi(u), \text{ where } \psi(u) \leq 1. 
\end{equation*} 
Note that \begin{align*}
    \frac{\partial}{\partial y_{i}} \{ - \psi(u) y_{i} \}& = y_{i}\frac{\partial}{\partial y_{i}} \{ - \psi(u) \} + (-\psi(u))\frac{\partial}{\partial y_{i}} \{ y_{i} \}\\ 
    &= -y_{i} \frac{\partial \psi (u)}{\partial y_{i}} - \psi(u)\\
    &= - y_{i} \frac{\partial \psi}{\partial u} \frac{\partial u}{\partial y_{i}} - \psi(u)\\ 
    & = -y_{i}\psi'(u) (2y_{i}) - \psi(u)\\ 
    &= -2y_{i}^{2}\psi'(u) - \psi(u). 
\end{align*}  
Therefore, 
\begin{align*}
    div(g(\boldsymbol
    {y})) = \sum\limits_{i=1}^{p} \frac{\partial g_{i}}{\partial y_{i}} &= \sum\limits_{i=1}^{p} [-\psi(u) - 2y_{i}^{2} \psi'(u) ]\\ 
    &= -p\psi(u) -2\psi'(u) \sum\limits_{i=1}^{p}y_{i}^{2}\\ 
    &= - p\psi(u) - 2||\boldsymbol{y}||^{2}\psi'(u). 
\end{align*}
So, 
\begin{equation*}
    |\nabla \cdot g(\boldsymbol{y})| \leq p|\psi(||\boldsymbol{y}||^{2})| + 2||\boldsymbol{y}||^{2} |\psi'(||\boldsymbol{y}||^{2})| \leq p + ||\boldsymbol{y}||^{2}. 
\end{equation*}
Hence, 
\begin{equation*}
    \mathbb{E}_{\theta}|\nabla\cdot g(\boldsymbol{Y})| \leq p + \mathbb{E}_{\theta}[||\boldsymbol{Y}||^{2}]  = 2p + ||\boldsymbol{\theta}||^{2} < \infty. 
\end{equation*}
Recall $\delta(\boldsymbol{y}) = g(\boldsymbol{y}) + \boldsymbol{y} $ with $g(\boldsymbol{y}) = -\psi(||\boldsymbol{y}||^{2})\boldsymbol{y}$.  From above we know $0 \geq \psi'(u) \geq - \frac{1}{2}$, that is $|\psi'(u)| \leq \frac{1}{2}$. From above we know that $g(\boldsymbol{y})$ is differentiable  and thus also weakly differentiable. Therefore we may use Stein's unbiased risk estimator (\cite{stein1981estimation}). Furthermore, 
\begin{align*}
    ||g(\boldsymbol{y})||^{2} &= ||-\psi(||\boldsymbol{y}||^{2}) \boldsymbol{y}||^{2}\\ 
    &= ||-\psi(||\boldsymbol{y}||^{2})||^{2}||\boldsymbol{y}||^{2}\\ 
    &= ||\psi(||\boldsymbol{y}||^{2})|| ^{2}||\boldsymbol{y}||^{2}\\ 
    &= \psi(||\boldsymbol{y}||^{2}) ^{2}||\boldsymbol{y}||^{2}. 
\end{align*} Therefore, 
\begin{align*}
    R(\boldsymbol{\theta}, \boldsymbol{\delta}) &= \mathbb{E}_{\boldsymbol{\theta}}[p+||g(\boldsymbol{Y})||^{2} +2div(g(\boldsymbol{Y}))] \\ 
    &= p + \mathbb{E}_{\boldsymbol{\theta}}[\psi(U)^{2} U - 2p\psi(U) - 4U\psi'(U)]. 
\end{align*} Then the excess risk integrand is given by, 
\begin{equation*}
    B(u) := \psi(u)^{2} u - 2p\psi(u) -4u\psi'(u). 
\end{equation*} Then $R(\boldsymbol{\theta}, \boldsymbol{\delta}) = p + \mathbb{E}_{\boldsymbol{\theta}}[B(U)]$. From above $\psi'(u) \leq 0$. Therefore $-4u\psi'(u) \geq 0$. Thus, \begin{equation*}
    B(u) \geq \psi(u)^{2}u - 2p\psi(u) = \psi(u) (u\psi(u) -2p). 
\end{equation*} Therefore, whenever $u\psi(u) > 2p $ we have $B(u) > 0$. If we can show that $u\psi(u) \rightarrow \infty$ as $u \rightarrow \infty$ then we can deduce non-minimaxity. Fix $\boldsymbol{k} = (k_{1}, \dots,k_{d-1})$. Conditional on $\boldsymbol{K} = \boldsymbol{k}, V = C \prod\limits_{\ell=1}^{d-1} T_{\ell}, T_{\ell}|\boldsymbol{K} = \boldsymbol{k} \sim \Gamma(\alpha_{\ell},1), \alpha_{\ell} = k_{\ell}/2$.
Recall, \begin{align*}
    f_{U|V=v}(u|v) & \propto u^{\frac{p}{2}-1} (1+v)^{-\frac{p}{2}} \exp \{ - \frac{u}{2(1+v)} \} \\
    & \text{ and }\\ 
    f_{U|T=t, \boldsymbol{K} = \boldsymbol{k}}(u|t, \boldsymbol{k}) &\propto u^{\frac{p}{2}-1} (1+C\prod\limits_{\ell=1}^{d-1}t_{\ell})^{-\frac{p}{2}} \exp \{ - \frac{u}{2(1+C\prod\limits_{\ell = 1}^{d-1}t_{\ell})} \}. 
\end{align*} By Bayes' theorem, 
\begin{equation*}
    \pi(t|U=u, \boldsymbol{K} = \boldsymbol{k}) = \frac{f_{U|t, \boldsymbol{k}}(u|t, \boldsymbol{k})\pi_{T| \boldsymbol{K}}(t|\boldsymbol{k})}{m_{U|\boldsymbol{K}= \boldsymbol{k}}(u|\boldsymbol{k})}, m_{U|\boldsymbol{K}}(u) = \int f_{U|t, \boldsymbol{k}}(u|t, \boldsymbol{k})\pi_{T| \boldsymbol{K}}(t|\boldsymbol{k}) dt. 
\end{equation*}
That is $\pi(t|U=u, \boldsymbol{K} = \boldsymbol{k}) \propto f_{U|t, \boldsymbol{k}}(u|t, \boldsymbol{k}) \pi_{T|\boldsymbol{K}}(t)$. Therefore, \begin{equation*}
    \pi(t|U=u, \boldsymbol{K}=\boldsymbol{k}) \propto u^{\frac{p}{2}-1}(1+C\prod\limits_{\ell = 1}^{d-1}t_{\ell})^{-\frac{p}{2}} \exp \{ - \frac{u}{2(1+C\prod\limits_{\ell =1}^{d-1} t_{\ell})} \} \pi_{T|\boldsymbol{K}}(t). 
\end{equation*} Now, 
\begin{align*}
    \pi_{T_{1},\dots, T_{d-1}|\boldsymbol{K}}(t_{1}, \dots, t_{d-1}|\boldsymbol{k}) &= \prod\limits_{\ell = 1}^{d-1} f_{T_{\ell}|{K}_{\ell}}(t_{\ell}|k_{\ell}), \text{ by independence}, \\ 
    &= \prod\limits_{\ell = 1}^{d-1} \frac{1}{\Gamma(\alpha_{\ell})} t_{\ell}^{\alpha_{\ell}-1} \exp \{  - t_{\ell}\}\\ 
    &= (\prod\limits_{\ell=1}^{d-1} \Gamma(\alpha_{\ell})^{-1}) (\prod\limits_{\ell = 1}^{d-1} t_{\ell}^{\alpha_{\ell}-1})\exp\{ - \sum\limits_{\ell = 1}^{d-1} t_{\ell} \} \\ 
    &\propto (\prod\limits_{\ell = 1}^{d-1} t_{\ell}^{\alpha_{\ell}-1}) \exp \{ - \sum\limits_{\ell = 1}^{d-1}t_{\ell} \}. 
\end{align*} Thus, \begin{align*}
    \pi(t|U=u, \boldsymbol{K} = \boldsymbol{k}) &\propto u^{\frac{p}{2}-1} (1+C\prod\limits_{\ell = 1}^{d-1} t_{\ell})^{-\frac{p}{2}} \exp \{ - \frac{u}{2(1+C\prod\limits_{\ell = 1}^{d-1}t_{\ell})} \} \\&\times(\prod\limits_{\ell = 1}^{d-1} t_{\ell}^{\alpha_{\ell}-1})\exp \{ - \sum\limits_{\ell = 1}^{d-1} t_{\ell} \} \\ 
    &\propto (1+C\prod\limits_{\ell = 1}^{d-1} t_{\ell})^{- \frac{p}{2}} (\prod\limits_{\ell = 1}^{d-1} t_{\ell}^{\alpha _{\ell}-1})  \exp \{- \sum\limits_{\ell = 1}^{d-1} t_{\ell} - \frac{u}{2(1+C\prod\limits_{\ell = 1}^{d-1} t_{\ell})} \}. 
\end{align*} 
Let $\boldsymbol{t} = (t_{1}, \dots, t_{d-1}) \in (0,\infty)^{d-1}.$ Let $x = (\prod\limits_{\ell = 1}^{d-1} t_{\ell})^{\frac{1}{d-1}}$. Then $\prod\limits_{\ell = 1}^{d-1} t_{\ell} = x^{d-1}$. Now $s_{j} := x^{-1} t_{j}$ for $j = 1, \dots, d-1$. Then, $t_{j} = xs_{j}$. This implies \begin{equation*}
    \prod\limits_{\ell = 1}^{d-1} s_{\ell} = \prod\limits_{\ell = 1}^{d-1}\frac{t_{\ell}}{x} = \frac{\prod\limits_{\ell = 1}^{d-1}t_{\ell}}{x^{d-1}} = \frac{x^{d-1}}{x^{d-1}} = 1. 
\end{equation*} 
Therefore, $\prod\limits_{\ell = 1}^{d-1}s_{\ell} = 1$. This means not all the $s_{j}$'s are independent. We will keep $s_{1}, \dots, s_{d-2}$ free and define $s_{d-1} := (\prod\limits_{\ell = 1}^{d-2} s_{\ell})^{-1}$. Then $\prod\limits_{\ell = 1}^{d-1} s_{\ell} = (\prod\limits_{\ell = 1}^{d-2}s_{\ell})s_{d-1} = (\prod\limits_{\ell = 1}^{d-2} s_{\ell}) (\prod\limits_{\ell = 1}^{d-2}s_{\ell})^{-1}$ which equals $1$. In summary we have $\boldsymbol{t} \rightarrow (x, s_{1},\dots, s_{d-2})$ where $x = (\prod\limits_{\ell = 1}^{d-1}t_{\ell})^{\frac{1}{d-1}}$, $s_{i} = x^{-1}t_{i}, i = 1,\dots, d-2$. The inverse map is given by $(x, s_{1}, \dots, s_{d-2}) \rightarrow \boldsymbol{t}$ where $t_{i} = xs_{i}, i = 1,\dots, d-2, t_{d-1} = xs_{d-1} = x(\prod\limits_{\ell = 1}^{d-2}s_{\ell})^{-1}$. Let $r:= \log(x), u_{i}:= \log(s_{i}),$ $i = 1, \dots, d-2$. Then
\begin{equation*}
    s_{d-1} = (\prod\limits_{\ell = 1}^{d-2}s_{\ell})^{-1} \implies \log(s_{d-1}) = \log(\prod\limits_{\ell = 1}^{d-2}s_{\ell})^{-1} = - \sum\limits_{\ell = 1}^{d-2} \log(s_{\ell}) = - \sum\limits_{\ell = 1}^{d-2}u_{\ell}.
\end{equation*}
Now $z_{j}:= \log(t_{j}) = \log(xs_{j}) = \log(x) + \log(s_{j}) = r + u_{j}, j = 1, \dots, d-2$. Then  $z_{d-1} = r- \sum\limits_{\ell = 1}^{d-2} u_{\ell} $. Let \begin{equation*}
    A = \frac{\partial(z_{1}, \dots, z_{d-1})}{\partial(r, u_{1}, \dots, u_{d-2})}. 
\end{equation*} Then for $i \leq d-2, $ \begin{equation*}
    \frac{\partial z_{i}}{\partial r} = 1, \frac{\partial z_{i}}{\partial u_{i}} = 1, \text{ and } \frac{\partial z_{i}}{\partial u_{j}} = 0 \text{ for } i\neq j. 
\end{equation*} Then row $i$ of A is given by $[1,\dots,0,1,0,\dots, 0]$ where the second $1$ is in the column corresponding to $u_{i}$. For row $d-1$ of A, \begin{equation*}
    \frac{\partial z_{d-1}}{\partial r} = 1\text{ and } \frac{\partial z_{d-1}}{\partial u_{j}} = -1 \text{ for } j = 1, \dots, d-2. 
\end{equation*} So the last row of A is $[1,-1, \dots, -1]$. Then,  \[
A =
\begin{bmatrix}
1 & 1 & 0 & 0 & \dots & 0\\ 
1 & 0 & 1 & 0 & \dots &0\\ 
1 & 0 & 0 & 1 & \dots & 0 \\ 
\vdots & \vdots & \vdots & \vdots & \vdots & \vdots\\ 
1& 0 & 0 & 0 & \dots & 1\\ 
1 & - 1& -1 & -1 & \dots & -1
\end{bmatrix}. 
\] Now $R_{2}, \dots , R_{d-1} \rightarrow R_{2}, \dots, R_{d-1}-R_{1}$. Then $R_{2}, \dots, R_{d-2} = [0,-1,1,0,\dots, 0]$ with a $-1$ in the column corresponding to $u_{1}$ and a 1 in column $u_{i}, i = 2, \dots, d-2$. Furthermore, $R_{d-1} = [0,-2, -1, \dots, -1]$. Then the transformed matrix is given by \[
A^{'} =
\begin{bmatrix}
1 & 1& 0 & 0 & \dots &0\\ 
0 & -1 & 1&0 & \dots &0\\ 
0 & -1 & 0 & 1 & \dots & 0\\ 
\vdots & \vdots & \vdots & \vdots & \vdots & \vdots \\ 
0 & -2 & -1 & -1 & \dots & -1
\end{bmatrix}. 
\] Then $\det(A) = \det(A^{'})$. We can use Laplace expansion by deleting row $1$ and column $1$. This gives, \[
B =
\begin{bmatrix}
-1 & 1 & 0 & 0 & \dots & 0\\
-1 & 0 & 1 & 0 & \dots &0\\ 
-1 & 0 & 0 & 1 & \dots & 0\\ 
\vdots & \vdots & \vdots & \vdots & \vdots & \vdots\\ 
-1 & 0 & 0 & 0 & \dots & 1\\ 
-2 & -1 & -1 & -1 & \dots & -1
\end{bmatrix}. 
\] Now $R_{d-2} \rightarrow R_{d-2} + R_{1} + R_{2} + \dots+ R_{d-3}.$ So in column 1 row $d-2$,  $-2 + (d-3)(-1)  = -2 - (d-3) = -2 -d+3 = -(d-1)$. Then column $2$ to column $d-2$ we have $-1+1 = 0$. Therefore the last row becomes $(-(d-1), 0, 0 , \dots, 0)$. So the transformed B matrix is lower triangular except the last row. Therefore the determinant is equal to $|\det(A)| = d-1$. We have $t_{j} = e^{z_{j}}$. This gives \begin{equation*}
    \frac{dt_{j}}{dz_{j}} = e^{z_{j}} \text{ i.e } dt_{j} = e^{z_{j}} dz_{j} .
\end{equation*} Then \begin{align*}
    dt_{1}\dots dt_{d-1} &= (e^{z_{1}}dz_{1})(e^{z_{2}}dz_{2}) \dots (e^{z_{d-1}}dz_{d-1})\\
    &= (\prod\limits_{j=1}^{d-1}e^{z_{j}})dz_{1} \dots dz_{d-1}\\
    &= (\prod\limits_{j=1}^{d-1}t_{j}) dz_{1}\dots dz_{d-1}.
\end{align*}
So \begin{equation*}
    dz_{1}\dots dz_{d-1} = |\det(A)|drdu_{1}\dots du_{d-2} = (d-1)drdu_{1}\dots du_{d-2}
\end{equation*}
  We also know $dr = x^{-1}dx$ and $du_{i} = s_{i}^{-1} ds_{i}$. Then  \begin{equation*}
      dt_{1}\dots dt_{d-1} = x^{d-1} (d-1)x^{-1}dx \prod\limits_{j=1}^{d-2} s_{j}^{-1}ds_{j} = (d-1)x^{d-2}(\prod\limits_{j=1}^{d-2} s_{j}^{-1})dxds_{1}\dots ds_{d-2}.
  \end{equation*}
  Now we can rewrite the kernel of $p(t|U=u, \boldsymbol{K} = \boldsymbol{k})$. Observe 
  \begin{equation*}
      (1+C\prod\limits_{\ell = 1}^{d-1} t_{\ell})^{-\frac{p}{2}} = (1+Cx^{d-1})^{-\frac{p}{2}} \text{ and } \exp \{ - \frac{u}{2(1+C \prod\limits_{\ell =1}^{d-1}t_{\ell})} \} = \exp \{ - \frac{u}{2(1+Cx^{d-1})} \}.
  \end{equation*}
Furthermore, $\prod\limits_{\ell = 1}^{d-1} t_{\ell}^{\alpha_{\ell}-1} = \prod\limits_{\ell = 1}^{d-1} (xs_{\ell})^{\alpha_{\ell}-1} = x^{\sum\limits_{\ell = 1}^{d-1}(\alpha_{\ell} -1)} \prod\limits_{\ell = 1}^{d-1} s_{\ell}^{\alpha_{\ell}-1}$. Let $\alpha_{\cdot} = \sum\limits_{\ell=1}^{d-1} \alpha_{\ell}$. Then $\sum\limits_{\ell = 1}^{d-1} (\alpha_{\ell} -1) = \alpha_{\cdot} - (d-1)$. So, $\prod\limits_{\ell = 1}^{d-1} t_{\ell}^{\alpha_{\ell}-1} = x^{\alpha_{\cdot}-(d-1)} \prod\limits_{\ell = 1}^{d-1} s_{\ell}^{\alpha_{\ell}-1}$. Finally \\$\exp \{ -\sum\limits_{\ell = 1}^{d-1} t_{\ell} \} = \exp \{ - xA(\boldsymbol{s}) \}$, where $A(\boldsymbol{s}) := \sum\limits_{\ell = 1}^{d-1} s_{\ell}$. Now, \begin{align*}
    \pi_{u, \boldsymbol{k}}(x,\boldsymbol{s}) &\propto \pi_{u, \boldsymbol{k}} (t(x,\boldsymbol{s})) \cdot (d-1) x^{d-2} \prod\limits_{\ell = 1}^{d-2}s_{\ell}^{-1} \\ 
    &\propto (1+Cx^{d-1}) ^{-\frac{p}{2}} x^{\alpha_{\cdot}-(d-1)} \prod\limits_{\ell=1}^{d-1} s_{\ell}^{\alpha_{\ell} -1} \exp \{ -xA(\boldsymbol{s}) \}\exp \{ - \frac{u}{2(1+Cx^{d-1})} \}\\
    &\times(d-1)x^{d-2}\prod\limits_{\ell = 1}^{d-2}s_{\ell}^{-1}\\ 
    &\propto (1+Cx^{d-1})^{-\frac{p}{2}} \exp \{ -xA(\boldsymbol{s}) - \frac{u}{2(1+Cx^{d-1})} \}x^{\alpha_{\cdot} -d+1+d-2}\\ 
    &\times (\prod\limits_{\ell = 1}^{d-1}s_{\ell}^{\alpha_{\ell}-1}) (\prod\limits_{\ell = 1}^{d-2}s_{\ell}^{-1}) \\ 
    &= (1+Cx^{d-1}) ^{-\frac{p}{2}} \exp \{ -xA(\boldsymbol{s}) - \frac{u}{2(1+Cx^{d-1})}\}x^{\alpha_{\cdot}-1}
    (\prod\limits_{\ell=1}^{d-1}s_{\ell}^{\alpha_{\ell}-1})(\prod\limits_{\ell = 1}^{d-2}s_{\ell}^{-1}). 
\end{align*} Let \begin{align*}
    Q(\boldsymbol{s}) &:= (\prod\limits_{\ell = 1}^{d-1} s_{\ell}^{\alpha_{\ell}-1})(\prod\limits_{\ell =1}^{d-2}s_{\ell}^{-1})\\ 
    & = (\prod\limits_{\ell = 1}^{d-2}s_{\ell}^{\alpha_{\ell}-1})s_{d-1}^{\alpha_{d-1}-1} (\prod\limits_{\ell = 1}^{d-2} s_{\ell}^{-1})\\ 
    & = (\prod\limits_{\ell =1}^{d-2} s_{\ell}^{\alpha_{\ell}-1})[(\prod\limits_{\ell =1}^{d-2}s_{\ell})^{-1}]^{\alpha_{d-1}-1} (\prod\limits_{\ell = 1}^{d-2}s_{\ell}^{-1})\\ 
    &= (\prod\limits_{\ell =1}^{d-2}s_{\ell}^{\alpha_{\ell}-1})(\prod\limits_{\ell = 1}^{d-2}s_{\ell})^{-(\alpha_{d-1}-1)}(\prod\limits_{\ell = 1}^{d-2}s_{\ell}^{-1}) \\ 
    & = (\prod\limits_{\ell = 1}^{d-2} s_{\ell}^{\alpha_{\ell}-1})(\prod\limits_{\ell = 1}^{d-2} s_{\ell}^{-(\alpha_{d-1}-1)})(\prod\limits_{\ell = 1}^{d-2} s_{\ell}^{-1})\\ 
    &= \prod\limits_{\ell = 1}^{d-2}s_{\ell}^{\alpha_{\ell}-1}s_{\ell}^{-\alpha_{d-1}+1}s_{\ell}^{-1}\\ 
    & = \prod\limits_{\ell =1}^{d-2} s_{\ell}^{\alpha_{\ell}-\alpha_{d-1}-1}.
\end{align*} Therefore, \begin{equation*}
    \pi_{u, \boldsymbol{k}}(x,\boldsymbol{s}) \propto (1+Cx^{d-1})^{-\frac{p}{2}}\exp \{ -xA(\boldsymbol{s}) - \frac{u}{2(1+Cx^{d-1})} \}x^{\alpha_{\cdot}-1}Q(\boldsymbol{s}). 
\end{equation*} Then, \begin{align*}
    \pi(x|u, \boldsymbol{k}) &\propto \int (1+Cx^{d-1})^{-\frac{p}{2}}\exp \{ - \frac{u}{2(1+Cx^{d-1})} \} x^{\alpha_{\cdot}-1} \exp \{-xA(\boldsymbol{s}) \}Q(\boldsymbol{s}) d\boldsymbol{s}\\ 
    &= (1+Cx^{d-1})^{-\frac{p}{2}}\exp \{ -\frac{u}{2(1+Cx^{d-1})} \}x^{\alpha_{\cdot}-1} \int \exp \{-xA(\boldsymbol{s}) \} Q(\boldsymbol{s}) ds\\ 
    &= (1+Cx^{d-1})^{-\frac{p}{2}} \exp \{ -\frac{u}{2(1+Cx^{d-1})} \}x^{\alpha_{\cdot}-1}H(x),
\end{align*} 
where 
\begin{equation*}
     H(x) := \int e^{-xA(\boldsymbol{s})}Q(\boldsymbol{s}) d\boldsymbol{s}.
\end{equation*}

Now we need to understand each term in $H(x)$. Firstly we will look at $A(\boldsymbol{s})$. Particularly, 
\begin{equation*}
    A(\boldsymbol{s}) = \sum\limits_{\ell = 1}^{d-1}s_{\ell} = \sum\limits_{\ell = 1}^{d-2}s_{\ell} + s_{d-1} = \sum\limits_{\ell = 1}^{d-2}s_{\ell} + (\prod\limits_{\ell = 1}^{d-2}s_{\ell})^{-1}. 
\end{equation*} By the AM-GM inequality \begin{equation*}
    \frac{1}{d-1}\sum\limits_{\ell = 1}^{d-1} s_{\ell} \geq (\prod\limits_{\ell = 1}^{d-1} s_{\ell})^{\frac{1}{d-1}} = 1. 
\end{equation*} That is $A(\boldsymbol{s}) = \sum\limits_{\ell =1}^{d-1} s_{\ell} \geq d-1$.Note that equality occurs when $s_{1} = \dots = s_{d-1} = \eta$. Then $\prod\limits_{\ell =1}^{d-1} s_{\ell} = \eta^{d-1} = 1 \implies \eta = 1$. Therefore, $s_{1} = \dots = s_{d-1} = 1$. Therefore $A(\boldsymbol{s})$ has a unique global minimizer at $\boldsymbol{s}^{*} = (1,\dots, 1) \in (0, \infty )^{d-2}$ with $A(\boldsymbol{s}^{*}) = d-1. $ Observe that for $i \in \{ 1, \dots, d-2 \}$, \begin{align*}
    \frac{\partial s_{d-1}}{\partial s_{i}} &= \frac{\partial }{\partial s_{i}} \{ (\prod\limits_{\ell =1}^{d-2}s_{\ell})^{-1} \} \\ 
    & = -(\prod\limits_{\ell = 1}^{d-2} s_{\ell })^{-2}\frac{\partial}{\partial s_{i}}\{\prod \limits_{\ell =1}^{d-2} s_{\ell} \} \\ 
    & = - (\prod\limits_{\ell =1}^{d-2} s_{\ell})^{-2} \frac{\partial}{\partial s_{i}} \{ s_{i} \prod\limits_{\ell \neq i}^{d-2} s_{\ell} \} \\ 
    & = -(\prod\limits_{\ell =1}^{d-2}s_{\ell})^{-2}\prod\limits_{\ell \neq i}^{d-2} s_{\ell}\\ 
    & = - (\prod\limits_{\ell =1}^{d-2}s_{\ell})^{-2}s_{i}^{-1}\prod\limits_{\ell = 1}^{d-2}s_{\ell} \\ 
    & = - \frac{(\prod\limits_{\ell =1}^{d-2}s_{\ell})^{-1}}{s_{i}}\\ 
    & = -\frac{s_{d-1}}{s_{i}}. 
\end{align*} Then, \begin{align*}
    \frac{\partial^{2}s_{d-1}}{\partial s_{i}^{2}} & = \frac{\partial}{\partial s_{i}} \{ -s_{i}^{-1}s_{d-1} \}\\ 
    & = -s_{i}^{-1} \frac{\partial s_{d-1}}{\partial s_{i}} + s_{d-1} \frac{\partial }{\partial s_{i}} \{- s_{i}^{-1} \}\\
    & = - \frac{1}{s_{i}}(-\frac{s_{d-1}}{s_{i}}) +s_{d-1}(s_{i}^{-2})\\ 
    & = \frac{s_{d-1}}{s_{i}^{2}} + \frac{s_{d-1}}{s_{i}^{2}}\\ 
    & = \frac{2s_{d-1}}{s_{i}^{2}}. 
\end{align*} Furthermore, \begin{equation*}
    \frac{\partial^{2}s_{d-1}}{\partial s_{i}\partial s_{j}} = \frac{\partial}{\partial s_{j}} \{ -\frac{s_{d-1}}{s_{i}} \} = - \frac{1}{s_{i}} \frac{\partial s_{d-1}}{\partial s_{j}} = - \frac{1}{s_{i}}(-\frac{s_{d-1}}{s_{j}}) = \frac{s_{d-1}}{s_{i}s_{j}}. 
\end{equation*} Then \begin{equation*}
    \frac{\partial A}{\partial s_{i}} = \sum\limits_{\ell = 1}^{d-2} \frac{\partial}{\partial s_{i}}s_{\ell} + \frac{\partial s_{d-1}}{\partial s_{i}} = 1- \frac{s_{d-1}}{s_{i}}. 
\end{equation*} Then at $\boldsymbol{s}^{*}$ we have $\nabla A(\boldsymbol{s}^{*}) = 0$. Then at $\boldsymbol{s}^{*}$ we have $\nabla^{2} A(\boldsymbol{s}^{*}) = I_{d-2} + \boldsymbol{1}\boldsymbol{1}^{T}$. The eigenvalues of $\nabla^{2}A(\boldsymbol{s}^{*})$ are 1 with multiplicity $d-3$ and $d-1$ with multiplicity 1. Therefore $\nabla^{2}A(\boldsymbol{s}^{*})$ is positive definite. Let $w_{i} := s_{i}^{-1}$ and $\boldsymbol{w}= (w_{1}, \dots, w_{d-2})^{T}$. Then the matrix $\boldsymbol{w}\boldsymbol{w}^{T}$ has entries $w_{i}w_{j} = (s_{i}s_{j})^{-1}$. Furthermore $\text{diag}(w_{i}^{2})$ has entries $s_{i}^{-2}$. Then 
\begin{equation*}
    \nabla^{2}A(\boldsymbol{s}) = s_{d-1}(\boldsymbol{w}\boldsymbol{w}^{T} + \text{diag}(w_{1}^{2},\dots, w_{d-2}^{2})).
\end{equation*}
We know $\boldsymbol{w}\boldsymbol{w}^{T} $ is positive semi-definite and $diag(w_{1}^{2},\dots, w_{d-2}^{2})$ is positive definite. Then $\nabla^{2}A(\boldsymbol{s})$ is positive definite for every $\boldsymbol{s} \in (0,\infty)^{d-2}$. Now fix $r \in (0,1)$. If $||\boldsymbol{s}-\boldsymbol{s}_{*}|| \leq r$ then $r^{2} \geq ||\boldsymbol{s}-\boldsymbol{s}_{*}||^{2} = \sum\limits_{i=1}^{d-2} (s_{i}-1)^{2} \geq (s_{i}-1)^{2}, i \in \{ 1, \dots, d-2 \}$. Therefore, $|s_{i}-1| \leq r \iff - r \leq s_{i}-1 \leq r \iff 1-r \leq s_{i} \leq 1+r$. This implies \begin{equation*}
    (1-r)^{d-2} \leq \prod\limits_{i=1}^{d-2}s_{i} \leq (1+r)^{d-2}. 
\end{equation*} Hence, \begin{equation*}
    (1+r)^{-(d-2)} \leq (\prod\limits_{i=1}^{d-2} s_{i})^{-1} \leq (1-r)^{-(d-2)} \text{ i.e. } (1+r)^{-(d-2)} \leq s_{d-1} \leq (1-r)^{-(d-2)}. 
\end{equation*} We know for any vector $\boldsymbol{z}$, \begin{equation*}
    \boldsymbol{z}^{T}(\boldsymbol{w}\boldsymbol{w}^{T} + diag(w_{1}^{2},\dots,w_{d-2}^{2}))\boldsymbol{z} \geq \boldsymbol{z}^{T} diag(w_{1}^{2}, \dots, w_{d-2}^{2}) \boldsymbol{z} \geq (\min\limits_{1 \leq i \leq d-2} w_{i}^{2})||\boldsymbol{z}||^{2}. 
\end{equation*} Then, \begin{align*}
    \lambda_{\text{min}}(\nabla^{2} A(\boldsymbol{s})) &= \min\limits_{||\boldsymbol{z}|| =1} \boldsymbol{z}^{T} \nabla^{2}A(\boldsymbol{s}) \boldsymbol{z} \\ 
    &\geq \min\limits_{||\boldsymbol{z}|| = 1} s_{d-1}(\min\limits_{1\leq i \leq d-2}w_{i}^{2})||\boldsymbol{z}||^{2}\\
    &= s_{d-1}\min\limits_{1 \leq i \leq d-2}w_{i}^{2}\\ 
    &= s_{d-1}\min\limits_{1\leq i \leq d-2} s_{i}^{-2}\\ 
    &= \frac{s_{d-1}}{(\max\limits_{1\leq i \leq d-2}s_{i})^{2}}.
\end{align*} On $||\boldsymbol{s}-\boldsymbol{1}_{d-2}|| \leq r$, $\max\limits_{1 \leq i \leq d-2}s_{i} \leq 1+r$. Furthermore, $s_{d-1} \geq (1+r)^{-(d-2)}$. Then, \begin{equation*}
    \lambda_{\text{min}} (\nabla^{2} A(\boldsymbol{s})) \geq \frac{(1+r)^{-(d-2)}}{(1+r)^{2}} = (1+r)^{-d +2} (1+r)^{-2}= (1+r)^{-d}. 
\end{equation*}  Recall that for PSD matrices $A$ and $B$ $\lambda_{\text{max}}(A+B) \leq \lambda_{\text{max}}(A) + \lambda_{\text{max}}(B)$. Then, \begin{align*}
    \lambda_{\text{max}}(\nabla^{2}A(\boldsymbol{s})) &= \lambda_{\text{max}}(s_{d-1}[\boldsymbol{w}\boldsymbol{w}^{T} + diag(w_{1}^{2}, \dots, w_{d-2}^{2})])\\ 
    & = s_{d-1} \lambda_{\text{max}}(\boldsymbol{w}\boldsymbol{w}^{T} + diag(w_{1}^{2}, \dots, w_{d-2}^{2})) \\ 
    & \leq s_{d-1}[\lambda_{\text{max}}(\boldsymbol{w}\boldsymbol{w}^{T}) + \lambda_{\text{max}}(diag(w_{1}^{2}, \dots, w_{d-2}^{2}))]\\ 
    & = s_{d-1}[||\boldsymbol{w}||^{2} + \max\limits_{1\leq i \leq d-2} w_{i}^{2}] . 
\end{align*} Recall, on $||\boldsymbol{s}-\boldsymbol{1}_{d-2}|| \leq r, s_{i} \geq 1- r$. So $w_{i} = s_{i}^{-1} \leq (1-r)^{-1}$. Therefore, \begin{align*}
    ||\boldsymbol{w}||^{2} = \sum\limits_{i=1}^{d-2} w_{i}^{2} \leq \sum\limits_{i=1}^{d-2} \max\limits_{1\leq i \leq d-2} w_{i}^{2}  &\leq \sum\limits_{i=1}^{d-2} (\frac{1}{1-r})^{2} = \frac{d-2}{(1-r)^{2}},\\ 
    &\text{ and }\\ 
    \max\limits_{1 \leq i \leq d-2} w_{i}^{2} &\leq \frac{1}{(1-r)^{2}}.
\end{align*} 
 Thus, \begin{equation*}
    ||\boldsymbol{w}||^{2} + \max\limits_{1 \leq i \leq d-2} w_{i}^{2} \leq \frac{d-2}{(1-r)^{2}} + \frac{1}{(1-r)^{2}} = \frac{d-2+1}{(1-r)^{2}} = \frac{d-1}{(1-r)^{2}}. 
\end{equation*} Recall, on $||\boldsymbol{s}-\boldsymbol{1}_{d-2}|| \leq r$, $s_{d-1} \leq (1-r)^{-(d-2)}$. Thus, \begin{equation*}
    \lambda_{\text{max}}(\nabla^{2}A(\boldsymbol{s})) \leq (1-r)^{-(d-2)}\frac{d-1}{(1-r)^{2}} = (d-1)(1-r)^{-d+2}(1-r)^{-2}= (d-1)(1-r)^{-d}. 
\end{equation*} Hence, on $||\boldsymbol{s}-\boldsymbol{1}_{d-2}|| \leq r$, \begin{equation*}
    (1+r)^{-d} \leq \lambda_{\min}(\nabla^{2}A(\boldsymbol{s})) \leq\lambda_{\max}(\nabla^{2}A(\boldsymbol{s}))  \leq (d-1)(1-r)^{-d}. 
\end{equation*} Let $\Delta := \boldsymbol{s}-\boldsymbol{s}^{*}$ and recall that $\boldsymbol{s}^{*} = \boldsymbol{1}_{d-2}$. Let $g(t) := A(\boldsymbol{s}^{*} + t\boldsymbol{\Delta}), t \in [0,1]$. By Taylor's integral remainder theorem, \begin{align*}
    g(1) &= g(0) + \frac{g'(0)}{1!}(1-0) + \int\limits_{0}^{1}\frac{g''(t)}{1!} (1-t)^{1} dt\\ 
    &= g(0) + g'(0) + \int\limits_{0}^{1} (1-t)g''(t) dt. 
\end{align*} Note\begin{align*}
    \frac{dg}{dt} = \frac{d}{dt} {A}(\boldsymbol{s}^{*}+t\boldsymbol{\Delta}) &= \nabla {A}(\boldsymbol{s}^{*}+t\boldsymbol{\Delta})^{T}\boldsymbol{\Delta} \\ 
    &\text{and } \\ 
    \frac{d^{2}g}{dt^{2}} = \frac{d}{dt}(\nabla {A}(\boldsymbol{s}^{*}+t\boldsymbol{\Delta})^{T}\boldsymbol{\Delta})& = \boldsymbol{\Delta}^{T}\nabla^{2}A(\boldsymbol{s}^{*}+t\boldsymbol{\Delta}) \boldsymbol{\Delta}. 
\end{align*}  Thus, 
\begin{align*}
    A(\boldsymbol{s}) &= A(\boldsymbol{s}^{*}) + \nabla A(\boldsymbol{s}^{*})^{T} \boldsymbol{\Delta} + \int\limits_{0}^{1} (1-t) \boldsymbol{\Delta}^{T}(\nabla^{2}A(\boldsymbol{s}^{*} + t\boldsymbol{\Delta}))\Delta dt \\ 
    \text{ i.e  } A(\boldsymbol{s}) - A(\boldsymbol{s}^{*}) &= \int\limits_{0}^{1} (1-t)\boldsymbol{\Delta}^{T}(\nabla^{2}A(\boldsymbol{s}^{*} + t \boldsymbol{\Delta})) \boldsymbol{\Delta} dt. 
\end{align*} Then by the eigenvalue bounds, 
\begin{align*}
    \boldsymbol{\Delta}^{T}(\nabla^{2} A(\boldsymbol{s}^{*}+t\boldsymbol{\Delta}))\boldsymbol{\Delta} & \geq \lambda_{\text{min}}(\nabla^{2}A(\boldsymbol{s}^{*} + t\boldsymbol{\Delta})) ||\boldsymbol{\Delta}||^{2} \geq (1+r)^{-d}||\boldsymbol{\Delta}||^{2} \\ 
    & \text{and }\\ 
    \boldsymbol{\Delta}^{T}(\nabla^{2}A(\boldsymbol{s}^{*} + t\boldsymbol{\Delta}))\boldsymbol{\Delta} & \leq (d-1)(1-r)^{-d} ||\boldsymbol{\Delta}||^{2}. 
\end{align*}
Observe that $\int\limits_{0}^{1} (1-t)dt = 0.5$. Hence, \begin{equation*}
     \frac{1}{2}(1+r)^{-d} ||\boldsymbol{\Delta}||^{2} \leq A(\boldsymbol{s}) - A(\boldsymbol{s}^{*}) \leq \frac{1}{2}(d-1)(1-r)^{-d} ||\boldsymbol{\Delta}||^{2}
\end{equation*} 
if and only if 
\begin{equation*}
    (d-1) + \frac{1}{2}(1+r)^{-d} ||\boldsymbol{s}-\boldsymbol{1}_{d-2}||^{2} \leq A(\boldsymbol{s}) \leq (d-1) + \frac{1}{2}(d-1)(1-r)^{-d}||\boldsymbol{s}-\boldsymbol{1}_{d-2}||^{2},
\end{equation*}
on $||\boldsymbol{s}-\boldsymbol{1}_{d-2}|| \leq r.$ That is, on $||\boldsymbol{s}-\boldsymbol{s}^{*}|| \leq r$, 
\begin{equation*}
    (d-1) + C_{-} ||\boldsymbol{s}-\boldsymbol{1}_{d-2}||^{2} \leq A(\boldsymbol{s}) \leq (d-1) + C_{+}||\boldsymbol{s}-\boldsymbol{1}_{d-2}||^{2}, 
\end{equation*} 
\begin{equation*}
    C_{-} := \frac{1}{2}(1+r)^{-d} \text{ and } C_{+} := \frac{1}{2} (d-1)(1-r)^{-d}. 
\end{equation*}
Now we need to bound $Q(\boldsymbol{s})$ which will allow us to bound $H(x)$. Recall \begin{align*}
    Q(\boldsymbol{s})  &= \prod\limits_{\ell = 1}^{d-2}s_{\ell}^{\alpha_{\ell}-\alpha_{d-1} - 1}\\ 
    & = \prod\limits_{\ell = 1}^{d-2} s_{\ell}^{\gamma_{\ell}}, \gamma_{\ell}: = \alpha_{\ell} - \alpha_{d-1}-1, \ell = 1, \dots, d-2. 
\end{align*} Just like for $A(\boldsymbol{s})$ consider the ball $B(r) := \{ \boldsymbol{s} \in (0,\infty)^{d-2}:||\boldsymbol{s}-\boldsymbol{1}_{d-2}|| \leq r \}$. Again, this implies $1-r \leq s_{\ell} \leq 1+r$. Fix some $r_{0} \in (0,1)$. Then on $B(r_{0}), s_{\ell} \in [a,b] \forall \ell$, where $a:= 1- r_{0} >0$ and $b:= 1+r_{0}<\infty$. Now, fix an $\ell$. On the interval $[a,b] \subset (0,\infty)$, $t \rightarrow t^{\gamma_{\ell}}$ is continuous and monotone. This means for $\gamma_{\ell} \geq 0, a^{\gamma_{\ell}} \leq s_{\ell}^{\gamma_{\ell}} \leq b^{\gamma_{\ell}}$ and $ \gamma_{\ell} \leq 0, b^{\gamma_{\ell}} \leq s_{\ell}^{\gamma_{\ell}} \leq a^{\gamma_{\ell}}$. That is, \begin{equation*}
    \min \{ a^{\gamma_{\ell}}, b^{\gamma_{\ell}} \} \leq s_{\ell}^{\gamma_{\ell}} \leq \max \{ a^{\gamma_{\ell}}, b^{\gamma_{\ell}} \}.
\end{equation*} Then \begin{equation*}
    \prod\limits_{\ell = 1}^{d-2} \min \{ a^{\gamma_{\ell}}, b^{\gamma_{\ell}} \} \leq \prod\limits_{\ell = 1}^{d-2} s_{\ell}^{\gamma_{\ell}} \leq \prod\limits_{\ell =1}^{d-2} \max \{ a^{\gamma_{\ell}}, b^{\gamma_{\ell}} \}. 
\end{equation*} That is $q_{-} \leq Q(\boldsymbol{s}) \leq q_{+}, \forall s \in B(r_{0}),$ where $q_{-} := \prod\limits_{\ell = 1}^{d-2} \min \{ (1-r_{0})^{\gamma_{\ell}}, (1+r_{0})^{\gamma_{\ell}} \}$  and $q_{+} := \prod\limits_{\ell =1}^{d-2} \max \{ (1-r_{0})^{\gamma_{\ell}}, (1+r_{0})^{\gamma_{\ell}} \}$. Now, \begin{align*}
    H(x) & = \int e^{-xA(\boldsymbol{s})}Q(\boldsymbol{s}) d\boldsymbol{s}\\ 
    & \geq \int\limits_{B} e^{-xA(\boldsymbol{s})}Q(\boldsymbol{s}) d\boldsymbol{s}\\ 
    &\geq \int\limits_{B} e^{-x[(d-1)+C_{+}||\boldsymbol{s}-\boldsymbol{s}^{*}||^{2}]}q_{-}d\boldsymbol{s}\\ 
    & = q_{-} e^{-(d-1)x} \int\limits_{B} e^{-xC_{+}||\boldsymbol{s}-\boldsymbol{s}^{*}||^{2}}d\boldsymbol{s}. 
\end{align*} Let $\boldsymbol{z} = \sqrt{x} (\boldsymbol{s}-\boldsymbol{s}^{*})$. Then $\boldsymbol{s} = \boldsymbol{s}^{*} + x^{-\frac{1}{2}}\boldsymbol{z}$. Then $\partial\boldsymbol{s} = x^{-\frac{1}{2}} I_{d-2} \partial\boldsymbol{z}$. Therefore, \begin{equation*}
    |\det(\frac{d \boldsymbol{s}}{d \boldsymbol{z}})| = |\det(x^{-\frac{1}{2}}I_{d-2})| = x^{-\frac{(d-2)}{2}} \text{ i.e. } d\boldsymbol{s} = x^{- \frac{(d-2)}{2}}d\boldsymbol{z}. 
\end{equation*} Furthermore if $\boldsymbol{s} \in B$, then $||\boldsymbol{z}|| = \sqrt{x} ||\boldsymbol{s}-\boldsymbol{s}^{*}|| \leq \sqrt{x} r_{0}$. Therefore, \begin{align*}
    \int\limits_{B} e^{-xC_{+} ||\boldsymbol{s}-\boldsymbol{s}_{*}||^{2}} &= \int\limits_{||\boldsymbol{z}|| \leq r_{0} \sqrt{x}}e^{-C_{+} ||\boldsymbol{z}||^{2}}x^{- \frac{(d-2)}{2}}d\boldsymbol{z} \\ & = x^{-\frac{(d-2)}{2}} \int\limits_{||\boldsymbol{z}|| \leq r_{0} \sqrt{x}} e^{-C_{+}||\boldsymbol{z}||^{2}}d\boldsymbol{z} \\ 
    & = x^{-\frac{(d-2)}{2}} I(x) , I(x) := \int\limits_{||\boldsymbol{z}|| \leq r_{0}\sqrt{x}} e^{-C_{+}||\boldsymbol{z}||^{2}}d\boldsymbol{z}. 
\end{align*} Observe that $I(x)$ is monotone increasing in $x$. If $x_{2} \geq x_{1}$ then $r_{0}\sqrt{x_{2}} \geq r_{0}\sqrt{x_{1}}$ so $||\boldsymbol{z}|| \leq r_{0} \sqrt{x_{1}} \subseteq ||\boldsymbol{z}|| \leq r_{0}\sqrt{x_{2}}$. Since the integrand is strictly positive $I(x_{2}) \geq I(x_{1})$. By the monotone convergence theorem  \begin{align*}
    \lim\limits_{x \rightarrow \infty} \int\limits_{||\boldsymbol{z}|| \leq r_{0} \sqrt{x}} e^{-C_{+}||\boldsymbol{z}||^{2}}dz &= \int\limits_{\mathbb{R}^{d-2}} e^{-C_{+}||\boldsymbol{z}||^{2}}dz\\ 
    &= \prod\limits_{\ell=1}^{d-2} \int\limits_{\mathbb{R}} e^{-C_{+}t^{2}} dt,\text{ by Tonelli's theorem} \\ 
    &= \prod\limits_{\ell = 1}^{d-2} (\frac{\sqrt{\pi}}{\sqrt{C_{+}}}) \\ 
    & = (\frac{\pi}{C_{+}})^{\frac{d-2}{2}}. 
\end{align*}  Then, $\exists x_{1}:\forall x \geq x_{1}, I(x) \geq b_{0}> 0$ for some number $b_{0}$.  Thus, \begin{equation*}
    \exists b_{1}>0, \exists x_{1}: \forall x\geq x_{1}, H(x) \geq q_{-} e^{-(d-1)x}x^{-\frac{(d-2)}{2}}b_{0} = b_{1}x^{-\frac{(d-2)}{2}} e^{-(d-1)x}, b_{1}:= q_{-}b_{0}. 
\end{equation*}Now we need to upper bound $H(x)$. To do this we will split the integral over $B$ and $B^{c}$. So, \begin{align*}
    H_{B}(x) &= \int\limits_{B} e^{-xA(\boldsymbol{s})} Q(\boldsymbol{s}) d\boldsymbol{s} \\ &\leq q_{+} e^{-(d-1)x} \int\limits_{B}e^{-xC_{-}||\boldsymbol{s}-\boldsymbol{s}^{*}||^{2}} d\boldsymbol{s} \\ 
    & \leq q_{+}e^{-(d-1)x}x^{-\frac{(d-2)}{2}} \int\limits_{\mathbb{R}^{d-2}}e^{-C_{-}||\boldsymbol{z}||^{2}} d\boldsymbol{z}\\ 
    &= b_{2}x^{-\frac{(d-2)}{2}} e^{-(d-1)x}. 
\end{align*} Recall that by the AM-GM inequality $A(\boldsymbol{s}) \geq d-1$ with equality if $\boldsymbol{s}=(1,\dots,1) = \boldsymbol{s}^{*}$. We now aim to show that that $\exists \eta > 0 : \forall \boldsymbol{s} \in B^{c}, A(\boldsymbol{s}) \geq d-1+\eta$. Fix $M>0$ and consider $S_{M} = \{ \boldsymbol{s} \in (0, \infty)^{d-2}: A(\boldsymbol{s}) \leq M \}$. As $A$ is continuous, $S_{M}$ is closed. Now we need boundedness away from $0$ and $\infty$. If $A(\boldsymbol{s}) \leq M$, for each $i$, 
\begin{equation*}
    s_{i} \leq \sum\limits_{\ell = 1}^{d-2} s_{\ell} \leq A(\boldsymbol{s}) \leq M.
\end{equation*}
So $s_{i} \leq M$. Also, $(\prod\limits_{\ell = 1}^{d-2}s_{\ell})^{-1} \leq A(\boldsymbol{s}) \leq M$ which implies $\prod\limits_{\ell = 1}^{d-2}s_{\ell} \geq \frac{1}{M}$. Therefore $\prod\limits_{\ell = 1}^{d-2}s_{\ell} = s_{i}\prod\limits_{\ell \neq i} s_{\ell} \leq s_{i} M^{d-3}.$ So $M^{-1} \leq \prod\limits_{\ell = 1}^{d-2} s_{\ell} \leq s_{i} M^{d-3}$ which implies $s_{i} \geq M^{-(d-2)}$. Thus $\forall \boldsymbol{s} \in S_{M}, M^{-(d-2)} \leq s_{i} \leq M, \forall i$. Therefore, $S_{M} \subseteq [M^{-(d-2)}, M]^{d-2}$ which is compact. As $S_{M}$ is closed $S_{M}$ is compact as a closed subset of a compact set is also compact. Note that $B^{c}:= \{ ||\boldsymbol{s}-\boldsymbol{s}^{*}|| > r_{0} \}$. Let $F:= \{ ||\boldsymbol{s}-\boldsymbol{s}^{*}|| \geq r_{0} \}$. Then $B^{c} \subseteq F$. Now we will show that $\exists \eta >0$ such that $A(\boldsymbol{s}) \geq (d-1) + \eta, \forall \boldsymbol{s} \in F$. Assume no such $\eta$ exists. Then $\inf\limits_{\boldsymbol{s} \in F} A(\boldsymbol{s}) = d-1$. Then we can pick $\boldsymbol{s}^{(n)} \in F$ such that $A(\boldsymbol{s}^{(n)}) \downarrow d-1$. In particular, $\forall$ sufficiently large $n$ we have $A(\boldsymbol{s}^{(n)}) \leq d$. Now consider $K: = F \cap S_{d}$. We know $\boldsymbol{s}^{(n)} \in K$, $\forall$ sufficiently large $n$. We know $F$ and $S_{d}$ are closed, therefore $K$ is closed. Since $S_{d}$ is bounded, $K$ is also bounded. Therefore $K$ is compact. Since $K$ is compact $\{ \boldsymbol{s}^{(n)} \} \subseteq K$ has a convergent subsequence $\boldsymbol{s}^{n_{k}} \rightarrow \bar{\boldsymbol{s}} \in K$. By continuity of $A, A(\bar{\boldsymbol{s}}) = \lim\limits_{k \rightarrow \infty} A(\boldsymbol{s}^{n_{k}}) = d-1$. From the AM-GM inequality we know $A(\boldsymbol{s}) = d-1$ iff $\boldsymbol{s}=\boldsymbol{s}^{*}$. Therefore $\bar{\boldsymbol{s}} = \boldsymbol{s}^{*}$. But $\bar{\boldsymbol{s}} \in K \subset F$ and every point in $F$ satisfies $||\boldsymbol{s}-\boldsymbol{s}^{*}|| \geq r_{0}$. Therefore, $\boldsymbol{s} \neq \boldsymbol{s}^{*}$ is a contradiction. Therefore $\inf\limits_{\boldsymbol{s} \in F} A(\boldsymbol{s}) > d-1$. Let $\eta : = \inf\limits_{\boldsymbol{s} \in F} A(\boldsymbol{s}) -(d-1) > 0$. Therefore, $A(\boldsymbol{s}) \geq (d-1) + \eta$. Now we are in a position to upper bound $H(x)$ on $B^{c}$. Note that \begin{align*}
    e^{-xA(\boldsymbol{s})} & = e^{-\frac{x}{2}((d-1) + \eta)}e^{-x(A(\boldsymbol{s}) - \frac{(d-1)+\eta}{2})}\\ 
    &\leq e^{-\frac{x}{2}((d-1) + \eta)} e^{-\frac{x}{2}A(\boldsymbol{s})}.
\end{align*} Then, 
\begin{align*}
    H_{B^{c}}(x) &= \int\limits_{B^{c}} e^{-xA(\boldsymbol{s})}Q(\boldsymbol{s}) ds \\ 
    &\leq e^{-\frac{x}{2} (d-1+\eta)} \int\limits_{B^{c}} e^{-\frac{x}{2}A(\boldsymbol{s})} Q(\boldsymbol{s}) d\boldsymbol{s}\\ 
    &\leq e^{-\frac{x}{2}(d-1+\eta)} \int\limits_{(0,\infty)^{d-2}} e^{-\frac{x}{2} A(\boldsymbol{s})}Q(\boldsymbol{s}) d\boldsymbol{s}\\ 
    & = e^{-\frac{x}{2}(d-1+\eta)} H(\frac{x}{2}).
\end{align*} Then, \begin{align*}
    H(x) &= H_{B}(x) + H_{B^{c}}(x) \\ 
    & \leq K_{B} x^{-\frac{d-2}{2}}e^{-(d-1)x} + e^{-\frac{x}{2}(d-1+\eta)}H(\frac{x}{2}), K_{B} = b_{2}. 
\end{align*} Let $M(x):= x^{\frac{d-2}{2}}e^{(d-1)x}H(x)$. This implies \begin{equation*}
    M(x) = x^{\frac{(d-2)}{2}}e^{(d-1)x} H(x) \leq K_{B}+x^{\frac{(d-2)}{2}}e^{(d-1)x}e^{-\frac{x}{2}(d-1+\eta)} H(\frac{x}{2}). 
\end{equation*} Observe that \begin{equation*}
    H(\frac{x}{2}) = (\frac{x}{2})^{-\frac{(d-2)}{2}}e^{-\frac{(d-1)x}{2}}M(\frac{x}{2}). 
\end{equation*} Hence, 
\begin{align*}
    M(x) &\leq K_{B} + x^{\frac{(d-2)}{2}}e^{(d-1)x}e^{-\frac{x}{2}(d-1+\eta)} (\frac{x}{2})^{- \frac{(d-2)}{2}}e^{- \frac{(d-1)x}{2}}M(\frac{x}{2})\\ 
    &= K_{B} + 2^{\frac{(d-2)}{2}}e^{-\frac{\eta x}{2}}M(\frac{x}{2}). 
\end{align*} That is \begin{equation*}
    M(x) \leq K_{B} + q(x) M(\frac{x}{2}), q(x) := 2^{\frac{(d-2)}{2}}e^{-\frac{\eta x}{2}}. 
\end{equation*} Choose $x_{0}$ such that $2^{\frac{d-2}{2}}e^{-\frac{\eta x_{0}}{2}} \leq \frac{1}{2}$. Then, \begin{equation*}
    M(x) \leq K_{B} + \frac{1}{2}M(\frac{x}{2}
    ) \forall x \geq x_{0}. 
\end{equation*} Since $A(\boldsymbol{s}) \geq 0$, $x \rightarrow H(x)$ is nonincreasing in $x$. So for any $x \in [x_{0}, 2x_{0}]$,\\ $ H(x) \leq H(x_{0})$. Therefore \begin{equation*}
    M(x) = x^{\frac{(d-2)}{2}}e^{(d-1)x}H(x) \leq (2x_{0})^{\frac{(d-2)}{2}}e^{2(d-1)x_{0}}H(x_{0}). 
\end{equation*} Let $M_{*}:= (2x_{0})^{\frac{(d-2)}{2}}e^{2(d-1)x_{0}}H(x_{0})$. Now we need to prove $M_{*}$ is finite. Once this is done we can utilize a recursive application of the $M(x) \leq K_{B} + \frac{1}{2}M(\frac{x}{2})$ inequality which will yield an upper bound on $H(x)$. Recall $H(x) = \int\limits_{(0,\infty)^{d-2}} e^{-xA(\boldsymbol{s})}Q(\boldsymbol{s}) d\boldsymbol{s}$ where\\ $A(\boldsymbol{s}) = \sum\limits_{\ell =1}^{d-2} s_{\ell} + (\prod\limits_{\ell=1}^{d-2}s_{\ell})^{-1}$ and $Q(\boldsymbol{s}) = \prod\limits_{\ell = 1 }^{d-2} s_{\ell}^{\beta_{\ell}}, \beta_{\ell} \in \mathbb{R}$. For each subset \\$J \subseteq \{ 1,\dots, d-2\}$ let $R_{J}= \{\boldsymbol{s}: 0 < s_{i} < 1, \text{ for } i \in J, s_{i} \geq 1 \text{ for } i \in J^{c} \}$. Then the $2^{d-2}$ sets are disjoint and their union is $(0, \infty)^{d-2}$. So $H(x) = \sum\limits_{J} \int\limits_{R_{J}} e^{-xA(\boldsymbol{s})}Q(\boldsymbol{s}) d\boldsymbol{s}$. On $R_{\phi} = [1,\infty) ^{d-2}, A(\boldsymbol{s}) = \sum\limits_{\ell = 1}^{d-2} s_{\ell} + (\prod\limits_{\ell = 1}^{d-2} s_{\ell})^{-1} \geq \sum\limits_{\ell = 1}^{d-2} s_{\ell}$.  Then $e^{-xA(\boldsymbol{s})} \leq e^{-x \sum\limits_{\ell = 1}^{d-2} s_{\ell}} $ since $x>0$. Therefore, \begin{align*}
    \int\limits_{R_{\phi}} e^{-xA(\boldsymbol{s})}Q(\boldsymbol{s}) &\leq \int\limits_{[1,\infty)^{d-2}} e^{-x \sum\limits_{\ell =1}^{d-2} s_{\ell} } \prod\limits_{\ell =1}^{d-2}  s_{\ell}^{\beta_{\ell}}d\boldsymbol{s} \\ 
    &= \prod\limits_{\ell = 1}^{d-2} \int\limits_{1}^{\infty} e^{-xs_{\ell}}s_{\ell}^{\beta_{\ell}} ds_{\ell}. 
\end{align*} Let $g(s) = {s}^{-N} e^{\frac{x}{2}{s}}$. Then \begin{equation*}
    \log(g({s})) = \frac{x}{2}{s} - N\log({s}) \rightarrow \infty \text{ as } s \rightarrow \infty. 
\end{equation*} Therefore $\exists S$ such that $\forall s \geq S, \log(g(s)) \geq 0.$ Then $g(s) \geq e^{0} = 1, \forall s \geq S$. Then, $\exists S$ such that $ \forall s \geq S, s^{-N}e^{\frac{x}{2}s} \geq 1$. That is $s^{N} \leq e^{\frac{x}{2}s}$. If $s \geq 1$ and $N > \beta $ then $s^{N - \beta} \geq 1. $ Then, $s^{\beta}s^{N-\beta} \geq s^{\beta}$. That is $s^{\beta} \leq s^{N}$. Then $e^{-xs}s^{\beta} \leq e^{-xs}s^{N} \leq e^{-xs}e^{\frac{x}{2}s}  = e^{-\frac{x}{2}s}$. This implies \begin{align*}
    \int\limits_{1}^{\infty} e^{-xs} s^{\beta} ds &= \int\limits_{1}^{S}e^{-xs}s^{\beta}ds + \int\limits_{S}^{\infty} e^{-xs}s^{\beta}ds \\ 
    & \leq \int\limits_{1}^{S} e^{-xs}s^{\beta} ds + \int\limits_{S}^{\infty} e^{- \frac{x}{2}s}ds \\ 
    &= \int\limits_{1}^{S} e^{-xs}s^{\beta}ds + \frac{2}{x} e^{-\frac{x}{2}S}< \infty. 
\end{align*} This implies $\int\limits_{R_{\phi}} e^{-xA(\boldsymbol{s})}Q(\boldsymbol{s}) d\boldsymbol{s} < \infty$. For a non-empty $J \subset \{ 1, \dots, d-2 \}$ with \\$|J| = k \geq 1$. Let $\boldsymbol{v} = (v_{\ell})_{\ell \in J^{c}} \in [1,\infty)^{d-2-k}$ and $\boldsymbol{u} = (u_{j})_{j \in J} \in (0,1)^{k}$. So on $R_{J}, s= (\boldsymbol{v},\boldsymbol{u})$ and $\prod\limits_{\ell =1}^{d-2} s_{\ell} = (\prod\limits_{\ell \in J^{c}} v_{\ell})(\prod\limits_{j \in J}u_{j})$ On $R_{J}$, all $u_{j} >0$ so $\sum\limits_{j \in J} u_{j} \geq 0$. Hence \begin{align*}
    A(\boldsymbol{v},\boldsymbol{u}) &= \sum\limits_{\ell \in J^{c}} v_{\ell} + \sum\limits_{j \in J} u_{j} + (\prod\limits_{\ell \in J^{c}}v_{\ell} \prod\limits_{j \in J} u_{j})^{-1} \\ 
    &\geq \sum\limits_{\ell \in J^{c}} v_{\ell} + (\prod\limits_{\ell  \in J^{c}} v_{\ell} \prod\limits_{j \in J} u_{j})^{-1}. 
\end{align*} Since $x>0$, \begin{equation*}
    \exp \{ - x A(\boldsymbol{v},\boldsymbol{u}) \} \leq \exp \{ - x \sum\limits_{\ell \in J^{c}} v_{\ell} \} \exp \{ - x(\prod\limits_{\ell \in J^{c}} v_{\ell} \prod\limits_{j \in J}u_{j})^{-1} \}. 
\end{equation*} So  \begin{equation*}
    \int\limits_{R_{J}}e^{-xA(\boldsymbol{s})}Q(\boldsymbol{s}) d\boldsymbol{s} \leq \int\limits_{[1,\infty)^{d-2-k}} \exp \{ - x \sum\limits_{\ell \in J^{c}} v_{\ell} \}(\prod\limits_{\ell \in J^{c}} v_{\ell}^{\beta_{\ell}})I_{J}(c(\boldsymbol{v}))d\boldsymbol{v},
\end{equation*} 
by Tonelli's theorem, where
\begin{equation*}
    c_{v}: = x \prod\limits_{\ell \in J^{c}} v_{\ell}^{-1} \in (0,x] \text{ and } I_{J}(c) : = \int\limits_{(0,1)^{k}} \exp \{ - c(\prod\limits_{j \in J} u_{j})^{-1} \} \prod\limits_{j \in J} u_{j}^{\beta_{j}}d\boldsymbol{u}. 
\end{equation*} Let $t_{j} = u_{j}^{-1} \in [1,\infty)$. Then $u_{j} = t_{j}^{-1}, du_{j} = -t_{j}^{-2}dt_{j}$, and $u_{j}^{\beta_{j}} = t_{j}^{-\beta_{j}}$. Furthermore $(\prod u_{j})^{-1} = \prod t_{j}$. Then, \begin{equation*}
    I_{J}(c) = \int\limits_{[1,\infty)^{k}} \exp \{ - c\prod\limits_{j=1}^{k}t_{j} \} \prod\limits_{j=1}^{k} t_{j}^{-(\beta_{j}+2)}d\boldsymbol{t}. 
\end{equation*} Let $M_{j} := \max \{ 0, -(\beta_{j}+2) \} \geq 0$. Since $t_{j} \geq 1, t_{j}^{- (\beta_{j}+2)} \leq t_{j}^{M_{j}}. $ Hence \begin{equation*}
    I_{J}(c) \leq \int\limits_{[1, \infty)^{k}} \exp \{ - c \prod\limits_{j=1}^{k}t_{j} \} \prod\limits_{j=1}^{k} t_{j}^{M_{j}}d\boldsymbol{t}. 
\end{equation*} Let $S_{\ell} := \{ \boldsymbol{t} \in [1, \infty)^{k}: t_{\ell} = \max\limits_{1\leq j \leq k} t_{j} \}$. On $S_{\ell}$, $\prod\limits_{j=1}^{k} t_{j} \geq t_{\ell}$ with $[1,\infty)^{k} = \bigcup\limits_{\ell = 1}^{k}S_{\ell}$. So 
\begin{equation*}
    \exp \{ - c\prod\limits_{j=1}^{k}t_{j} \} \leq \exp \{ - ct_{\ell}\} \text{ and }\prod\limits_{j=1}^{k}t_{j}^{M_{j}} \leq t_{\ell}^{\sum\limits_{j=1}^{k}M_{j}}.
\end{equation*}
Therefore, \begin{equation*}
    \int\limits_{S_{\ell}} \exp \{ - c \prod\limits_{j=1}^{k} t_{j} \} \prod\limits_{j=1}^{k} t_{j}^{M_{j}} d\boldsymbol{t} \leq \int\limits_{S_{\ell}} e^{-ct_{\ell}} t_{\ell}^{\sum\limits_{j=1}^{k}M_{j}}d\boldsymbol{t}. 
\end{equation*} For fixed $t_{\ell} = t$, the remaining $k-1$ coordinates lie in $[1,t]^{k-1}$ whose volume is \\$(t-1)^{k-1} \leq t^{k-1}$. So, 
\begin{align*}
    \int\limits_{S_{\ell}} \exp \{ - c \prod\limits_{j=1}^{k} t_{j} \} \prod\limits_{j=1}^{k}t_{j}^{M_{j}} d\boldsymbol{t} &\leq \int\limits_{1}^{\infty}e ^{-ct}t^{\sum\limits_{j=1}^{k}M_{j}}t^{k-1}dt\\
    &= \int\limits_{1}^{\infty} e^{-ct}t^{P}dt, P:= \sum\limits_{j=1}^{k}M_{j} + k-1 \geq 0. 
\end{align*} Then, 
\begin{equation*}
    \int\limits_{1}^{\infty} e^{-ct}t^{P} dt \leq \int\limits_{0}^{\infty} e^{-ct}t^{P}dt. 
\end{equation*}
Let $u = ct$ then $t = c^{-1}u$ and $dt = c^{-1}du$. So 
\begin{equation*}
    \int\limits_{0}^{\infty}e^{-ct}t^{P} dt = \int\limits_{0}^{\infty} e^{-u}(\frac{u}{c})^{P} \frac{du}{c} = c^{-(P+1)}\int\limits_{0}^{\infty} e^{-u}u^{P+1-1} du = c^{-(P+1)}\Gamma(P+1). 
\end{equation*} Then \begin{equation*}
    I_{J}(c) \leq \sum\limits_{\ell = 1}^{k}  c^{-(P+1)} \Gamma(P+1) = k \Gamma(P+1)  c^{-(P+1)}. 
\end{equation*} That is \begin{equation*}
    I_{J}(c) \leq C_{J} c^{-q_{J}}, q_{J} = P+1 = \sum\limits_{j=1}^{k} M_{j} + k, C_{J} = k\Gamma(P+1).  
\end{equation*} Recall $c(\boldsymbol{v}) = x\prod\limits_{\ell \in J^{c}} v_{\ell}^{-1}$. This gives \begin{equation*}
    I_{J}(c(\boldsymbol{v})) \leq C_{J} c(\boldsymbol{v})^{-q_{J}} = C_{J} x^{-q_{J}} \prod\limits_{\ell \in J^{c}}v_{\ell}^{q_{J}}. 
\end{equation*} So, 
\begin{align*}
    \int\limits_{R_{J}} e^{-xA(\boldsymbol{s})}Q(\boldsymbol{s}) d\boldsymbol{s} & \leq C_{J}x^{-q_{J}} \int\limits_{[1,\infty)^{d-2-k}} \exp \{ - x \sum\limits_{\ell \in J^{c}} v_{\ell}\} \prod\limits_{\ell \in J^{c}} v_{\ell}^{\beta_{\ell} +q_{J}}d\boldsymbol{v} \\ 
    & = C_{J}x^{-q_{J}} \prod\limits_{\ell \in J^{c}} \int\limits_{1}^{\infty} e^{-xv_{\ell}}v_{\ell}^{\beta_{\ell} + q_{J}} dv_{\ell}. 
\end{align*}
Note \begin{equation*}
    t^{N} \leq e^{\frac{x}{2}t}, N > 0 \iff \log(t^{N}) \leq \log(e^{\frac{x}{2}t}) \iff N\log(t) \leq \frac{x}{2}t \iff 0 \leq \frac{x}{2}t - N\log(t). 
\end{equation*}
Let \begin{equation*}
    h(t):= \frac{x}{2}t - N\log(t), t \geq 1 \implies h'(t) = \frac{x}{2} - \frac{N}{t}. 
\end{equation*} 
Let $T_{1}:= x^{-1} 2N$. Then for any $t \geq T_{1}$, \begin{equation*}
    h'(t) = \frac{x}{2} - \frac{N}{t} \geq \frac{x}{2} - \frac{N}{T_{1}} = \frac{x}{2} - \frac{N}{\frac{2N}{x}} = \frac{x}{2} - \frac{x}{2} = 0. 
\end{equation*} So for $t \geq T_{1}, h(t)$ is nondecreasing. Furthermore, \begin{equation*}
    \frac{h(t)}{t} = \frac{x}{2} - N\frac{\log(t)}{t} \rightarrow \frac{x}{2} >0. 
\end{equation*} Since $h(t) \rightarrow \infty, \exists T_{2}$ such that $h(T_{2}) \geq 0 $. Now set $T:= \max \{1, T_{1}, T_{2} \}$. Since $T \geq T_{1}, h $ is nondecreasing on $[T,\infty)$. Since $T \geq T_{2}, h(T) \geq 0$. Therefore, $\forall t \geq T, h(t) \geq h(T) \geq 0$. This means $\forall t \geq T,$ \begin{equation*}
    \frac{x}{2}t - N\log(t) \geq 0 \iff N\log(t) \leq \frac{x}{2}t \iff t^{N} \leq e^{\frac{x}{2}t}. 
\end{equation*}  Since $ t \geq 1$ and for $N > \beta_{\ell} + q_{J}$ we have $t^{\beta_{\ell} + q_{J}} \leq t^{N}$. Hence, \begin{equation*}
    e^{-x v_{\ell}} v_{\ell}^{\beta_{\ell} + q_{J}} \leq e^{-x v_{\ell}}v_{\ell}^{N} \leq e^{-xv_{\ell}}e^{(\frac{x}{2})v_{\ell}} = e^{-(\frac{x}{2}) v_{\ell}}.
\end{equation*} Then, \begin{align*}
    \int\limits_{1}^{\infty} e^{-x v_{\ell}}v_{\ell}^{\beta_{\ell} + q_{J}} dv_{\ell} & = \int\limits_{1}^{T} e^{-xv_{\ell}}v_{\ell}^{\beta_{\ell} + q_{J}} dv_{\ell} + \int\limits_{T}^{\infty}e^{-xv_{\ell}}v_{\ell}^{\beta_{\ell}+q_{J}} dv_{\ell} \\ 
    &\leq \int\limits_{1}^{T} e^{-xv_{\ell}}v_{\ell}^{\beta_{\ell} + q_{J}}dv_{\ell} + \int\limits_{T}^{\infty} e^{-(\frac{x}{2})v_{\ell}} dv_{\ell}\\ 
    &= \int\limits_{1}^{T} e^{-xv_{\ell}} v_{\ell}^{\beta_{\ell} + q_{J}} dv_{\ell} + \frac{2}{x} e^{-\frac{x}{2}T} \\ 
    & < \infty.
\end{align*} Thus, $\int\limits_{R_{J}} e^{-xA(\boldsymbol{s})} Q(\boldsymbol{s}) d\boldsymbol{s} < \infty$ for every nonempty $J$. Therefore each region is finite and there are only finitely many subsets of $J \subseteq \{1, \dots, d-2 \}$. Therefore, summing over finitely regions we have finite integrals. Thus $H(x) < \infty$ for $x>0$. Therefore \begin{equation*}
    M_{*} = (2x_{0})^{\frac{(d-2)}{2}}e^{2(d-1)x_{0}}H(x_{0}) < \infty. 
\end{equation*} Then $M(x) \leq M_{*} \forall x \in [x_{0}, 2x_{0}]$. Fix any $x \geq x_{0}$. Then choose $k \in \mathbb{N}$ such that $x2^{-k} \in [x_{0}, 2x_{0}]$. Particularly, let $r:= x_{0}^{-1}x \geq 1$. Then \begin{equation*}
    \frac{x}{2^{k}} \in [x_{0}, 2x_{0}] \iff 1 \leq \frac{x}{2^{k}x_{0}} \leq 2 \iff 2^{k} \leq r \leq 2^{k+1}. 
\end{equation*} Let \begin{equation*}
    K: = \lfloor \log_{2}(\frac{x}{x_{0}}) \rfloor. 
\end{equation*} Then, \begin{equation*}
    K \leq \log_{2}(\frac{x}{x_{0}}) < K+1. 
\end{equation*} So, \begin{equation*}
    2^{K} \leq \frac{x}{x_{0}} \leq 2^{K+1} \iff 2^{K}x_{0} \leq x \leq 2^{K+1}x_{0} \iff x_{0} \leq \frac{x}{2^{K}} \leq 2x_{0}. 
\end{equation*}
So we have $M(x) \leq K_{B} + \frac{1}{2} M(\frac{x}{2})$ for $x \geq x_{0}, M(x) \leq M_{*} \forall x \in [x_{0}, 2x_{0}]$, and \\$2^{-k}x \in [x_{0}, 2x_{0}]$. Therefore \begin{align*}
    M(x) & \leq K_{B} + \frac{1}{2}M(\frac{x}{2})\\ 
    & \leq K_{B} + \frac{1}{2}(K_{B} + \frac{1}{2} M(\frac{x}{4})) \\ 
    & = K_{B} + \frac{1}{2}K_{B} + \frac{1}{4}M(\frac{x}{4})\\ 
    & \leq K_{B} (1+ \frac{1}{2} + \dots \frac{1}{2^{K-1}}) + \frac{1}{2^{K}}M(\frac{x}{2^{K}}). 
\end{align*} We know that $\sum\limits_{j=0}^{K-1} 2^{-j} = 2-2^{-K+1} \leq 2. $ So, \begin{equation*}
    M(x) \leq 2K_{B} + \frac{1}{2^{K}}M(\frac{x}{2^{K}}) \leq 2K_{B} + \frac{1}{2^{K}} M_{*}. 
\end{equation*} Since $K \geq 0 \iff - K \leq 0 \implies 2^{-K} \leq 2^{0} = 1$. Therefore, \begin{equation*}
    M(x) \leq 2K_{B} + M_{*} \implies \sup\limits_{x\geq x_{0}} M(x) \leq 2K_{B} + M_{*} < \infty. 
\end{equation*} Recall $M(x) = x^{\frac{d-2}{2}}e^{(d-1)x}H(x)$. This implies \begin{align*}
    H(x) & = x^{-\frac{d-2}{2}} e^{-(d-1)x}M(x)\\ 
    & \leq (2K_{B} + M_{*}) x^{-\frac{(d-2)}{2}}e^{-(d-1)x}\\ 
    &= \tilde{b}_{2}x^{-\frac{d-2}{2}} e^{-(d-1)x}, \tilde{b}_{2} := 2K_{B} + M_{*}. 
\end{align*} Therefore we have shown that \begin{equation*}
    b_{1}x^{-\frac{d-2}{2}}e^{-(d-1)x} \leq H(x) \leq \tilde{b}_{2}x^{-\frac{d-2}{2}}e^{-(d-1)x}.
\end{equation*} That is, \begin{equation*}
    H(x) \asymp x^{-\frac{d-2}{2}}e^{-(d-1)x}. 
\end{equation*} Recall \begin{equation*}
    \pi(x|u,\boldsymbol{k}) \propto (1+Cx^{d-1})^{-\frac{p}{2}} \exp \{ - \frac{u}{2(1+Cx^{d-1})} \}x^{\alpha_{\cdot}-1}H(x). 
\end{equation*} For sufficiently large $x, H(x) \asymp x^{-\frac{d-2}{2}}e^{-(d-1)x}$ which implies \begin{align*}
    K_{1} x^{\alpha_{\cdot}-1-\frac{d-2}{2}}(1+Cx^{d-1})^{-\frac{p}{2}} \exp \{ -f_{u}(x) \} &\leq \pi (x|u, \boldsymbol{k})\\ &\leq K_{2}x^{\alpha_{\cdot}-1-\frac{d-2}{2}}(1+Cx^{d-1})^{-\frac{p}{2}} \exp \{ - f_{u}(x) \}
\end{align*} 
where
\begin{equation*}
    f_{u}(x) : = (d-1)x + \frac{u}{2(1+Cx^{d-1})}. 
\end{equation*} 
Let $m: = d-1, x_{u} = (\frac{u}{2C})^{\frac{1}{m+1}}$ and $A_{u} := Cx_{u}^{m}$. For $x = \lambda x_{u}, \lambda >0$, \begin{equation*}
    \frac{u}{2(1+Cx^{m})} = \frac{u}{2(1+C\lambda^{m}x_{u}^{m})} = \frac{u}{2(1+\lambda^{m}A_{u})} = \frac{u}{2A_{u}} \cdot \frac{1}{\lambda^{m} + \frac{1}{A_{u}}}. 
\end{equation*} We know $u(2A_{u})^{-1} = u(2Cx_{u}^{m})^{-1}.$ Note, \begin{equation*}
    x_{u} = (\frac{u}{2C})^{\frac{1}{m+1}} \iff x_{u}^{m+1} = \frac{u}{2C} \iff 2C x_{u}^{m+1} = u.
\end{equation*} This gives \begin{equation*}
    \frac{u}{2A_{u}} = \frac{2Cx_{u}^{m+1}}{2Cx_{u}^{m}} = x_{u}. 
\end{equation*} Therefore \begin{equation*}
    f_{u}(\lambda x_{u}) = m\lambda x_{u} + \frac{x_{u}}{\lambda^{m} + \frac{1}{A_{u}}} = x_{u}\phi_{u}(\lambda), \phi_{u}(\lambda): = m\lambda + \frac{1}{\lambda^{m} + \frac{1}{A_{u}}}. 
\end{equation*} Also, 
\begin{align*}
    A_{u} = Cx_{u}^{m} &= C(\frac{u}{2C})^{\frac{m}{m+1}}\\ 
    &= Cu^{\frac{m}{m+1}} (2C)^{- \frac{m}{m+1}} \\ 
    & = Cu^{\frac{m}{m+1}} 2^{-\frac{m}{m+1}} C^{-\frac{m}{m+1}}\\ 
    &= 2^{- \frac{m}{m+1}} C^{1- \frac{m}{m+1}} u^{\frac{m}{m+1}}\\ 
    & = 2^{-\frac{m}{m+1}} C^{\frac{1}{m+1}} u^{\frac{m}{m+1}} \rightarrow \infty  \text{ as } u \rightarrow \infty. 
\end{align*} Therefore $\phi_{u}(\lambda) \rightarrow \phi_{\infty}(\lambda) = m\lambda + \lambda^{-m}$. Observe that for $\lambda > 0$, $\lim\limits_{\lambda \rightarrow 0^{+}} \phi_{\infty}(\lambda) = \infty$ and $\lim\limits_{\lambda \rightarrow \infty} \phi_{\infty}(\lambda) = \infty$. Note that \begin{equation*}
    \phi_{\infty}'(\lambda) = m - m\lambda^{-m-1}. 
\end{equation*} Setting equal to zero gives \begin{equation*}
    0 = m - m\lambda^{-m-1} \iff m = m\lambda^{-m-1} \iff 1 = \lambda^{-m-1} \iff \lambda = 1. 
\end{equation*} Furthermore, \begin{equation*}
    \phi''_{\infty}(\lambda) = - m(-m-1)\lambda^{-m-2} = m(m+1)\lambda^{-m-2}>0, \forall \lambda >0. 
\end{equation*} Therefore, \begin{equation*}
    \Delta_{\infty}(\epsilon) : = \inf\limits_{|\lambda -1| \geq \epsilon} (\phi_{\infty}(\lambda) - \phi_{\infty}(1)) >0. 
\end{equation*} Now we will show that for fixed $\epsilon \in (0,1),\exists u_{0}$ and $c(\epsilon) >0$ such that $\forall u\geq u_{0},$
\begin{equation*}
    \inf\limits_{|\lambda -1| \geq \epsilon}(\phi_{u}(\lambda) - \phi_{u}(1)) \geq c(\epsilon) >0.
\end{equation*}
Set $K := \phi_{\infty} (1) + 2\Delta_{\infty}(\epsilon)$. We know that $\lim\limits_{\lambda \rightarrow \infty} \phi_{\infty}(\lambda) = \infty$. For every number $K, \exists R>0$ such that $\lambda \geq R \implies \phi_{\infty}(\lambda) \geq K$.  We also know that $\lim\limits_{\lambda\rightarrow 0^{+}} \phi_{\infty}(\lambda) = \infty$. Then, $\exists \delta> 0$ such that $0 < \lambda \leq \delta \implies \phi_{\infty} (\lambda) \geq K$. Let $M_{0}:= \delta^{-1}$ then $\lambda \leq M_{0}^{-1} \iff \lambda \leq \delta$. Thus, if $0 < \lambda \leq M_{0}^{-1}$ then $\phi_{\infty} (\lambda) \geq K$. Now let $M:= \max \{ R, M_{0}, 1 \}$. Then $M > 1$. If $\lambda \geq M$ then $\lambda \geq R$ so $\phi_{\infty}(\lambda) \geq K$. If $0< \lambda < M^{-1}$ then $\lambda \leq M^{-1} \leq M_{0}^{-1}$. Hence $\phi_{\infty}(\lambda) \geq K$. Therefore, when $\lambda \notin [M^{-1}, M]$ i.e. $\lambda \geq M$ or $\lambda \leq M^{-1}$,\\ $\phi_{\infty} \geq K = \phi_{\infty}(1) + 2\Delta_{\infty}(\epsilon)$.  This implies $\phi_{\infty}(\lambda) - \phi_{\infty}(1) \geq 2\Delta_{\infty} (\epsilon)$. On $[M^{-1}, M],$ \\ $\phi_{u}(\lambda)\rightarrow \phi_{\infty}(\lambda)$ uniformly because \begin{align*}
    |\frac{1}{\lambda^{m} + \frac{1}{A_{u}}} - \lambda^{-m}| &= |\frac{\lambda^{m}}{\lambda^{m}(\lambda^{m} + \frac{1}{A_{u}})} - \frac{\lambda^{m} + \frac{1}{A_{u}}}{\lambda^{m}(\lambda^{m} + \frac{1}{A_{u}})}|\\ 
    & = |\frac{-\frac{1}{A_{u}}}{\lambda^{m}(\lambda^{m} + \frac{1}{A_{u}})}|\\ 
    & = \frac{\frac{1}{A_{u}}}{\lambda^{m}(\lambda^{m} + \frac{1}{A_{u}})} \\ 
    &= \frac{1}{A_{u}\lambda^{m}(\lambda^{m} + \frac{1}{A_{u}})}\\ 
    &= \frac{1}{\lambda^{m}(A_{u}\lambda^{m} + 1)} \\ 
    &= \frac{1}{\lambda^{m}} \cdot \frac{1}{A_{u}\lambda^{m}+1}. 
\end{align*} If $\lambda \in [M^{-1}, M]$ then $\lambda \geq M^{-1}$. Therefore $\lambda^{m} \geq M^{-m} $ which implies $\lambda^{-m} \leq M^{m}$. Since $A_{u}\lambda^{m} + 1 \geq A_{u}\lambda^{m}, (A_{u}\lambda^{m} + 1)^{-1} \leq (A_{u}\lambda^{m})^{-1} \leq A_{u}^{-1} M^{m}$. Therefore, \begin{equation*}
    |\frac{1}{\lambda^{m} + \frac{1}{A_{u}}} - \frac{1}{\lambda^{m}}| \leq (M^{m}) (\frac{M^{m}}{A_{u}}) = \frac{M^{2m}}{A_{u}}.
\end{equation*} Thus, \begin{equation*}
    \sup\limits_{\lambda \in [\frac{1}{M}, M]} |\phi_{u}(\lambda) - \phi_{\infty}(\lambda)| \leq A_{u}^{-1} M^{2m}. 
\end{equation*} Then for sufficiently large $u, A_{u}$ is large enough. So \begin{equation*}
    \sup\limits_{\lambda \in [M^{-1}, M]}|\phi_{u}(\lambda) - \phi_{\infty}(\lambda)| \leq \frac{\Delta_{\infty}(\epsilon)}{4}. 
\end{equation*} If $\lambda \in[M^{-1}, M]$ and $|\lambda -1| \geq \epsilon$, then \begin{align*}
    \phi_{u}(\lambda) - \phi_{u}(1)  &= (\phi_{u}(\lambda) - \phi_{\infty}(\lambda)) + (\phi_{\infty}(\lambda) - \phi_{\infty}(1)) + (\phi_{\infty}(1) - \phi_{u}(1))\\ 
    &= (\phi_{\infty}(\lambda) - \phi_{\infty}(1)) + (\phi_{u}(\lambda) - \phi_{\infty}(\lambda)) - (\phi_{u}(1)- \phi_{\infty}(1))\\ 
    &\geq (\phi_{\infty}(\lambda) - \phi_{\infty}(1)) - |\phi_{u}(\lambda) - \phi_{\infty}(\lambda)|  - |\phi_{u}(1) - \phi_{\infty}(1)|\\ 
    &\geq \Delta_{\infty} (\epsilon) - |\phi_{u}(\lambda) - \phi_{\infty}(\lambda)| - |\phi_{u}(1) - \phi_{\infty}(1)|, \text{ by definition of } \Delta_{\infty}(\epsilon) \\ 
    & \geq \Delta_{\infty}(\epsilon) - \frac{\Delta_{\infty}(\epsilon)}{4} - \frac{\Delta_{\infty}(\epsilon)}{4} \\ 
    &= \frac{\Delta_{\infty}(\epsilon)}{2}\\ 
    &>0.
\end{align*} That is $\phi_{u} (\lambda)  - \phi_{u}(1) > 0 $ for $\lambda \in [M^{-1}, M]$ and $|\lambda - 1| \geq \epsilon. $ Note that \begin{align*}
    \phi_{u}(1) &= m + \frac{1}{1+\frac{1}{A_{u}}}\\ 
    & \leq m + 1,\forall u, \text{ since } 1+ \frac{1}{A_{u}} \geq 1, \frac{1}{1+\frac{1}{A_{u}}} \leq 1. 
\end{align*} Now, lower bound $\phi_{u}(\lambda)$ when $\lambda \notin [M^{-1}, M]$. For $\lambda \geq M, $\begin{equation*}
    \phi_{u}(\lambda) = m\lambda + \frac{1}{\lambda^{m} + \frac{1}{A_{u}}} \geq m\lambda. 
\end{equation*} That is, $\phi_{u}(\lambda) \geq mM$ if $\lambda \geq M$. Then, $\phi_{u}(\lambda) - \phi_{u}(1) \geq mM - (m+1)$. Then by choosing $M$ large enough, $mM - (m+1) \geq c$ for $c>0, \forall \lambda \geq M$. Now, consider $0 < \lambda < M^{-1}$. If $0 < \lambda < M^{-1}$, then $\lambda^{m} \leq M^{-m}$. Hence $\lambda^{m} + A_{u}^{-1} \leq M^{-m} + A_{u}^{-1}$. This implies \begin{equation*}
    \frac{1}{\lambda^{m} + \frac{1}{A_{u}}} \geq \frac{1}{\frac{1}{M^{m}} + \frac{1}{A_{u}}}. 
\end{equation*} Now choose $u$ large enough so $A_{u} \geq M^{m}$. Then $A_{u}^{-1} \leq M^{-m}$. So \begin{equation*}
    \frac{1}{\frac{1}{M^{m}} + \frac{1}{A_{u}}} \geq \frac{1}{\frac{1}{M^{m}}+ \frac{1}{M^{m}}} = \frac{M^{m}}{2}. 
\end{equation*} Therefore for large $u$ such that $A_{u} \geq M^{m}$ and $\lambda \leq M^{-1}$ we have $\phi_{u}(\lambda) \geq 2^{-1}M^{m}$. This implies $\phi_{u}(\lambda) - \phi_{u}(1) \geq 2^{-1}M^{m} - (m+1)$. Then choosing M large enough \\$\phi_{u}(\lambda) - \phi_{u}(1) \geq C_{tail} := \min \{mM-(m+1), 2^{-1}M^{m}-(m+1) \}>0$. So we have that \begin{align*}
    \Delta_{\infty}(\epsilon) &:= \inf\limits_{|\lambda -1| \geq \epsilon} (\phi_{\infty}(\lambda) - \phi_{\infty}(1))>0 \\ 
    \phi_{\infty}(\lambda) - \phi_{\infty}(1) & \geq 2\Delta_{\infty}(\epsilon), \text{ for } \lambda \notin [\frac{1}{M}, M], \\ 
    \sup\limits_{\lambda \in [\frac{1}{M}, M]} |\phi_{u}(\lambda) - \phi_{\infty}(\lambda)| &\leq \frac{\Delta_{\infty}(\epsilon)}{4}, \text{ which gives}\\ 
    |\phi_{u}(1) - \phi_{\infty}(1)| &\leq \frac{\Delta_{\infty}(\epsilon)}{4}, \text{ and } \\ 
    \phi_{u}(\lambda) - \phi_{u}(1)  &\geq  (\phi_{\infty}(\lambda) - \phi_{\infty}(1)) - |\phi_{u} (\lambda) - \phi_{\infty}(\lambda)| - |\phi_{u}(1) - \phi_{\infty}(1)|. 
\end{align*} This gives \begin{equation*}
    \phi_{u}(\lambda) - \phi_{u}(1) \geq \frac{\Delta_{\infty}(\epsilon)}{2}>0. 
\end{equation*} Now choose $u_{0}$ such that both large $u$ inequalities hold. Consider $\lambda \in [M^{-1}, M]$ and $|\lambda - 1| \geq \epsilon$. By definition of $\Delta_{\infty}(\epsilon), \phi_{\infty}(\lambda) - \phi_{\infty}(1)   \geq \Delta_{\infty}(\epsilon).$ Then \\$|\phi_{u}(\lambda) - \phi_{\infty}(\lambda)| \leq 4^{-1} \Delta_{\infty}(\epsilon)$ and $|\phi_{u}(1) - \phi_{\infty}(1)| \leq 4^{-1} \Delta_{\infty}(\epsilon)$. Then $\phi_{u}(\lambda) - \phi_{u}(1) > 0$. For $\lambda \notin [M^{-1}, M]$ we know $\phi_{u}(\lambda) - \phi_{u}(1) \geq C_{tail}$. Therefore $\forall u \geq u_{0}$ and $\forall \lambda$ with $|\lambda - 1| \geq \epsilon,  \phi_{u}(\lambda) - \phi_{u}(1) \geq C(\epsilon) := \min\{\frac{\Delta_{\infty}(\epsilon)}{2}, C_{tail} \}>0$. Now take any $x>0$ with $x \notin [(1-\epsilon)x_{u}, (1+\epsilon)x_{u}].$ Let $\lambda:= x_{u}^{-1}x>0$. Then $x = \lambda x_{u}$. This implies $x \notin [(1-\epsilon)x_{u}, (1+\epsilon)x_{u}] \iff \lambda \notin [1-\epsilon, 1+\epsilon] \iff |\lambda -1| \geq \epsilon.$ Therefore \begin{align*}
    & \phi_{u}(\lambda) \geq \phi_{u}(1) + c(\epsilon) \\ 
&\iff x_{u}\phi_{u}(\lambda) \geq x_{u}\phi_{u}(1) + x_{u}c(\epsilon)\\ 
&\iff f_{u}(x) = f_{u}(\lambda x_{u}) \geq f_{u}(x_{u}) + c(\epsilon)x_{u}. 
\end{align*} That is $\forall \epsilon \in (0,1), \exists u_{0}, c(\epsilon) >0: \forall u \geq u_{0}, \forall x \notin [(1-\epsilon)x_{u}, (1+\epsilon)x_{u}], $\\$ f_{u}(x) \geq f_{u}(x_{u}) + c(\epsilon)x_{u}$. Fix $\epsilon \in (0,2^{-1})$. Recall \begin{equation*}
    x_{u} = (\frac{u}{2C})^{\frac{1}{m+1}}, A_{u} = Cx_{u}^{m} \text{ such that as }  x_{u} \rightarrow \infty, A_{u} \rightarrow \infty. 
\end{equation*} Then $x_{u}A_{u} = Cx_{u}^{m+1} = 2^{-1}u$. Let $\Pi_{u, \boldsymbol{k}}(\cdot) = \mathbb{P}(X \in \cdot|U=u, \boldsymbol{K} = \boldsymbol{k}).$ We will prove that for large u \begin{equation*}
    \Pi_{u, \boldsymbol{k}}(|\frac{X}{x_{u}}-1| \geq \epsilon) \leq \exp \{ - c(\epsilon) x_{u} \}. 
\end{equation*} Since $x_{0}$ is fixed and $x_{u} \rightarrow \infty$, for all large $u$ we know $(1-\epsilon) x_{u} \geq x_{0}$. Hence \begin{align*}
    \{ |\frac{X}{x_{u}}-1| \geq \epsilon \} &= \{ \{ |\frac{X}{x_{u}}-1| \geq \epsilon \} \cap \{ X \leq x_{0} \}\} \cup \{ \{ |\frac{X}{x_{u}}-1| \geq \epsilon \} \cap \{ X \geq x_{0} \} \} \\
    &\subseteq  \{ X \leq x_{0} \} \cup \{\{ |\frac{X}{x_{u}}-1| \geq \epsilon \} \cap \{ X \geq x_{0} \}\}. 
\end{align*} We will show that $\Pi_{u, \boldsymbol{k}}((0, x_{0}))$ is negligible.  Let \begin{equation*}
    g(x) := \frac{1}{2(1+Cx^{m})} \text{ on } x\in (0, x_{0}), g(x) \geq g(x_{0}) =: c_{0} > 0. 
\end{equation*} Hence \begin{equation*}
    \exp \{ - \frac{u}{2(1+Cx^{m})} \} = \exp \{ - ug(x) \} \leq \exp \{ - c_{0} u \} \text{ for } x\in (0, x_{0}). 
\end{equation*} Therefore up to a normalizing constant \begin{align*}
    \int\limits_{0}^{x_{0}} \pi_{u, \boldsymbol{k}}(x) &\leq e^{-c_{0}u} \int\limits_{0}^{x_{0}}x^{\alpha_{\cdot}-1}H(x) dx \\ 
    &= e^{-c_{0}u} \int\limits_{(0,\infty)^{m-1}} Q(\boldsymbol{s})[\int\limits_{0}^{x_{0}} x^{\alpha_{\cdot}-1}e^{-xA(\boldsymbol{s})}dx]d\boldsymbol{s}, \text{ by Tonelli's theorem,}\\  
    &\leq e^{-c_{0}u} \int\limits_{(0,\infty)^{m-1}}Q(\boldsymbol{s}) [\int\limits_{0}^{\infty}(\frac{t}{A(\boldsymbol{s})})^{\alpha_{\cdot}-1}e^{-t}\frac{dt}{A(\boldsymbol{s})}]d\boldsymbol{s}, t = xA(\boldsymbol{s}), \\ 
    &= e^{-c_{0}u}\int\limits_{(0,\infty)^{m-1}}Q(\boldsymbol{s}) [A(\boldsymbol{s})^{-\alpha_{\cdot}}\int\limits_{0}^{\infty} t^{\alpha_{\cdot}-1}e^{-t}dt]d\boldsymbol{s}\\ 
    &= e^{-c_{0}u}\int\limits_{(0,\infty)^{m-1}}Q(\boldsymbol{s}) [A(\boldsymbol{s})^{-\alpha_{\cdot}}\Gamma(\alpha_{\cdot})]d\boldsymbol{s}\\ 
    &= \Gamma(\alpha_{\cdot})e^{-c_{0}u}\int\limits_{(0,\infty)^{m-1}}Q(\boldsymbol{s}) A(\boldsymbol{s}) ^{-\alpha_{\cdot}}d\boldsymbol{s}. 
\end{align*}
Recall \begin{equation*}
    f_{T_{1},\dots, T_{m}|\boldsymbol{K}}(t_{1},\dots, t_{m}|\boldsymbol{k}) = \prod\limits_{i=1}^{m} \frac{1}{\Gamma(\alpha_{i})}t_{i}^{\alpha_{i}-1}e^{-t_{i}}. 
\end{equation*}
Then under the reparameterization the joint density becomes \begin{equation*}
    f_{X,\boldsymbol{S}|\boldsymbol{K}}(x,\boldsymbol{s}|\boldsymbol{k}) = \frac{mx^{\alpha_{\cdot}-1}e^{-xA(\boldsymbol{s})}Q(\boldsymbol{s})}{\prod\limits_{i=1}^{m}\Gamma(\alpha_{i})}. 
\end{equation*}  Then the marginal density of $\boldsymbol{S}$ is given by 
\begin{align*}
    f_{\boldsymbol{S}}(\boldsymbol{s}) &= \frac{m}{\prod\limits_{i=1}^{m}\Gamma(\alpha _{i})}Q(\boldsymbol{s}) \int\limits_{0}^{\infty} x^{\alpha_{\cdot}-1} e^{-xA(\boldsymbol{s})}dx \\ 
    &=\frac{m}{\prod\limits_{i=1}^{m}\Gamma(\alpha_{i})} Q(\boldsymbol{s}) \int\limits_{0}^{\infty} (\frac{u}{A(\boldsymbol{s})})^{\alpha_{\cdot}-1}e^{-u}\frac{du}{A(\boldsymbol{s})}, u = xA(\boldsymbol{s}), \\ 
    & = \frac{m}{\prod\limits_{i=1}^{m}\Gamma(\alpha_{i})}Q(\boldsymbol{s}) A(\boldsymbol{s})^{-\alpha_{\cdot}} \int\limits_{0}^{\infty} u^{\alpha_{\cdot}-1}e^{-u}du\\  
    & = \frac{m\Gamma(\alpha_{\cdot})}{\prod\limits_{i=1}^{m}\Gamma(\alpha_{i})}Q(\boldsymbol{s}) A(\boldsymbol{s}) ^{-\alpha_{\cdot}}, \boldsymbol{s} \in (0,\infty)^{m-1 }.
\end{align*}
We know the density integrates to 1. So \begin{equation*}
    1 = \int\limits_{(0,\infty)^{m-1}} f_{\boldsymbol{S}}(\boldsymbol{s}) d\boldsymbol{s}  = \frac{m\Gamma(\alpha_{\cdot})}{\prod\limits_{i=1}^{m}\Gamma(\alpha_{i})} \int\limits_{(0,\infty)^{m-1}} Q(\boldsymbol{s}) A(\boldsymbol{s}) ^{-\alpha_{\cdot}}d\boldsymbol{s}. 
\end{equation*} Therefore \begin{equation*}
    \frac{\prod\limits_{i=1}^{m}\Gamma(\alpha_{i})}{m\Gamma(\alpha_{\cdot})} = \int\limits_{(0,\infty)^{m-1}}Q(\boldsymbol{s}) A(\boldsymbol{s}) ^{-\alpha_{\cdot}}d\boldsymbol{s}. 
\end{equation*} 
Hence \begin{align*}
    \int\limits_{0}^{x_{0}}\pi_{u,\boldsymbol{k}}(x) dx &\leq e^{-c_{0}u}\Gamma(\alpha_{\cdot}) \frac{\prod\limits_{\ell = 1}^{m}\Gamma(\alpha_{\ell})}{m\Gamma(\alpha_{\cdot})} \\ 
    & = \frac{\prod\limits_{\ell = 1}^{m}\Gamma(\alpha_{\ell})}{m}e^{-c_{0}u}\\ 
    &= K_{0}e^{-c_{0}u}, K_{0}:= \frac{\prod\limits_{\ell=1}^{m}\Gamma(\alpha_{\ell})}{m}. 
\end{align*}
Note that \begin{align*}
    \gamma > -1 &\iff \alpha_{\cdot} - 1- \frac{m-1}{2} > -1 \iff \alpha_{\cdot} - \frac{m-1}{2} > 0 \\ 
    &\iff \frac{1}{2}\sum\limits_{\ell = 1}^{m}k_{\ell} - \frac{m-1}{2} > 0 \iff \frac{1}{2} \sum\limits_{\ell = 1}^{m} k_{\ell} > \frac{m-1}{2} \iff \sum\limits_{\ell = 1}^{m} k_{\ell} > m-1. 
\end{align*} Since each $k_{\ell}$ is an integer $\sum\limits_{\ell = 1}^{m}k_{\ell} \geq m \implies \sum\limits_{\ell = 1}^{m} k_{\ell} > m-1. $ Now fix $\delta \in (0,1)$ and consider $I_{\delta} := [x_{0} +1- \delta, x_{0}+1] \subset [x_{0}, \infty). $ On $I_{\delta}$ \begin{equation*}
     (1+Cx^{m})^{-\frac{p}{2}} \geq (1+C(x_{0} + 1)^{m}) ^{-\frac{p}{2}} =: c_{*}, \exp \{ - mx \} \geq \exp \{ -m(x_{0}+1) \}. 
\end{equation*} Furthermore $H(x) \geq h_{*} = b_{1} \min\limits_{x \in I_{\delta}} (x^{-\frac{(m-1)}{2}}e^{-mx})$. Recall \begin{equation*}
    \pi_{u,\boldsymbol{k}}(x) \geq K_{1} x^{\gamma} (1+Cx^{m})^{-\frac{p}{2}} e^{-mx} \exp \{ - \frac{u}{2(1+Cx^{m})} \}. 
\end{equation*} Let $A(x) := K_{1} x^{\gamma} (1+Cx^{m})^{-\frac{p}{2}} e^{-mx}$. Observe $A(x)$ is a continuous, strictly positive function for $x>0$. Note that $I_{\delta}$ is compact and bounded away from $0$. Then by the extreme value theorem $A_{\delta} := \inf\limits_{x \in I_{\delta}} A(x) > 0$. Then for every $x \in I_{\delta}$ we have $K_{1} x^{\gamma}(1+Cx^{m}) ^{-\frac{p}{2}} e^{-mx} \geq A_{\delta}$. Therefore \begin{equation*}
    \pi_{u, \boldsymbol{k}}(x) \geq A_{\delta} \exp \{ - \frac{u}{2(1+Cx^{m})} \}.
\end{equation*} On $I_{\delta}, g(x)$ is decreasing in $x$ so for $x \in I_{\delta}$ \begin{equation*}
    \exp \{ - \frac{u}{2(1+Cx^{m})} \} = \exp \{ - ug(x) \} \geq \exp \{ -ug(x_{0} + 1 - \delta) \} =: e^{-c_{1}u}, 
\end{equation*}
where $c_{1} := g(x_{0} + 1- \delta).$
Then \begin{equation*}
    \int\limits_{0}^{\infty}\pi_{u, \boldsymbol{k}} (x) dx \geq \int\limits_{I_{\delta}} \pi_{u, \boldsymbol{k}}(x) dx \geq \int\limits_{I_{\delta}} A_{\delta}e^{-c_{1}u}dx = A_{\delta}\cdot\delta e^{-c_{1}u} =: K_{1}e^{-c_{1}u}. 
\end{equation*} Then \begin{equation*}
    \Pi_{u, \boldsymbol{k}}((0, x_{0})) = \frac{\int\limits_{0}^{x_{0}}\pi_{u}(x)dx}{ \int\limits_{0}^{\infty} \pi_{u}(x)dx} \leq \frac{K_{0}e^{-c_{0}u}}{K_{1}e^{-c_{1}u}} = \frac{K_{0}}{K_{1}}e^{-(c_{0}-c_{1})u}. 
\end{equation*} That is for large $u$ and $x_{0}$ $\Pi_{u, \boldsymbol{k}}((0,x_{0})) \leq e^{-cu}$ for some $c:= c_{0}-c_{1}>0$. For \\$x = \lambda x_{u} \geq x_{0}$. Then $dx = x_{u}d\lambda$. Recall \begin{equation*}
    \phi_{u}(\lambda) = m\lambda + \frac{1}{\lambda^{m} + \frac{1}{A_{u}}}. 
\end{equation*} Then \begin{align*}
    -x_{u} \phi_{u}(\lambda) & = -x_{u}(m\lambda + \frac{1}{\lambda^{m} + \frac{1}{A_{u}}}) \\ 
    &= -(m(\lambda x_{u}) + \frac{x_{u}A_{u}}{\lambda^{m} A_{u} + 1}) \\ 
    & = - (m(\lambda x_{u}) + \frac{u}{2(1+\lambda^{m}A_{u})}) \\ 
    & = -(m(\lambda x_{u}) + \frac{u}{2(1+C(\lambda x_{u})^{m})}) \\ 
    & = -f_{u}(\lambda x_{u}). 
\end{align*} Therefore $\exp \{ - f_{u}(\lambda x_{u}) \} = \exp \{ - x_{u}\phi_{u}(\lambda) \}. $ Then \begin{align*}
    \tilde{\pi}_{u, \boldsymbol{k}}(\lambda) := x_{u}\pi_{u, \boldsymbol{k}}(\lambda x_{u}) &\leq K_{2} x_{u} (\lambda x_{u})^{\gamma} (1+A_{u}\lambda ^{m})^{-\frac{p}{2}} \exp \{ -x_{u}\phi_{u}(\lambda) \}\\ 
    & = K_{2}x_{u}^{\gamma + 1} \lambda ^{\gamma} (1+A_{u}\lambda^{m})^{-\frac{p}{2}} \exp \{  - x_{u}\phi_{u}(\lambda) \} \\ 
    &= K_{2} x_{u}^{\gamma+1}q_{u}(\lambda), q_{u}(\lambda):= \lambda^{\gamma}(1+A_{u}\lambda^{m})^{-\frac{p}{2}} \exp \{ -x_{u}\phi_{u}(\lambda) \}. 
\end{align*} Similarly, $\tilde{\pi}_{u, \boldsymbol{k}} (\lambda) \geq K_{1}x_{u}^{\gamma + 1}q_{u}(\lambda)$. That is \begin{equation*}
    K_{1}x_{u}^{\gamma + 1}q_{u}(\lambda) \leq \tilde{\pi} _{u, \boldsymbol{k}}(\lambda) \leq K_{2}x_{u}^{\gamma + 1} q_{u}(\lambda). 
\end{equation*} Consider measurable $S \subseteq [x_{0}x_{u}^{-1}, \infty ). $ Then \begin{align*}
    \Pi_{u, \boldsymbol{k}}(S|X\geq x_{0}) & = \mathbb{P}(\lambda \in S|X \geq x_{0}, U= u) \\ 
    &= \frac{\int\limits_{S} \tilde{\pi}_{u, \boldsymbol{k}}(\lambda)d\lambda}{\int\limits_{\frac{x_{0}}{x_{u}}}^{\infty} \tilde{\pi}_{u, \boldsymbol{k}}(\lambda) d\lambda}\\ 
    &\leq \frac{K_{2}x_{u}^{\gamma + 1} \int\limits_{S}q_{u}(\lambda)d\lambda}{ K_{1}x_{u}^{\gamma+1} \int\limits_{\frac{x_{0}}{x_{u}}}^{\infty} q_{u}(\lambda) d\lambda} \\ 
    & = \frac{K_{2}}{K_{1}}\frac{\int\limits_{S} q_{u}(\lambda) d\lambda}{\int\limits_{\frac{x_{0}}{x_{u}}} ^{\infty} q_{u}(\lambda) d\lambda }. 
\end{align*} Similarly, \begin{equation*}
    \Pi_{u|\boldsymbol{k}}(S|X\geq x_{0}) \geq \frac{K_{1}}{K_{2}} \frac{\int\limits_{S} q_{u}(\lambda) d\lambda}{ \int\limits_{\frac{x_{0}}{x_{u}}}^{\infty}q_{u}(\lambda)d\lambda} . 
\end{equation*} That is \begin{equation*}
    \frac{K_{1}}{K_{2}} \frac{\int\limits_{S} q_{u}(\lambda
    ) d\lambda}{ \int\limits_{\frac{x_{0}}{x_{u}}}^{\infty} q_{u}(\lambda) d\lambda} \leq \Pi_{u, \boldsymbol{k}}(S|X \geq x_{0}) \leq \frac{K_{2}}{K_{1}} \frac{\int\limits_{S}q_{u}(\lambda)d\lambda}{\int\limits_{\frac{x_{0}}{x_{u}}}^{\infty} q_{u}(\lambda) d\lambda}. 
\end{equation*} Now set $S = S_{\xi} := \{ |\lambda - 1| \geq \xi \} \cap [ x_{u}^{-1} x_{0}, \infty)$. Recall $\phi_{u}(\lambda) = m\lambda + (\lambda^{m} + A_{u}^{-1})^{-1}.$ Then \begin{align*}
    \phi_{u}^{'}(\lambda) &= m - (\lambda^{m} + A_{u}^{-1})^{-2}(m\lambda^{m-1}) \\ 
    & = m - \frac{m \lambda^{m-1}}{(\lambda ^{m} + A_{u}^{-1})^{2}}\\ 
    & = m - (m\lambda^{m-1})(\lambda^{m} + A_{u}^{-1})^{-2}. 
\end{align*} This also implies \begin{align*}
    \phi_{u}^{''}(\lambda) &= [-m\lambda^{m-1}][-2(\lambda^{m} + A_{u}^{-1})^{-3}](m\lambda^{m-1}) + (\lambda^{m} + A_{u}^{-1})^{-2} [-m(m-1)\lambda^{m-2}]\\ 
    & = 2m\lambda^{m-1}(\lambda^{m} + A_{u}^{-1})^{-3}m\lambda^{m-1} - (\lambda^{m} + A_{u}^{-1})^{-2} m(m-1)\lambda^{m-2}\\ 
    &= 2m^{2}\lambda^{2m - 2} (\lambda^{m} + A_{u}^{-1})^{-3} - m(m-1) \lambda ^{m-2}(\lambda^{m} + A_{u}^{-1})^{-2} \\ 
    &= \lambda^{m-2}[2m^{2} \lambda ^{m} (\lambda^{m} + A_{u}^{-1})^{-3} - m(m-1)(\lambda^{m} + A_{u}^{-1})^{-2}]\\ 
    &= m\lambda^{m-2} [2m\lambda ^{m} (\lambda^{m} + A_{u}^{-1})^{-3} - (m-1)(\lambda^{m} + A_{u}^{-1})^{-2}]\\ 
    & = \frac{m\lambda^{m-2}}{(\lambda^{m} + A_{u}^{-1})^{3}} (2m\lambda^{m} - (m-1)(\lambda^{m} + A_{u}^{-1})) \\ 
    & = \frac{m \lambda ^{m-2}}{(\lambda ^{m} + A_{u}^{-1})^{3}} (2m \lambda^{m} - m \lambda^{m} -mA_{u}^{-1} +\lambda^{m} + A_{u}^{-1})\\ 
    & = \frac{m \lambda ^{m-2}}{(\lambda^{m} + A_{u}^{-1})^{3}}((m+1)\lambda^{m} - (m-1) A_{u}^{-1})\\ 
    & = \frac{m \lambda^{m-2}}{(\lambda^{m} + A_{u}^{-1})^{3}}((m+1)\lambda^{m} - (m-1)A_{u}^{-1})\\ 
    &= \frac{m\lambda^{m-2}}{(\lambda^{m} + \frac{1}{A_{u}})^{3}} ((m+1)\lambda^{m} - \frac{m-1}{A_{u}}). 
\end{align*} For $\lambda \geq 2^{-1}$ we have $\lambda^{m} \geq 2^{-m}$. Since $A_{u} \rightarrow \infty$, for all large $u$ \begin{equation*}
    (m+1)\lambda^{m} - \frac{m-1}{A_{u}} \geq \frac{m+1}{2^{m}} - \frac{m-1}{A_{u}} > \frac{m+1}{2^{m+1}} >0. 
\end{equation*} Therefore $\phi_{u}^{''}(\lambda) >0$ on $[2^{-1}, \infty)$. So $\phi_{u}$ is strictly convex and has at most one critical point on $[2^{-1}, \infty).$ At $\lambda = 1, $ \begin{equation*}
    \phi_{u}^{'}(1) = m - \frac{m}{(1+\frac{1}{A_{u}})^{2}} >0 \text{ since } (1+\frac{1}{A_{u}})^{2}>1. 
\end{equation*} For $\lambda = 2^{-1}$ and sufficiently large $u$, $A_{u} \geq 2^{(m+1)}$ iff $A_{u}^{-1} \leq 2^{-(m+1)}$. Therefore, \begin{equation*}
    (\frac{1}{2})^{m} + \frac{1}{A_{u}} \leq (\frac{1}{2})^{m} + (\frac{1}{2})^{m+1} = (\frac{1}{2})^{m} + \frac{1}{2}(\frac{1}{2})^{m} = \frac{3}{2}(\frac{1}{2})^{m}. 
\end{equation*} Thus, \begin{equation*}
    \frac{(\frac{1}{2})^{m-1}}{((\frac{1}{2})^{m} + \frac{1}{A_{u}})^{2}} \geq \frac{(\frac{1}{2})^{m-1}}{[\frac{3}{2}(\frac{1}{2})^{m}]^{2}} = (\frac{1}{2})^{m-1} \frac{4}{9}(\frac{1}{2})^{-2m} = \frac{4}{9}(\frac{1}{2})^{-m-1} = \frac{4}{9} 2^{m+1}. 
\end{equation*} Hence for large $u$, \begin{equation*}
    \phi_{u}^{'}(\frac{1}{2}) = m(1-\frac{(\frac{1}{2})^{m-1}}{((\frac{1}{2})^{m} + \frac{1}{A_{u}})^{2}}) \leq m (1-\frac{4}{9}2^{m+1}) <0, \text{ since }
\end{equation*} \begin{align*}
    1- \frac{4}{9}2^{m+1} <0 &\iff 1 < \frac{4}{9}2^{m+1}\\ &\iff \frac{9}{4} < 2^{m+1}\\ &\iff \log(\frac{9}{4}) < (m+1)\log(2)\\
    &\iff \frac{\log(\frac{9}{4})}{\log(2)} - 1 < m,
\end{align*} which is true since $m \geq 1$. So for sufficiently large $u$ $\phi_{u}^{'}(2^{-1}) <0$ and $\phi_{u}^{'}(1) >0$. By continuity $\phi_{u}^{'}$ has a zero in $(2^{-1}, 1)$. By strict convexity on $[2^{-1}, \infty)$ that zero is unique and the unique minimizer $\lambda_{u}^{*} \in (2^{-1}, 1).$ From above we know as $A_{u} \rightarrow \infty,$\\$ \phi_{u}(\lambda) \rightarrow \phi_{\infty}(\lambda)$ pointwise on $(0,\infty). $ Furthermore the convergence is uniform on $[2^{-1}, 2]$ because \begin{equation*}
    |\frac{1}{\lambda^{m} + \frac{1}{A_{u}}} - \lambda^{-m}| = \frac{\frac{1}{A_{u}}}{\lambda^{m}(\lambda^{m}+ \frac{1}{A_{u}})} \leq \frac{\frac{1}{A_{u}}}{(\frac{1}{2})^{2m}} = \frac{2^{2m}}{A_{u}} \rightarrow 0 . 
\end{equation*} We know $\phi_{\infty}(\lambda)$ has a unique minimizer at $\lambda = 1$. Therefore we know that \\$|\lambda_{u}^{*}-1|<\epsilon, \epsilon > 0, $ with $\lambda_{u}^{*} \rightarrow 1.$ Therefore for fixed $\epsilon \in (0,2^{-1})$ for all sufficiently large u, \begin{equation*}
    |\lambda -1| = |\lambda - \lambda_{u}^{*} + \lambda_{u}^{*} -1| \leq |\lambda - \lambda_{u}^{*}| + |\lambda_{u}^{*} -1|. 
\end{equation*} Then $|\lambda - 1| \geq \epsilon, |\lambda -\lambda_{u}^{*}| \geq \epsilon - |\lambda_{u}^{*}-1|.$ Then for sufficiently large $u$ we have \\$|\lambda_{u}^{*}-1| \leq 2^{-1} \epsilon.$ This gives $|\lambda - \lambda_{u}^{*}| \geq \epsilon - 2^{-1}\epsilon = 2^{-1} \epsilon.$ That is for sufficiently large $u$, $|\lambda -1| \geq \epsilon$ which implies $|\lambda - \lambda_{u}^{*}| \geq 2^{-1}\epsilon.$ Therefore we can focus on bounding the posterior mass outside $[\lambda_{u}^{*} - 2^{-1}\epsilon, \lambda_{u}^{*} + 2^{-1}\epsilon].$ Let $\eta_{u} := x_{u}^{-\frac{1}{2}}$ and $B_{u}:= [\lambda_{u}^{*}- \eta_{u}, \lambda_{u}^{*} + \eta_{u}]. $ For all sufficiently large $u$, $\lambda_{u}^{*} \in (2^{-1}, 1)$ and $\eta_{u} \rightarrow 0$. Hence $B_{u} \subset [2^{-1}, 2]$ and \\$B_{u} \subset [x_{u}^{-1}x_{0}, \infty). $ We will show that $\exists c_{in}>0$ such that for all large $u$ \begin{equation*}
    \int\limits_{\frac{x_{0}}{x_{u}}}^{\infty} q_{u}(\lambda)d\lambda \geq \int\limits_{B_{u}} q_{u}(\lambda) d\lambda \geq c_{in} \eta_{u} (1+A_{u})^{-\frac{p}{2}}\exp\{ - x_{u}\phi_{u}(\lambda_{u}^{*}) \}. 
\end{equation*}  Since $B_{u} \subset [2^{-1}, 2], \lambda^{\gamma} \geq \min\limits_{\lambda \in [2^{-1},2]} \lambda^{\gamma} = 2^{-|\gamma|} =: c_{\gamma} >0, \forall \lambda \in B_{u}$. For $\lambda \in [2^{-1}, 2],$\\$ 1+A_{u}\lambda^{m} \leq 1+A_{u}2^{m} \leq 2^{m} + A_{u}2^{m}  = 2^{m}(1+A_{u})$ for $A_{u} \geq 1$ which is true for sufficiently large $u$. Hence $(1+A_{u}\lambda^{m})^{-\frac{p}{2}} \geq 2^{-\frac{mp}{2}}(1+A_{u})^{-\frac{p}{2}} =: c_{p}(1+A_{u})^{-\frac{p}{2}}$. On $[2^{-1}, 2]$ for $A_{u} \geq 1$, \begin{align*}
    \lambda^{m-2} &\leq 2^{|m-2|},\\ (\lambda^{m} + \frac{1}{A_{u}}) ^{-3} &\leq \lambda ^{-3m} \leq 2^{3m},\\ &\text{ and }\\ (m+1)\lambda^{m} - \frac{m-1}{A_{u}} &\leq (m+1) \lambda^{m} \leq (m+1)2^{m}. 
\end{align*} Therefore, for all large $u$, \begin{align*}
    \sup\limits_{\lambda \in [2^{-1},2]} \phi_{u}^{''}(\lambda) \leq M &:= m2^{|m-2|}2^{3m}(m+1)2^{m}\\
    & = m(m+1)2^{|m-2|+3m + m}\\ 
    & = m(m+1)2^{|m-2| + 4m}< \infty, \text{ which is independent of } u. 
\end{align*} By Taylor's theorem,  \begin{align*}
    \phi_{u}(\lambda) &\leq \phi_{u}(\lambda_{u}^{*}) + \phi_{u}^{'}(\lambda_{u}^{*}) (\lambda - \lambda_{u}^{*}) + \frac{1}{2}\phi_{u}^{''}(\xi) (\lambda - \lambda_{u}^{*})^{2} \\ 
    &\leq \phi_{u}(\lambda_{u}^{*}) + \frac{1}{2}\sup\limits_{\xi \in[\frac{1}{2}, 2]} \phi_{u}^{''}(\xi) (\lambda - \lambda_{u}^{*})^{2}\\ 
    &\leq \phi_{u}(\lambda_{u}^{*}) + \frac{M}{2}(\lambda- \lambda_{u}^{*})^{2}. 
\end{align*} For $\lambda \in B_{u}, \lambda_{u}^{*} - \eta_{u} \leq \lambda \leq \lambda_{u}^{*} + \eta_{u}$ iff $|\lambda - \lambda_{u}^{*}| \leq \eta_{u}$. Then, \begin{equation*}
    \phi_{u}(\lambda) \leq \phi_{u}(\lambda_{u}^{*}) + \frac{M}{2} \eta_{u}^{2} = \phi_{u}(\lambda_{u}^{*}) + \frac{M}{2x_{u}}. 
\end{equation*} Thus \begin{equation*}
    \exp \{ - x_{u}\phi_{u}(\lambda) \} \geq \exp \{ - x_{u} \phi_{u}(\lambda_{u}^{*}) \} \exp \{ - \frac{M}{2} \}, \forall \lambda \in B_{u}. 
\end{equation*} So \begin{equation*}
    q_{u}(\lambda ) \geq c_{\gamma}c_{p}(1+A_{u})^{-\frac{p}{2}} e^{-\frac{M}{2}} \exp \{ - x_{u}\phi_{u}(\lambda_{u}^{*}) \}, \text{ for } \lambda \in B_{u}. 
\end{equation*} Then \begin{equation*}
    \int\limits_{B_{u}}q_{u}(\lambda) d\lambda \geq 2\eta_{u} c_{\gamma} c_{p} e^{-\frac{M}{2}}(1+A_{u})^{-\frac{p}{2}} \exp \{ - x_{u} \phi_{u}(\lambda_{u}^{*}) \}, |B_{u}| = 2\eta_{u}. 
\end{equation*} That is \begin{equation*}
    \int\limits_{B_{u}} q_{u}(\lambda) d\lambda \geq c_{in}\eta_{u}(1+A_{u})^{-\frac{p}{2}} \exp \{ -x_{u}\phi_{u}(\lambda_{u}^{*}) \}, c_{in} := 2c_{\gamma}c_{p}e^{-\frac{M}{2}}>0. 
\end{equation*}
Fix $\epsilon \in (0,2^{-1})$. Furthermore, $O_{u}:= \{ \lambda \geq x_{u}^{-1}x_{0}: |\lambda - \lambda_{u}^{*}| \geq 2^{-1}\epsilon \}.$ Choose $b>2$ and $a \in (0,2^{-1})$ such that $(2a^{m})^{-1} > m+2$. Then $O_{u} = O_{u}^{(L)} \cup O_{u}^{(M)} \cup O_{u}^{(R)}$ where $O_{u}^{(L)} := O_{u} \cap (0,a], O_{u}^{(M)} := O_{u} \cap [a,b], $ and $O_{u}^{(R)} := O_{u} \cap [b,\infty).$ On $[a,b]$ for large $u$, since $\lambda \in [a,b],$ \[
\lambda^{m-2} \geq \ell_{1}=
\begin{cases}
a^{m-2}, & m \geq 2,\\
b^{-1},  & m=1
\end{cases}.
\] For sufficiently large $u, A_{u} \geq 1$. Hence $A_{u}^{-1} \leq 1$. Then for $\lambda \in [a,b],$ \begin{equation*}
    \lambda^{m} + \frac{1}{A_{u}} \leq b^{m} + 1\implies (\lambda^{m} + \frac{1}{A_{u}})^{-3} \geq (b^{m} + 1)^{-3}. 
\end{equation*} For $\lambda \in [a,b], \lambda^{m} \geq a^{m}.$ So \begin{equation*}
    (m+1)\lambda^{m} - \frac{m-1}{A_{u}} \geq (m+1)a^{m} - \frac{m-1}{A_{u}}. 
\end{equation*}
If $m = 1, A_{u}^{-1}(m-1) = 0$. If $m \geq 2,$ then for sufficiently large $u$, \begin{equation*}
    A_{u} \geq \frac{2(m-1)}{(m+1)a^{m}}.
\end{equation*} Then \begin{equation*}
    \frac{m-1}{A_{u}} \leq \frac{m-1}{\frac{2(m-1)}{(m+1)a^{m}}} = \frac{(m+1)a^{m}}{2}. 
\end{equation*} This implies \begin{equation*}
    (m+1)\lambda^{m} - \frac{m-1}{A_{u}} \geq (m+1)a^{m} - \frac{(m+1)a^{m}}{2} = \frac{(m+1)a^{m}}{2}. 
\end{equation*} Therefore \begin{equation*}
    \phi_{u}^{''}(\lambda) \geq m\ell_{1} (b^{m}+1)^{-3}\frac{(m+1)a^{m}}{2} =: c_{0}. 
\end{equation*} So, $\inf\limits_{\lambda \in [a,b]} \phi_{u}^{''}(\lambda) \geq c_{0}>0$. By Taylor's theorem \begin{equation*}
    \phi_{u}(\lambda ) \geq \phi_{u}(\lambda_{u}^{*}) + \phi_{u}'(\lambda_{u}^{*}) (\lambda - \lambda_{u}^{*}) + \frac{\phi_{u}^{''}(\xi)}{2}(\lambda - \lambda_{u}^{*})^{2} \geq  \phi_{u}(\lambda_{u}^{*}) + \frac{c_{0}}{2}(\lambda - \lambda_{u}^{*})^{2}. 
\end{equation*} Hence on $O_{u}^{(M)}, |\lambda - \lambda_{u}^{*}| \geq 2^{-1}\epsilon.$ Then, \begin{equation*}
    \phi_{u}(\lambda) \geq \phi_{u}(\lambda_{u}^{*}) + \frac{c_{0}}{2}(\frac{\epsilon}{2})^{2} = \phi_{u}(\lambda_{u}^{*}) + c_{m}(\epsilon), c_{m}(\epsilon):= 8^{-1} c_{0}\epsilon^{2}. 
\end{equation*} Note that on $[a,b]$ \[
\lambda ^{\gamma} \leq c_{a,b} :=
\begin{cases}
b^{\gamma}, & \gamma > 0,\\
1 & \gamma = 0, \\
a^{\gamma},  & \gamma < 0
\end{cases}.
\] Furthermore $(1+A_{u}\lambda^{m})^{-\frac{p}{2}} \leq (1+A_{u}a^{m})^{-\frac{p}{2}}.$ Observe \\$1+A_{u}a^{m} \geq a^{m} + A_{u}a^{m} = a^{m}(1+A_{u}).$ This implies \begin{equation*}
    (1+A_{u}a^{m})^{-\frac{p}{2}} \leq a^{-\frac{mp}{2}}(1+A_{u})^{-\frac{p}{2}} = (1+A_{u})^{-\frac{p}{2}} c''(a),a^{-\frac{mp}{2}} := c''(a). 
\end{equation*} Therefore \begin{align*}
    \int\limits_{O_{u}^{(M)}} q_{u}(\lambda )d\lambda & = \int\limits_{O_{u}^{(M)}} \lambda ^{\gamma}(1+A_{u}\lambda ^{m})^{-\frac{p}{2}} \exp \{ - x_{u} \phi_{u}(\lambda) \} d\lambda\\ 
    & \leq c_{a,b}c''(a) (1+A_{u})^{-\frac{p}{2}} \int\limits_{O_{u}^{(M)}} \exp \{ - x_{u}\phi_{u}(\lambda) \} d\lambda. 
\end{align*} Note that \begin{align*}
    \exp\{ - x_{u}\phi_{u}(\lambda) \} &\leq \exp \{ -x_{u}(\phi_{u}(\lambda_{u}^{*})  + c_{m}(\epsilon)) \}\\ 
    &= \exp \{ - x_{u}\phi_{u} (\lambda_{u}^{*})\} \exp \{ - x_{u}c_{m}(\epsilon) \}. 
\end{align*} Then \begin{align*}
    \int\limits_{O_{u}^{(M)}} \exp \{ -x_{u}\phi_{u}(\lambda) \} & \leq |O_{u}^{(M)}| \exp \{ - x_{u}\phi_{u}(\lambda_{u}^{*}) \} \exp \{ - x_{u}c_{m}(\epsilon) \}\\ 
    &\leq (b-a)\exp \{ - x_{u} \phi_{u}(\lambda_{u}^{*}) \} \exp \{ - x_{u}c_{m}(\epsilon) \}. 
\end{align*} Therefore \begin{align*}
    \int\limits_{O_{u}^{(M)}} q_{u}(\lambda)d\lambda &\leq c_{a,b} c''(a) (1+A_{u})^{-\frac{p}{2}}(b-a) \exp \{ - x_{u}(\phi_{u}(\lambda_{u}^{*} )+c_{m}(\epsilon)) \}\\ 
    & = K_{m}(\epsilon) (1+A_{u})^{-\frac{p}{2}} \exp \{- x_{u}(\phi_{u}(\lambda_{u}^{*})+c_{m}(\epsilon)) \},
\end{align*} 
$ K_{m}(\epsilon) := c_{a,b}c''(a) (b-a).$ We will now show that \begin{equation*}
    \int\limits_{O_{u}^{(R)}} q_{u}(\lambda) d\lambda \leq K_{R} (1+A_{u})^{-\frac{p}{2}} \exp \{-x_{u}(\phi_{u}(\lambda_{u}^{*}) + C_{R}) \}. 
\end{equation*} Consider $\lambda \geq b$. We know that $\phi_{u}(\lambda) \geq m\lambda$ and for large $u$, \\$(1+A_{u}\lambda^{m})^{-\frac{p}{2}} \leq (A_{u}\lambda^{m})^{-\frac{p}{2}} = A_{u}^{-\frac{p}{2}}\lambda^{-\frac{mp}{2}}. $
For large $u, A_{u} \geq 1$. Then, \begin{equation*}
    1+A_{u} \leq 2A_{u} \implies (1+A_{u})^{-\frac{p}{2}} \geq 2^{-\frac{p}{2}}A_{u}^{-\frac{p}{2}} \implies A_{u}^{-\frac{p}{2}} \leq 2^{\frac{p}{2}} (1+A_{u})^{-\frac{p}{2}}. 
\end{equation*} Therefore, \begin{equation*}
    (1+A_{u}\lambda^{m})^{-\frac{p}{2}} \leq 2^{\frac{p}{2}} (1+A_{u})^{-\frac{p}{2}}\lambda^{-\frac{mp}{2}} = c_{p} (1+A_{u})^{-\frac{p}{2}} \lambda^{-\frac{mp}{2}}, c_{p}:= 2^{\frac{p}{2}}. 
\end{equation*} Thus for $\lambda \geq b$, \begin{align*}
    q_{u}(\lambda) &\leq \lambda^{\gamma} c_{p} (1+A_{u})^{-\frac{p}{2}}\lambda^{-\frac{mp}{2}}\exp \{ - m \lambda x_{u} \} \\ 
    & = c_{p}(1+A_{u})^{-\frac{p}{2}} \lambda ^{\gamma - \frac{mp}{2}} \exp \{ - m \lambda x_{u} \}. 
\end{align*} Let $r : = \gamma - 2^{-1} (mp)$, $a := mx_{u} >0 $ and $I:= \int\limits_{b}^{\infty} \lambda^{r} e^{-a\lambda} d\lambda.$ Let $s := a(\lambda -b) = a\lambda - ab$ which implies $\lambda = b + a^{-1}s$. Then $d\lambda = a^{-1} ds$. Then \begin{align*}
    I &= \int\limits_{0}^{\infty}(b+\frac{s}{a})^{r} \exp \{ - a(b + \frac{s}{a}) \} \frac{ds}{a}\\ 
    &= \frac{1}{a} \int\limits_{0}^{\infty} (b+ \frac{s}{a})^{r} \exp \{ - ab-s \} ds \\ 
    & = \frac{e^{-ab}}{a}\int\limits_{0}^{\infty}(b + \frac{s}{a})^{r}e^{-s}ds. 
\end{align*} Firstly consider $r\leq 0.$ Since $b + a^{-1}s \geq b$ and $ x \rightarrow x^{r}$ is decreasing, $(b + a^{-1}s)^{r} \leq b^{r}.$ So, \begin{align*}
    I & \leq \frac{e^{-ab}}{a} \int \limits_{0}^{\infty} b^{r}e^{-s}ds\\ 
    & = \frac{e^{-ab}b^{r}}{a} \int \limits_{0}^{\infty} e^{-s}ds \\ 
    & = \frac{e^{-ab}b^{r}}{a}[-e^{-s}]\\ 
    & = \frac{b^{r}e^{-ab}}{a}. 
\end{align*}
Now consider $r \geq 0$. Recall $(x+y)^{r} \leq 2^{r}(x^{r} + y^{r})$ for $x,y\geq 0.$ So \begin{equation*}
    (b + \frac{s}{a})^{r} \leq 2^{r} (b^{r} + (\frac{s}{a})^{r}). 
\end{equation*} Then \begin{align*}
    I &\leq \frac{e^{-ab}}{a}2^{r} (b^{r} \int\limits_{0}^{\infty} e^{-s} ds + \frac{1}{a^{r}} \int\limits_{0}^{\infty}s^{r}e^{-s}ds)\\ 
    &= \frac{e^{-ab}2^{r}}{a}(b^{r} + \frac{1}{a^{r}}\Gamma(r+1)). 
\end{align*} For sufficiently large $u, a\geq 1, $ so $a^{-r}\Gamma(r+1) \leq \Gamma(r+1)$. Hence \begin{align*}
    I &\leq \frac{e^{-ab}}{a}2^{r}(b^{r} + \Gamma(r+1))\\ 
    &= \frac{C_{r,b}}{mx_{u}} e^{-mbx_{u}}, C_{r,b}:= 2^{r}(b^{r} + \Gamma(r+1)). 
\end{align*} So for $r \leq 0$ \begin{equation*}
    I \leq \frac{b^{r}}{mx_{u}}e^{-mbx_{u}} \text{ and } r \geq 0 \text{ implies } I \leq \frac{C_{r,b}}{mx_{u}} e^{-mbx_{u}}. 
\end{equation*} As $x_{u} \rightarrow \infty$ as $u \rightarrow \infty$, we have for sufficiently large $u_{0}, x_{u}^{-1} \leq x_{u_{0}}^{-1}$. So \begin{equation*}
    I \leq \frac{C_{r,b}}{mx_{u_{0}}}e^{-mbx_{u}} =: Ke^{-mbx_{u}}. 
\end{equation*} So \begin{align*}
    I &\leq \frac{\tilde{C}_{r,b}}{a} e^{-ab}, \tilde{C}_{r,b} = \begin{cases}
b^{r}, & r \leq 0,\\
2^{r}(b^{r} + \Gamma(r+1)),& r\geq 0
\end{cases} \\ 
&= K_{r,b}e^{-ab}, \text{ with } K_{r,b} : = \frac{C_{r,b}}{mx_{u_{0}}} \\ 
&= K_{r,b}e^{-mbx_{u}}. 
\end{align*} Hence \begin{equation*}
    \int\limits_{O_{u}^{(R)}}q_{u}(\lambda) d\lambda \leq C(1+A_{u}) ^{-\frac{p}{2}}K_{r,b}e^{-mbx_{u}} = K(1+A_{u})^{-\frac{p}{2}}e^{-mbx_{u}}, K:= CK_{r,b}. 
\end{equation*}
We know that $\lambda_{u}^{*} \rightarrow 1$. Then for sufficiently large $u$, $\lambda_{u}^{*} \in [0.5,1.5]$. On that interval \begin{equation*}
    \phi_{u}(\lambda_{u}^{*}) = m\lambda_{u}^{*} + \frac{1}{(\lambda_{u}^{*})^{{m}} + \frac{1}{A_{u}}} \leq \frac{3m}{2} + \frac{1}{(\frac{1}{2})^{m}} = \frac{3m}{2} + 2^{m}. 
\end{equation*} Therefore $\phi_{u}(\lambda_{u}^{*})\leq M$ for sufficiently large $u$. Since $b>2, mb>m+1$. So \\$C_{R}:= 2^{-1}(mb - (m+1)) >0$. Then $2C_{R} + (m+1) = mb.$ Since $\lambda_{u}^{*} \rightarrow 1$ which implies $\phi_{u}(\lambda_{u}^{*}) \rightarrow \phi_{\infty}(1) = m+1$. Then $\exists u_{1}$ such that $\forall u \geq u_{1}, \phi_{u}(\lambda_{u}^{*}) \leq (m+1) + C_{R}$ which implies $\phi_{u}(\lambda_{u}^{*}) +C_{R} \leq (m+1) + 2C_{R}$ which implies $\phi_{u}(\lambda_{u}^{*}) + C_{R} \leq mb$. Then \begin{align*}
    &x_{u}(\phi_{u}(\lambda_{u}^{*}) + C_{R}) \leq mbx_{u}\\ 
    \implies -&x_{u}(\phi_{u}(\lambda_{u}^{*}) + C_{R}) \geq -mbx_{u}\\ 
    \implies & \exp \{ - x_{u}(\phi_{u}(\lambda_{u}^{*}) + C_{R}) \} \geq \exp\{-mbx_{u}  \}. 
\end{align*} Therefore \begin{equation*}
    \int\limits_{O_{u}^{(R)}} q_{u}(\lambda) d\lambda \leq K_{R}(1+A_{u})^{-\frac{p}{2}} \exp \{ - x_{u}(\phi_{u}(\lambda_{u}^{*})+C_{R}) \}.
\end{equation*}
Now we will show that \begin{equation*}
    \int\limits_{O_{u}^{(L)}} q_{u}(\lambda) \leq K_{L} \exp \{ -x_{u}(\phi_{u}(\lambda_{u}^{*}) +C_{L}) \}. 
\end{equation*} Note that \begin{equation*}
    \int\limits_{O_{u}^{(L)}} q_{u}(\lambda) d\lambda \leq \int\limits_{0}^{a}q_{u}(\lambda) d\lambda. 
\end{equation*} Let $\lambda_{*}:= A_{u}^{-\frac{1}{m}}$ which implies $\lambda_{*}^{m} = A_{u}^{-1}$. Observe that for sufficiently large $u, \lambda_{*}<a.$ Then \begin{equation*}
    \int\limits_{O_{u}^{(L)}} q_{u}(\lambda) d\lambda \leq \int\limits_{0}^{\lambda_{*}} q_{u}(\lambda) d\lambda + \int\limits_{\lambda_{*}}^{a} q_{u}(\lambda) d\lambda. 
\end{equation*} If $\lambda \leq \lambda_{*}$ then $\lambda ^{m} \leq A_{u}^{-1}.$ So $\lambda^{m} + A_{u}^{-1} \leq 2A_{u}^{-1}. $ Hence \begin{equation*}
    \phi_{u}(\lambda) \geq \frac{1}{\lambda^{m} + \frac{1}{A_{u}}} \geq \frac{1}{\frac{2}{A_{u}}}  = \frac{A_{u}}{2}. 
\end{equation*} Furthermore $(1+A_{u}\lambda ^{m})^{-\frac{p}{2}} \leq 1. $ Therefore \begin{align*}
    \int\limits_{0}^{\lambda_{*}} q_{u}(\lambda) d\lambda &\leq \int\limits_{0}^{\lambda_{*}} \lambda ^{\gamma} \exp \{ - x_{u} \frac{A_{u}}{2} \} d\lambda \\ 
    &= \exp \{ - \frac{x_{u}A_{u}}{2} \} \int\limits_{0}^{\lambda_{*}} \lambda ^{\gamma} d\lambda \\ 
    & = \exp \{ - \frac{x_{u}A_{u}}{2} \} [\frac{\lambda ^{\gamma + 1}}{\gamma + 1}]_{0}^{\lambda _{*}}\\ 
    &= \exp \{ -\frac{x_{u}A_{u}}{2} \} \frac{\lambda_{*}^{\gamma + 1}}{ \gamma +1} \\ 
    &= \exp \{ - \frac{u}{4} \} \frac{\lambda_{*}^{\gamma +1}}{ \gamma + 1}. 
\end{align*} Since $\lambda_{*} = A_{u}^{-\frac{1}{m}} \rightarrow 0$ and $\gamma +1> 0, $ we have $\lambda_{*}^{\gamma +1} \leq 2^{-1}$ for sufficiently large $u$. So for large $u$, $\lambda ^{\gamma + 1}(\gamma +1)^{-1} \leq (\gamma + 1)^{-1} =: c$. Therefore \begin{equation*}
    \int\limits_{0}^{\lambda_{*}} q_{u}(\lambda) d\lambda \leq ce^{- \frac{u}{4}}. 
\end{equation*} If $\lambda \geq \lambda_{*}$ then $\lambda^{m} \geq A_{u}^{-1}$. Hence $2\lambda ^{m} \geq \lambda ^{m} + A_{u}^{-1}. $ Therefore $\phi_{u}(\lambda) \geq (2\lambda^{m})^{-1}.$ Thus \begin{align*}
    \int\limits_{\lambda_{*}}^{a} q_{u}(\lambda)d\lambda &\leq \int\limits_{\lambda_{*}} ^{a} \lambda ^{\gamma} (1+A_{u}\lambda^{m})^{-\frac{p}{2}} \exp \{ - \frac{x_{u}}{2\lambda^{m}}  \}d\lambda \\ 
    &\leq \int\limits_{\lambda_{*}}^{a} \lambda ^{\gamma} \exp \{ - \frac{x_{u}}{2\lambda^{m}} \} d\lambda.
\end{align*}
Let \begin{equation*}
    t := \frac{x_{u}}{2\lambda^{m}} \implies \lambda = (\frac{x_{u}}{2t})^{\frac{1}{m}} = (\frac{x_{u}}{2})^{\frac{1}{m}}t^{-\frac{1}{m}}. 
\end{equation*}  Then \begin{equation*}
    d\lambda = (\frac{x_{u}}{2})^{\frac{1}{m}}(- \frac{1}{m}t^{-\frac{1}{m}-1})dt \text{ and } \lambda^{\gamma } = (\frac{x_{u}}{2})^{\frac{\gamma}{m}}t^{-\frac{\gamma}{m}}. 
\end{equation*} So \begin{equation*}
    \lambda^{\gamma} d\lambda = (\frac{x_{u}}{2})^{\frac{\gamma}{m}}t^{-\frac{\gamma}{m}}(\frac{x_{u}}{2})^{\frac{1}{m}}(-\frac{1}{m})t^{-\frac{1}{m}-1}dt = - \frac{1}{m}(\frac{x_{u}}{2})^{\frac{(\gamma+1)}{m}} t^{-\frac{\gamma +1}{m}-1}dt. 
\end{equation*} When $\lambda = a, t_{a} := (2a^{m})^{-1}x_{u}$ and when $\lambda = \lambda_{*},t_{*} = (2\lambda_{*}^{m})^{-1}x_{u}$. Then \begin{align*}
    \int\limits_{\lambda_{*}}^{a} \lambda^{\gamma} \exp \{ - \frac{x_{u}}{2 \lambda ^{m}} \} d\lambda &= \int\limits_{t_{*}}^{t_{a}} -\frac{1}{m} (\frac{x_{u}}{2})^{\frac{\gamma +1}{m}}t^{-\frac{\gamma  + 1}{m}-1}e^{-t} dt \\ 
    &= \frac{1}{m}(\frac{x_{u}}{2})^{\frac{\gamma +1}{m}} \int\limits_{t_{a}} ^{t_{*}} t^{-\frac{\gamma +1}{m}-1} e^{-t} dt.  
\end{align*} Since $\gamma >-1, m^{-1}(\gamma +1) >0$. So $-m^{-1} (\gamma + 1) < 0$. Hence $\forall t\geq 1, t^{-\frac{\gamma +1}{m} - 1} \leq 1. $ We know that $t_{a} = (2a^{m})^{-1}x_{u} \rightarrow \infty$ as $u \rightarrow \infty. $ So for all sufficiently large u $t_{a} \geq 1.$ Therefore \begin{equation*}
    \int\limits_{t_{a}}^{t_{*}}t^{- \frac{\gamma + 1}{m}-1}e^{-t}dt \leq \int\limits_{t_{a}}^{t_{*}} e^{-t}dt \leq \int\limits_{t_{a}}^{\infty}e^{-t}dt = [-e^{-t}]_{t_{a}}^{\infty} = e^{-t_{a}}. 
\end{equation*} Therefore \begin{align*}
    \int\limits_{\lambda_{*}}^{a} \lambda ^{\gamma} \exp \{ - \frac{x_{u}}{2\lambda ^{m}} \} d\lambda &\leq \frac{1}{m}(\frac{x_{u}}{2})^{\frac{\gamma +1}{m}} e^{-t_{a}} \\ 
    &= \frac{1}{m2^{\frac{\gamma +1}{m}}}x_{u}^{\frac{\gamma +1}{m}} \exp \{ - \frac{x_{u}}{2a^{m}} \} \\ 
    &= Kx_{u}^{\frac{\gamma +1}{m}} \exp \{ - \frac{x_{u}}{2a^{m}} \}, K:= \frac{1}{m}2^{-\frac{\gamma +1}{m}}. 
\end{align*} Let \begin{equation*}
    4C_{L}:= \frac{1}{2a^{m}} - (m+2).
\end{equation*}  
Then,
\begin{equation*}
4C_{L} > 0 \iff \frac{1}{2a^{m}} - (m+2) > 0 \iff \frac{1}{2a^{m}} > m+2, \text{by definition of } a. 
\end{equation*} Let \begin{equation*}
    M_{u} := \min\limits_{\lambda \in [\frac{1}{2}, 2]}\phi_{u}(\lambda) = \phi_{u}(\lambda_{u}^{*}) \text{ and } M_{\infty} := \min\limits_{\lambda \in [\frac{1}{2}, 2]} \phi_{\infty}(\lambda) = \phi_{\infty}(1) = m+1. 
\end{equation*} Recall for functions $f$ and $g$ on a compact set $K$ that $|\min\limits_{k}f - \min\limits_{k}g| \leq \sup\limits_{k}|f-g|. $Then $|\min\limits_{\lambda \in [\frac{1}{2},2]}\phi_{u} - \min\limits_{\lambda \in [\frac{1}{2},2]}\phi_{\infty}| = |\phi_{u}(\lambda_{u}^{*}) - (m+1)| \leq \sup\limits_{\lambda \in [\frac{1}{2},2]}|\phi_{u}(\lambda) - \phi_{\infty}(\lambda)| \rightarrow 0.$ Therefore $\phi_{u}(\lambda_{u}^{*}) \rightarrow m+1. $ Then for all large $u, \phi_{u}(\lambda_{u}^{*}) \leq m+2$. Hence \begin{align*}
    \exp \{ - \frac{x_{u}}{2a^{m}} \} &\leq \exp \{ -x_{u}((m+2) + 4C_{L}) \} \\ 
    &\leq \exp \{ - x_{u} (\phi_{u}(\lambda_{u}^{*})+4C_{L}) \}. 
\end{align*} Therefore \begin{align*}
    \int\limits_{\lambda_{*}}^{a} \lambda^{\gamma} \exp \{ - \frac{x_{u}}{2\lambda^{m}} \} d\lambda &\leq Kx_{u}^{\frac{\gamma +1}{m}}\exp \{- x_{u}(\phi_{u}(\lambda_{u}^{*})+4C_{L}) \} \\ 
    &= K \exp \{ - x_{u}\phi_{u}(\lambda_{u}^{*}) \}[x_{u}^{\frac{\gamma +1}{m}}e^{-4C_{L}x_{u}}]. 
\end{align*} Now, for all sufficiently large $u$ we will show that $x_{u}^{\frac{\gamma +1}{m}} \exp \{ - 4C_{L}x_{u} \} \leq \exp \{ - C_{L}x_{u} \}$. That is \begin{align*}
    x_{u}^{\frac{\gamma+1}{m}} &\leq \exp \{ 3C_{L}x_{u} \}\\ 
    \iff \frac{\gamma +1}{m}\log(x_{u}) &\leq 3C_{L}x_{u}\\ 
    \iff \frac{\log(x_{u})}{x_{u}} &\leq \frac{3mC_{L}}{\gamma +1} \text{ which is true for sufficiently large }u. 
\end{align*} So , \begin{equation*}
    \int\limits_{\lambda_{*}}^{a} q_{u}(\lambda)d\lambda \leq K \exp \{ - x_{u}(\phi_{u}(\lambda_{u}^{*}) + C_{L}) \}. 
\end{equation*} Then \begin{equation*}
    \int\limits_{O_{u}^{(L)}} q_{u}(\lambda) d\lambda \leq ce^{-\frac{u}{4}} + K \exp \{ - x_{u}(\phi_{u}(\lambda_{u}^{*}) + C_{L}) \}. 
\end{equation*} We know $\phi_{u}(\lambda_{u}^{*}) \leq \phi_{u}(1),$ but $\phi_{u}(1)  \leq m+1. $ Hence $\phi_{u}(\lambda_{u}^{*}) + C_{L} \leq m+1 + C_{L}$. Let $B := m+1 +C_{L} >0$. We want \begin{align*}
    \exp \{ -\frac{u}{4} \} &\leq \exp \{ - x_{u}(\phi_{u}(\lambda_{u}^{*}) + C_{L}) \} \text{ i.e. } \\
    \frac{u}{4} & \geq x_{u}(\phi_{u}(\lambda_{u}^{*}) + C_{L}). 
\end{align*} Then it is sufficient to check that $4^{-1}u \geq Bx_{u}.$ That is \begin{equation*}
    \frac{C}{2}x_{u}^{m+1} \geq Bx_{u} \iff \frac{C}{2}x_{u}^{m} \geq B \text{ which holds for sufficiently large }u. 
\end{equation*} 
Then, 
\begin{align*}
    \int\limits_{O_{u}^{(L)}} q_{u}(\lambda) d\lambda & \leq (c+K) \exp \{ - x_{u}(\phi_{u}(\lambda_{u}^{*}) + C_{L}) \}\\ 
    &= K_{L}\exp\{ - x_{u}(\phi_{u}(\lambda_{u}^{*})+C_{L}) \}.\\ 
\end{align*} This gives \begin{align*}
    \int\limits_{O_{u}}q_{u}(\lambda)d\lambda &= \int\limits_{O_{u}^{(L)}}q_{u}(\lambda) d\lambda + \int\limits_{O_{u}^{(M)}}q_{u}(\lambda) d\lambda + \int\limits_{O_{u}^{(R)}}q_{u}(\lambda) d\lambda \\
    &\leq (1+A_{u})^{-\frac{p}{2}}\exp \{ -x_{u}\phi_{u}(\lambda_{u}^{*}) \} [K_{M} \exp\{ - C_{M} x_{u} \} \\&+ K_{R}\exp\{ -C_{R}x_{u} \}]+ K_{L}\exp\{ - x_{u}(\phi_{u}(\lambda_{u}^{*}) +C_{L}) \} .
\end{align*} Let $C_{\text{min}} := \min \{ C_{M}, C_{R}, C_{L} \}$. Then $e^{-C_{M}x_{u}} \leq e^{-C_{\text{min}}x_{u}}, e^{-C_{R}x_{u}}\leq e^{-C_{\text{min}}x_{u}}$, \\and $e^{-C_{L}x_{u}} \leq e^{-C_{\text{min}}x_{u}}$. So \begin{align*}
    \int\limits_{O_{u}} q_{u}(\lambda) d\lambda &\leq (1+A_{u})^{-\frac{p}{2}}e^{-x_{u}(\phi_{u}(\lambda_{u}^{*}) + C_{\text{min}})} (K_{M}+K_{R}) +K_{L}e^{-x_{u}(\phi_{u}(\lambda_{u}^{*}) + C_{\text{min}})}\\
    &= e^{-x_{u}(\phi_{u}(\lambda_{u}^{*}) + C_{\text{min}})} [(K_{M} + K_{R})(1+A_{u})^{-\frac{p}{2}} + K_{L}]\\ 
    &\leq e^{-x_{u}(\phi_{u}(\lambda_{u}^{*}) + C_{\text{min}})}[K_{M} + K_{R} + K_{L}]\\ 
    &= K_{\epsilon} e^{-x_{u}(\phi_{u}(\lambda_{u}^{*}) + C_{\text{min}})}, K_{\epsilon}:=K_{M} + K_{R} + K_{L}. 
\end{align*} Since $|\lambda -1| \geq \epsilon $ which  implies $|\lambda - \lambda_{u}^{*}| \geq 2^{-1}\epsilon$ for sufficiently large $u$ and \begin{equation*}
    O_{u}:=\{ \lambda \geq \frac{x_{0}}{x_{u}} : |\lambda - \lambda_{u}^{*}| \geq \frac{\epsilon}{2} \}. 
\end{equation*} Then \begin{equation*}
    \Pi_{u,\boldsymbol{k}}(|\lambda -1| \geq \epsilon|X\geq x_{0}) \leq \Pi_{u,\boldsymbol{k}}(\lambda \in O_{u}|X\geq x_{0}). 
\end{equation*} We showed that \begin{equation*}
    \int\limits_{\frac{x_{0}}{x_{u}}}^{\infty} q_{u}(\lambda) d\lambda \geq \int\limits_{B_{u}}q_{u}(\lambda)d\lambda \geq C_{in} \eta_{u} (1+A_{u})^{-\frac{p}{2}} e^{-x_{u}\phi_{u}(\lambda_{u}^{*})}. 
\end{equation*} So \begin{align*}
    \Pi_{u, \boldsymbol{k}}(\lambda \in O_{u}|X \geq x_{0}) &\leq \frac{K_{2}}{K_{1}}\frac{K_{\epsilon} \exp \{ -x_{u}(\phi_{u}(\lambda_{u}^{*})+C_{\text{min}}) \}}{C_{in}\eta_{u}(1+A_{u})^{-\frac{p}{2}} \exp \{ -x_{u}\phi_{u}(\lambda_{u}^{*}) \}}  \\ 
    &= \frac{K_{2}}{K_{1}} \frac{K_{\epsilon}}{C_{in}}\eta_{u}^{-1} (1+A_{u})^{\frac{p}{2}} e^{-x_{u}C_{\text{min}}}. 
\end{align*} Recall $\eta_{u}^{-1} = \sqrt{x_{u}}$ and $A_{u} = Cx_{u}^{m}$. Then $\eta_{u}^{-1}(1+A_{u})^{\frac{p}{2}}.$ For sufficiently large $u$, $A_{u} \geq 1.$ Hence $1+A_{u} \leq 2A_{u} = 2Cx_{u}^{m}.$ Then $(1+A_{u})^{\frac{p}{2}} \leq (2C)^{\frac{p}{2}} x_{u}^{\frac{mp}{2}}. $ Then \begin{align*}
    \eta_{u}^{-1} (1+A_{u})^{\frac{p}{2}} &\leq x_{u}^{\frac{1}{2}}(2C)^{\frac{p}{2}}x_{u}^{\frac{mp}{2}}\\ &=(2C)^{\frac{p}{2}}x_{u}^{\frac{mp+1}{2}}\\ &= C_{0}x_{u}^{B}, C_{0}:=(2C)^{\frac{p}{2}}, B:=\frac{mp+1}{2}.
\end{align*} Therefore \begin{equation*}
    \Pi_{u, \boldsymbol{k}}(\lambda \in O_{u}|X\geq x_{0}) \leq \frac{K_{2}}{K_{1}} \frac{K_{\epsilon}}{C_{in}} x_{u}^{B} e^{-C_{\text{min}}x_{u}} = C_{1}x_{u}^{B} e^{-C_{\text{min}}x_{u}}. 
\end{equation*} We know for sufficiently large $u, C_{1}x_{u}^{B} \leq e^{\frac{C_{\text{min}}}{2}x_{u}}. $ Then \begin{equation*}
    \Pi_{u, \boldsymbol{k}}(\lambda \in O_{u}|X\geq x_{0}) \leq e^{\frac{C_{\text{min}}}{2}x_{u}}e^{-C_{\text{min}}x_{u}} = e^{-\frac{C_{\text{min}}}{2}x_{u}}. 
\end{equation*} Therefore, \begin{equation*}
    \Pi_{u, \boldsymbol{k}}(|\lambda -1| \geq \epsilon|X\geq x_{0}) \leq e^{-c^{'}x_{u}}, c':=\frac{C_{\text{min}}}{2}. 
\end{equation*} Then \begin{equation*}
    \Pi_{u, \boldsymbol{k}}(|\frac{X}{x_{u}} -1| \geq \epsilon |X\geq x_{0}) \leq e^{-c^{'}x_{u}}. 
\end{equation*} Note that \begin{equation*}
    \Pi_{u, \boldsymbol{k}}(|\frac{X}{x_{u}}-1| \geq \epsilon) = \Pi_{u, \boldsymbol{k}}(|\frac{X}{x_{u}}-1| \geq \epsilon , X < x_{0}) + \Pi_{u, \boldsymbol{k}}(|\frac{X}{x_{u}}-1| \geq \epsilon, X \geq x_{0}). 
\end{equation*} Then \begin{align*}
    \Pi_{u, \boldsymbol{k}}(|\frac{X}{x_{u}}-1| \geq \epsilon , X\leq x_{0}) &\leq \Pi_{u, \boldsymbol{k}}(X<x_{0}) \\ 
    &\text{ and }\\
\Pi_{u, \boldsymbol{k}}(|\frac{X}{x_{u}}-1| \geq \epsilon, X\geq x_{0}) &= \Pi_{u, \boldsymbol{k}}(X \geq x_{0}) \Pi_{u, \boldsymbol{k}}(|\frac{X}{x_{u}}-1| \geq \epsilon|X\geq x_{0})\\ &\leq \Pi_{u, \boldsymbol{k}}(|\frac{X}{x_{u}}-1| \geq \epsilon|X\geq x_{0}). 
\end{align*} So \begin{equation*}
    \Pi_{u, \boldsymbol{k}}(|\frac{X}{x_{u}}-1| \geq \epsilon) \leq \Pi_{u, \boldsymbol{k}}(X<x_{0})+e^{-c'x_{u}}. 
\end{equation*} Note that $e^{-C_{0}u} \leq e^{-C_{0} Kx_{u}}$ for sufficiently large $u$ and constant $K$. Then for sufficiently large $u$ \begin{equation*}
    \Pi_{u, \boldsymbol{k}}(|\frac{X}{x_{u}}-1| \geq \epsilon) \leq e^{-c_{0}Kx_{u}} + e^{-c^{'}x_{u}} = 2e^{-c_{\epsilon}x_{u}}, c_{\epsilon} := \min \{c_{0}K,c' \}. 
\end{equation*} We want \begin{equation*}
    2e^{-c_{\epsilon}x_{u}} \leq e^{-\frac{c_{\epsilon}}{2}x_{u}} \iff 2 \leq e^{\frac{c_{\epsilon}x_{u}}{2}} \iff \log(2) \leq \frac{c_{\epsilon}}{2}x_{u} \iff \frac{2}{c_{\epsilon}}\log(2) \leq x_{u},
\end{equation*} which is true for sufficiently large $u$. Therefore \begin{equation*}
    \Pi_{u, \boldsymbol{k}}(|\frac{X}{x_{u}}-1| \geq \epsilon) \leq e^{-c_{\epsilon}x_{u}}. 
\end{equation*} That is \begin{equation*}
    \mathbb{P}(|\frac{X}{x_{u}}-1| \geq \epsilon|U=u, \boldsymbol{K} = \boldsymbol{k}) \leq e^{-c_{\epsilon}x_{u}}. 
\end{equation*} Observe that \begin{equation*}
    |\frac{X}{x_{u}}-1| < \epsilon \iff 1-\epsilon < \frac{X}{x_{u}} < 1+\epsilon.  
\end{equation*} By taking complements \begin{equation*}
    \{ |\frac{X}{x_{u}}-1| \geq \epsilon \} = \{ X\leq (1-\epsilon)x_{u} \} \cup \{ X \geq (1+\epsilon)x_{u} \}. 
\end{equation*} Therefore \begin{equation*}
    \mathbb{P}((1-\epsilon) x_{u} < X < (1+\epsilon)x_{u}|U=u, \boldsymbol{K} = \boldsymbol{k}) \geq 1-e^{-c_{\epsilon}x_{u}}. 
\end{equation*} Note that \begin{align*}
    (1-\epsilon)x_{u} < X < (1+\epsilon)x_{u} &\iff (1-\epsilon)^{m}x_{u}^{m} < X^{m} < (1+\epsilon)^{m} x_{u}^{m}\\ & \iff C(1-\epsilon)^{m}x_{u}^{m} < CX^{m} < C(1+\epsilon)^{m}x_{u}^{m}. 
\end{align*} Therefore, \begin{equation*}
    \mathbb{P}(V \in [C(1-\epsilon)^{m}x_{u}^{m}, C(1+\epsilon)^{m}x_{u}^{m}]|U=u, \boldsymbol{K} = \boldsymbol{k}) \geq 1- e^{-c_{\epsilon}x_{u}}. 
\end{equation*} Observe that \begin{equation*}
    Cx_{u}^{m} = C(\frac{u}{2C}) ^{\frac{m}{m+1}} = 2^{-\frac{m}{m+1}}C^{\frac{1}{m+1}} u^{\frac{m}{m+1}}. 
\end{equation*} Recall $m = d-1$ so $m(m+1)^{-1} = d^{-1}(d-1) = 1- d^{-1}. $ Therefore 
\begin{equation*}
    Cx_{u}^{m} = 2^{-(1-\frac{1}{d})} C^{\frac{1}{d}} u^{1-\frac{1}{d}} = C_{d}u^{1-\frac{1}{d}}, C_{d}:=2^{-(1-\frac{1}{d})}C^{\frac{1}{d}}
\end{equation*}
So 
\begin{equation*}
    \mathbb{P}(V \asymp u^{1-\frac{1}{d}}|U=u, \boldsymbol{K} = \boldsymbol{k}) \geq 1- e^{-c_{\epsilon}x_{u}}. 
\end{equation*} Hence for sufficiently large $u$ \begin{equation*}
    \mathbb{P}(V \asymp u^{1-\frac{1}{d}}|U=u, \boldsymbol{K} = \boldsymbol{k}) \geq \frac{1}{2}. 
\end{equation*} 
That is, there exists constants $c_{1,k}, c_{2,k}>0$ such that for sufficiently large $u$
\begin{equation*}
    \mathbb{P}(c_{1,k} u^{1-\frac{1}{d}} \leq V \leq c_{2,k} u^{1-\frac{1}{d}}|U=u, \boldsymbol{K} = \boldsymbol{k}) \geq \frac{1}{2}.
\end{equation*}
Let $\mathcal{A}_{u} := \{ c_{1}u^{1- \frac{1}{d}} \leq V \leq c_{2}u^{1-\frac{1}{d}} \}, c_{1} = \min\limits_{k} c_{1,k}, c_{2} = \max\limits_{k} c_{2,k}$. Observe that 
\begin{align*}
    \mathbb{P}(\mathcal{A}_{u} |U = u) &= \sum\limits_{k} \mathbb{P}(\mathcal{A}_{u}|U=u, \boldsymbol{K} = \boldsymbol{k})\mathbb{P}(\boldsymbol{K} = \boldsymbol{k}|U = u)\\
    &\geq \sum\limits_{k} \frac{1}{2} \mathbb{P}(\boldsymbol{K} = \boldsymbol{k}|U=u)\\ 
    &= \frac{1}{2}\sum\limits_{k} \mathbb{P}(\boldsymbol{K} = \boldsymbol{k}|U= u ) \\ 
\end{align*}
Therefore, $\mathbb{P}(\mathcal{A}_{u}|U=u)\geq \frac{1}{2}$.
Since $1_{\mathcal{A}_{u}} \geq 0$ and $(1+V)^{-1} \geq 0$, \begin{equation*}
    \psi(u) = \mathbb{E}[\frac{1}{1+V}|U=u] \geq \mathbb{E}[\frac{1}{1+V}1_{\mathcal{A}_{u}}|U=u]. 
\end{equation*} On $\mathcal{A}_{u}$ we have $V \leq c_{2}u^{1-\frac{1}{d}}.$ Since $V \rightarrow (1+V)^{-1}$ is decreasing on $[0,\infty)$ we have \begin{equation*}
    \frac{1}{1+V} \geq \frac{1}{1+c_{2}u^{1-\frac{1}{d}}} \text{ on } \mathcal{A}_{u}. 
\end{equation*} Equivalently \begin{equation*}
    \frac{1}{1+V} 1_{\mathcal{A}_{u}} \geq \frac{1}{1+c_{2}u^{1-\frac{1}{d}}} 1_{\mathcal{A}_{u}}. 
\end{equation*} Then, \begin{equation*}
    \mathbb{E}[\frac{1}{1+V}1_{\mathcal{A}_{u}}|U=u] \geq \frac{1}{1+c_{2}u^{1-\frac{1}{d}}} \mathbb{E}[1_{\mathcal{A}_{u}}|U=u]. 
\end{equation*} Note that $\mathbb{E}[1_{\mathcal{A}_{u}}|U=u] = \mathbb{P}(\mathcal{A}_{u}|U=u) \geq 2^{-1}$. Hence \begin{equation*}
    \psi(u) \geq \frac{1}{2(1+c_{2}u^{1-\frac{1}{d}})}, \text{ for sufficiently large }u. 
\end{equation*} Let $a:= 1-d^{-1} \in (0,1).$ For any $u$ such that $c_{2} u^{a} \geq 1 $ iff $u \geq c_{2}^{-\frac{1}{a}}$ which is true for sufficiently large $u$. Then $1+c_{2}u^{a} \leq 2c_{2}u^{a}$ which implies $(1+c_{2}u^{a})^{-1} \geq (2c_{2}u^{a})^{-1}.$ Hence for all sufficiently large $u,$ \begin{equation*}
    \psi(u) \geq \frac{1}{2}\frac{1}{1+c_{2}u^{a}} \geq \frac{1}{2}\frac{1}{2c_{2}u^{a}} = \frac{1}{4c_{2}}u^{-a} = \frac{1}{4c_{2}}u^{-(1-\frac{1}{d})}. 
\end{equation*} This implies \begin{equation*}
    u\psi(u) \geq \frac{1}{4c_{2}}u^{\frac{1}{d}} \rightarrow \infty \text{ as } u \rightarrow \infty. 
\end{equation*} Then $\exists u_{*}<\infty$ such that $\forall u \geq u_{*}, u\psi(u) \geq (4c_{2})^{-1}u^{\frac{1}{d}}.$ Now let $u_{0}:= \max\{ u_{*}, (16pc_{2})^{d}\}.$ Then for any $u \geq u_{0}$\begin{equation*}
    u\psi(u) \geq \frac{1}{4c_{2}}u^{\frac{1}{d}} \geq \frac{1}{4c_{2}}u_{0}^{\frac{1}{d}} \geq \frac{1}{4c_{2}} (16pc_{2}) = 4p. 
\end{equation*} Recall $B(u) \geq \psi(u)(u\psi(u) - 2p).$ For $u \geq u_{0}, u\psi(u) - 2p \geq 2p$. Hence $u  \geq u_{0} $ implies 
\begin{equation*}
    B(u) \geq 2p\psi(u) \geq \frac{2p}{4c_{2}}u^{\frac{1}{d}-1} = \frac{p}{2c_{2}}u^{\frac{1}{d}-1}. 
\end{equation*} Let $A_{\boldsymbol{\theta}} := \{ u_{0}  \leq U \leq b_{\boldsymbol{\theta}}\}$ with $b_{\boldsymbol{\theta}}$ to be chosen. Note $B(u) \geq 0$ is not guaranteed everywhere. Therefore we will lower bound the expectation on $A_{\boldsymbol{\theta}}.$ In addition, observe that $B(u) \geq -2p$. So, 
\begin{align*}
    \mathbb{E}_{\boldsymbol{\theta}}[B(U)] & = \mathbb{E}_{\boldsymbol{\theta}}[B(U)\boldsymbol{1}_{A_{\boldsymbol{\theta}}}] + \mathbb{E}_{\boldsymbol{\theta}}[B(U)\boldsymbol{1}_{A_{\boldsymbol{\theta}}^{c}}]\\
    &\geq \mathbb{E}_{\boldsymbol{\theta}}[B(U) \boldsymbol{1}_{A_{\boldsymbol{\theta}}}]-2p\mathbb{P}_{\boldsymbol{\theta}}(A_{\boldsymbol{\theta}}^{c}). 
\end{align*}
On $A_{\boldsymbol{\theta}}, U \geq u_{0}$, so $B(U)1_{A_{\boldsymbol{\theta}}} \geq (2c_{2})^{-1}pU^{\frac{1}{d}-1}1_{A_{\boldsymbol{\theta}}}.$ Now we will choose $b_{\boldsymbol{\theta}}$ and use the fact that $u\rightarrow u^{\frac{1}{d}-1}$ is decreasing since $d\geq 2$. Thus on the event $\{ U \leq b_{\boldsymbol{\theta}} \}$, $U^{\frac{1}{d}-1} \geq b_{\boldsymbol{\theta}}^{\frac{1}{d}-1}$. Hence on $A_{\boldsymbol{\theta}}$, $U^{\frac{1}{d}-1}1_{A_{\boldsymbol{\theta}}} \geq b_{\boldsymbol{\theta}}^{\frac{1}{d}-1}1_{A_{\boldsymbol{\theta}}}.$ Therefore, $B(u) 1_{A_{\boldsymbol{\theta}}} \geq p(2c_{2})^{-1} b_{\boldsymbol{\theta}}^{\frac{1}{d}-1}1_{A_{\boldsymbol{\theta}}}.$ Then, \begin{equation*}
    \mathbb{E}_{\boldsymbol{\theta}}[B(U)] \geq \frac{p}{2c_{2}} b_{\boldsymbol{\theta}}^{\frac{1}{d}-1}\mathbb{E}[1_{A_{\boldsymbol{\theta}}}] -2p\mathbb{P}_{\boldsymbol{\theta}}(A_{\boldsymbol{\theta}}^{c}) = \frac{p}{2c_{2}}b_{\boldsymbol{\theta}}^{\frac{1}{d}-1}\mathbb{P}_{\boldsymbol{\theta}}(A_{\boldsymbol{\theta}}) - 2p \mathbb{P}_{\boldsymbol{\theta}}(A_{\boldsymbol{\theta}}^{c}). 
\end{equation*} We know $\boldsymbol{Y}= \boldsymbol{\theta} +\boldsymbol{Z}$ with $\boldsymbol{Z} \sim N_{p}(\boldsymbol{0}_{p},I_{p}).$ The $U = ||\boldsymbol{Y}||^{2} = ||\boldsymbol{\theta} + \boldsymbol{Z}||^{2}.$ By the triangle inequality $||\boldsymbol{\theta} + \boldsymbol{Z}|| \leq ||\boldsymbol{\theta}|| + ||\boldsymbol{Z}||.$ By the reverse triangle inequality \\$|||\boldsymbol{\theta}+\boldsymbol{Z}|| - ||\boldsymbol{Z}||| \leq ||\boldsymbol{\theta}+\boldsymbol{Z}-\boldsymbol{Z}|| =||\boldsymbol{\theta}|| .$ This implies $||\boldsymbol{\theta} + \boldsymbol{Z}|| \geq ||\boldsymbol{\theta}|| - ||\boldsymbol{Z}||.$Therefore $||\boldsymbol{\theta}|| - ||\boldsymbol{Z}|| \leq ||\boldsymbol{\theta}+\boldsymbol{Z}|| \leq ||\boldsymbol{\theta}|| + ||\boldsymbol{Z}||.$\\ This implies $(||\boldsymbol{\theta}||-||\boldsymbol{Z}||)^{2}\leq ||\boldsymbol{\theta}+\boldsymbol{Z}||^{2} \leq (||\boldsymbol{\theta}|| + ||\boldsymbol{Z}||)^{2}. $ Consider the event \\$E_{\boldsymbol{\theta}} := \{ ||\boldsymbol{Z}|| \leq 2^{-1}||\boldsymbol{\theta}|| \}.$ On $E_{\boldsymbol{\theta}}, U \geq (||\boldsymbol{\theta}-2^{-1}||\boldsymbol{\theta}||)^{2} = 4^{-1}||\boldsymbol{\theta}||^{2}.$ Similarly, \\$U \leq (||\boldsymbol{\theta}|| + 2^{-1}||\boldsymbol{\theta}||)^{2} = 2.25||\boldsymbol{\theta}||^{2}.$ Now choose $b_{\boldsymbol{\theta} } = 2.25||\boldsymbol{\theta}||^{2}.$ Then \\$E_{\boldsymbol{\theta}} \subseteq \{ 4^{-1}||\boldsymbol{\theta}||^{2} \leq U \leq b_{\boldsymbol{\theta}} \}.$ Now pick any $\boldsymbol{\theta}$ such that $4^{-1}||\boldsymbol{\theta}||^{2} \geq u_{0}.$ That is, \\$||\boldsymbol{\theta}|| \geq 2 \sqrt{u_{0}}.$  This implies $E_{\boldsymbol{\theta}} \subseteq \{ u_{0} \leq U \leq b_{\boldsymbol{\theta}} \} = A_{\boldsymbol{\theta}}.$ Therefore, whenever \\$||\boldsymbol{\theta}|| \geq 2\sqrt{u_{0}}, \mathbb{P}_{\boldsymbol{\theta}}(A_{\boldsymbol{\theta}}) \geq \mathbb{P}(E_{\boldsymbol{\theta}}) = \mathbb{P}(||\boldsymbol{Z}|| \leq 2^{-1} ||\boldsymbol{\theta}|| ).$ Similarly, $\mathbb{P}_{\boldsymbol{\theta}}(A_{\boldsymbol{\theta}}^{c}) \leq \mathbb{P}(E_{\boldsymbol{\theta}}^{c})$. Then, 
\begin{equation*}
    \mathbb{E}_{\boldsymbol{\theta}}[B(U)] \geq \frac{p}{2C_{2}}(\frac{9}{4}||\boldsymbol{\theta}||^{2})^{\frac{1}{d}-1}\mathbb{P}(E_{\boldsymbol{\theta}}) - 2p \mathbb{P}(E_{\boldsymbol{\theta}}^{c}).
\end{equation*}
Let 
\begin{equation*}
    C_{*} := \frac{p}{2c_{2}} (\frac{9}{4})^{\frac{1}{d}-1}. 
\end{equation*}
So, 
\begin{equation*}
    \mathbb{E}_{\boldsymbol{\theta}}[B(U)] \geq C_{*} ||\boldsymbol{\theta}||^{\frac{2}{d}-2} \mathbb{P}(E_{\boldsymbol{\theta}})-2p\mathbb{P}(E_{\boldsymbol{\theta}}^{c}). 
\end{equation*}
Let $S = ||\boldsymbol{Z}||^{2}.$ Then $S \sim \chi_{p}^{2}.$ Observe \begin{equation*}
    \mathbb{P}(E_{\boldsymbol{\theta}}) = \mathbb{P}(S \leq (\frac{1}{2}||\boldsymbol{\theta}||)^{2}) = \mathbb{P}(S \leq \frac{1}{4}||\boldsymbol{\theta}||^{2}) = 1- \mathbb{P}(S \geq \frac{1}{4}||\boldsymbol{\theta}||^{2}). 
\end{equation*} Let $x = 0.25||\boldsymbol{\theta}||^{2}.$ By Markov's inequality, \begin{equation*}
    \mathbb{P}(S \geq x) = \mathbb{P}(e^{\lambda S} \geq e^{\lambda x}) \leq e^{-\lambda x}\mathbb{E}[e^{\lambda S}] = e^{-\lambda x} (1-2\lambda)^{-\frac{p}{2}}, \text{ for } \lambda < \frac{1}{2}. 
\end{equation*} Now optimize over $\lambda \in (0,2^{-1}).$ Let \begin{equation*}
    \phi(\lambda) := - \lambda x - \frac{p}{2} \log(1-2\lambda). 
\end{equation*} So $e^{-\lambda x}(1-2\lambda)^{-\frac{p}{2}} = e^{\phi(\lambda)}.$ Then, \begin{equation*}
    \phi'(\lambda) = -x - \frac{p}{2}(-\frac{2}{1-2\lambda}) = -x + \frac{p}{1-2\lambda}. 
\end{equation*} Then setting the first derivative equal to 0 gives \begin{equation*}
    -x+ \frac{p}{1-2\lambda} = 0 \implies 1- 2\lambda = \frac{p}{x} \implies \lambda^{*} = \frac{1}{2}(1-\frac{p}{x}) \in (0,\frac{1}{2}). 
\end{equation*} Furthermore \begin{equation*}
    \phi''(\lambda) = -p(1-2\lambda)^{-2}(-2) = \frac{2p}{(1-2\lambda)^{2}} >0. 
\end{equation*} Therefore $\lambda^{*}$ minimizes $\phi$. Therefore, \begin{align*}
    e^{-\lambda^{*}x} = \exp\{ -\frac{x}{2}(1-\frac{p}{x}) \} &= \exp \{ -\frac{x-p}{2} \} \\ 
    & \text{ and } \\ 
    (1-2\lambda^{*})^{-\frac{p}{2}} = (\frac{p}{x})^{-\frac{p}{2}} & = (\frac{x}{p})^{\frac{p}{2}}. 
\end{align*} Then \begin{equation*}
    \mathbb{P}(S \geq x) \leq \exp \{ -\frac{x-p}{2} \} (\frac{x}{p})^{\frac{p}{2}}. 
\end{equation*} Then, \begin{align*}
    \mathbb{P}(S \geq \frac{1}{4}||\boldsymbol{\theta}||^{2}) & \leq \exp \{ - \frac{\frac{1}{4}||\boldsymbol{\theta}||^{2}-p}{2} \} (\frac{\frac{1}{4}||\boldsymbol{\theta}||^{2}}{p})^{\frac{p}{2}} \\ 
    &= \exp\{ - \frac{||\boldsymbol{\theta}||^{2}-4p}{8} \} (\frac{||\boldsymbol{\theta}||^{2}}{4p})^{\frac{p}{2}}\\ 
    &= \exp \{ - \frac{||\boldsymbol{\theta}||^{2}}{8} + \frac{p}{2} \} (\frac{||\boldsymbol{\theta}||^{2}}{4p})^{\frac{p}{2}}\\
    &= \exp \{ - \frac{||\boldsymbol{\theta}||^{2}}{8} + \frac{p}{2} + \frac{p}{2}\log(\frac{||\boldsymbol{\theta}||^{2}}{4p}) \}. 
\end{align*} 
Choose a constant e.g. $1/16$. Since $\log(x)/x \rightarrow 0$ as $x \rightarrow \infty, \exists x_{0}>0$ such that $\log(x) \leq x/8, \forall x \geq x_{0}$. Now choose $||\boldsymbol{\theta}||$ large enough such that $||\boldsymbol{\theta}||^{2}/4p \geq x_{0}$. Then 
\begin{equation*}
    \log(\frac{||\boldsymbol{\theta}||^{2}}{4p}) \leq \frac{1}{8} \frac{||\boldsymbol{\theta}||^{2}}{4p} = \frac{||\boldsymbol{\theta}||^{2}}{32p}. 
\end{equation*}
So 
\begin{equation*}
    \frac{p}{2}\log(\frac{||\boldsymbol{\theta}||^{2}}{4p}) \leq \frac{p}{2}\frac{||\boldsymbol{\theta}||^{2}}{32p} = \frac{||\boldsymbol{\theta}||^{2}}{64}.
\end{equation*}
Then 
\begin{equation*}
    -\frac{||\boldsymbol{\theta}||^{2}}{8} + \frac{p}{2}+ \frac{p}{2}\log(\frac{||\boldsymbol{\theta}||^{2}}{4p}) \leq - \frac{||\boldsymbol{\theta}||^{2}}{8} + \frac{p}{2} + \frac{||\boldsymbol{\theta}||^{2}}{64}= - \frac{7}{64}||\boldsymbol{\theta}||^{2} + \frac{p}{2}. 
\end{equation*}
Then for large enough $\||\boldsymbol{\theta}||$ such that $p/2 \leq 3||\boldsymbol{\theta}||^{2}/64$ we obtain 
\begin{equation*}
    -\frac{7}{64} ||\boldsymbol{\theta}||^{2} + \frac{p}{2} \leq -\frac{7}{64} ||\boldsymbol{\theta}||^{2} + \frac{3}{64}||\boldsymbol{\theta}||^{2} \leq - \frac{1}{16} ||\boldsymbol{\theta}||^{2}. 
\end{equation*}
That is, for sufficiently large $||\boldsymbol{\theta}||$, 
\begin{equation*}
    \mathbb{P}(E_{\boldsymbol{\theta}}^{c}) \leq \exp \{- \frac{1}{16}||\boldsymbol{\theta}||^{2} \}. 
\end{equation*}
So, 
\begin{align*}
    \mathbb{E}_{\boldsymbol{\theta}}[B(U)] &\geq C_{*}||\boldsymbol{\theta}||^{\frac{2}{d}-2}\mathbb{P}(E_{\boldsymbol{\theta}}) -2p\mathbb{P}(E_{\boldsymbol{\theta}}^{c})\\
    &= C_{*} ||\boldsymbol{\theta}||^{\frac{2}{d}-2}(1-\mathbb{P}(E_{\boldsymbol{\theta}}^{c})) -2p\mathbb{P}(E_{\boldsymbol{\theta}}^{c})\\
    &\geq C_{*}||\boldsymbol{\theta}||^{\frac{2}{d}-2}(1-\exp\{ - \frac{1}{16}||\boldsymbol{\theta}||^{2} \})-2p\exp\{ - \frac{1}{16}||\boldsymbol{\theta}||^{2} \}.
\end{align*}
For sufficiently large $||\boldsymbol{\theta}||$,
\begin{equation*}
    \exp \{- \frac{1}{16}||\boldsymbol{\theta}||^{2} \} \leq \frac{1}{2}. 
\end{equation*}
So for sufficiently large $||\boldsymbol{\theta}||$, 
\begin{equation*}
    \mathbb{E}[B(U)] \geq \frac{C_{*}}{2}||\boldsymbol{\theta}||^{\frac{2}{d}-2} - 2p\exp\{ - \frac{1}{16}||\boldsymbol{\theta}||^{2}\}. 
\end{equation*}
We know 
\begin{equation*}
    \exp\{ \frac{||\boldsymbol{\theta}||^{2}}{16} \} \geq \frac{1}{2}(\frac{||\boldsymbol{\theta}||^{2}}{16})^{2} = \frac{||\boldsymbol{\theta}||^{4}}{512}. 
\end{equation*}
Hence 
\begin{equation*}
    \exp\{ - \frac{||\boldsymbol{\theta}||^{2}}{16}\} \leq \frac{512}{||\boldsymbol{\theta}||^{4}}.
\end{equation*}
So
\begin{equation*}
    2p\exp \{ - \frac{||\boldsymbol{\theta}||^{2}}{16} \} \leq \frac{1024p}{||\boldsymbol{\theta}||^{4}}. 
\end{equation*}
Now choose $||\boldsymbol{\theta}||$ large enough such that $||\boldsymbol{\theta}||^{2+\frac{2}{d}} \geq 4096p/C_{*}$. Then,  
\begin{equation*}
    1024p \leq \frac{C_{*}}{4}||\boldsymbol{\theta}||^{2+\frac{2}{d}}.
\end{equation*}
This gives 
\begin{equation*}
    \frac{1024p}{||\boldsymbol{\theta}||^{4}} \leq \frac{C_{*}}{4}||\boldsymbol{\theta}||^{\frac{2}{d}-2}. 
\end{equation*}
So for sufficiently large $||\boldsymbol{\theta}||$
\begin{equation*}
    2p\exp\{ - \frac{1}{16} ||\boldsymbol{\theta}||^{2} \} \leq \frac{C_{*}}{4}||\boldsymbol{\theta}||^{\frac{2}{d}-2}. 
\end{equation*}
Therefore 
\begin{align*}
    \mathbb{E}_{\boldsymbol{\theta}}[B(U)] &\geq \frac{C_{*}}{2}||\boldsymbol{\theta}||^{\frac{2}{d}-2} - \frac{C_{*}}{4}||\boldsymbol{\theta}||^{\frac{2}{d}-2}\\
    &= \frac{C_{*}}{4}||\boldsymbol{\theta}||^{\frac{2}{d}-2}\\
    &>0. 
\end{align*}
Therefore $\mathbb{E}_{\boldsymbol{\theta}}[B(u)] >0$ for sufficiently large $\boldsymbol{\theta}.$ Thus by Stein's identity \begin{equation*}
    R(\boldsymbol{\theta}, \boldsymbol{\delta}) = p + \mathbb{E}_{\boldsymbol{\theta}}[B(U)]>p \text{ for sufficiently large } ||\boldsymbol{\theta}||. 
\end{equation*} Therefore $\sup\limits_{\boldsymbol{\theta} \in \Theta} R(\boldsymbol{\theta}, \boldsymbol{\delta})>p.$ Thus $\boldsymbol{\delta}$ is not minimax. 
\end{proof}
\section{Proofs of Section 3}
\subsection{Proof of Theorem 3.1}
\begin{proof}
From the proof of Lemma 1.1
\begin{equation*}
    p_{d}^{lin}(\boldsymbol{\theta}|\boldsymbol{x}) = [\prod\limits_{\ell=1}^{d-1} \int_{0}^{\infty} dt_{\ell} \frac{t_{\ell}^{\frac{n_{\ell}}{2}-1} e^{-t_{\ell}}}{\Gamma(\frac{n_{\ell}}{2})}] (2^{d}\pi \kappa_{d}^{2})^{-\frac{p}{2}}(t_{1}\dots t_{d-1})^{-\frac{p}{2}} \exp \{- \frac{||\boldsymbol{\theta}||^{2}}{2^{d}\kappa_{d}^{2}t_{1}\dots t_{d-1}}\}. 
\end{equation*} Let $S:= t_{1}\dots t_{d-1}$ and $W:= 2^{d-1}\kappa_{d}^{2}S$. Then, $2^{d}\kappa_{d}^{2}S = 2W$. Therefore, \begin{equation*}
    \exp \{- \frac{||\boldsymbol{\theta}||^{2}}{2^{d}\kappa_{d}^{2}S} \}= \exp \{ -\frac{||\boldsymbol{\theta}||^{2}}{2W} \}. 
\end{equation*} Furthermore, 
\begin{equation*}
    (2\pi W)^{-\frac{p}{2}} = (2\pi 2^{d-1} \kappa_{d}^{2}S)^{-\frac{p}{2}} = (2^{d}\pi \kappa_{d}^{2})^{-\frac{p}{2}}S^{-\frac{p}{2}}.
\end{equation*} Thus, 
\begin{equation*}
    (2\pi W) ^{-\frac{p}{2}} \exp \{ - \frac{||\boldsymbol{\theta}||^{2}}{2W} \} = \phi_{p}(\boldsymbol{\theta}; 0 , WI_{p}),
\end{equation*}
where $\phi_{p}$ is the Gaussian probability density function. Then, \begin{equation*}
    p_{d}^{lin}(\boldsymbol{\theta}|\boldsymbol{x}) = [\prod\limits_{\ell = 1}^{d-1} \int\limits_{0}^{\infty} dt_{\ell} \frac{t_{\ell}^{\frac{n_{\ell}}{2}-1}e^{-t_{\ell}}}{\Gamma(\frac{n_{\ell}}{2})}] \phi_{p}(\boldsymbol{\theta}; 0 , WI_{p}). 
\end{equation*}
Observe that $f_{T_{\ell}} (t_{\ell}):= (\Gamma(\frac{n_{\ell}}{2}))^{-1}t_{\ell}^{\frac{n_{\ell}}{2}-1}e^{-t_{\ell}}$ is the probability density function of $T_{\ell} \sim \Gamma(\frac{n_{\ell}}{2},1)$. Therefore, 
\begin{align*}
    p_{d}^{lin}(\boldsymbol{\theta}|\boldsymbol{x}) &= \int\limits_{0}^{\infty}\dots\int\limits_{0}^{\infty} \phi_{p}(\boldsymbol{\theta}; \boldsymbol{0}_{p} , (2^{d-1}\kappa_{d}^{2}\prod\limits_{\ell=1}^{d-1}t_{\ell})I_{p}) \prod\limits_{\ell = 1}^{d-1} f_{T_{\ell}}(t_{\ell}) dt_{1}\dots dt_{d-1}\\
    & = \mathbb{E}[\phi_{p}(\boldsymbol{\theta}; \boldsymbol{0}_{p}, (2^{d-1}\kappa_{d}^{2} \prod\limits_{\ell = 1}^{d-1}T_{\ell})I_{p})] \\ 
    & = \mathbb{E}[\phi_{p}(\boldsymbol{\theta}; \boldsymbol{0}_{p} , WI_{p})]. 
\end{align*} That is, \begin{equation*}
    \boldsymbol{\theta}|W \sim N_{p}(0, WI_{p}). 
\end{equation*}
We know W is a function of S. Therefore, \begin{equation*}
    \boldsymbol{\theta}| S \sim N_{p}(0, (2^{d-1}\kappa_{d}^{2}S)I_{p}). 
\end{equation*}
Let $U:= 2^{d-1}||\boldsymbol{x}||^{2}\prod\limits_{\ell = 1}^{d}\sigma_{\ell}^{2} = 2^{d-1} \kappa_{d}^{2}$. Then,\begin{equation*}
    \boldsymbol{\theta}|U,S \sim N_{p}(\boldsymbol{0}_{p}, (US)I_{p}). 
\end{equation*} Then for fixed k, \begin{equation*}
    \boldsymbol{\theta} |U,S,k \sim N_{p}(\boldsymbol{0}_{p}, (US)I_{p}), S = \prod\limits_{\ell = 1}^{d-1} T_{\ell}, T_{\ell} \sim \Gamma(\frac{k_{\ell}}{2},1). 
\end{equation*}
Let $b:= \frac{p}{2} - 2>0 \iff p \geq 5$ which is the minimum dimension for the existence of proper minimax Bayes estimators. Now define $W \sim \mathrm{BetaPrime}(1,b)$ where $ h(w) = b(1+w)^{-(b+1)} = b(1+w)^{1- \frac{p}{2}}$. Now we define the hyperprior on $U$ conditional on $(S, k)$ by  observing $U = WS^{-1}$. Note $f_{W|S=s, k} = h(w)$. Now, $w = us, u>0$. This implies $dw = sdu$. Therefore, 
\begin{equation*}
    f_{U|S=s, k}(u|S, k) = f_{W|S=s, k} (w = us|S,k)|\frac{dw}{du}| = h(us)s, u>0. 
\end{equation*} Then, 
\begin{equation*}
    f_{U|k}(u) = \int\limits_{0}^{\infty} h(us)sf_{S|k}(s) ds \geq 0, u>0. 
\end{equation*} We also know \begin{equation*}
    \int\limits_{0}^{\infty} f_{U|k}(u) du = \int\limits_{0}^{\infty}\int\limits_{0}^{\infty}h(us)sf_{S|k}(s) dsdu = \int\limits_{0}^{\infty}(\int\limits_{0}^{\infty} h(w) dw) f_{S|k}(s) ds = 1, 
\end{equation*}
as $S$ is a product of Gamma distributed random variables which have a proper probability density function. Therefore, $f_{U|k}$ is a proper density.  Observe that, $US = WS^{-1}S = W$. This implies $US \sim \mathrm{BetaPrime}(1,b)$. Hence, $\boldsymbol{\theta} | U, S, k \sim N_{p}(\boldsymbol{0}_{p}, (US)I_{p}) = N_{p}(\boldsymbol{0}_{p}, WI_{p})$. Then, $\pi_{k}(\boldsymbol{\theta}) = \int\limits_{0}^{\infty} \phi_{p}(\boldsymbol{\theta}; \boldsymbol{0}_{p}, wI_{p})h(w)dw$, where the right hand side does not depend on k. So, 
\begin{align*}
    \pi(\boldsymbol{\theta}) &= \sum\limits_{k}\omega_{k}\pi_{k}(\boldsymbol{\theta}) \\ 
    & = \sum\limits_{k}\omega_{k} \int\limits_{0}^{\infty}\phi_{p}(\boldsymbol{\theta} ; \boldsymbol{0}_{p} , wI_{p})h(w)dw \\ 
    & = \int\limits_{0}^{\infty} \phi_{p}(\boldsymbol{\theta}; \boldsymbol{0}_{p}, wI_{p})h(w) dw \sum\limits_{k} \omega_{k}. 
\end{align*}
For a given $\ell$, 
\begin{equation*}
    \sum\limits_{k_{\ell}=1}^{n_{\ell}} \begin{pmatrix}
        n_{\ell}\\ 
        k_{\ell}
    \end{pmatrix} = (\sum\limits_{k_{\ell = 0}}^{n_{\ell}} \begin{pmatrix}
        n_{\ell}\\ 
        k_{\ell}
    \end{pmatrix}) - \begin{pmatrix}
        n_{\ell}\\
        0
    \end{pmatrix} = 2^{n_{\ell}}-1. 
\end{equation*} Hence, 
\begin{equation*}
    \sum\limits_{k_{1} = 1}^{n_{1}} \dots \sum\limits_{k_{d-1} = 1}^{n_{d-1}}\prod\limits_{\ell = 1}^{d-1}\begin{pmatrix}
        n_{\ell}\\ 
        k_{\ell}
    \end{pmatrix} = \prod\limits_{\ell =1}^{d-1}(\sum\limits_{k_{\ell} = 1}^{n_{\ell}} \begin{pmatrix}
        n_{\ell}\\ 
        k_{\ell}
    \end{pmatrix}) = \prod\limits_{\ell = 1}^{d-1} (2^{n_{\ell}}-1).
\end{equation*} Therefore, 
\begin{equation*}
    N= \sum\limits_{k} \omega _{k}=\frac{\prod\limits_{\ell = 1}^{d-1}(2^{n_{\ell}}-1)}{2^{n_{1}+\dots+n_{d-1}}}<\infty. 
\end{equation*} So, 
\begin{equation*}
    \pi(\boldsymbol{\theta}) = N\int\limits_{0}^{\infty} \phi_{p}(\boldsymbol{\theta}; \boldsymbol{0}_{p}, wI_{p}) h(w)dw. 
\end{equation*} Then, 
\begin{align*}
    m(\boldsymbol{y}) &= \int\limits_{\mathbb{R}^{p}} \phi_{p}(\boldsymbol{y}-\boldsymbol{\theta}; \boldsymbol{0}_{p}, I_{p}) \pi (\boldsymbol{\theta}) d\boldsymbol{\theta} \\ 
    &= N\int\limits_{\mathbb{R}^{p}} \phi_{p}(\boldsymbol{y}- \boldsymbol{\theta}; \boldsymbol{0}_{p}, I_{p})[\int\limits_{0}^{\infty}\phi_{p}(\boldsymbol{\theta}; \boldsymbol{0}_{p}, wI_{p})h(w)dw]d\boldsymbol{\theta}\\ 
    & = N\int\limits_{0}^{\infty}[\int\limits_{\mathbb{R}^{p}} \phi_{p}(\boldsymbol{y}- \boldsymbol{\theta}; \boldsymbol{0}_{p}, I_{p}) \phi_{p}(\boldsymbol{\theta}; \boldsymbol{0}_{p} , wI_{p}) d\theta]h(w) dw, \text{ by Tonelli's theorem,}\\ 
    & = N\int\limits_{0}^{\infty} \phi_{p}(\boldsymbol{y}; \boldsymbol{0}_{p}, (1+w)I_{p})h(w)dw, \text{ by convolution of Gaussians}\\ 
    & = N\int\limits_{0}^{\infty}(2\pi(1+w))^{- \frac{p}{2}} \exp \{ - \frac{||\boldsymbol{y}||^{2}}{2(1+w)} \} h(w) dw\\ 
    & = Nb(2\pi)^{-\frac{p}{2}}\int\limits_{0}^{\infty} (1+w)^{1-p} \exp \{ - \frac{||\boldsymbol{y}||^{2}}{2(1+w)} \} dw
\end{align*}
Let $u = (1+w)^{-1}$, so $w = (1-u)u^{-1}, dw = - u^{-2}$, and $(1+w)^{1-p} = u^{p-1}$, Then,  \begin{align*}
    m(\boldsymbol{y}) &= Nb(2\pi)^{- \frac{p}{2}} \int\limits_{1}^{0} u^{p-1}\exp \{ - \frac{u||\boldsymbol{y}||^{2}}{2} \} (-u^{-2})du\\ 
    & =  Nb(2\pi)^{- \frac{p}{2}} \int\limits_{0}^{1} u^{p-3} \exp \{ - \frac{u||\boldsymbol{y}||^{2}}{2} \}du. 
\end{align*} 
Let $C:= Nb(2\pi)^{- \frac{p}{2}}>0$ and $I_{0}(\lambda) := \int\limits_{0}^{1} u^{p-3}e^{-\lambda u}du$ with $\lambda = \frac{1}{2}r^{2}$ and $r = ||\boldsymbol{y}||$. Then $m(r) = CI_{0}(\lambda)$. Similarly, let $I_{1}(\lambda) := \int\limits_{0}^{1} u^{p-2} e^{-\lambda u}du$ and $I_{2}(\lambda) = \int\limits_{0}^{1} u^{p-1}e^{-\lambda u}du$. Then, \begin{equation*}
    \frac{d}{d\lambda} I_{0}(\lambda) = - I_{1}(\lambda) \text{ and } \frac{d}{d\lambda} I_{1}(\lambda) = - I_{2}(\lambda), 
\end{equation*}
since $\frac{d}{d\lambda} (u^{a} e^{-\lambda u}) = - u^{a+1}e^{-\lambda u}$ and $|u^{a+1}e^{-\lambda u}| \leq u^{a+1} \in L^{1}(0,1), a\geq2$ and using the Dominated Convergence Theorem. Then \begin{equation*}
    m'(r) = C \frac{dI_{0}}{d\lambda} \frac{d \lambda}{dr} = - CrI_{1}(\lambda)
\end{equation*} and 
\begin{align*}
    m''(r) &= - CI_{1}(\lambda) - Cr \frac{d}{dr}\{I_{1}(\lambda) \}\\
    &= -CI_{1}(\lambda) +Cr^{2}I_{2}(\lambda). 
\end{align*}
For radial $m$ in $\mathbb{R}^{p}$ we know 
\begin{align*}
    \Delta m(r) &= m''(r) + \frac{p-1}{r}m'(r)\\ 
    & = -CI_{1}(\lambda) + Cr^{2}I_{2}(\lambda) + \frac{p-1}{r}(-Cr I_{1}(\lambda))\\ 
    & = - CI_{1}(\lambda) + Cr^{2}I_{2}(\lambda) - C(p-1)I_{1}(\lambda
    ) \\ 
    & = C[-I_{1}(\lambda) + r^{2}I_{2}(\lambda) - pI_{1}(\lambda) + I_{1}(\lambda)]\\ 
    & = C[r^{2}I_{2}(\lambda) - pI_{1}(\lambda)]. 
\end{align*} Observe \begin{equation*}
    \nabla (\sqrt{m}) = \frac{1}{2\sqrt{m}}\nabla m. 
\end{equation*} Then, 
\begin{equation*}
    \Delta (\sqrt{m}) = \nabla \cdot (\frac{1}{2\sqrt{m}}\nabla m) = \frac{1}{2\sqrt{m}}\Delta m + \nabla(\frac{1}{2\sqrt{m}})\cdot\nabla m. 
\end{equation*} We also have \begin{equation*}
    \nabla (\frac{1}{2\sqrt{m}}) = - \frac{1}{4}m^{-\frac{3}{2}}\nabla m. 
\end{equation*} Then 
\begin{equation*}
    \Delta (\sqrt{m}) = \frac{1}{2\sqrt{m}}\Delta m - \frac{1}{4}m^{-\frac{3}{2}} ||\nabla m ||^{2} \iff 4m^{\frac{3}{2}} \Delta (\sqrt{m}) = 2m\Delta m - ||\nabla m||^{2}. 
\end{equation*} Thus, \begin{equation*}
    \Delta (\sqrt{m}) \leq 0 \iff 2m\Delta m - ||\nabla m||^{2} \leq 0 \iff 2m\Delta m\leq ||\nabla m||^{2}. 
\end{equation*}
Let $m(\boldsymbol{y}) = \tilde{m}(r)$. By the chain rule, \begin{equation*}
    \frac{\partial m}{\partial y_{i}} = \tilde{m}'(r)\frac{\partial r}{\partial y_{i}} = \tilde{m}'(r)\frac{y_{i}}{r}, r>0. 
\end{equation*}So, \begin{equation*}
    \nabla m(\boldsymbol{y}) = \tilde{m}'(r) (\frac{y_{1}}{r}, \dots, \frac{y_{p}}{r}) = \tilde{m}'(r)\frac{\boldsymbol{y}}{r}. 
\end{equation*} Then, \begin{equation*}
    ||\nabla m(\boldsymbol{y})||^{2} = (\tilde{m}'(r))^{2}||\frac{\boldsymbol{y}}{r}||^{2}= (\tilde{m}'(r))^{2}\frac{||\boldsymbol{y}||^{2}}{r^{2}} = (\tilde{m}'(r))^{2} = (m'(r))^{2}. 
\end{equation*} Therefore, 
\begin{equation*}
    \Delta \sqrt{m} \leq 0 \iff 2m\Delta m \leq (m')^{2}. 
\end{equation*} Thus, 
\begin{equation*}
    \Delta \sqrt{m} \leq 0 \iff 2C^{2}I_{0}(\lambda)(r^{2}I_{2}(\lambda)- pI_{1}(\lambda)) \leq C^{2} r^{2}I_{1}(\lambda)^{2}. 
\end{equation*} That is, 
\begin{equation*}
    2I_{0}(\lambda) (r^{2}I_{2}(\lambda) - pI_{1}(\lambda)) \leq r^{2}I_{1}(\lambda)^{2}. 
\end{equation*} Recall $r^{2} = 2\lambda$ which gives \begin{equation*}
    2I_{0}(\lambda) (2\lambda I_{2}(\lambda) - pI_{1} (\lambda)) \leq 2 \lambda I_{1}(\lambda)^{2}. 
\end{equation*} This becomes \begin{equation*}
    I_{0}(\lambda)(2\lambda I_{2}(\lambda) - pI_{1}(\lambda)) \leq \lambda I_{1}(\lambda)^{2}. 
\end{equation*}
Let $a:= p-2$ so $p = a+2$ with $a \geq 3$. Also let $t = \lambda u$ i.e. $u = t \lambda^{-1}$ with $du = \lambda^{-1} dt$.  Therefore, \begin{align*}
    I_{0}(\lambda) &= \int\limits_{0}^{1} u^{p-2-1}e^{-\lambda u}du\\ &= \lambda^{-1} \int\limits_{0}^{\lambda} (\frac{t}{\lambda})^{a-1}e^{-t} dt \\ 
    & = \lambda^{-a} \int\limits_{0}^{\lambda} t^{a-1}e^{-t}dt\\ 
    & = \lambda^{-a} \gamma(a, \lambda), \text{ where } \gamma(.,.) \text{ is the lower incomplete gamma function. }
\end{align*} Similarly, $I_{1}(\lambda) = \lambda ^{-(a+1)}\gamma(a+1, \lambda)$ and $I_{2}(\lambda) = \lambda ^{-(a+2)}\gamma (a+2, \lambda)$. Thus, \begin{align*}
    & \lambda^{-a}\gamma(a, \lambda) (2\lambda \lambda ^{-(a+2)} \gamma (a+2, \lambda) - p \lambda^{-(a+1)}\gamma(a+1, \lambda)) \leq \lambda \lambda^{-2(a+1)}\gamma(a+1, \lambda)^{2}\\
    \iff &  \lambda^{-a}\gamma(a,\lambda) (2\lambda^{-(a+1)}\gamma(a+2, \lambda) - p\lambda^{-(a+1)}\gamma(a+1, \lambda)) \leq \lambda^{-2a-1}\gamma(a+1, \lambda)^{2} \\ 
    \iff &\lambda^{-a}\lambda^{-(a+1)} \gamma(a,\lambda)(2\gamma(a+2, \lambda )) - p \gamma(a+1, \lambda))\leq \lambda^{-2a-1} \gamma(a+1, \lambda)^{2} \\ 
    \iff &\lambda^{-2a-1}\gamma(a, \lambda) (2\gamma(a+2, \lambda) - p\gamma(a+1, \lambda)) \leq \lambda^{-2a-1}\gamma(a+1, \lambda)^{2} \\ 
    \iff & \gamma(a, \lambda)(2\gamma(a+2, \lambda) - (a+2)\gamma(a+1, \lambda)) \leq \gamma(a+1, \lambda)^{2}. 
\end{align*} Note that \begin{equation*}
    \gamma(a+1, \lambda) = a \gamma(a, \lambda) - \lambda^{a}e^{-\lambda} \text{ and } \gamma(a+2, \lambda) = (a+1) \gamma(a+1, \lambda)- \lambda^{a+1}e^{- \lambda}. 
\end{equation*} 
Let $G:= \gamma(a, \lambda)$ and $E:= \lambda^{a} e^{-\lambda}$. Then, $\gamma(a+1, \lambda) = aG - E$. Furthermore, $\gamma (a+2, \lambda) = (a+1)(aG - E) - \lambda E = a(a+1)G - (a+1+\lambda)E$. Therefore, 
\begin{align*}
    & G(2[a(a+1)G - (a+1+\lambda)E] - (a+2)[aG-E]) \leq  (aG - E)^{2}\\ 
    \iff & G(2a(a+1)G - 2(a+1+\lambda)E - (a+2)aG + (a+2)E) \leq (aG - E)^{2}\\ 
    \iff & G([2a(a+1) - (a+2)a] G + [a+2 - 2(a+1 + \lambda)]E) \leq (aG  - E)^{2} \\ 
    \iff & G([2(a+1) - (a+2)] aG + [a+2 - 2a - 2- 2\lambda]E) \leq (aG - E)^{2} \\ 
    \iff & G([2a + 2 - a- 2]aG + [-a-2\lambda]E) \leq (aG - E)^{2}\\ 
    \iff & G(a^{2}G - (a+2\lambda)E) \leq a^{2}G^{2}-2aGE + E^{2} \\ 
    \iff &a^{2}G^{2}  - aGE - 2\lambda GE \leq a^{2}G^{2} - 2aGE + E^{2}\\ 
    \iff & 0 \leq -aGE + 2\lambda GE + E^{2}\\ 
    \iff & 0 \leq E^{2} + (2\lambda - a)GE\\ 
    \iff & 0 \leq E(E + (2\lambda - a)G)\\ 
    \iff & 0 \leq E + (2\lambda -a) G \\ 
    \iff & 0 \leq \lambda^{a}e^{-\lambda} + (2\lambda - a)\gamma(a, \lambda) =:F
\end{align*}
Note that $G'(\lambda) = \lambda^{a-1}e^{- \lambda}$ and $E'(\lambda) = a\lambda^{a-1} e^{-\lambda} - \lambda^{a}e^{-\lambda}= \lambda^{a-1} e^{-\lambda}(a-\lambda)$. Therefore, \begin{align*}
    F'(\lambda) &= \lambda^{a-1} e^{-\lambda}(a-\lambda) + 2\gamma(a,\lambda) + (2\lambda - a) \lambda^{a-1}e^{- \lambda}\\ 
    & = \lambda^{a-1}e^{- \lambda}[a-\lambda + 2\lambda - a] + 2\gamma(a, \lambda) \\ 
    & = \lambda^{a}e^{-\lambda} + 2\gamma(a, \lambda)\\
    & >0,  \forall \lambda > 0.
\end{align*}
Recall, $\gamma(a, \lambda) = \int\limits_{0}^{\lambda} t^{a-1}e^{-t} dt$. For $0 \leq t \leq \lambda, e^{-t}$ is decreasing. Therefore, $e^{-\lambda} \leq e^{-t} \leq 1$. Then, $e^{-\lambda} \int\limits_{0}^{\lambda} t^{a-1}dt  \leq \gamma(a,\lambda ) \leq \int\limits_{0}^{\lambda} t^{a-1} dt$. Observe that $\int\limits_{0}^{\lambda} t^{a-1} dt = a^{-1} \lambda ^{a}$. Therefore, \begin{equation*}
    e^{-\lambda} \frac{\lambda^{a}}{a} \leq \gamma(a, \lambda) \leq \frac{\lambda^{a}}{a} \iff e^{-\lambda} \leq \frac{\gamma(a,\lambda)}{\frac{\lambda ^{a}}{a}} \leq 1.
\end{equation*} This implies, \begin{equation*}
    \lim\limits_{\lambda \rightarrow 0 } \frac{\gamma(a, \lambda)}{\frac{\lambda^{a}}{a}} = 1 \text{ i.e }\lim\limits_{\lambda \rightarrow 0 } \frac{\gamma(a, \lambda)}{\lambda^{a}} = \frac{1}{a}. 
\end{equation*} Note that \begin{equation*}
    \frac{F(\lambda)}{\lambda^{a}} = e^{-\lambda} + (2\lambda - a) \frac{\gamma(a, \lambda)}{\lambda^{a}}. 
\end{equation*}Then, 
\begin{equation*}
    \lim\limits_{\lambda \rightarrow 0 } \frac{F(\lambda)}{\lambda^{a}} = 1- \frac{a}{a} = 0. 
\end{equation*} 
Since $F'(\lambda) > 0,  \forall \lambda > 0$ and $\lim\limits_{\lambda \rightarrow 0} F(\lambda)  = 0$ we have that $F(\lambda) > 0, \forall \lambda > 0$. Therefore, $\Delta \sqrt{m(\boldsymbol{y})} \leq 0, \forall \boldsymbol{y} \neq \boldsymbol{0}_{p}$.  Recall, \begin{equation*}
    m(\boldsymbol{y}) = N\int\limits_{0}^{\infty}\phi_{p}(\boldsymbol{y}; \boldsymbol{0}_{p}, (1+w)I_{p})h(w)dw. 
\end{equation*} For $j = 1,\dots, p$, \begin{align*}
    \frac{\partial }{\partial y_{j}} \phi_{p}(\boldsymbol{y}; \boldsymbol{0}_{p} , (1+w)I_{p}) &= \frac{\partial }{\partial y_{j}} \{ (2\pi)^{- \frac{p}{2}} (1+w)^{-\frac{p}{2}} \exp \{ - \frac{||\boldsymbol{y}||^{2}}{2(1+w)} \} \} \\ 
    &= (2\pi)^{-\frac{p}{2}} (1+w)^{-\frac{p}{2}} \frac{\partial}{\partial y_{j}} \exp \{ - \frac{||\boldsymbol{y}||^{2}}{2(1+w)} \}\\ 
    & = - \frac{y_{j}}{(1+w)} \phi_{p}(\boldsymbol{y}; \boldsymbol{0}_{p} , (1+w)I_{p}). 
\end{align*}
Consider compact set $K \subset \mathbb{R}^{p}$ and $M:= \sup\limits_{\boldsymbol{y} \in K}\max\limits_{1 \leq i \leq p} |y_{i}| < \infty$. Since \begin{equation*}
    \exp \{ - \frac{||\boldsymbol{y}||^{2}}{2(1+w)} \} \leq 1 \text{ we know } \phi_{p}(\boldsymbol{y}; \boldsymbol{0}_{p} , (1+w)I_{p}) \leq (2\pi)^{-\frac{p}{2}} (1+w)^{- \frac{p}{2}}. 
\end{equation*} Hence, $\forall \boldsymbol{y} \in K$ \begin{align*}
    |\frac{\partial}{\partial y_{j}} \phi_{p}(\boldsymbol{y}; \boldsymbol{0}_{p} , (1+w)I_{p})h(w)| &= \frac{|y_{j}|}{1+w} \phi_{p}(\boldsymbol{y}; 0, (1+w)I_{p})h(w)\\ 
    &\leq M(2\pi) ^{-\frac{p}{2}} (1+w)^{-\frac{p}{2}-1}h(w).
\end{align*} Let $g(w):= M(2\pi)^{-\frac{p}{2}} (1+w)^{-\frac{p}{2}-1}h(w)$. Then $g(w) \geq 0$ and integrable since $(1+w)^{-\frac{p}{2}-1} \leq 1$ and $h$ is a probability density function. That is, \begin{equation*}
    \int\limits_{0}^{\infty} g(w) dw \leq M(2\pi)^{-\frac{p}{2}} \int\limits_{0}^{\infty} h(w) dw = M(2\pi)^{-\frac{p}{2}} < \infty.  
\end{equation*}
Let $\boldsymbol{e}_{j}$ be the j-th unit vector and fix $\boldsymbol{y} \in \mathbb{R}^{p}$. For $t \neq 0$, \begin{equation*}
    \frac{m(\boldsymbol{y}+t\boldsymbol{e}_{j})-m(\boldsymbol{y})}{t} = N\int\limits_{0}^{\infty} \frac{f(\boldsymbol{y}+t\boldsymbol{e}_{j}, w) - f(\boldsymbol{y},w)}{t}dw
\end{equation*} where 
\begin{equation*}
    f(\boldsymbol{y},w):= \phi_{p}(\boldsymbol{y}; \boldsymbol{0}_{p} , (1+w)I_{p})h(w).
\end{equation*}
Let $\psi(s) = f(\boldsymbol{y}+s\boldsymbol{e}_{j}, w).$ By the mean value theorem \begin{equation*}
    \frac{f(\boldsymbol{y}+t\boldsymbol{e}_{j} , w)-f(\boldsymbol{y}, w)}{t} = \frac{\partial}{\partial y_{j}} f(\boldsymbol{y} + \eta t \boldsymbol{e}_{j}, w), \text{ for some } \eta \in (0,1). 
\end{equation*} Now, choose $K$ containing $\{ \boldsymbol{y} + s\boldsymbol{e}_{j}: |s| \leq 1 \}$. Then, $\forall |t| \leq 1$ we have $\boldsymbol{y}+\eta t \boldsymbol{e}_{j} \in K$. Then, \begin{equation*}
    |\frac{f(\boldsymbol{y}+t\boldsymbol{e}_{j}, w)-f(\boldsymbol{y},w)}{t}| = |\frac{\partial}{\partial y_{j}} f(\boldsymbol{y} + \eta t \boldsymbol{e}_{j}, w)| \leq g(w). 
\end{equation*} For fixed $w$, as $t \rightarrow 0$ \begin{equation*}
    \frac{f(\boldsymbol{y} + t\boldsymbol{e}_{j} , w) - f(\boldsymbol{y},w)}{t} \rightarrow \frac{\partial}{\partial y_{j}} f(\boldsymbol{y},w). 
\end{equation*} Then by the dominated convergence theorem, 
\begin{equation*}
    \lim\limits_{t \rightarrow 0} \frac{m(\boldsymbol{y} + t\boldsymbol{e}_{j}) - m(\boldsymbol{y})}{t} = N\int\limits_{0}^{\infty} \lim\limits_{t \rightarrow 0 } \frac{f(\boldsymbol{y} + t\boldsymbol{e}_{j},w) - f(\boldsymbol{y},w)}{t} dw = N\int\limits_{0}^{\infty} \frac{\partial}{\partial y_{j}} f(\boldsymbol{y},w)dw. 
\end{equation*} That is, \begin{align*}
    \frac{\partial}{\partial y_{j}} m(\boldsymbol{y}) &= N\int\limits_{0}^{\infty} \frac{\partial }{\partial y_{j}} f(\boldsymbol{y},w) dw \\ 
    & = N\int\limits_{0}^{\infty} \frac{\partial }{\partial y_{j}} \phi_{p}(\boldsymbol{y}; \boldsymbol{0}_{p} , (1+w)I_{p})h(w) dw \\ 
    & = - Ny_{j} \int\limits_{0}^{\infty} \frac{1}{1+w} \phi_{p}(\boldsymbol{y}; \boldsymbol{0}_{p} , (1+w)I_{p})h(w) dw. 
\end{align*}
Note that for $i \neq j$, 
\begin{equation*}
    \frac{\partial ^{2}}{\partial y_{i} \partial y_{j}} \phi_{p}(\boldsymbol{y}; \boldsymbol{0}_{p}, (1+w)I_{p}) = \frac{y_{i}y_{j}}{(1+w)^{2}}\phi_{p}(\boldsymbol{y}; \boldsymbol{0}_{p} , (1+w)I_{p}), 
\end{equation*}
and if $i = j$, 
\begin{align*}
    \frac{\partial ^{2}}{\partial y_{i}^{2}} \phi_{p} &= - \frac{\partial}{\partial y_{j}} \{ \frac{y_{i}}{(1+w)}\phi_{p}(\boldsymbol{y}; 0, (1+w)I_{p}) \}\\ 
    &= - \frac{1}{(1+w)} \frac{\partial }{\partial y_{i}} \{ y_{i} \phi_{p}(\boldsymbol{y}; \boldsymbol{0}_{p}, (1+w)I_{p}) \} \\ 
    & = - \frac{1}{(1+w)}[\phi_{p}(\boldsymbol{y}; \boldsymbol{0}_{p} , (1+w)I_{p}) - \frac{y_{i}^{2}}{(1+w)}\phi_{p}(\boldsymbol{y}; \boldsymbol{0}_{p} , (1+w)I_{p})] \\ 
    & = (\frac{y_{i}^{2}}{(1+w)^{2}}-\frac{1}{1+w}) \phi_{p}(\boldsymbol{y}; \boldsymbol{0}_{p}, (1+w)I_{p}). 
\end{align*}
For $i \neq j$, 
\begin{align*}
    |\frac{\partial ^{2}}{\partial y_{i}\partial y_{j}} \phi_{p}(\boldsymbol{y}; \boldsymbol{0}_{p} , (1+w)I_{p})h(w)| & = \frac{|y_{i}y_{j}|}{(1+w)^{2}} \phi_{p}(\boldsymbol{y}; \boldsymbol{0}_{p}, (1+w)I_{p})h(w) \\
    & \leq \frac{M^{2}}{(1+w)^{2}} \phi_{p}(\boldsymbol{y}; \boldsymbol{0}_{p}, (1+w)I_{p})h(w) \\ 
    & \leq \frac{M^{2}}{(1+w)^{2}} (2\pi)^{-\frac{p}{2}}(1+w)^{-\frac{p}{2}}h(w) \\ 
    &= M^{2}(2\pi)^{-\frac{p}{2}} (1+w)^{-\frac{p}{2}-2}h(w) \\ 
    &\leq M^{2} (2\pi)^{-\frac{p}{2}} (1+w)^{-\frac{p}{2}-1}h(w), 
\end{align*}
and for $i = j$, 
\begin{align*}
    |\frac{\partial ^{2}}{\partial y_{i}^{2}}\phi_{p}(\boldsymbol{y}; \boldsymbol{0}_{p}, (1+w)I_{p})h(w)| &= |\frac{y_{i}^{2}}{(1+w)^{2}}-\frac{1}{1+w}| \phi_{p}(\boldsymbol{y}; \boldsymbol{0}_{p} , (1+w)I_{p})h(w) \\ 
    & \leq (\frac{y_{i}^{2}}{(1+w)^{2}} + \frac{1}{1+w}) \phi_{p}(\boldsymbol{y}; \boldsymbol{0}_{p}, (1+w)I_{p})h(w) \\ 
    &\leq (\frac{M^{2}}{(1+w)^{2}} + \frac{1}{1+w}) (2\pi)^{-\frac{p}{2}}(1+w)^{-\frac{p}{2}}h(w) \\ 
    & = (2\pi)^{-\frac{p}{2}} [M^{2}(1+w)^{-\frac{p}{2}-2} + (1+w)^{-\frac{p}{2}-1}]h(w) \\ 
    &\leq (2\pi)^{-\frac{p}{2}} (M^{2} + 1) (1+w)^{-\frac{p}{2} - 1} h(w). 
\end{align*}
Then, 
\begin{equation*}
    |\frac{\partial ^{2}}{\partial y_{i}\partial y_{j}} \phi_{p} (\boldsymbol{y};\boldsymbol{ 0}_{p} , (1+w)I_{p})h(w)| \leq C_{k}(2\pi)^{-\frac{p}{2}} (1+w)^{- \frac{p}{2}-1} h(w), C_{k} = (1+M^{2}). 
\end{equation*}
Therefore, by a similar dominated convergence theorem argument, $m \in C^{2}(\mathbb{R}^{p})$. Since $m \in C^{2}(\mathbb{R}^{p}), m(\boldsymbol{y}) > 0$, and $\phi(t) = \sqrt{t}$ is $C^{\infty}((0, \infty))$, so $\sqrt{m} \in C^{2}(\mathbb{R}^{p})$. Then $\Delta \sqrt{m}$ is continuous on $\mathbb{R}^{p}$. We know $\Delta \sqrt{m(\boldsymbol{y})} \leq 0, \forall \boldsymbol{y} \neq \boldsymbol{0}_{p}$. Now let $\boldsymbol{y}_{n} \rightarrow \boldsymbol{0}_{p}$ with $\boldsymbol{y}_{n} \neq \boldsymbol{0}_{p}$. By continuity, $\Delta \sqrt{m}(0) = \lim\limits_{n \rightarrow \infty} \Delta \sqrt{m}(\boldsymbol{y}_{n}) \leq 0$. That is $\Delta \sqrt{m}(\boldsymbol{y}) \leq 0 , \forall \boldsymbol{y} \in \mathbb{R}^{p}$. 
Now, 
\begin{equation*}
    \nabla m(\boldsymbol{y})  = - N\boldsymbol{y} \int\limits_{0}^{\infty} \frac{1}{1+w} \phi_{p}(\boldsymbol{y}; 0 , (1+w)I_{p}) h(w) dw. 
\end{equation*} Now, \begin{equation*}
    \pi(w|y) : = \frac{\phi_{p}(\boldsymbol{y}; \boldsymbol{0}_{p} , (1+w)I_{p})h(w)}{\int\limits_{0}^{\infty} \phi_{p}(\boldsymbol{y}; \boldsymbol{0}_{p}, (1+w)I_{p})h(w)dw}. 
\end{equation*} Then, 
\begin{equation*}
    \frac{\nabla m(\boldsymbol{y})}{m(\boldsymbol{y})} = - \boldsymbol{y} \int\limits_{0}^{\infty} \frac{1}{1+w} \pi (w|\boldsymbol{y}) dw = - \boldsymbol{y} \mathbb{E}[\frac{1}{1+W}|\boldsymbol{Y}= \boldsymbol{y}]. 
\end{equation*} Since \begin{equation*}
    0 < \frac{1}{1+w} \leq 1 \text{ we have } 0 < \mathbb{E}[\frac{1}{1+W}|\boldsymbol{Y}=\boldsymbol{y}] \leq 1. 
\end{equation*} Therefore, 
\begin{equation*}
    ||\frac{\nabla m(\boldsymbol{y})}{m(\boldsymbol{y})}|| = ||\boldsymbol{y}|| \mathbb{E}[\frac{1}{1+W} | \boldsymbol{Y} = \boldsymbol{y}] \leq ||\boldsymbol{y}||. 
\end{equation*} This implies, \begin{equation*}
    ||\frac{\nabla m(\boldsymbol{y})}{m(\boldsymbol{y})}||^{2} \leq ||\boldsymbol{y}||^{2}. 
\end{equation*} Hence, 
\begin{align*}
    \mathbb{E}_{\boldsymbol{\theta}} [||\frac{\nabla m(\boldsymbol{Y})}{m(\boldsymbol{Y})}||^{2}] &\leq \mathbb{E}_{\boldsymbol{\theta}}[||\boldsymbol{Y}||^{2}]\\ &= \mathbb{E}_{\boldsymbol{\theta}}[\sum\limits_{j=1}^{p}Y_{j}^{2}]\\
    &= \sum\limits_{j=1}^{p}\mathbb{E}_{\theta}[Y_{j}^{2}]\\
    &= \sum\limits_{j=1}^{p} (1+\theta_{j}^{2}) \\
    &= (p + ||\boldsymbol{\theta}||^{2})\\
    &< \infty. 
\end{align*} 
Then by Theorem 3.1 in \cite{fourdrinier2018shrinkage} \begin{equation*}
    \delta_{\pi}(\boldsymbol{Y})  = \boldsymbol{Y} + \frac{\nabla m(\boldsymbol{Y})}{m(\boldsymbol{Y})} \text{ is minimax. }
\end{equation*}
\end{proof}
\subsection{Proof of Corollary 3.2}
\begin{proof}
    From the proof of Theorem 3.1 the prior density can be written as \begin{equation*}
        \pi(\boldsymbol{\theta}) = N\int\limits_{0}^{\infty} \phi_{p}(\boldsymbol{\theta}; \boldsymbol{0}_{p}, wI_{p}) h(w) dw, \text{ where } h(w) = b(1+w)^{-(b+1)}, b = \frac{p}{2}-2. 
    \end{equation*}
That is \begin{align*}
    \pi(\boldsymbol{\theta}) &= N\int\limits_{0}^{\infty} (2\pi)^{-\frac{p}{2}} \det(wI_{p})^{-\frac{1}{2}}\exp \{ - \frac{1}{2}\boldsymbol{\theta}^{T}(wI_{p})^{-1}\boldsymbol{\theta} \}h(w) dw\\ 
    & = N(2\pi)^{-\frac{p}{2}} \int\limits_{0}^{\infty} w^{-\frac{p}{2}} \exp \{ - \frac{||\boldsymbol{\theta}||^{2}}{2w}\}h(w) dw. 
\end{align*} This is exactly of the form of expression (3.4) in \cite{fourdrinier2018shrinkage}. Note that \begin{align*}
    \lim\limits_{w \rightarrow \infty} 
    \frac{h(w)}{bw^{-(b+1)}} &= \lim\limits_{w \rightarrow \infty} \frac{b(1+w)^{-(b+1)}}{bw^{-(b+1)}} \\
    & = \lim\limits_{w \rightarrow \infty} \frac{(1+w)^{-(b+1)}}{w^{-(b+1)}}\\ 
    & = \lim\limits_{w \rightarrow \infty} (\frac{1+w}{w})^{-(b+1)}\\ 
    & = \lim\limits_{w \rightarrow \infty}(\frac{1}{w}+1)^{-(b+1)}\\
    &=1. 
\end{align*}
That is $h(w) \sim bw^{-(b+1)}$. In view of Theorem 3.15 from \cite{fourdrinier2018shrinkage} $a = -(b+1)= - (\frac{p}{2} - 2+ 1) = 1- \frac{p}{2}<-1 \iff 4 <p$ which is true for our case of $p \geq 5$. Therefore, by Theorem 3.15 of \cite{fourdrinier2018shrinkage} the induced decision rule is admissible. 
\end{proof}
\subsection{Proof of Corollary 3.3}
\begin{proof}
Throughout this proof, condition on a fixed subnetwork index vector $\boldsymbol{k}$, and suppress this conditioning in the density and Mellin-transform notation. Let $b= \frac{p}{2}-2, \alpha_{\ell} = \frac{k_{\ell}}{2}, \ell = 1, \dots, d-1, c:=2^{d-1}||\boldsymbol{x}||^{2}$. Similarly, set $R :=\prod\limits_{\ell=1}^{d} \sigma_{\ell}^{2}$ and $Z:= \prod\limits_{\ell = 1}^{d-1}T_{\ell}$. Then $W = cRZ$. The Mellin transform of $W$ is given by 
    \begin{align*}
        M_{W}(s) &= \mathbb{E}[W^{s-1}]\\
        &= \int\limits_{0}^{\infty} w^{s-1} b(1+w)^{-(b+1)} dw\\ 
        &= b\int\limits_{0}^{\infty} w^{s-1}(1+w)^{-(b+1)}dw\\
        &= bB(s, b-s+1)\\
        & = b \frac{\Gamma(s)\Gamma(b-s+1)}{\Gamma(s+b-s+1)}\\
        &= b\frac{\Gamma(s) \Gamma(b-s+1)}{\Gamma(1+b)}\\
        &= b\frac{\Gamma(s)\Gamma(b-s+1)}{b\Gamma(b)}\\ 
        & = \frac{\Gamma(s)\Gamma(b-s+1)}{\Gamma(b)}. 
    \end{align*}
Similarly, 
\begin{align*}
    M_{T_{\ell}}(s) &= \mathbb{E}[T_{\ell}^{s-1}]\\ 
    & = \int\limits_{0}^{\infty} \frac{1}{\Gamma(\alpha_{\ell})} t_{\ell}^{\alpha_{\ell}-1} e^{-t_{\ell}}t_{\ell}^{s-1}dt_{\ell}\\ 
    & = \frac{1}{\Gamma(\alpha_{\ell})} \int\limits_{0}^{\infty} t_{\ell}^{\alpha_{\ell}+s-2} e^{-t_{\ell}}dt_{\ell}\\
    &= \frac{\Gamma(\alpha_{\ell} + s-1)}{\Gamma(\alpha_{\ell})}. 
\end{align*}
This implies, 
\begin{align*}
    M_{Z}(s) &= \mathbb{E}[Z^{s-1}]\\ 
    & = \mathbb{E}[(\prod\limits_{\ell=1}^{d-1} T_{\ell})^{s-1}]\\
    &= \mathbb{E}[\prod\limits_{\ell=1}^{d-1} T_{\ell}^{s-1}]\\
    &= \prod\limits_{\ell=1}^{d-1} \mathbb{E}[T_{\ell}^{s-1}], \text{by independence,}\\
    &= \prod\limits_{\ell=1}^{d-1} \frac{\Gamma(\alpha_{\ell}+s-1)}{\Gamma(\alpha_{\ell})}. 
\end{align*}
Let $U:= cR = W/Z$. By the construction in Theorem 3.1, $W$ and $Z$ are independent conditional on $k$. Therefore, 
\begin{align*}
    M_{U}(s) &= \mathbb{E}[W^{s-1}Z^{1-s}]\\
    &= M_{W}(s)M_{Z}(2-s)\\
    &= \frac{\Gamma(s)\Gamma(b+1-s)}{\Gamma(b)}\prod\limits_{\ell=1}^{d-1} \frac{\Gamma(\alpha_{\ell}+1-s)}{\Gamma(\alpha_{\ell})}. 
\end{align*}
Since $U=cR$, this implies $f_{R}(r) = cf_{U}(cr)$. Let $c'$ denote the real part of a line in the complex plane, which is to be specified. Observe, 
\begin{align*}
    f_{U}(u) &= \frac{1}{2\pi i} \int\limits_{c'-i\infty}^{c'+i\infty} M_{U}(s) u^{-s}ds\\ 
    & = \frac{1}{2\pi i} \int\limits_{c'-i\infty}^{c'+i\infty}\frac{\Gamma(s)\Gamma(b+1-s)}{\Gamma(b)}[\prod\limits_{\ell=1}^{d-1} \frac{\Gamma(\alpha_{\ell}+1-s)}{\Gamma(\alpha_{\ell})}]u^{-s}ds\\
    &= \frac{1}{2\pi i \Gamma(b) \prod\limits_{\ell =1}^{d-1} \Gamma(\alpha_{\ell})}\int\limits_{c^{'}-i\infty}^{c'+i\infty}\Gamma(s)\Gamma(b+1-s) \prod\limits_{\ell = 1}^{d-1} \Gamma(\alpha_{\ell} + 1-s)u^{-s}ds\\
    &= \frac{1}{\Gamma(b)\prod\limits_{\ell=1}^{d-1}\Gamma(\alpha_{\ell})}G_{d,1}^{1,d}(u|\begin{matrix}
        -b, -\alpha_{1}, \dots,-\alpha_{d-1}\\
        0
    \end{matrix}), u >0. 
\end{align*}
Therefore, 
\begin{equation*}
    f_{R}(r)  = \frac{c}{\Gamma(b) \prod\limits_{\ell = 1}^{d-1} \Gamma(\alpha_{\ell})}G_{d,1}^{1,d}(cr|\begin{matrix}
        -b , - \alpha_{1}, \dots, -\alpha_{d-1}\\
        0
    \end{matrix}), r>0. 
\end{equation*}
Now we need to check the assumptions of the Meijer-G function are satisfied. The index conditions $0 \leq 1 \leq 1$ and $0 \leq d \leq d$ are satisfied. All upper parameters, $-b, -\alpha_{1}, \dots, - \alpha_{d-1}$, are negative, while the sole lower parameter is zero; hence none of the relevant parameter differences is a positive integer. Choose 
\begin{equation*}
    0 < c' < \min \{ b+1, \alpha_{1}+1, \dots, \alpha_{d-1}+1 \}. 
\end{equation*}
This places the poles of $\Gamma(s)$ to the left of the contour and those of $\Gamma(b+1-s)$ and $\Gamma(\alpha_{\ell} + 1-s)$ to the right. The Mellin-Barnes integral converges for $u>0$. Since $c>0$ and $r>0$, its argument $cr$ is also positive. 
\end{proof}
\section{Proofs of Section 4}
\subsection{Proof of Theorem 4.1}
\begin{proof}
Let $b = (p/2)-2$ and $a = p-2$. From Section 3, $\boldsymbol{\theta}|W = w \sim N_{p}(\boldsymbol{0}_{p}, wI_{p})$. Now, let $W=v_{0}B$ with $v_{0}>0$. Then conditioning on $B=s$ implies $W=v_{0}s$. Therefore, 
\begin{equation*}
    \boldsymbol{\theta}|B=s,k \sim N_{p}(\boldsymbol{0}_{p}, v_{0}sI_{p}),
\end{equation*}
where 
\begin{equation*}
    h(s) = b(1+s)^{-(b+1)} = b(1+s)^{-\frac{a}{2}}, s>0. 
\end{equation*}
Therefore, 
\begin{align*}
    \pi_{v_{0},k}(\boldsymbol{\theta}) &= \int\limits_{0}^{\infty} \pi_{v_{0},k}(\boldsymbol{\theta}|B=s)h(s) ds\\
    &= \int\limits_{0}^{\infty} \phi_{p}(\boldsymbol{\theta}; \boldsymbol{0}_{p}, v_{0}sI_{p}) h(s) ds.
\end{align*}
Then, 
\begin{align*}
\pi_{v_{0}}(\boldsymbol{\theta}) &= \sum\limits_{k} w_{k} \pi_{v_{0},k}(\boldsymbol{\theta}) \\
&= \sum\limits_{k} w_{k} \int\limits_{0}^{\infty} \phi_{p}(\boldsymbol{\theta}; \boldsymbol{0}_{p}, v_{0}sI_{p}) h(s) ds\\
&= N\int\limits_{0}^{\infty} \phi_{p}(\boldsymbol{\theta}; \boldsymbol{0}_{p}, v_{0}sI_{p})h(s) ds, \qquad N=\sum\limits_{k}w_{k}. 
\end{align*}
Now, for $v>0$, consider 
\begin{align*}
    m_{\pi_{v_{0}}}(\boldsymbol{z};v) &= \int\limits_{\mathbb{R}^{p}} \phi_{p}(\boldsymbol{z}-\boldsymbol{\theta};\boldsymbol{0}_{p}, vI_{p})\pi_{v_{0}}(\boldsymbol{\theta})d\boldsymbol{\theta}\\
    &= \int\limits_{\mathbb{R}^{p}} \phi_{p}(\boldsymbol{z}-\boldsymbol{\theta}; \boldsymbol{0}_{p}, vI_{p})[N\int\limits_{0}^{\infty}\phi_{p}(\boldsymbol{\theta}, \boldsymbol{0}_{p}, v_{0}sI_{p})h(s) ds]d\boldsymbol{\theta}.
\end{align*}
By Tonelli's theorem, 
\begin{align*}
    m_{\pi_{v_{0}}}(\boldsymbol{z};v) &= N\int\limits_{0}^{\infty}[\int\limits_{\mathbb{R}^{p}} \phi_{p}(\boldsymbol{z} - \boldsymbol{\theta};\boldsymbol{0}_{p}, vI_{p})\phi_{p}(\boldsymbol{\theta}; \boldsymbol{0}_{p}, v_{0}sI_{p})d\boldsymbol{\theta}]h(s) ds. \\
    &= N \int\limits_{0}^{\infty} \phi_{p}(\boldsymbol{z}; \boldsymbol{0}_{p}, (v+v_{0}s)I_{p})h(s) ds, \text{by convolution of Gaussians,}\\
    &= Nb(2\pi)^{-\frac{p}{2}} \int\limits_{0}^{\infty} (v+v_{0}s)^{-\frac{p}{2}}\exp\{ - \frac{||\boldsymbol{z}||^{2}}{2(v+v_{0}s)} \} (1+s)^{-\frac{a}{2}}ds. 
\end{align*}
Now we will prove square root superharmonicity for every $0<v\leq v_{0}$. Define 
\begin{equation*}
    \rho := \frac{v}{v_{0}} \in (0,1] \text{ and } \boldsymbol{y}:= \frac{\boldsymbol{z}}{\sqrt{v}}. 
\end{equation*}
Now let $s = \rho t$ so $ds = \rho dt$. Since $v_{0}\rho = v$, 
\begin{equation*}
    v+v_{0}s = v+v_{0}\rho t = v+vt = v(1+t) \text{ and } 1+s = 1+ \rho t. 
\end{equation*}
Therefore, 
\begin{equation*}
    m_{\pi_{v_{0}}}(\boldsymbol{z}; v) = Nb\rho(2\pi v)^{-\frac{p}{2}} \int\limits_{0}^{\infty} (1+t) ^{-\frac{p}{2}}(1+\rho t)^{-\frac{a}{2}} \exp \{ - \frac{||\boldsymbol{y}||^{2}}{2(1+t)} \}dt.
\end{equation*}
Now let $u = 1/(1+t)$. Then $t = (1-u)/u$ and $dt = -u^{-2}du$. Moreover, 
\begin{equation*}
    1+t = \frac{1}{u} \text{ and } 1+\rho t = 1+ \rho(\frac{1-u}{u}) = \frac{\rho + (1-\rho)u}{u}. 
\end{equation*}
Therefore, 
\begin{equation*}
    (1+t)^{-\frac{p}{2}} = u^{\frac{p}{2}} \text{ and } (1+\rho t)^{-\frac{a}{2}} = u^{\frac{a}{2}} \{ \rho + (1-\rho)u \}^{-\frac{a}{2}}. 
\end{equation*}
Note that 
\begin{equation*}
    \frac{p}{2} + \frac{a}{2} -2 = p -3 = a-1. 
\end{equation*}
Thus, 
\begin{align*}
    m_{\pi_{v_{0}}}(\boldsymbol{z}; v) &= Nb\rho (2\pi v)^{-\frac{p}{2}}\int\limits_{1}^{0}u^{\frac{p}{2}}u^{\frac{a}{2}} \{ \rho + (1-\rho)u \}^{-\frac{a}{2}}\exp\{ - \frac{||\boldsymbol{y}||^{2}u}{2} \}(-u^{-2})du\\
    &= Nb\rho (2\pi v)^{-\frac{p}{2}} \int\limits_{0}^{1}u^{a-1} \{ \rho + (1-\rho)u \}^{-\frac{a}{2}}\exp \{ - \frac{||\boldsymbol{y}||^{2}u}{2} \}du\\
    &= C_{\rho,v}G_{\rho}(\boldsymbol{y}), 
\end{align*}
where 
\begin{align*}
    C_{\rho,v} &:= Nb\rho(2\pi v)^{-\frac{p}{2}}>0, \\
    G_{\rho}(\boldsymbol{y}) &: = \int\limits_{0}^{1} u^{a-1} q_{\rho}(u) \exp\{ - \frac{||\boldsymbol{y}||^{2}u}{2} \}du, \text{ and }\\
    q_{\rho}(u)&:= \{ \rho + (1-\rho)u  \}^{-\frac{a}{2}}. 
\end{align*}
Since $y = \boldsymbol{z}/\sqrt{v}$, we have 
\begin{align*}
    \frac{\partial y_{i}}{\partial z_{j}} = \frac{1}{\sqrt{v}} \delta_{ij}. 
\end{align*}
Furthermore, 
\begin{equation*}
    \sqrt{m_{\pi_{v_{0}}}(\boldsymbol{z}; v)} = C_{\rho,v}^{\frac{1}{2}}\sqrt{G_{\rho}(\boldsymbol{y})}.
\end{equation*}
Then for every $i$, 
\begin{align*}
    \frac{\partial ^{2}}{\partial z_{i}^{2}}\sqrt{m_{\pi_{v_{0}}}(\boldsymbol{z};v)} &= C_{\rho,v}^{\frac{1}{2}}(\frac{\partial y_{i}}{\partial z_{i}})^{2}\frac{\partial^{2}}{\partial y_{i}^{2}}\sqrt{G_{\rho}(\boldsymbol{y})}\\
    &= C_{\rho,v}^{\frac{1}{2}} v^{-1} \frac{\partial^{2}}{\partial y_{i}^{2}}\sqrt{G_{\rho}(\boldsymbol{y})}. 
\end{align*}
Then summing over $i$ yields 
\begin{equation*}
    \Delta_{\boldsymbol{z}} \sqrt{m_{\pi_{v_{0}}}(\boldsymbol{z};v)} = C_{\rho,v}^{\frac{1}{2}}v^{-1} \Delta_{\boldsymbol{y}} \sqrt{G_{\rho}(\boldsymbol{y})}. 
\end{equation*}
Because $C_{\rho,v}^{\frac{1}{2}} v^{-1}>0$, it is sufficient to prove $\Delta_{\boldsymbol{y}} \sqrt{G_{\rho}(\boldsymbol{y})} \leq 0$. Let $r := ||\boldsymbol{y}||$ and $\lambda := r^{2}/2$. Define for $k = 0,1,2$, 
\begin{equation*}
    I_{k}(\lambda) := \int\limits_{0}^{1} u^{a-1+k} q_{\rho}(u) e^{-\lambda u} du. 
\end{equation*}
Then $G_{\rho}(\boldsymbol{y}) = I_{0}(\lambda)$. For fixed $\rho \in (0,1], q_{\rho}(u) \leq \rho^{-\frac{a}{2}}$. For $\lambda \geq 0$ the first two $\lambda$-derivatives of the integrand defining $I_{0}$ are therefore bounded in absolute value by $\rho^{-\frac{a}{2}}u^{a}$ and $\rho^{-\frac{a}{2}}u^{a+1}$, respectively. These functions are integrable on $(0,1)$. Therefore, differentiation under the integral sign is justified by the same dominated convergence theorem argument as in the proof of Theorem 3.1. Differentiation under the integral sign gives $I_{0}^{'}(\lambda) = - I_{1}(\lambda)$, $I_{1}^{'}(\lambda) = - I_{2}(\lambda)$. Since $d\lambda / dr = r$, 
\begin{equation*}
    G_{\rho}^{'}(r) = I_{0}^{'}(\lambda) \frac{d\lambda}{dr} = -rI_{1}(\lambda)
\end{equation*}
and 
\begin{equation*}
    G_{\rho}^{''}(r) = - \frac{d}{dr}\{ rI_{1}(\lambda) \} = - I_{1}(\lambda) -rI_{1}^{'}(\lambda)\frac{d\lambda}{dr} = -I_{1}(\lambda) + r^{2}I_{2}(\lambda). 
\end{equation*}
Because $G_{\rho}$ is radial
\begin{align*}
    \Delta_{\boldsymbol{y}}G_{\rho}(\boldsymbol{y}) &= G_{\rho}^{''}(r) + \frac{p-1}{r}G_{\rho}^{'}(r)\\
    &= - I_{1} + r^{2}I_{2}- (p-1)I_{1}\\
    &= r^{2}I_{2} - pI_{1}. 
\end{align*}
Furthermore, 
\begin{equation*}
    ||\nabla _{\boldsymbol{y}} G_{\rho}(\boldsymbol{y}) ||^{2} = \{ G_{\rho}^{'}(r) \}^{2} = r^{2}I_{1}^{2}. 
\end{equation*}
Note that for any positive twice differentiable function $G$, 
\begin{equation*}
    \Delta \sqrt{G} = \frac{1}{2\sqrt{G}} \Delta G - \frac{1}{4G^{\frac{3}{2}}} ||\nabla G||^{2}. 
\end{equation*}
Therefore, 
\begin{equation*}
    \Delta_{\boldsymbol{y}}\sqrt{G_{\rho}(\boldsymbol{y})} \leq 0 \iff 2 G_{\rho}(\boldsymbol{y}) \Delta_{\boldsymbol{y}} G_{\rho}(y) - ||\nabla_{\boldsymbol{y}}G_{\rho}(\boldsymbol{y})||^{2} \leq 0. 
\end{equation*}
Substituting in the expressions above give 
\begin{equation*}
    2I_{0}(r^{2}I_{2} - pI_{1})- r^{2}I_{1}^{2} \leq 0. 
\end{equation*}
Since $r^{2} = 2\lambda$, this is equivalent to 
\begin{equation*}
    D_{\rho}(\lambda):= I_{0}(2\lambda I_{2} - pI_{1}) - \lambda I_{1}^{2} \leq 0. 
\end{equation*}
Now we need to prove this inequality holds. Define 
\begin{equation*}
    d_{\rho}(u) := \rho + (1-\rho)u \text{ and } c_{\rho} : = \frac{a(1-\rho)}{2} \geq 0. 
\end{equation*}
Since 
\begin{equation*}
    q_{\rho}(u) = d_{\rho}(u) ^{-\frac{a}{2}} \text{, we have } q^{'}_{\rho}(u) = - c_{\rho} \frac{q_{\rho}(u)}{d_{\rho}(u)}. 
\end{equation*}
Define 
\begin{equation*}
    K_{0}:= \int\limits_{0}^{1} \frac{u^{a}q_{\rho}(u)}{d_{\rho}(u)}e^{-\lambda u} du \text{ and } K_{1} := \int\limits_{0}^{1} \frac{u^{a+1}q_{\rho}(u)}{d_{\rho}(u)}e^{-\lambda u} du. 
\end{equation*}
Observe, 
\begin{align*}
    \frac{d}{du} \{ u^{a}q_{\rho}(u) e^{-\lambda u} \} &= au^{a-1}q_{\rho}(u)e^{-\lambda u} + u^{a} q_{\rho}^{'}(u)e^{-\lambda u} - \lambda u^{a}q_{\rho}(u) e^{-\lambda u}\\
    &= au^{a-1} q_{\rho}(u)e^{-\lambda u} - c_{\rho}\frac{u^{a}q_{\rho}(u)}{d_{\rho}(u)}e^{-\lambda u} - \lambda u^{a} q_{\rho}(u) e^{-\lambda u}. 
\end{align*}
Then by the fundamental theorem of calculus, 
\begin{equation*}
    \int\limits_{0}^{1} \frac{d}{du} \{ u^{a}q_{\rho}(u)e^{-\lambda u} \} du = [u^{a}q_{\rho}(u)e^{-\lambda u }]_{0}^{1} = e^{-\lambda}. 
\end{equation*}
Therefore, 
\begin{equation*}
    e^{-\lambda } = aI_{0} -c_{\rho}K_{0} - \lambda I_{1}. 
\end{equation*}
That is 
\begin{equation}
    \lambda I_{1} = aI_{0} - c_{\rho}K_{0} - e^{-\lambda}.
\end{equation}
Similarly, 
\begin{equation*}
    \frac{d}{du} \{ u^{a+1} q_{\rho}(u)e^{-\lambda u} \} = (a+1)u^{a}q_{\rho}(u) e^{-\lambda u} - c_{\rho}\frac{u^{a+1}q_{\rho}(u)}{d_{\rho}(u)} e^{-\lambda u} - \lambda u^{a+1}q_{\rho}(u) e^{-\lambda u}. 
\end{equation*}
Then,
\begin{equation*}
    \int\limits_{0}^{1} \frac{d}{du} \{ u^{a+1} q_{\rho}(u)e^{-\lambda u} \} du = e^{-\lambda}. 
\end{equation*}
This gives 
\begin{equation*}
    e^{-\lambda } = (a+1)I_{1}- c_{\rho}K_{1} - \lambda I_{2}. 
\end{equation*}
Hence, 
\begin{equation}
    \lambda I_{2} = (a+1) I_{1} - c_{\rho}K_{1} - e^{-\lambda}.
\end{equation}
Since $p= a+2$, we obtain 
\begin{align*}
    2\lambda I_{2} - pI_{1} &= 2\{ (a+1)I_{1} - c_{\rho}K_{1} - e^{-\lambda}  \} - (a+2)I_{1}\\
    &= aI_{1} - 2 c_{\rho}K_{1} - 2e^{-\lambda}. 
\end{align*}
Therefore 
\begin{equation*}
    D_{\rho}(\lambda) = I_{0}(aI_{1} - 2c_{\rho}K_{1}-2e^{-\lambda}) - \lambda I_{1}^{2}
\end{equation*}
Using (S4.1) gives, 
\begin{equation*}
    \lambda I_{1}^{2} = I_{1}\{ aI_{0} - c_{\rho}K_{0} - e^{-\lambda} \}. 
\end{equation*}
Then substitution yields 
\begin{align}
    D_{\rho}(\lambda) & = aI_{0}I_{1} - 2c_{\rho}I_{0}K_{1} - 2e^{-\lambda}I_{0} - aI_{0}I_{1} + c_{\rho}I_{1}K_{0}+ e^{-\lambda } I_{1}\\
    &= c_{\rho} \{ I_{1}K_{0} - 2I_{0}K_{1} \} + e^{-\lambda} \{ I_{1} - 2I_{0} \}. 
\end{align}
Now consider the probability measure 
\begin{equation*}
    \mu_{\lambda, \rho} (du) = \frac{u^{a-1}q_{\rho}(u)e^{-\lambda u}}{I_{0}(\lambda)}du,\qquad 0 < u<1. 
\end{equation*}
Using this measure gives 
\begin{equation*}
    \frac{I_{1}}{I_{0}} = \mathbb{E}_{\mu_{\lambda,\rho}}[U]\text{, }\frac{K_{0}}{I_{0}} =  \mathbb{E}_{\mu_{\lambda, \rho}}[\frac{U}{d_{\rho}(U)}]\text{, and} \frac{K_{1}}{I_{0}} = \mathbb{E}_{\mu_{\lambda, \rho}}[\frac{U^{2}}{d_{\rho}(U)}]. 
\end{equation*}
Note that the functions 
\begin{equation*}
    u \mapsto u \text{ and } u \mapsto \frac{u}{d_{\rho}(u)}
\end{equation*}
are both nondecreasing on $[0,1]$, since 
\begin{equation*}
    \frac{d}{du} \{ \frac{u}{d_{\rho}(u)} \} = \frac{\rho}{d_{\rho}(u)^{2}} \geq 0.
\end{equation*}
By Chebyshev's association inequality $(\mathbb{E}[f(u)g(u)] \geq \mathbb{E}[f(u)]\mathbb{E}[g(u)]), $
\begin{equation*}
    \mathbb{E}_{\mu_{\lambda, \rho}}[U\frac{U}{d_{\rho}(U)}] \geq \mathbb{E}_{\mu_{\lambda, \rho}} [U] \mathbb{E}_{\mu_{\lambda, \rho}}[\frac{U}{d_{\rho}(U)}]. 
\end{equation*}
Substituting the preceding expectation identities and multiplying by $I_{0}(\lambda)^{2}$ gives,
\begin{equation*}
    \frac{K_{1}}{I_{0}} \geq \frac{I_{1}}{I_{0}} \frac{K_{0}}{I_{0}} \iff I_{1}K_{0} \leq I_{0}K_{1}. 
\end{equation*}
Hence, 
\begin{equation*}
    I_{1}K_{0} - 2I_{0} K_{1} \leq - I_{0} K_{1} \leq 0. 
\end{equation*}
Since $0 < u < 1$, 
\begin{equation*}
    I_{1} = \int\limits_{0}^{1}uu^{a-1} q_{\rho}(u)e^{-\lambda u}du \leq I_{0}. 
\end{equation*}
Therefore, 
\begin{equation*}
    I_{1} -2I_{0} \leq -I_{0} < 0. 
\end{equation*}
Since $c_{\rho} \geq 0$ and $e^{-\lambda} > 0$, (S4.4) implies $D_{\rho}(\lambda) <0$. Therefore, $\Delta_{\boldsymbol{y}}\sqrt{G_{\rho}(\boldsymbol{y})} \leq 0$. The case $\boldsymbol{y} = \boldsymbol{0}_{p}$ follows by the same $C^{2}$-regularity and continuity argument as in the proof of Theorem 3.1, using $v+v_{0}s \geq v > 0$ in the Gaussian mixture representation. Thus, $\Delta_{\boldsymbol{z}}\sqrt{m_{\pi_{v_{0}}}(\boldsymbol{z};v)} \leq 0$ for every $\boldsymbol{z} \in \mathbb{R}^{p}$ and every $0 < v \leq v_{0}$. Now define 
\begin{equation*}
    v_{w} := \frac{v_{x}v_{y}}{v_{x}+v_{y}}. 
\end{equation*}
Since $0 < v_{w} < v_{x} \leq v_{0}$ we have $[v_{w}, v_{x}] \subseteq (0,v_{0}]$. Thus $\Delta_{\boldsymbol{z}} \sqrt{m_{\pi_{v_{0}}}(\boldsymbol{z};v)} \leq 0$ for every $\boldsymbol{z} \in \mathbb{R}^{p}$ and every $v \in [v_{w}, v_{x}]$. Finally, for fixed $\boldsymbol{z}, \boldsymbol{\theta} \in \mathbb{R}^{p}$ and $v>0$, 
\begin{equation*}
    0 < \phi_{p} (\boldsymbol{z} - \boldsymbol{\theta}; \boldsymbol{0}_{p}, vI_{p}) \leq (2\pi v)^{-\frac{p}{2}}. 
\end{equation*}
Because $\pi_{v_{0}}$ is a finite measure, 
\begin{align*}
    0 < m_{\pi_{v_{0}}}(\boldsymbol{z};v) &= \int \limits_{\mathbb{R}^{p}} \phi_{p}(\boldsymbol{z}- \boldsymbol{\theta}; \boldsymbol{0}_{p}, vI_{p}) \pi_{v_{0}}(d\boldsymbol{\theta})\\
    &\leq (2\pi v)^{-\frac{p}{2}} \int\limits_{\mathbb{R}^{p}} \pi_{v_{0}}(d\boldsymbol{\theta}) < \infty. 
\end{align*}
Therefore, $m_{\pi_{v_{0}}}(\boldsymbol{z};v_{x})$ is finite for every $\boldsymbol{z}$, and its square root is superharmonic for every $v \in [v_{w}, v_{x}]$. Theorem 1(ii) of \cite{george2006improved} therefore implies that $\hat{p}_{\pi_{v_{0}}}(\boldsymbol{y}|\boldsymbol{x})$ is minimax under KL-loss. 
\end{proof}
\subsection{Proof of Corollary 4.2}
\begin{proof}
Consider $A_{0} = \{ g: \mathbb{R}^{p} \rightarrow \mathbb{R}, g\geq 0 \text{ and } \int g(\boldsymbol{y}) d\boldsymbol{y} = 1 \}$. We know $\pi_{v_{0}}(\boldsymbol{\theta}) >0, \forall \boldsymbol{\theta} \in \mathbb{R}^{p}$. The predictive Bayes rule in view of \cite{brown2008admissible} can be written as \begin{equation*}
    \hat{p}_{M}(\boldsymbol{y}|\boldsymbol{x}) = \frac{\int p(\boldsymbol{x}|\boldsymbol{\theta})p(\boldsymbol{y}|\boldsymbol{\theta})M(d\boldsymbol{\theta})}{\int p(\boldsymbol{x}|\boldsymbol{\theta}) M(d\boldsymbol{\theta})} = \int p(\boldsymbol{y}|\boldsymbol{\theta}) M(d\boldsymbol{\theta} |\boldsymbol{x}).
\end{equation*}
We know $p(\boldsymbol{x}|\boldsymbol{\theta}) \leq (2\pi v_{x})^{-\frac{p}{2}}$ and $p(\boldsymbol{y}|\boldsymbol{\theta}) \leq (2\pi v_{y})^{-\frac{p}{2}}$. Then, \begin{equation*}
    p(\boldsymbol{x}|\boldsymbol{\theta})p(\boldsymbol{y}|\boldsymbol{\theta}) \leq (2\pi v_{x})^{-\frac{p}{2}}(2\pi v_{y})^{-\frac{p}{2}}. 
\end{equation*}
This implies \begin{equation*}
    \int p(\boldsymbol{x}|\boldsymbol{\theta}) p(\boldsymbol{y}|\boldsymbol{\theta})\pi_{v_{0}}(\boldsymbol{\theta})d\boldsymbol{\theta}  \leq (2\pi v_{x})^{-\frac{p}{2}}(2\pi v_{y})^{-\frac{p}{2}} \int \pi_{v_{0}}(\boldsymbol{\theta})d\boldsymbol{\theta} = N(2\pi v_{x})^{-\frac{p}{2}}(2\pi v_{y})^{-\frac{p}{2}}<\infty. 
\end{equation*}
Similarly, \begin{equation*}
    \int p(\boldsymbol{x}|\boldsymbol{\theta}) \pi_{v_{0}}(\boldsymbol{\theta}) d\boldsymbol{\theta} \leq (2\pi v_{x})^{-\frac{p}{2}} \int \pi_{v_{0}}(\boldsymbol{\theta}) d\boldsymbol{\theta} = N(2\pi v_{x})^{-\frac{p}{2}} < \infty. 
\end{equation*}
From \cite{aitchison1975goodness} $\hat{p}_{\pi_{v_{0}}}(\boldsymbol{y}|\boldsymbol{x})$ minimizes the average KL risk with respect to the finite prior measure $\pi_{v_{0}}(d\boldsymbol{\theta})$ 
\begin{equation*}
    B_{KL}(\pi_{v_{0}}, \hat{p}) = \int R_{KL}(\boldsymbol{\theta}, \hat{p}) \pi_{v_{0}}(d\boldsymbol{\theta}). 
\end{equation*}
Assume for contradiction there exists some predictive rule $\tilde{p}$ which dominates $\hat{p}_{\pi_{v_{0}}}$. Then \begin{equation*}
    R_{KL}(\boldsymbol{\theta}, \tilde{p}) \leq R_{KL}(\boldsymbol{\theta}, \hat{p}_{\pi_{v_{0}}}), \forall \boldsymbol{\theta}, 
\end{equation*}
with strict inequality for at least one $\boldsymbol{\theta}$. Integrating against $\pi_{v_{0}}(\boldsymbol{\theta}) d\boldsymbol{\theta}$ gives \begin{equation*}
    B_{KL}(\pi_{v_{0}}, \tilde{p}) \leq B_{KL}(\pi_{v_{0}},\hat{p}_{\pi_{v_{0}}}). 
\end{equation*}
But we also know that $\hat{p}_{\pi_{v_{0}}}$ minimizes $B_{KL}(\pi_{v_{0}}, \cdot)$. Hence $B_{KL}(\pi_{v_{0}}, \tilde{p}) = B_{KL}(\pi_{v_{0}}, \hat{p}_{\pi_{v_{0}}})$. Then 
\begin{align*}
    0 &= B_{KL}(\pi_{v_{0}}, \tilde{p}) - B_{KL}(\pi_{v_{0}}, \hat{p}_{\pi_{v_{0}}})\\
    & = \int \int\int p(\boldsymbol{y}|\boldsymbol{\theta}) \log(\frac{p(\boldsymbol{y}|\boldsymbol{\theta})}{\tilde{p}(\boldsymbol{y}|\boldsymbol{x})})d\boldsymbol{y}p(\boldsymbol{x}|\boldsymbol{\theta})d\boldsymbol{x} \pi_{v_{0}}(d\boldsymbol{\theta})\\ &- \int\int\int p(\boldsymbol{y}|\boldsymbol{\theta})\log(\frac{p(\boldsymbol{y}|\boldsymbol{\theta})}{\hat{p}_{\pi_{v_{0}}}(\boldsymbol{y}|\boldsymbol{x})})d\boldsymbol{y} p(\boldsymbol{x}|\boldsymbol{\theta}) d\boldsymbol{x} \pi_{v_{0}}(d\boldsymbol{\theta})\\ 
    & = \int\int\int p(\boldsymbol{y}|\boldsymbol{\theta}) \log(\frac{\hat{p}_{\pi_{v_{0}}}(\boldsymbol{y}|\boldsymbol{x})}{\tilde{p}(\boldsymbol{y}|\boldsymbol{x})}) d\boldsymbol{y} p(\boldsymbol{x}|\boldsymbol{\theta}) d\boldsymbol{x} \pi_{v_{0}}(d\boldsymbol{\theta}). 
\end{align*}
Now we will show that 
\begin{equation*}
    \int\pi_{v_{0}}(d\boldsymbol{\theta}) \int p(\boldsymbol{x}|\boldsymbol{\theta}) \int p(\boldsymbol{y}|\boldsymbol{\theta}) |\log(\frac{\hat{p}_{\pi_{v_{0}}}(\boldsymbol{y}|\boldsymbol{x})}{\tilde{p}(\boldsymbol{y}|\boldsymbol{x})})| d\boldsymbol{y}d\boldsymbol{x} < \infty. 
\end{equation*}
By the triangle inequality it suffices to show 
\begin{align*}
    \int \pi_{v_{0}}(d \boldsymbol{\theta}) \int p(\boldsymbol{x}|\boldsymbol{\theta}) \int &p(\boldsymbol{y}|\boldsymbol{\theta}) |\log(\hat{p}_{\pi_{v_{0}}}(\boldsymbol{y}|\boldsymbol{x}))|d\boldsymbol{y}d\boldsymbol{x} < \infty\\ 
    &\text{and }\\ 
    \int \pi_{v_{0}} (d\boldsymbol{\theta}) \int p(\boldsymbol{x}|\boldsymbol{\theta}) \int &p(\boldsymbol{y}|\boldsymbol{\theta})|\log(\tilde{p}(\boldsymbol{y}|\boldsymbol{x}))|d\boldsymbol{y}d\boldsymbol{x} < \infty. 
\end{align*}
Let $(\log(t))^{+} := \max \{ \log(t),0 \}$ and $(\log(t))^{-}:= \max\{ -\log(t),0 \} $ with \\$\log(t) = (\log(t))^{+} - (\log(t))^{-}$. Then for any density function \\$q(\boldsymbol{y}|\boldsymbol{x}), (\log(q(\boldsymbol{y}|\boldsymbol{x}))) ^{+} \leq q(\boldsymbol{y}|\boldsymbol{x}).$ Therefore, 
\begin{equation*}
    \int p(\boldsymbol{y}|\boldsymbol{\theta}) (\log(q(\boldsymbol{y}|\boldsymbol{x})))^{+}d\boldsymbol{y} \leq (2\pi v_{y})^{-\frac{p}{2}} \int q(\boldsymbol{y}|\boldsymbol{x})d\boldsymbol{y} = (2\pi v_{y}) ^{-\frac{p}{2}}. 
\end{equation*}
Then, 
\begin{equation*}
    \int \pi_{v_{0}}(d\boldsymbol{\theta}) \int p(\boldsymbol{x}|\boldsymbol{\theta})\int p(\boldsymbol{y}|\boldsymbol{\theta})(\log(q(\boldsymbol{y}|\boldsymbol{x})))^{+} d\boldsymbol{y}d\boldsymbol{x} \leq N(2 \pi v_{y}) ^{-\frac{p}{2}} < \infty, 
\end{equation*} 
which holds for both $q = \hat{p}_{\pi_{v_{0}}}$ and $q = \tilde{p}$. Observe $(\log(q(\boldsymbol{y}|\boldsymbol{x})))^{-} = (\log(q(\boldsymbol{y}|\boldsymbol{x})))^{+} - \log(q(\boldsymbol{y}|\boldsymbol{x}))$. 
Therefore, 
\begin{align*}
    &\int p(\boldsymbol{x}|\boldsymbol{\theta}) \int p(\boldsymbol{y}|\boldsymbol{\theta})(\log(q(\boldsymbol{y}|\boldsymbol{x})))^{-}d\boldsymbol{y}d\boldsymbol{x}\\ &= \int p(\boldsymbol{x}|\boldsymbol{\theta}) \int p(\boldsymbol{y}|\boldsymbol{\theta}) (\log(q(\boldsymbol{y}|\boldsymbol{x})))^{+}d\boldsymbol{y}d\boldsymbol{x}
    - \int p(\boldsymbol{x}|\boldsymbol{\theta}) \int p(\boldsymbol{y}|\boldsymbol{\theta}) (\log(q(\boldsymbol{y}|\boldsymbol{x})))d\boldsymbol{y}d\boldsymbol{x}\\ 
    &\leq (2\pi v_{y})^{-\frac{p}{2}}+ \int p(\boldsymbol{x}|\boldsymbol{\theta}) [\int p(\boldsymbol{y}|\boldsymbol{\theta})\log(\frac{p(\boldsymbol{y}|\boldsymbol{\theta})}{q(\boldsymbol{y}|\boldsymbol{x})})d\boldsymbol{y}- \int p(\boldsymbol{y}|\boldsymbol{\theta}) \log(p(\boldsymbol{y}|\boldsymbol{\theta}))d\boldsymbol{y}]d\boldsymbol{x}\\ 
    & = (2\pi v_{y})^{-\frac{p}{2}} + R_{KL}(\boldsymbol{\theta}, q)- \int p(\boldsymbol{y}|\boldsymbol{\theta})\log(p(\boldsymbol{y}|\boldsymbol{\theta}))d\boldsymbol{y} \\ 
    & = (2\pi v_{y})^{-\frac{p}{2}} + R_{KL}(\boldsymbol{\theta}, q) + \frac{p}{2}\log(2 \pi v_{y}) + \frac{1}{2v_{y}} \int p(\boldsymbol{y}|\boldsymbol{\theta})||\boldsymbol{y}\boldsymbol{- \theta}||^{2}d\boldsymbol{y}\\ 
    & = (2\pi v_{y})^{- \frac{p}{2}} + R_{KL}(\boldsymbol{\theta}, q) + \frac{p}{2}\log(2 \pi v_{y})+ \frac{1}{2v_{y}} \mathbb{E}[||\boldsymbol{Y}- \boldsymbol{\theta}||^{2}]\\ 
    & = (2\pi v_{y})^{- \frac{p}{2}} + R_{KL}(\boldsymbol{\theta}, q) + \frac{p}{2}\log(2 \pi v_{y}) + \frac{1}{2v_{y}}tr(v_{y}I_{p})\\ 
    & =  (2\pi v_{y})^{-\frac{p}{2}} + R_{KL}(\boldsymbol{\theta}, q) + \frac{p}{2}\log(2\pi v_{y}) + \frac{p}{2}\\ 
    & = (2\pi v_{y})^{-\frac{p}{2}} + R_{KL}(\boldsymbol{\theta}, q) + \frac{p}{2}[\log(2\pi v_{y}) + \log(e)] \\ 
    & = (2\pi v_{y})^{-\frac{p}{2}} + R_{KL}(\boldsymbol{\theta}, q) + \frac{p}{2}\log(2 \pi e v_{y}). 
\end{align*}
Thus, 
\begin{align*}
    &\int \pi_{v_{0}}(d\boldsymbol{\theta}) \int p(\boldsymbol{x}|\boldsymbol{\theta}) \int p(\boldsymbol{y}|\boldsymbol{\theta}) (\log(q(\boldsymbol{y}|\boldsymbol{x})))^{-} d\boldsymbol{y}d\boldsymbol{x}\\ &\leq N(2\pi v_{y})^{-\frac{p}{2}} + \int R_{KL}(\boldsymbol{\theta}, q) \pi_{v_{0}} (d\boldsymbol{\theta}) + N\frac{p}{2}\log(2 \pi ev_{y})\\ 
    & = N(2 \pi v_{y})^{- \frac{p}{2}} + B_{KL}(\pi_{v_{0}}, q) + N\frac{p}{2} \log (2 \pi e v_{y}) \\ 
    & < \infty,
\end{align*}
since $B_{KL}(\pi_{v_{0}}, \hat{p}_{\pi_{v_{0}}}) < \infty$ by Theorem 4.1 and $B_{KL} (\pi_{v_{0}}, \tilde{p}) \leq B_{KL}(\pi_{v_{0}}, \hat{p}_{\pi_{v_{0}}}) $ by assumption. Since both positive and negative parts are finite and $|\log(t)| = (\log(t))^{+} + (\log(t))^{-}$, 
\begin{equation*}
    \int \pi_{v_{0}} (d\boldsymbol{\theta}) \int p(\boldsymbol{x}|\boldsymbol{\theta}) \int p(\boldsymbol{y}|\boldsymbol{\theta}) |\log(q(\boldsymbol{y}|\boldsymbol{x}))|d\boldsymbol{y}d\boldsymbol{x} < \infty. 
\end{equation*}
Thus, 
\begin{equation*}
    \int \pi_{v_{0}} (d\boldsymbol{\theta}) \int p(\boldsymbol{x}|\boldsymbol{\theta}) \int p(\boldsymbol{y}|\boldsymbol{\theta})|\log(\frac{\hat{p}_{\pi_{v_{0}}}(\boldsymbol{y}|\boldsymbol{x})}{\tilde{p}(\boldsymbol{y}|\boldsymbol{x})})| d\boldsymbol{y}d\boldsymbol{x} < \infty. 
\end{equation*}
Then by Fubini's theorem \begin{equation*}
    0 = \int [\int p(\boldsymbol{y}|\boldsymbol{\theta})p(\boldsymbol{x}|\boldsymbol{\theta})\pi_{v_{0}}(d\boldsymbol{\theta})]\log(\frac{\hat{p}_{\pi_{v_{0}}}(\boldsymbol{y}|\boldsymbol{x})}{\tilde{p}(\boldsymbol{y}|\boldsymbol{x})}) d\boldsymbol{y}d\boldsymbol{x}. 
\end{equation*}
Observe that \begin{align*}
    \hat{p}_{\pi_{v_{0}}}(\boldsymbol{y}|\boldsymbol{x}) & = \int p(\boldsymbol{y}|\boldsymbol{\theta}) \pi_{v_{0}}(\boldsymbol{\theta} |\boldsymbol{x}) d\boldsymbol{\theta} \\ 
    & = \int p(\boldsymbol{y}|\boldsymbol{\theta})\frac{p(\boldsymbol{x}|\boldsymbol{\theta})\pi_{v_{0}}(\boldsymbol{\theta})}{m_{\pi_{v_{0}}}(\boldsymbol{x})} d\boldsymbol{\theta}\\ 
    & = \frac{1}{m_{\pi_{v_{0}}}(\boldsymbol{x})} \int p(\boldsymbol{y}|\boldsymbol{\theta})p(\boldsymbol{x}|\boldsymbol{\theta}) \pi_{v_{0}}(d \boldsymbol{\theta}). 
\end{align*}
Then, \begin{equation*}
    \int p(\boldsymbol{y}|\boldsymbol{\theta}) p(\boldsymbol{x}|\boldsymbol{\theta}) \pi_{v_{0}}(d\boldsymbol{\theta}) = m_{\pi_{v_{0}}}(\boldsymbol{x})\hat{p}_{\pi_{v_{0}}}(\boldsymbol{y}|\boldsymbol{x}). 
\end{equation*}
So, \begin{align*}
    0 &= \int \int m_{\pi_{v_{0}}}(\boldsymbol{x})\hat{p}_{\pi_{v_{0}}}(\boldsymbol{y}|\boldsymbol{x}) \log(\frac{\hat{p}_{\pi_{v_{0}}}(\boldsymbol{y}|\boldsymbol{x})}{\tilde{p}(\boldsymbol{y}|\boldsymbol{x})})d\boldsymbol{y}d\boldsymbol{x}\\ 
    & = \int m_{\pi_{v_{0}}}(\boldsymbol{x}) [\int \hat{p}_{\pi_{v_{0}}}(\boldsymbol{y}|\boldsymbol{x})\log(\frac{\hat{p}_{\pi_{v_{0}}}(\boldsymbol{y}|\boldsymbol{x})}{\tilde{p}(\boldsymbol{y}|\boldsymbol{x})})d\boldsymbol{y}]d\boldsymbol{x}\\ 
    & = \int m_{\pi_{v_{0}}}(\boldsymbol{x}) KL(\hat{p}_{\pi_{v_{0}}}, \tilde{p})d\boldsymbol{x}. 
\end{align*}
Since the integrand is non-negative, it follows that 
\begin{equation*}
    KL(\hat{p}_{\pi_{v_{0}}}(\boldsymbol{y}|\boldsymbol{x}) , \tilde{p}(\boldsymbol{y}|\boldsymbol{x})) = 0 \text{ for } m_{\pi_{v_{0}}}-a.e. \boldsymbol{x}. 
\end{equation*}
Hence, for $m_{\pi_{v_{0}}}-a.e. \boldsymbol{x}$,
\begin{equation*}
    \tilde{p}(\boldsymbol{y}|\boldsymbol{x}) = \hat{p}_{\pi_{v_{0}}}(\boldsymbol{y}|\boldsymbol{x}), a.e. \boldsymbol{y}. 
\end{equation*}
Note that $m_{\pi_{v_{0}}}(\boldsymbol{x}) > 0, \forall \boldsymbol{x} \in \mathbb{R}^{p}$ because $p(\boldsymbol{x}|\boldsymbol{\theta}) >0$ and $\pi_{v_{0}} (\boldsymbol{\theta}) >0$. Therefore, $m_{\pi_{v_{0}}}-a.e.\boldsymbol{x}$ implies Lebesgue-a.e.$\boldsymbol{x}$. Since $p(\boldsymbol{x}|\boldsymbol{\theta}) d\boldsymbol{x} = p(d\boldsymbol{x}|\boldsymbol{\theta})$ is absolutely continuous with respect to the Lebesgue measure, for every $\boldsymbol{\theta}$, 
\begin{equation*}
    \tilde{p}(\boldsymbol{y}|\boldsymbol{x}) = \hat{p}_{\pi_{v_{0}}}(\boldsymbol{y}|\boldsymbol{x}) a.e. \text{ in } \boldsymbol{y}, p(d\boldsymbol{x}|\boldsymbol{\theta}) -a.e. \boldsymbol{x}. 
\end{equation*}
Therefore, for every $\boldsymbol{\theta},$
\begin{equation*}
    R_{KL}(\boldsymbol{\theta}, \tilde{p}) = R_{KL}(\boldsymbol{\theta}, \hat{p}_{\pi_{v_{0}}}). 
\end{equation*}
This contradicts strict domination which requires strict inequality for at least one $\boldsymbol{\theta}$. Therefore, no such $\tilde{p}$ exists and $\hat{p}_{\pi_{v_{0}}}$ is admissible. 
\end{proof}
\bibliographystyle{plainnat}
\bibliography{bibliography}       

\end{document}